\documentclass[11pt,a4paper]{article}

\usepackage[top=1.3in, bottom=1.3in, left=1.3in, right=1.3in]{geometry}
\usepackage{parskip} 
\usepackage{enumitem} 
\setlist{topsep=3pt, itemsep=3pt, parsep=0pt}
\usepackage{titlesec}

\usepackage{setspace}
\titlespacing{\section}{0pt}{*1.5}{*0.8}
\titlespacing{\subsection}{0pt}{*1.2}{*0.5}

\usepackage[hidelinks]{hyperref}

\newcommand{\HypOne}{\textbf{Hyp 1}}
\newcommand{\HypTwo}{\textbf{Hyp 2}}
\newcommand{\Eell}{\mathcal{E}_{\mathrm{ell}}}

\usepackage{amsmath}
\usepackage{amsfonts}
\usepackage{amssymb}
\usepackage{amsthm}
\usepackage{mathptmx} 
\DeclareSymbolFont{cmlargesymbols}{OMX}{cmex}{m}{n}
\let\coprod\relax
\DeclareMathSymbol{\coprod}{\mathop}{cmlargesymbols}{"60}
\DeclareMathOperator{\inv}{inv}
\DeclareMathOperator{\obs}{obs}

\usepackage{stmaryrd}
\usepackage{amscd}
\usepackage{enumitem}
\usepackage[dvipsnames]{xcolor}
\usepackage{tikz-cd}
\usepackage{tikz}
\usetikzlibrary{cd}
\usepackage{mathrsfs}

\theoremstyle{definition}
\newtheorem{shared}{dummy}[subsection]
\newtheorem{D}[shared]{Definition}
\newtheorem{const}[shared]{Construction}

\theoremstyle{plain}

\newtheorem{Le}[shared]{Lemma}
\newtheorem{C}[shared]{Corollary}

\theoremstyle{remark}
\newtheorem{R}[shared]{Remark}
\newtheorem{Rc}[shared]{Complement}

\theoremstyle{plain}
\newtheorem*{Q}{Question}
\newtheorem*{LemA}{Lemma A}
\newtheorem*{LemB}{Lemma B}
\newtheorem*{T}{Theorem}

\DeclareMathOperator{\vol}{vol}
\DeclareMathOperator{\Aut}{Aut}
\DeclareMathOperator{\Frob}{Frob}
\DeclareMathOperator{\tr}{tr}
\DeclareMathOperator{\ST}{ST}
\DeclareMathOperator{\Isoc}{Isoc}
\DeclareMathOperator{\Rep}{Rep}
\DeclareMathOperator{\Vect}{Vec}
\DeclareMathOperator{\Hom}{Hom}
\DeclareMathOperator{\GL}{GL}
\DeclareMathOperator{\val}{val}
\DeclareMathOperator{\pr}{pr}
\DeclareMathOperator{\im}{im}
\DeclareMathOperator{\Lie}{Lie}
\DeclareMathOperator{\Out}{Out}
\DeclareMathOperator{\Int}{Int}
\DeclareMathOperator{\Ad}{Ad}
\DeclareMathOperator{\coker}{coker}
\DeclareMathOperator{\Br}{Br}
\DeclareMathOperator{\SL}{SL}
\DeclareMathOperator{\Gal}{Gal}
\DeclareMathOperator{\Res}{Res}
\DeclareMathOperator{\Ind}{Ind}
\DeclareMathOperator{\id}{id}
\DeclareMathOperator{\charpoly}{charpoly}
\DeclareMathOperator{\roots}{roots}
\DeclareMathOperator{\Spec}{Spec}
\DeclareMathOperator{\semisimplepart}{sspart}
\DeclareMathOperator{\GSp}{GSp}
\DeclareMathOperator{\Cl}{Cl}
\DeclareMathOperator{\Or}{O}
\DeclareMathOperator{\TO}{TO}
\DeclareMathOperator{\SO}{SO}

\title{Stable trace formula for Newton strata of Shimura varieties}
\author{Dhruva Kelkar}
\date{}

\begin{document}

\maketitle

\begin{abstract}
    \noindent Let $(G,X)$ be a Shimura datum of abelian type satisfying the hypotheses of the main theorem, and suppose that the associated Shimura varieties have hyperspecial good reduction at $p$. Their special fibres are stratified by the $\sigma$-conjugacy classes $b\in B(G_{\mathbb{Q}_p},\mu_h^{-1})$. For an individual Newton stratum, this paper constructs a stabilized formula for the alternating traces of Frobenius--Hecke correspondences on its compactly supported $\ell$-adic cohomology, with coefficients in an $\ell$-adic local system. The formula, obtained by modifying the Langlands--Kottwitz method, expresses these Lefschetz numbers as elliptic stable geometric distributions on endoscopic groups of $G$, placing the cohomological traces in a form suitable for comparison with automorphic spectral data. We also give a group-theoretic criterion determining which elliptic endoscopic groups can contribute to a given Newton stratum. For certain unitary Shimura varieties, we exhibit an intermediate Newton stratum with no non-trivial endoscopic contribution, although there are non-trivial endoscopic contributions in the cohomology of the ambient Shimura variety.
\end{abstract}

\setcounter{tocdepth}{2}
\tableofcontents

\section{Introduction}
\subsection{General context and motivation}\label{informalintro}
The Langlands program seeks deep connections between representations of Galois groups of number fields and automorphic representations of reductive algebraic groups. Shimura varieties offer a natural framework for exploring this relationship: their cohomology groups carry commuting actions of the Galois group of the reflex field and a Hecke algebra, and their decompositions reflect both Galois and automorphic representations. Many important Shimura varieties also admit moduli interpretations, making them central objects in algebraic geometry. This paper develops new methods for investigating the mod $p$ geometry of Shimura varieties using the theory of automorphic forms. Specifically, it studies the Newton stratification on the special fibre at a prime of good reduction and constructs a stable trace formula for the Lefschetz numbers of Frobenius--Hecke correspondences on an individual Newton stratum. The formula, obtained by modifying the Langlands--Kottwitz method, expresses these alternating traces as stable geometric distributions on elliptic endoscopic groups. A key feature is a group-theoretic criterion determining which elliptic endoscopic groups can contribute to a given Newton stratum.

Over $\mathbb{C}$, Shimura varieties are finite disjoint unions of arithmetic quotients of Hermitian symmetric domains, and they admit canonical models over number fields \cite{MR546620}. A classical example is the Siegel modular variety $\mathcal{A}_g$, which, with a suitable level structure, is the moduli space of $g$-dimensional principally polarised abelian varieties; the associated group is $\GSp_{2g}$. At a prime of good reduction, the Newton stratification on its special fibre decomposes the variety according to the Newton polygon, which records the slopes of the $p$-divisible group, or equivalently of the $F$-isocrystal, attached to an abelian variety. The geometry of the Newton stratification has been widely studied in this setting, starting with the work of Oort \cite{MR1722536}.

For the local group $G_{\mathbb{Q}_p}$, Kottwitz \cite{MR809866} introduced the set $B(G_{\mathbb{Q}_p})$ of isocrystals with $G_{\mathbb{Q}_p}$-structure (see Definition \ref{B(G)def}), and Rapoport--Richartz \cite{MR1411570} developed the corresponding Newton stratification for $F$-isocrystals with additional structure. This provides the group-theoretic parameter set for the Newton strata on the special fibre of an integral model of a Shimura variety associated with a Shimura datum $(G,X)$. The possible strata are indexed by the finite, partially ordered subset $B(G_{\mathbb{Q}_p},\mu_h^{-1})\subset B(G_{\mathbb{Q}_p})$, where $\mu_h$ is the minuscule cocharacter associated with the Shimura datum; see Remark \ref{BGmudef}. Group-theoretic methods have been used to study geometric questions about Newton strata in this setting; see \cite{MR4439210} for a survey. In particular, questions about non-emptiness
of the strata, the dimensions of the strata and closure relations between strata have been extensively studied, for
example in \cite{MR1781927}, \cite{MR3430453}, \cite{MR4484214}. The paper concerns the following fundamental question:

\begin{Q}
Is it possible to compute the Lefschetz numbers of Frobenius--Hecke correspondences on the compactly supported cohomology of a Newton stratum in terms of automorphic representations?
\end{Q}

Such an automorphic description can, in favourable cases, yield geometric information about Newton strata, such as their dimension or number of irreducible components. Kret \cite{MR3048605} obtained an automorphic description of these Lefschetz numbers, and hence of the zeta function, for the basic stratum of certain Kottwitz varieties at split primes of good reduction. Liu subsequently treated certain intermediate strata at split primes \cite{liu2025cohomologycertainintermediatestrata} and the basic stratum at inert primes \cite{liu2024automorphicdescriptionzetafunction}, again in cases of good reduction. Kottwitz varieties form a restrictive class of Shimura varieties, and the stabilized trace formulas arising in these results have no non-trivial endoscopic terms. This paper explores the question beyond that setting, when non-trivial endoscopic contributions occur.

The approach compares two trace formulas: (1) the Lefschetz--Verdier trace formula, which relates the alternating trace of a Frobenius--Hecke correspondence on compactly supported cohomology to its fixed points, and (2) the Arthur--Selberg trace formula, which relates automorphic representations to orbital integrals. This is an adaptation of the Langlands--Kottwitz method, originally developed to describe the zeta functions of Shimura varieties in terms of automorphic representations. A key step is to stabilize the fixed-point expression so that it can be compared with stable trace formulas on endoscopic groups. The main result of this paper carries out this stabilization for an individual Newton stratum, under the technical assumptions stated in Section \ref{intro2}, by adapting the stabilization for whole Shimura varieties developed by Kottwitz \cite{MR1044820} and, in greater generality, by Kisin--Shin--Zhu \cite{kisin2021stabletraceformulashimura}. It thereby puts the Newton-stratum Lefschetz numbers into the stable form required for comparison with automorphic spectral data.

\subsection{Statement of main result}\label{intro2}
\begin{T}\label{maintheorem}
    With notation as in Section \ref{details}, so that in particular $p$ is good for $G$, the level is $K=K^pK_p$ with $K_p=\mathcal{G}(\mathbb{Z}_p)$ hyperspecial and $K^p$ sufficiently small, and $b\in B(G_{\mathbb{Q}_p},\mu_h^{-1})$, $f^p\in\mathcal{H}(G(\mathbb{A}_f^p)//K^p)$ and a local system $\mathcal{F}_{\xi}$ as above are fixed, let $(G,X)$ be a Shimura datum of abelian type satisfying the conditions:
    \begin{enumerate}
        \item $G^{\text{der}}$ is simply connected
        \item The maximal $\mathbb{Q}$-split torus is equal to the maximal $\mathbb{R}$-split torus in the center $Z(G)$
    \end{enumerate}
    Then for sufficiently large $n$,
    \begin{equation*}
		 \sum_i (-1)^i \tr(\Frob_q^n \times f^p \mid H^i_{c,\mathrm{et}}(\overline{S}_{K^p}(b)_{\overline{\mathbb{F}}_q}, \mathcal{F}_{\xi})) = \sum_{\mathcal{H} \in \Eell(G)_b} \iota (G,H) \ST_e^H(h^p \cdot h_p \chi_{b,n}^{\mathcal{H}} \cdot h_{\infty})
	\end{equation*}
    where $\ST_e^H(-)$ denotes the elliptic part of the geometric side of the stable trace formula for $H$, taken with the central character datum determined by $\xi$, and $h = h^p h_p h_{\infty}$ is the function on $H(\mathbb{A})$ which appears in the stable trace formula for Shimura varieties (Theorem 7.2 of \cite{MR1044820}). The factor $h_{\infty}$ depends on $\xi$; the other two do not.
\end{T}

The local system affects the formula only at the archimedean place, through the function $h_{\infty}$, whereas the passage to a single Newton stratum is a condition at $p$. The two modifications are therefore independent, and allowing a general $\xi$ costs no extra work. We nevertheless carry out the proof in this generality, so that the theorem can be applied with a non-constant local system.

The second hypothesis is a simplifying assumption in Kottwitz's form of the stabilization. Writing $A_G$ for the maximal $\mathbb{Q}$-split torus in $Z(G)$, it ensures the finiteness of the volume factors $c_{\infty}$ and $c_1(\gamma_0,\gamma,\delta)$ occurring in the Lefschetz formula; see Appendix \ref{hyp2appendix}. Kisin--Shin--Zhu remove this hypothesis in their more general treatment.
Due to the technical nature of the result, we define all the terms in the formula precisely only in Section \ref{details}. We first discuss it informally to motivate the result and explain its significance. The left-hand side is the Lefschetz number of the correspondence $\Frob_q^n \times f^p$ on the compactly supported cohomology of the Newton stratum $\overline{S}_{K^p}(b)$, where $\Frob_q$ denotes Frobenius and $f^p$ a Hecke operator, as in Section \ref{informalintro}. The right-hand side consists of terms on the geometric sides of stable Arthur--Selberg trace formulas for elliptic endoscopic groups of $G$; these are the stable distributions that can subsequently be compared with automorphic spectral data. In the rest of this section, we explain the Arthur--Selberg trace formula, its stabilization, and the endoscopic groups that occur.

To explain the Arthur-Selberg trace formula for a connected reductive group $G$ over a number field $F$, it is easiest to start with the case where the group $G$ satisfies the property that $G / A_G$ is anisotropic, where $A_G$ denotes the maximal $F$-split torus in the center of $G$. Note that a group is said to be anisotropic if it does not contain any proper parabolic subgroups defined over $F$. The group $G$ associated with Kottwitz varieties satisfies this condition.

Let $\mathbb{A}_F$ denote the ring of adeles of $F$. Equip $G(\mathbb{A}_F)$ with a Haar measure and the subgroup $G(F) \subset G(\mathbb{A}_F)$ with the counting measure. The starting point for the theory of automorphic representations is the space $L^2(G(F) \backslash G(\mathbb{A}_F))$, which has an action of $G(\mathbb{A}_F)$ by right translation. Instead of studying this representation directly, one often fixes a unitary character $\chi$ of $A_G^+$ and works with the fixed-central-character Hilbert space $L^2_{\chi}(G(F) \backslash G(\mathbb{A}_F))$. Here $A_G^+ \subset \prod_{v \mid \infty} A_G(F_v)$ is a certain subgroup. The precise definitions are recalled in Appendix \ref{hyp2appendix}.

Under the assumption that $G/ A_G$ is anisotropic, there are two important consequences which result in a significant simplification of the Arthur-Selberg trace formula. The first is that the $G(\mathbb{A}_F)$ representation $L^2_{\chi}(G(F) \backslash G(\mathbb{A}_F))$ has a purely discrete spectrum i.e. it decomposes as a Hilbert direct sum of irreducible representations
$$ L^2_{\chi}(G(F) \backslash G(\mathbb{A}_F)) = \widehat{\bigoplus}_{\pi} m(\pi) \pi$$
where the sum runs over all irreducible admissible representations $\pi$ of $G(\mathbb{A}_F)$ and $m(\pi)$ is the multiplicity with which the representation $\pi$ of $G(\mathbb{A}_F)$ appears. If $m(\pi) > 0$, then the representation is automorphic.

The second fact is that for a function $f \in \mathcal{C}^{\infty}_c(G(\mathbb{A}_F))$, the operator $$R(f): L^2_{\chi}(G(F) \backslash G(\mathbb{A}_F)) \xrightarrow[]{} L^2_{\chi}(G(F) \backslash G(\mathbb{A}_F))$$ is of trace class where the operator $R(f)$ is defined as $$(R(f)\phi)(x) : = \int_{G(\mathbb{A}_F)} f(y) \phi(xy)dy \text{ , for any } \phi \in L^2_{\chi}(G(F) \backslash G(\mathbb{A}_F)). $$

Using these two assumptions, and suppressing finite component-group factors and measure normalizations, one obtains the schematic identity
$$ \sum_{\pi}  m(\pi) \tr \pi (f) = \sum_{\gamma} \tau (G_{\gamma}^0) \Or_{\gamma}(f).$$
The sum on the left is over irreducible admissible representations $\pi$ of $G(\mathbb{A}_F)$ and the sum on the right is over $G(F)$-conjugacy classes $\gamma$ in $G(F)$. Here $G_{\gamma}^0$ is the identity component of the centralizer of $\gamma$, $\Or_{\gamma}(f)$ is the orbital integral defined in Section \ref{lvformula}, and $\tau(G_{\gamma}^0)$ is the Tamagawa number discussed in Section \ref{prestab} and Remark \ref{Tamagawa}; see also \cite[\S\S 3.5, 5.3]{MR4615820}.

This is the Arthur-Selberg trace formula, in the case when $G/ A_G$ is anisotropic. Both sides of this formula:
\begin{enumerate}
    \item Spectral side: $f \mapsto \sum_{\pi} m(\pi) \tr \pi (f)$
    \item Geometric side: $f \mapsto \sum_{\gamma} \tau(G_{\gamma}^0)\Or_{\gamma}(f)$ 
\end{enumerate}
are invariant distributions for the test function $f$ and the trace formula should be viewed as an identity of invariant distributions. Note that a distribution $d$ is invariant if $d(f) = d (^g f)$ for any test function $f \in \mathcal{C}^{\infty}_c(G(\mathbb{A}_F))$ and $g \in G(\mathbb{A}_F)$, where $^g f(x) = f( g^{-1} x g)$.

The Arthur-Selberg trace formula is significantly more complicated in the case when $G/ A_G$ is not anisotropic. On the spectral side, complications arise due to contributions coming from a continuous spectrum of $L^2_{\chi}(G(F) \backslash G(\mathbb{A}_F))$. On the geometric side, the quotient volumes and orbital integrals occurring before truncation need not be finite. For example, Arthur considers $G=\GL_2$ and the non-semisimple unipotent element
$$
\gamma=
\begin{pmatrix}
1 & 1\\
0 & 1
\end{pmatrix}.
$$
For a non-negative test function $f \in \mathcal{C}_c^{\infty}(G(\mathbb{A}))$, he shows that the adelic orbital integral
$$
\int_{G(\mathbb{A})_{\gamma}\backslash G(\mathbb{A})}
f(x^{-1}\gamma x),dx
$$
reduces to an expression of the form
$$
c(f)\prod_p(1-p^{-1})^{-1}
$$
for a certain positive constant $c(f)$ associated with $f$. This is the divergent Euler product for the Riemann zeta function at $s=1$; see \cite[\S 4]{MR2192011} for more details about this example.

Apart from the difficulties caused by the noncompactness of $G(F)\backslash G(\mathbb{A}_F)$, there is a second important issue in applying and comparing trace formulas, namely the distinction between rational conjugacy and stable conjugacy. Two elements of $G(F)$ are rationally conjugate if they are conjugate by an element of $G(F)$, whereas they are stably conjugate if, roughly speaking, they become conjugate by an element of $G(\overline{F})$; see \ref{StableConjugacy} for the precise definition. Thus, a single stable conjugacy class may contain several distinct $G(F)$-conjugacy classes.

This distinction already plays a prominent role in the work of Jacquet and Langlands, where a comparison between the trace formulas for $\GL_2$ and for the multiplicative groups of quaternion algebras is used to establish the Jacquet--Langlands correspondence; see \cite{MR401654}. More generally, the need to compare geometric terms attached to different groups led to the theory of endoscopy, which forms an important part of the broader framework of Langlands functoriality. The stabilization of the trace formula reorganizes terms indexed by rational conjugacy classes into terms indexed by stable conjugacy classes on $G$ and on its endoscopic groups. The definitions and notation used here are reviewed in Section \ref{2.2}; for further details on endoscopy and its natural extension to the twisted setting, see Kottwitz--Shelstad \cite{MR1687096}.

As a result, the stable trace formula is not merely an identity of invariant distributions, but an identity of \textit{stably} invariant distributions. Its construction therefore requires the stabilization of both the geometric and the spectral sides of the Arthur--Selberg trace formula. It was subsequently realized that the same issue arises in the Langlands--Kottwitz method for studying the cohomology of Shimura varieties. Indeed, the comparison between the Lefschetz--Verdier trace formula and the Arthur--Selberg trace formula carried out by Kottwitz in \cite{MR1044820} also requires a stabilized trace formula. A principal objective of this paper is to address the corresponding stabilization problem in the new setting of an individual Newton stratum, rather than for the entire Shimura variety.

We will now discuss in more detail, the stabilization of the geometric side, as this is what appears in the main theorem. The stabilization of the spectral side in the most general case is conditional on Arthur's conjectures (see $\S$12 of \cite{MR0757954} and Part II of \cite{MR1044820})).

While the geometric side of the Arthur-Selberg trace formula is quite complicated, there are some terms in it, known as the elliptic terms that are more tractable. We denote the elliptic part of the trace formula by the distribution $$ f \mapsto T_e(f) \text{ , where } T_e(f) = \sum_{\gamma} |(G_{\gamma}/G_{\gamma}^0)(F)|^{-1}\tau(G_{\gamma}^0) \Or_{\gamma}(f) $$
where now the sum is taken only over elliptic conjugacy classes $\gamma$ in $G(F)$ (see \ref{ellipticdef}). Note that the if $G/ A_G$ is anisotropic, then every conjugacy class in $G(F)$ is elliptic and hence the elliptic part corresponds to the entire geometric side of the trace formula in this case.

The stabilization in the case relevant to this paper takes the form (see \cite[\S 9.1]{kisin2021stabletraceformulashimura})
$$ T_e(f) = \sum_{\mathcal{H}=(H,s,\eta) \in \Eell(G)} \iota(G,H) \ST_e^H(f^H).$$
The sum is over the set $\Eell(G)$ of equivalence classes of elliptic endoscopic data $\mathcal{H}=(H,s,\eta)$ for $G$, and includes the trivial endoscopic datum, whose endoscopic group is the quasi-split inner form of $G$. For example, $\GL_n$ has no non-trivial elliptic endoscopy, whereas split $\GSp_4$ has the non-trivial elliptic endoscopic group $(\GL_2 \times \GL_2)/\mathbb{G}_m$. For a group $H$, the distribution $$ f_0 \mapsto \ST_e^H(f_0) \text{ , where } \ST_e^H(f_0) = \sum_{\gamma_H} | (H_{\gamma_H}/H_{\gamma_H}^0)(F) |^{-1} \cdot \tau(H) \cdot \SO_{\gamma_H}(f_0)$$ is a stably invariant distribution, where
\begin{itemize}
    \item The sum is over elliptic stable conjugacy classes $\gamma_H \in H(F)$ (see \ref{StableConjugacy})
    \item $(H_{\gamma_H}/H_{\gamma_H}^0)(F)$ is the group of $F$-points of the group of connected components of the centralizer of $\gamma_H$ in $H$; its order depends only on the stable conjugacy class of $\gamma_H$ and not on the function $f_0$.
    \item $\tau(H)$ is the Tamagawa number of $H$ (a volume term)
    \item $\SO_{\gamma_H}(f_0)$ denotes the stable orbital integral of $f_0$; see Section \ref{stab}. The definition for a general semisimple element, whose centralizer may be disconnected, is recalled in Complement \ref{SOgeneral}.
    \item $\iota(G,H)$ is a positive rational constant attached to the pair $(G,H)$.
    \item The function $f^H \in \mathcal{C}^{\infty}_c(H(\mathbb{A}_F))$ is an endoscopic transfer of $f \in \mathcal{C}^{\infty}_c(G(\mathbb{A}_F))$. It is characterized by identities involving its stable orbital integrals and is therefore determined only up to its stable orbital integrals; see Sections 5--6 of \cite{MR858284} and Section 7.4 of \cite{kisin2021stabletraceformulashimura}.
\end{itemize}

Hence, the stabilization of the elliptic part of the trace formula rewrites the invariant distribution $T_e$ in terms of stably invariant distributions $\ST_e^H$ on elliptic endoscopic groups of $G$. The transfer results, the fundamental lemma, and the required results on Tamagawa numbers entering this statement are now theorems; see \cite{MR942522,MR2653248,MR3051198}. For the modern formulation in the setting of Shimura varieties, see \cite[\S 9.1]{kisin2021stabletraceformulashimura}; Kottwitz's work \cite{MR858284} provides the original framework for the elliptic stabilization.

Associated to the identity connected component $I = G_{\gamma}^0$ of a semisimple $\gamma \in G(F)$, there is an abelian group $\mathfrak{K}(I) = \mathfrak{K}( G_{\gamma}^0)$ (see \ref{kappadef}) which controls the endoscopic contributions in the stabilization of the trace formula (in the sense of \ref{keylemma}). In the Kottwitz-variety cases treated by Kret \cite{MR3048605} and Liu \cite{liu2025cohomologycertainintermediatestrata,liu2024automorphicdescriptionzetafunction}, the relevant groups $\mathfrak{K}(G_{\gamma}^0)$ are trivial. Consequently, only the trivial endoscopic datum occurs in the stabilization.

To summarize, going beyond the case of Kottwitz varieties involves dealing with two main challenges:
\begin{enumerate}
    \item The Arthur-Selberg trace formula is more complicated due to contributions from the continuous spectrum on the spectral side and divergent orbital integrals and volumes on the geometric side. However, the elliptic terms on the geometric side in this formula are more tractable.
    \item The stabilization of the trace formula involves non-trivial endoscopic contributions. One has to understand the contribution of these non-trivial elliptic endoscopic groups to the Newton stratum. 
\end{enumerate}

Theorem 7.2 of \cite{MR1044820} gives an expression for the Lefschetz number of the Frobenius-Hecke operator $\Frob_q^n \times f^p$ on the Shimura variety $S_K$ in terms of elliptic parts of the stable trace formulas for elliptic endoscopic groups $\Eell(G)$ of $G$ as follows (under the same hypothesis as our main theorem):
\begin{equation*}
		 \sum_i (-1)^i \tr(\Frob_q^n \times f^p \mid H^i_{c,\mathrm{et}}((S_K)_{\overline{E}}, \overline{\mathbb{Q}}_\ell)) = \sum_{\mathcal{H} = (H,s,\eta) \in \Eell(G)} \iota (G,H) \ST_e^H(h)
\end{equation*}
For every elliptic endoscopic group $H$ of $G$, Kottwitz constructs a smooth function $h=h^ph_ph_{\infty}$ on $H(\mathbb{A})$, where $h^p$ and $h_p$ are compactly supported and $h_{\infty}$ is compactly supported modulo $A_G(\mathbb{R})^0$; the function is well defined only up to its stable orbital integrals. Its local components are recalled in Complement \ref{hconstruction}. Note that the functions $h$ here are NOT endoscopic transfers of a function in $\mathcal{C}^{\infty}_c(G(\mathbb{A}))$ - which is precisely why stably invariant trace formulas are essential for this question.

The main theorem of this paper can be viewed as an analogue of the above result for the Newton stratum corresponding to $b \in B(G_{\mathbb{Q}_p}, \mu_h^{-1})$. The subset $\Eell(G)_b \subset \Eell(G)$ determines which elliptic endoscopic groups of $G$ contribute to the stratum corresponding to $b$. One of the difficult parts of the proof of this result consists of constructing suitable truncation functions $\chi_{b,n}^{\mathcal{H}}$, which isolate the contributions of the endoscopic triples $\mathcal{H} = (H, s , \eta)$ to the Newton stratum $b$, and proving that these functions are smooth.

\subsection{Further details on main result}\label{details}
\subsubsection{Setup}\label{detailssetup}
	    \begin{itemize}
	    \item For every field $F$, fix an algebraic closure $\overline{F}$ and denote by $\Gamma_F=\operatorname{Gal}(\overline{F}/F)$ its absolute Galois group. In the global setting, fix an embedding $\overline{\mathbb{Q}}\hookrightarrow\mathbb{C}$. If $E\subset\overline{\mathbb{Q}}$ is a number field, take $\overline{E}=\overline{\mathbb{Q}}$, so that $\Gamma_E$ is naturally a subgroup of $\Gamma_{\mathbb{Q}}$
	    \item Fix a prime number $\ell \neq p$ and an embedding $\overline{\mathbb{Q}}_\ell\hookrightarrow\mathbb{C}$. We use this embedding throughout to regard $\overline{\mathbb{Q}}_\ell$-valued test functions, together with the orbital integrals and other distributions attached to them, as complex valued. This allows the $\overline{\mathbb{Q}}_\ell$-valued cohomological traces in the Main Theorem to be compared with its complex-valued trace-formula terms
	    \item $\mathbb{A}$, $\mathbb{A}_f$ and $\mathbb{A}_f^p$ denote the rings of adeles, finite adeles and adeles away from $p$ resp. over $\mathbb{Q}$
	\end{itemize}

\textbf{Shimura varieties.}
Let $(G,X)$ be a Shimura datum, for a connected reductive group $G$ over $\mathbb{Q}$ and a $G(\mathbb{R})$-conjugacy class $X$ of homomorphisms $h: R_{\mathbb{C}/\mathbb{R}} \mathbb{G}_m \xrightarrow{} G_{\mathbb{R}}$ (in the sense of Deligne \cite{MR546620}). From this, one obtains:

\begin{enumerate}
    \item A $G(\mathbb{C})$-conjugacy class of cocharacters of $G_{\mathbb{C}}$, with a representative denoted by $\mu_h$ and known as the Hodge cocharacter associated to $h$. This is constructed as follows. After base change to $\mathbb{C}$, there is an isomorphism
    $$
        \left(\Res_{\mathbb{C}/\mathbb{R}}\mathbb{G}_m\right)_{\mathbb{C}}
        \xrightarrow[]{\sim}
        \mathbb{G}_{m,\mathbb{C}}\times\mathbb{G}_{m,\mathbb{C}}
    $$
    defined on $A$-points, for every $\mathbb{C}$-algebra $A$, by
    $$
        \left(\mathbb{C}\otimes_{\mathbb{R}}A\right)^{\times}
        \xrightarrow[]{\sim}
        A^{\times}\times A^{\times},
        \qquad
        \sum_i z_i\otimes a_i
        \longmapsto
        \left(
            \sum_i z_i a_i,
            \sum_i \overline{z_i}a_i
        \right).
    $$
    Thus, the first factor corresponds to the identity embedding
    $\mathbb{C}\hookrightarrow\mathbb{C}$ and the second factor corresponds to complex conjugation. For $h\in X$, the associated Hodge cocharacter $\mu_h$ is
    $$
        \mu_h:\mathbb{G}_{m,\mathbb{C}}
        \xrightarrow[]{}
        G_{\mathbb{C}},
        \qquad
        z\longmapsto h_{\mathbb{C}}(z,1).
    $$
    Its $G(\mathbb{C})$-conjugacy class is independent of the choice of $h\in X$.

    \item A number field $E\subset\overline{\mathbb{Q}}$, called the reflex field of $(G,X)$. Using the fixed embedding $\overline{\mathbb{Q}}\hookrightarrow\mathbb{C}$, the $G(\mathbb{C})$-conjugacy class of $\mu_h$ determines a $G(\overline{\mathbb{Q}})$-conjugacy class. The reflex field is the fixed field of the stabilizer of this conjugacy class; equivalently,
    $$
        \Gamma_E
        =
        \left\{
            \tau\in \Gamma_{\mathbb{Q}}
            \ \middle|\
            {}^{\tau}\mu_h
            \text{ is }G(\overline{\mathbb{Q}})\text{-conjugate to }\mu_h
        \right\}.
    $$
    Here $\Gamma_{\mathbb{Q}}$ acts on $\Hom_{\overline{\mathbb{Q}}}(\mathbb{G}_m,G_{\overline{\mathbb{Q}}})$ by acting on the coefficients.
    \item A smooth quasi-projective variety $S_K$, for a sufficiently small (for example, neat) compact open subgroup $K\subset G(\mathbb{A}_f)$, which is known as the Shimura variety at level $K$. The complex points of this variety have the following group theoretic description:
    $$S_K(\mathbb{C}) = G(\mathbb{Q})\backslash(X \times (G(\mathbb{A}_f) / K)).$$
\end{enumerate}

To get a handle on the double cosets, one uses the fact that $G(\mathbb{Q}) \backslash G(\mathbb{A}_f) / K$ is finite. Choose representatives $g_1, g_2, \cdots, g_r \in G(\mathbb{A}_f)$ and set $\Gamma_i = G(\mathbb{Q}) \cap g_i K g_i^{-1}$. Then the map $$ \coprod_{i=1}^r \Gamma_i \backslash X \xrightarrow{} S_K(\mathbb{C})$$ given by $x \in \Gamma_i \backslash X \mapsto (x, g_i)$ is a bijection. 

Each connected component of $X$ is a Hermitian symmetric domain and each $\Gamma_i$ is
an arithmetic subgroup of $G(\mathbb{Q})$. Applying the Baily--Borel theorem componentwise, the
arithmetic quotient $\Gamma_i\backslash X$ has a natural structure of a
quasi-projective algebraic variety over $\mathbb{C}$ and admits a canonical
normal projective compactification
$$
    \Gamma_i\backslash X
    \lhook\joinrel\longrightarrow
    \left(\Gamma_i\backslash X\right)^{\mathrm{BB}},
$$
called the Baily--Borel compactification; see \cite{MR216035}. Transporting
these algebraic structures through the decomposition
$$
    S_K(\mathbb{C})
    \simeq
    \coprod_{i=1}^{r}\Gamma_i\backslash X
$$
gives $S_K(\mathbb{C})$ the structure of a complex quasi-projective algebraic
variety.

Using Deligne's theory \cite{MR546620}, it follows that $S_K$ has a canonical model defined over $E$ and satisfies the following properties:

\begin{enumerate}
    \item For compact, open subgroups $K' \subset K$ there is a finite étale surjective map $\pi: S_{K'} \xrightarrow[]{} S_K$
    \item For $g \in G(\mathbb{A}_f)$ there is an isomorphism $\psi_g: S_K \xrightarrow[]{\sim} S_{g^{-1}Kg}$ and if $g \in K$ then this isomorphism is the identity.
\end{enumerate}

These maps are constructed as follows. For a sufficiently small
$K\subset G(\mathbb{A}_f)$, we start with the description:
$$
S_K(\mathbb{C})
=
G(\mathbb{Q})\backslash
\left(X\times G(\mathbb{A}_f)/K\right).
$$
We will construct the maps $\pi$ and $\psi_g$ over $\mathbb{C}$; however, they also descend to the canonical model over $E$. Denote the class of $(x,a)\in X\times G(\mathbb{A}_f)$ in this
quotient by $[x,a]_K$. For compact open subgroups $K'\subset K$, the projection map $\pi$ is given by
$$
\pi:S_{K'}(\mathbb{C})\longrightarrow S_K(\mathbb{C}),
\qquad
[x,a]_{K'}\longmapsto [x,a]_K.
$$
For $g\in G(\mathbb{A}_f)$, the isomorphism $\psi_g$ is given by
$$
\psi_g:S_K(\mathbb{C})
\xrightarrow[]{\sim}
S_{g^{-1}Kg}(\mathbb{C}),
\qquad
[x,a]_K\longmapsto [x,ag]_{g^{-1}Kg}.
$$
This map is well defined because, for every $k\in K$,
$$
akg=ag(g^{-1}kg),
$$
and $g^{-1}kg\in g^{-1}Kg$. Its inverse is $\psi_{g^{-1}}$. Moreover,
if $g\in K$, then $g^{-1}Kg=K$ and $\psi_g$ is the identity map on
$S_K(\mathbb{C})$.

\textbf{Hecke algebra.}
The Hecke algebra $\mathcal{H}(G(\mathbb{A}_f)//K)$ consists of compactly supported, smooth, $K$-bi-invariant, $\overline{\mathbb{Q}}_\ell$-valued functions on $G(\mathbb{A}_f)$. There is a canonical topology on $G(\mathbb{A}_f)$ satisfying the property that for any closed immersion of $\mathbb{Q}$-schemes $G \hookrightarrow \mathbb{A}_{\mathbb{Q}}^N$, the topology on $G(\mathbb{A}_f)$ is the subspace topology with respect to the natural product topology on $\mathbb{A}_f^N$. Smooth functions (taking values in $\overline{\mathbb{Q}}_\ell$) on $G(\mathbb{A}_f)$ are locally constant functions with respect to this topology. The group $G(\mathbb{A}_f)$ is unimodular and the algebra structure on $\mathcal{H}(G(\mathbb{A}_f)//K)$ is obtained from convolution of functions with respect to a choice of a Haar measure on $G(\mathbb{A}_f)$.

\textbf{Hecke action.}
The algebra $\mathcal{H}(G(\mathbb{A}_f)//K)$ acts on the compactly supported Betti cohomology groups $H_c^i(S_K(\mathbb{C}),\overline{\mathbb{Q}}_\ell)$. For $g\in G(\mathbb{A}_f)$, put
$$
    T_g=\frac{1}{\vol(K)}\mathbf{1}_{KgK}.
$$
The action of $T_g$ is induced by the Hecke correspondence
$$
\begin{tikzcd}
& S_{gKg^{-1} \cap K} \arrow{dl}{\pi_1} \arrow{r}{\psi_g} \arrow{drr}{\pi_2} & S_{K \cap g^{-1}Kg} \arrow{dr} & \\
S_K &  & & S_K
\end{tikzcd}
$$

Here $\pi_1$ is the level projection and $\pi_2$ is the composite of $\psi_g$ with the level projection to $S_K$. If $d = \dim S_K$, then $(\pi_1, \pi_2)$ defines a codimension $d$ cycle $Z$ on $S_K \times S_K$. Let $p_1,p_2:S_K\times S_K\to S_K$ denote the two projections. The map on the cohomology $H_c^i(S_K(\mathbb{C}), \overline{\mathbb{Q}}_\ell)$ induced by $T_g$ is
$$
    (p_1)_*\bigl([Z]\cup p_2^*(-)\bigr),
$$
where $[Z]$ is the cycle class of $Z$. The pushforward is applied to a class supported on $Z$, and is therefore defined because $p_1|_Z:Z\to S_K$ is proper.

\textbf{Galois action.}
The canonical model over $E$ gives a natural action of $\Gamma_E$ on
$$
    H^i_{c,\mathrm{et}}(S_{K,\overline{E}},\overline{\mathbb{Q}}_\ell).
$$
Through the fixed embedding $\overline{E}=\overline{\mathbb{Q}}\hookrightarrow\mathbb{C}$, the comparison isomorphism identifies this group with $H_c^i(S_K(\mathbb{C}),\overline{\mathbb{Q}}_\ell)$. Since the Hecke correspondences are defined over $E$, they also act on $\ell$-adic cohomology, and the actions of $\Gamma_E$ and $\mathcal{H}(G(\mathbb{A}_f)//K)$ on $H^i_{c,\mathrm{et}}(S_{K,\overline{E}},\overline{\mathbb{Q}}_\ell)$ commute.

\textbf{Coefficient systems.}
Let $\xi$ be an algebraic representation of $G$ on a finite dimensional vector space over a number field contained in $\overline{\mathbb{Q}}_\ell$. It determines an $\ell$-adic local system $\mathcal{F}_{\xi}$ on $S_K$ over $E$, and the Galois and Hecke actions described above are defined on $H^i_{c,\mathrm{et}}(S_{K,\overline{E}},\mathcal{F}_{\xi})$ in the same way; see $\S$ 1 of \cite{MR1044820}. Since $\ell \neq p$, the local system $\mathcal{F}_{\xi}$ extends over the integral model of Section \ref{newtonstratdetails} and restricts to its special fibre, hence to each Newton stratum; see $\S$ 1.8 of \cite{kisin2021stabletraceformulashimura}. Taking $\xi$ trivial returns the constant coefficients $\overline{\mathbb{Q}}_\ell$ used above, and the discussion of Hecke correspondences in this section, as well as the whole of Section \ref{GUexample}, is written with $\xi$ trivial for simplicity.

\subsubsection{Newton stratification and automorphic forms}\label{newtonstratdetails}
Now assume that $(G,X)$ is an abelian type Shimura datum. We say a rational prime $p$ is \textit{good for $G$} if $G_{\mathbb{Q}_p}$ is unramified, that is, if it is quasi-split and split over an unramified extension of $\mathbb{Q}_p$. Choose a reductive group scheme $\mathcal{G}$ over $\mathbb{Z}_p$ with generic fibre $G_{\mathbb{Q}_p}$. Recall that $\ell \neq p$, since below we use $\ell$-adic étale cohomology of the special fibre.

Put $K_{p,0} = \mathcal{G}(\mathbb{Z}_p) \subset G(\mathbb{Q}_p)$, a hyperspecial subgroup. Fix a place $v$ of the reflex field $E$ above $p$. For a sufficiently small compact open $K^p \subset G(\mathbb{A}_f^p)$, we denote by $\mathcal{S}_{K^p}$ the canonical integral model of $S_K$ over $\mathcal{O}_{E_v}$, where $K = K^p K_{p,0}$. Its existence for abelian type Shimura data at hyperspecial level was established by Kisin \cite{MR2669706} for $p>2$, building on Vasiu \cite{MR1796512}, and by Kim and Madapusi Pera \cite{MR3569319} for $p=2$. Let $\mathbb{F}_q$ denote the residue field of $\mathcal{O}_{E_v}$ and $\overline{S}_{K^p}$ denotes the special fibre of $\mathcal{S}_{K^p}$, hence a variety over $\mathbb{F}_q$.

\textbf{Newton stratification.}
Put $L=W(\overline{\mathbb{F}}_p)[1/p]$, and let $\sigma$ denote the Witt-vector Frobenius induced by $x\mapsto x^p$ on $\overline{\mathbb{F}}_p$. Denote by $B(G_{\mathbb{Q}_p})$ the set of $\sigma$-conjugacy classes in $G(L)$, where $b,b'\in G(L)$ are $\sigma$-conjugate if
$$
    b'=g b\sigma(g)^{-1}
$$
for some $g\in G(L)$. Equivalently, $B(G_{\mathbb{Q}_p})$ is the set of isomorphism classes of isocrystals with $G_{\mathbb{Q}_p}$-structure; see \cite[\S 1]{MR1485921} and Definition \ref{B(G)def}. Choose an embedding $\overline{\mathbb{Q}}\hookrightarrow\overline{\mathbb{Q}_p}$ inducing the place $v$. Via this embedding, the $G(\overline{\mathbb{Q}})$-conjugacy class of $\mu_h$ determines a $G(\overline{\mathbb{Q}_p})$-conjugacy class of cocharacters. Kottwitz associates to the inverse of this class a finite subset $B(G_{\mathbb{Q}_p}, \mu_h^{-1}) \subset B(G_{\mathbb{Q}_p})$; see \cite[\S 6]{MR1485921} and Remark \ref{BGmudef}. The construction of the integral model supplies a Newton map
$$
\overline{S}_{K^p}(\overline{\mathbb{F}}_q)\longrightarrow B(G_{\mathbb{Q}_p}),
\qquad x\longmapsto b_x.
$$
For the construction in the Hodge-type case and its passage to the abelian-type case, see \cite[\S 1.3]{MR4484214}, \cite[\S 4.6]{MR3630089}, and \cite[\S\S 5.2, 6.2]{kisin2021stabletraceformulashimura}. Throughout this paper $b_x$ is normalised \textit{covariantly}, as in \cite{MR1044820} and \cite{kisin2021stabletraceformulashimura}: when the Shimura variety is a moduli space of abelian varieties, $b_x$ is the class of the covariant Dieudonn\'e module of $A[p^{\infty}]$. Then $\kappa_G(b_x)$ is the restriction of $-\mu_h^{\natural}$, so the image of the map lies in $B(G_{\mathbb{Q}_p},\mu_h^{-1})$; see Remark \ref{BGmudef} and Convention \ref{newtonnormalisation}.

The strata are constructed in two stages. For $(G,X)$ of Hodge type there is a universal abelian scheme, and Kisin's construction of the integral model attaches to it an $F$-isocrystal with $G_{\mathbb{Q}_p}$-structure over $\overline{S}_{K^p}$; the loci of constant Newton point are then locally closed by the semicontinuity theorem of Rapoport--Richartz \cite[Thm.~3.6]{MR1411570}, which holds for such an object over an arbitrary base of characteristic $p$. For $(G,X)$ of abelian type there is no abelian scheme at all, and the stratification is obtained from the Hodge type case by passing to the adjoint group and an auxiliary Hodge type datum \cite[\S 2.3]{MR4557079}. In either case one obtains, for each $b\in B(G_{\mathbb{Q}_p},\mu_h^{-1})$, a locally closed subscheme $\overline{S}_{K^p}(b)\subset\overline{S}_{K^p}$ defined over $\mathbb{F}_q$ \cite[Thm.~2.2.2, Thm.~2.3.6]{MR4557079}; these are the Newton strata. In particular $\Frob_q$ acts on the cohomology of $\overline{S}_{K^p}(b)_{\overline{\mathbb{F}}_q}$, as Theorem \ref{maintheorem} requires.

Let $\pi$, $\psi_g$ denote the morphisms as before. In the case when $K$, $K'$ and $g$ are of the special forms $K' = K'^pK_{p,0} \subset K = K^p K_{p,0}$ and $g = (g^p, 1) \in G(\mathbb{A}_f^p) \times G(\mathbb{Q}_p)$, these morphisms can be extended to morphisms on the canonical integral models of the corresponding Shimura varieties. Denote by $\overline{\pi}$ and $\overline{\psi_g}$ their base change to the corresponding special fibres. The Newton stratification satisfies the following properties with respect to these maps:

\begin{enumerate}
    \item For compact, open subgroups $K'^p \subset K^p$, $\overline{S}_{K'^p}(b)$ is the pullback of $\overline{S}_{K^p}(b)$ via $\overline{\pi}$
    \item For $g^p \in G(\mathbb{A}_f^p)$ $\overline{S}_{K^p}(b)$ is the pullback of $\overline{S}_{(g^p)^{-1}K^pg^p}(b)$ via $\overline{\psi_g}$.
\end{enumerate}

Hence, the Hecke algebra $\mathcal{H}(G(\mathbb{A}_f^p) // K^p)$ acts on the compactly supported $\ell$-adic cohomology groups $H^i_{c,\mathrm{et}}(\overline{S}_{K^p}(b)_{\overline{\mathbb{F}}_q}, \overline{\mathbb{Q}}_\ell)$. For $g^p\in G(\mathbb{A}_f^p)$, put
$$
    T_{g^p}=\frac{1}{\vol(K^p)}\mathbf{1}_{K^pg^pK^p}.
$$
The action of $T_{g^p}$ is induced by the restriction of the preceding Hecke correspondence to the Newton stratum:
$$
\begin{tikzcd}
\overline{S}_{g^pK^p(g^p)^{-1} \cap K^p}(b) \arrow{d}{\overline{\pi_1}} \arrow{r}{\overline{\psi_g}} \arrow{dr}{\overline{\pi_2}} & \overline{S}_{K^p \cap (g^p)^{-1}K^pg^p}(b) \arrow{d} \\
\overline{S}_{K^p}(b) & \overline{S}_{K^p}(b)
\end{tikzcd}
$$
where $g=(g^p,1)$ and $\overline{\pi_1},\overline{\pi_2}$ are the restrictions of the previously defined maps to the Newton strata.

\textbf{Lefschetz numbers.}
Let $\Frob_q \in \Gamma_{\mathbb{F}_q}$ denote the geometric Frobenius. Fix a Hecke function $f^p \in \mathcal{H}(G(\mathbb{A}_f^p) // K^p)$ and a representation $\xi$ as above. The Lefschetz number $c_b(n, f^p, \xi)$ of the operator $\Frob_q^n \times f^p$ on the compactly supported cohomology of the Newton stratum $\overline{S}_{K^p}(b)$ with coefficients in $\mathcal{F}_{\xi}$ is defined by
$$
c_b(n,f^p,\xi) = \sum_i (-1)^i \tr\left(\Frob_q^n \times f^p \mathrel{\Big|} H^i_{c,\mathrm{et}}\left(\overline{S}_{K^p}(b)_{\overline{\mathbb{F}}_q}, \mathcal{F}_{\xi}\right)\right).
$$

Having phrased things more precisely, we can reformulate our earlier question:

\begin{Q}\label{lknumberquestion}
    Is there a description of the Lefschetz number $c_b(n,f^p,\xi)$ in terms of data associated with automorphic representations?
\end{Q}

\subsubsection{Main result}
\begin{itemize}
    \item $\Hat{G}$ denotes the dual group over $\mathbb{C}$ of a connected reductive group $G$ over a field $F$; see Definition \ref{Lgroupdef}
    \item $\pi_1(G)$ denotes the algebraic fundamental group for a connected reductive group $G$ over a field $F$; see the discussion preceding Remark \ref{observation}
    \item $\mathcal{N}(G)$ denotes the set of Newton points for a connected reductive group $G$ over a $p$-adic field $F$; see Section \ref{newton}
\end{itemize}

For a group $G$ over a $p$-adic field $F$, there are maps $\overline{\nu}_G: B(G) \xrightarrow[]{} \mathcal{N}(G)$, $\kappa_G: B(G) \xrightarrow[]{} \pi_1(G)_{\Gamma_F}$ and $\delta_G: \mathcal{N}(G) \xrightarrow[]{} \pi_1(G)^{\Gamma_F} \otimes \mathbb{Q}$, recalled respectively in Definition \ref{newtonmap}, Lemma \ref{kottwitzmap}, and the construction preceding Lemma \ref{Newonkottwitzcharisoc}, where $\pi_1(G)_{\Gamma_F}$ and $\pi_1(G)^{\Gamma_F}$ denote the coinvariants and invariants respectively under the canonical action of $\Gamma_F$ on $\pi_1(G)$ (see \cite{MR1411570}). The following square commutes by Lemma \ref{Newonkottwitzcharisoc}, where the horizontal arrow in the bottom is the map $\overline{\mu} \mapsto \mid \Gamma_F \cdot \mu \mid ^{-1} \sum_{\mu' \in \Gamma_F \cdot \mu} \mu'$

\begin{center}
\begin{tikzcd}
    B(G) \arrow{r}{\overline{\nu}_G} \arrow{d}{\kappa_G} & \mathcal{N}(G) \arrow{d}{\delta_G} \\
    \pi_1(G)_{\Gamma_F} \arrow{r} & \pi_1(G)^{\Gamma_F} \otimes \mathbb{Q}
\end{tikzcd}
\end{center}

Let $\mathcal{N}(G)_{\mathbb{Z}}$ denote the image of $B(G)$ under the map $\overline{\nu}_G$, which is a discrete subset of $\mathcal{N}(G)$. In the case when $G$ is quasi-split, this set has a description in terms of the root datum of $G$, using $\S 3$ of \cite{MR1781927}.

Let $(H,s,\eta)$ be an endoscopic triple for $G$, for a quasi-split connected reductive group $H$ over $F$, an element $s \in Z(\Hat{H})$ in the center and an embedding $\eta: \Hat{H} \xrightarrow[]{} \Hat{G}$ over $\mathbb{C}$ satisfying the conditions recalled in Lemma \ref{endotriples}; see also $\S 7$ of \cite{MR0757954}.

\begin{LemA}
For an endoscopic triple $\mathcal{H} =(H,s,\eta)$ for a connected reductive group $G$ over a $p$-adic field there are canonical maps $f^{\mathcal{H}}: \mathcal{N}(H) \xrightarrow[]{} \mathcal{N}(G)$ and a $\Gamma_F$-equivariant map $\psi^{\mathcal{H}}: \pi_1(H) \xrightarrow[]{} \pi_1(G)$. We abuse notation by also using $\psi^{\mathcal{H}}$ to denote the induced maps on the invariants and the coinvariants. The following diagram commutes
\begin{center}
\begin{tikzcd}
    \mathcal{N}(H) \arrow{r}{\delta_H} \arrow{d}{f^{\mathcal{H}}} & \pi_1(H)^{\Gamma_F} \otimes \mathbb{Q} \arrow{d}{\psi^{\mathcal{H}}} & \pi_1(H)_{\Gamma_F} \arrow{l} \arrow{d}{\psi^{\mathcal{H}}} \\
    \mathcal{N}(G) \arrow{r}{\delta_G} & \pi_1(G)^{\Gamma_F} \otimes \mathbb{Q} & \pi_1(G)_{\Gamma_F} \arrow{l}
\end{tikzcd}
\end{center}

The map $f^{\mathcal{H}}:\mathcal{N}(H) \xrightarrow[]{} \mathcal{N}(G)$ has finite fibres.
\end{LemA}

The four assertions in Lemma A are proved in Lemmas \ref{endoscopicnewtonmap}, \ref{endoscopicpimap}, \ref{endoscopiccompatibility}, and \ref{endoscopicfinitefibres}.

\textbf{Endoscopic contributions to a stratum.} 
For a Shimura datum $(G,X)$, let $\Eell(G)$ denote the set of equivalence classes of elliptic endoscopic triples for $G$. The set $\Eell(G)$ is finite and admits a description in terms of the root datum of $G$. For $\mathcal{H}=(H,s,\eta)\in\Eell(G)$, apply Lemma A to the localization of $\mathcal{H}$ at $p$, and continue to denote the resulting maps by $f^{\mathcal{H}}$ and $\psi^{\mathcal{H}}$. For $b\in B(G_{\mathbb{Q}_p},\mu_h^{-1})$, define $\Eell(G)_b$ to consist of those $\mathcal{H}\in\Eell(G)$ for which there exists $b_H\in B(H_{\mathbb{Q}_p})$ such that
$$
f^{\mathcal{H}}\bigl(\overline{\nu}_H(b_H)\bigr)=\overline{\nu}_G(b),
\qquad
\psi^{\mathcal{H}}\bigl(\kappa_H(b_H)\bigr)=\kappa_G(b).
$$
This condition can be checked from the root data of $H$ and $G$. Moreover, only finitely many elements of $B(H_{\mathbb{Q}_p})$ satisfy these two conditions, by \ref{finitelymanybH}.

\textbf{Truncation function.} Let $j=n[\mathbb{F}_q:\mathbb{F}_p]$, so that $p^j=q^n$. Define
$$
\chi_{b,n}^{\mathcal H}:H(\mathbb Q_p)\longrightarrow\{0,1\}
$$
to be the characteristic function of the set of elements $x_H\in H(\mathbb Q_p)$ satisfying
$$
f^{\mathcal H}\bigl(\overline{\nu}_H([x_H])\bigr)=j\cdot\overline{\nu}_G(b),
$$
where $[x_H]\in B(H_{\mathbb Q_p})$ denotes the $\sigma$-conjugacy class of the image of $x_H$ in $H(L)$. The subscript $n$ records the dependence of this function on $n$, through $j$. By contrast the set $\Eell(G)_b$ defined above does not depend on $n$: it is cut out by the condition $f^{\mathcal H}(\overline{\nu}_H(b_H))=\overline{\nu}_G(b)$, in which no factor of $j$ occurs. The two are compared at the end of Section \ref{completionmainproof}.

\begin{LemB}
    The function $\chi_{b,n}^{\mathcal{H}}$ is smooth and constant on stable conjugacy classes of semisimple elements in $H(\mathbb{Q}_p)$.
\end{LemB}

Although $\chi_{b,n}^{\mathcal{H}}$ need not be compactly supported, the product $h_p\chi_{b,n}^{\mathcal{H}}$ is compactly supported, since $h_p\in\mathcal{C}^{\infty}_c(H(\mathbb{Q}_p))$.

\textbf{Stable trace formula for Newton strata.}
With this notation, the main theorem takes the following form. It is the analogue for Newton strata of Theorem 7.2 of \cite{MR1044820}: the set $\Eell(G)_b$ determines which endoscopic triples contribute, while $\chi_{b,n}^{\mathcal H}$ isolates the contribution of the stratum $b$ on $H$.

\begin{T}
    Let $(G,X)$ be a Shimura datum of abelian type satisfying the conditions:
    \begin{enumerate}
        \item $G^{\text{der}}$ is simply connected
        \item The maximal $\mathbb{Q}$-split torus is equal to the maximal $\mathbb{R}$-split torus in the center $Z(G)$
    \end{enumerate}
    Fix $b\in B(G_{\mathbb Q_p},\mu_h^{-1})$ and $f^p\in\mathcal H(G(\mathbb A_f^p)//K^p)$, and retain the notation introduced above. For all sufficiently large $n$, the Lefschetz number $c_b(n, f^p)$ is given by
    \begin{equation*}
		\sum_{\mathcal{H} \in \Eell(G)_b} \iota (G,H) \ST_e^H(h^p \cdot h_p \chi_{b,n}^{\mathcal{H}} \cdot h_{\infty})
	\end{equation*}
    Here $\ST_e^H(-)$ denotes the elliptic part of the geometric side of the stable trace formula for $H$. For each $\mathcal H$, the factors $h^p$, $h_p$, and $h_\infty$ are those occurring in the function constructed for $\mathcal H$ in Theorem 7.2 of \cite{MR1044820}; in the formula above, its $p$-component is replaced by $h_p\chi_{b,n}^{\mathcal H}$.
\end{T}

\subsection{Example: $\mathrm{GU}(r,s)$}\label{GUexample}
The set $\Eell(G)_b$ introduced in Section \ref{details} records the elliptic endoscopic data that can contribute to the stabilized formula for a fixed Newton stratum. We now make this restriction explicit for a unitary Shimura variety: although there are non-trivial endoscopic contributions in the cohomology of the ambient Shimura variety, the intermediate Newton stratum considered below has only the trivial endoscopic contribution.

\subsubsection{The Shimura datum and its endoscopic groups}
Let $E_0/\mathbb{Q}$ be an imaginary quadratic extension, with non-trivial automorphism $a\mapsto\overline{a}$, and let $V$ be an $N$-dimensional Hermitian space over $E_0$ having signature $(r,s)$ at infinity, where $N=r+s$. Let $G=\mathrm{GU}(V)=\mathrm{GU}(r,s)$ be its group of unitary similitudes. Write $\mathbb{S}=\Res_{\mathbb{C}/\mathbb{R}}\mathbb{G}_m$. After identifying $V\otimes_{\mathbb Q}\mathbb R$ with $\mathbb C^N$ so that the Hermitian form has signature $(r,s)$, define
$$
h:\mathbb{S}\longrightarrow G_{\mathbb R},
\qquad
h(z)=\mathrm{diag}\bigl(zI_r,\overline{z}I_s\bigr).
$$
The multiplier of $h(z)$ is $z\overline{z}$. Let $X$ be the $G(\mathbb R)$-conjugacy class of $h$. Then $(G,X)$ is the PEL-type unitary Shimura datum considered here. In the case considered below, where $r\ne s$, its reflex field is $E_0$.

To describe its endoscopy, let $J_N$ be the anti-diagonal Hermitian matrix with entries $1$ on the anti-diagonal. The quasi-split unitary similitude group $\mathrm{GU}(N)$ is defined, for every $\mathbb{Q}$-algebra $R$, by
$$
\mathrm{GU}(N)(R)
=
\left\{
g\in\GL_N(E_0\otimes_{\mathbb Q}R)
\ \middle|\
{}^t\!\overline{g}J_Ng=c(g)J_N
\text{ for some }c(g)\in R^\times
\right\}.
$$
The scalar $c(g)$ defines the similitude character $c:\mathrm{GU}(N)\rightarrow\mathbb{G}_m$. This group is the quasi-split inner form of $G$.

For a decomposition $N=N_1+N_2$ with $N_1,N_2>0$, define
$$
H_{N_1,N_2}
=\mathrm{GU}(N_1)\times_{\mathbb G_m}\mathrm{GU}(N_2),
$$
where both maps to $\mathbb G_m$ are the similitude character $c$. Thus, for every $\mathbb Q$-algebra $R$,
$$
H_{N_1,N_2}(R)
=
\left\{
(g_1,g_2)\in\mathrm{GU}(N_1)(R)\times\mathrm{GU}(N_2)(R)
\ \middle|\ c(g_1)=c(g_2)
\right\}.
$$
This is precisely the group customarily denoted $\mathrm{G}\bigl(\mathrm{U}(N_1)\times\mathrm{U}(N_2)\bigr)$: both descriptions impose the equality of the two similitude factors.
Up to interchanging $N_1$ and $N_2$, these are the underlying groups of the non-trivial elliptic endoscopic data for $G$; the trivial endoscopic datum has underlying group $\mathrm{GU}(N)$. An endoscopic datum also includes an element $s$ and an $L$-embedding, which will not be needed for the present description.

\subsubsection{Integral Newton polygons for $\mathrm{GU}(N)$}
Fix a good prime $p$ which is inert in $E_0$, and write $F=\mathbb Q_p$ and $E=E_{0,p}$. Thus $E/F$ is unramified quadratic, and $G_F=\mathrm{GU}(N)_F$ is quasi-split because $p$ is good for $G$. Quasi-splitness at $p$ is a condition on the discriminant of the Hermitian form rather than a consequence of inertness, but it can always be arranged by choosing $V$ suitably, and is automatic when $N$ is odd. Its standard Levi subgroups are indexed by decompositions
$$
N=2(m_1+\cdots+m_t)+N_0.
$$
The Levi stabilizing the corresponding isotropic flag is
$$
M_{\boldsymbol m,N_0}\simeq
\prod_{i=1}^t\Res_{E/F}\GL_{m_i}\times
\begin{cases}
\mathrm{GU}(N_0),&N_0>0,\\
\mathbb G_m,&N_0=0.
\end{cases}
$$
Under this decomposition, the restriction of the similitude character to $M_{\boldsymbol m,N_0}$ is
$$
c_M((g_i)_i,g_0)=c(g_0)\quad\text{if }N_0>0,
\qquad
c_M((g_i)_i,z)=z\quad\text{if }N_0=0.
$$
Thus $c_M$ is projection to the final factor, followed by the similitude character when that factor is $\mathrm{GU}(N_0)$; the factors $\Res_{E/F}\GL_{m_i}$ do not affect it.

Now let $b\in B(G_F)$, let $d\in\mathbb Z$ be the slope of its similitude character, and write its slopes as
$$
\lambda_1^{m_1},\ldots,\lambda_t^{m_t},(d/2)^{N_0},
(d-\lambda_t)^{m_t},\ldots,(d-\lambda_1)^{m_1},
\qquad \lambda_1<\cdots<\lambda_t<d/2.
$$
By Lemma \ref{chailemma},
$$
M_{\nu_b}=Z_G(\nu_b)\simeq M_{\boldsymbol m,N_0},
$$
and $b$ is the image of a basic class in $B(M_{\nu_b})$. The same lemma gives the integrality condition $2m_i\lambda_i\in\mathbb Z$. Consequently, the component of $\mathcal N(G_{\mathbb Q_p})_{\mathbb Z}$ with similitude slope $d$ consists of the convex polygons from $(0,0)$ to $(N,Nd/2)$ whose breakpoints lie in $\mathbb Z\times\frac12\mathbb Z$ and whose ordered slopes satisfy
$$
\lambda_i+\lambda_{N+1-i}=d.
$$
The classes indexing Newton strata lie in $B(G_F,\mu_h^{-1})$ (Section \ref{newtonstratdetails}), so the component determined by the Shimura datum is the one with $d=-1$. Throughout this section we display instead the negatives of these polygons, so that $d=1$ and all slopes lie in $[0,1]$. Nothing is lost. Writing $P_{\nu}(k)=\lambda_1+\cdots+\lambda_k$ for the slopes in increasing order, the negated polygon is
$$
P_{-\nu}(k)=P_{\nu}(N-k)-P_{\nu}(N),
$$
from which one reads off that $P \mapsto -P$ is a bijection between the components $d=1$ and $d=-1$ of $\mathcal N(G_F)_{\mathbb Z}$ preserving integrality, and that it preserves the relation ``lies on or above'' between polygons with common endpoints. Moreover $f^{\mathcal H}$, computed below, takes the union of slope multisets and so commutes with it. The following two representative non-basic polygons illustrate the two possible forms of the final factor of $M_{\nu}$.

\begin{center}
\begin{minipage}{0.95\textwidth}
\centering
\begin{tikzpicture}[x=0.46cm,y=0.46cm]
    \begin{scope}
        \draw[->] (-0.3,0) -- (10.45,0);
        \draw[->] (0,-0.3) -- (0,5.45);
        \draw[thin,gray!55] (0,0) -- (10,5);
        \draw[densely dashed,gray] (5,-0.15) -- (5,5.25)
            node[above] {\scriptsize axis $x=N/2$};
        \draw[line width=1pt,RoyalBlue]
            (0,0) -- (1,0) -- (4,1) -- (6,2) -- (9,4) -- (10,5);
        \filldraw[black] (0,0) circle (1.8pt);
        \filldraw[black] (1,0) circle (1.8pt);
        \filldraw[black] (4,1) circle (1.8pt);
        \filldraw[black] (6,2) circle (1.8pt);
        \filldraw[black] (9,4) circle (1.8pt);
        \filldraw[black] (10,5) circle (1.8pt);
        \node[align=center] at (5,-1.05)
            {\scriptsize $0,(1/3)^3,(1/2)^2,(2/3)^3,1$\\[-1pt]
             \scriptsize $M_{\nu}\simeq\Res_{E/F}\GL_1\times\Res_{E/F}\GL_3\times\mathrm{GU}(2)$};
    \end{scope}

    \begin{scope}[xshift=7.25cm]
        \draw[->] (-0.3,0) -- (10.45,0);
        \draw[->] (0,-0.3) -- (0,5.45);
        \draw[thin,gray!55] (0,0) -- (10,5);
        \draw[densely dashed,gray] (5,-0.15) -- (5,5.25)
            node[above] {\scriptsize axis $x=N/2$};
        \draw[line width=1pt,BrickRed]
            (0,0) -- (2,0.5) -- (5,1.5) -- (8,3.5) -- (10,5);
        \filldraw[black] (0,0) circle (1.8pt);
        \filldraw[black] (2,0.5) circle (1.8pt);
        \filldraw[black] (5,1.5) circle (1.8pt);
        \filldraw[black] (8,3.5) circle (1.8pt);
        \filldraw[black] (10,5) circle (1.8pt);
        \node[align=center] at (5,-1.05)
            {\scriptsize $(1/4)^2,(1/3)^3,(2/3)^3,(3/4)^2$\\[-1pt]
             \scriptsize $M_{\nu}\simeq\Res_{E/F}\GL_2\times\Res_{E/F}\GL_3\times\mathbb G_m$};
    \end{scope}
\end{tikzpicture}

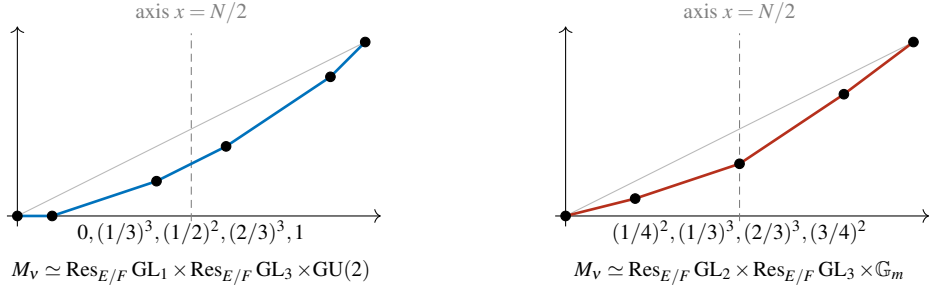
\small Two integral Newton polygons for $N=10$ and $d=1$. The dashed axis pairs segments whose slopes sum to $1$; the gray line has slope $1/2$.
\end{minipage}
\end{center}

\subsubsection{Integral Newton polygons for the endoscopic groups}
Let $H=H_{N_1,N_2}$. For $a=1,2$, choose a decomposition
$$
N_a=2(m_{a,1}+\cdots+m_{a,t_a})+N_{a,0}.
$$
The corresponding standard Levi subgroup of $H_F$ is
$$
M_H=M_{\boldsymbol m_1,N_{1,0}}
\times_{\mathbb G_m}
M_{\boldsymbol m_2,N_{2,0}}.
$$
Applying the preceding computation to the two factors gives
$$
\mathcal N(H_F)_{\mathbb Z}
=
\left\{
(P_1,P_2)\ \middle|\
P_a\in\mathcal N(\mathrm{GU}(N_a)_F)_{\mathbb Z},\quad
d(P_1)=d(P_2)
\right\},
\qquad
d(P_a)=\frac{2P_a(N_a)}{N_a}.
$$
Thus each $P_a$ has breakpoints in $\mathbb Z\times\frac12\mathbb Z$ and slopes satisfying
$$
\lambda_{a,i}+\lambda_{a,N_a+1-i}=d,
$$
for one common integer $d$. The equality of the two values of $d$ is exactly the condition imposed by the shared similitude factor in the fibre product.

\begin{center}
\begin{minipage}{0.92\textwidth}
\centering
\begin{tikzpicture}[x=0.72cm,y=0.72cm]
    \node[draw,rounded corners,inner sep=4pt] at (7.2,4.55)
        {\small shared similitude slope: $d(P_1)=d(P_2)=1$};

    \begin{scope}
        \draw[->] (-0.25,0) -- (6.4,0);
        \draw[->] (0,-0.25) -- (0,3.35);
        \draw[thin,gray!55] (0,0) -- (6,3);
        \draw[densely dashed,gray] (3,-0.1) -- (3,3.15)
            node[above] {\scriptsize $x=N_1/2$};
        \draw[line width=1pt,RoyalBlue]
            (0,0) -- (2,0.5) -- (4,1.5) -- (6,3);
        \filldraw[black] (0,0) circle (1.8pt);
        \filldraw[black] (2,0.5) circle (1.8pt);
        \filldraw[black] (4,1.5) circle (1.8pt);
        \filldraw[black] (6,3) circle (1.8pt);
        \node[align=center] at (3,-0.8)
            {\scriptsize $P_1:\ (1/4)^2,(1/2)^2,(3/4)^2$\\[-1pt]
             \scriptsize $N_1=6$};
    \end{scope}

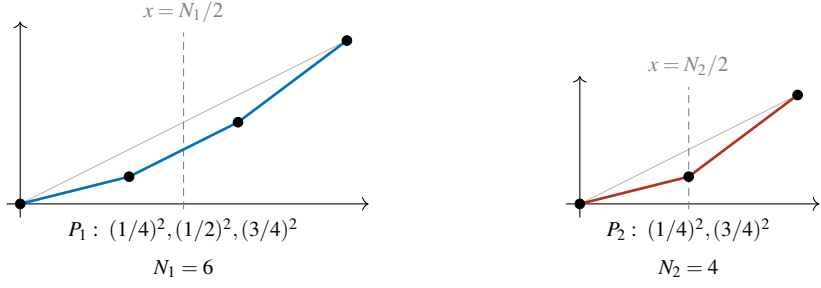
    \begin{scope}[xshift=7.4cm]
        \draw[->] (-0.25,0) -- (4.4,0);
        \draw[->] (0,-0.25) -- (0,2.35);
        \draw[thin,gray!55] (0,0) -- (4,2);
        \draw[densely dashed,gray] (2,-0.1) -- (2,2.15)
            node[above] {\scriptsize $x=N_2/2$};
        \draw[line width=1pt,BrickRed]
            (0,0) -- (2,0.5) -- (4,2);
        \filldraw[black] (0,0) circle (1.8pt);
        \filldraw[black] (2,0.5) circle (1.8pt);
        \filldraw[black] (4,2) circle (1.8pt);
        \node[align=center] at (2,-0.8)
            {\scriptsize $P_2:\ (1/4)^2,(3/4)^2$\\[-1pt]
             \scriptsize $N_2=4$};
    \end{scope}
\end{tikzpicture}

\small An integral Newton point of $H_{6,4}$. Each polygon satisfies the unitary integrality conditions separately; the boxed equality is the additional condition coming from the shared $\mathbb G_m$.
\end{minipage}
\end{center}

\subsubsection{Transfer of Newton polygons}
Let $\mathcal H$ be the endoscopic datum with underlying group $H_{N_1,N_2}$. The map
$$
f^{\mathcal H}:\mathcal N(H_F)\longrightarrow\mathcal N(\mathrm{GU}(N)_F)
$$
is constructed in Lemma \ref{endoscopicnewtonmap}, using the admissible embeddings of Definition \ref{admissibleembedding}. Under the resulting identification of maximal tori---equivalently, the identification coming from a common maximal torus of the dual groups---the Weyl group of $H_F$ permutes the slopes within the two blocks, while the Weyl group of $\mathrm{GU}(N)_F$ also permutes slopes between them. Thus $f^{\mathcal H}$ has the following simple description: break $P_1$ and $P_2$ into their straight-line segments, arrange all the segments in ascending order of slope, and concatenate them. Equivalently, the slope multiset of the transferred polygon is the union of the slope multisets of $P_1$ and $P_2$.

For the pair displayed above, the following picture shows this construction. Blue segments come from $P_1$ and red segments from $P_2$.

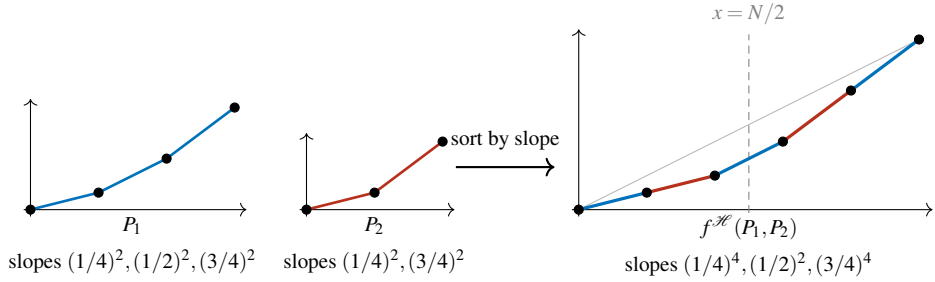
\begin{center}
\begin{minipage}{0.95\textwidth}
\centering
\begin{tikzpicture}[x=0.45cm,y=0.45cm]
    \begin{scope}
        \draw[->] (-0.2,0) -- (6.35,0);
        \draw[->] (0,-0.2) -- (0,3.25);
        \draw[line width=1pt,RoyalBlue]
            (0,0) -- (2,0.5) -- (4,1.5) -- (6,3);
        \filldraw[black] (0,0) circle (1.7pt);
        \filldraw[black] (2,0.5) circle (1.7pt);
        \filldraw[black] (4,1.5) circle (1.7pt);
        \filldraw[black] (6,3) circle (1.7pt);
        \node[align=center] at (3,-1.05)
            {\scriptsize $P_1$\\[-1pt]
             \scriptsize slopes $(1/4)^2,(1/2)^2,(3/4)^2$};
    \end{scope}

    \begin{scope}[xshift=3.65cm]
        \draw[->] (-0.2,0) -- (4.35,0);
        \draw[->] (0,-0.2) -- (0,2.25);
        \draw[line width=1pt,BrickRed]
            (0,0) -- (2,0.5) -- (4,2);
        \filldraw[black] (0,0) circle (1.7pt);
        \filldraw[black] (2,0.5) circle (1.7pt);
        \filldraw[black] (4,2) circle (1.7pt);
        \node[align=center] at (2,-1.05)
            {\scriptsize $P_2$\\[-1pt]
             \scriptsize slopes $(1/4)^2,(3/4)^2$};
    \end{scope}

    \draw[->,line width=0.8pt] (12.5,1.25) -- (15.4,1.25);
    \node at (13.95,2.05) {\scriptsize sort by slope};

    \begin{scope}[xshift=7.25cm]
        \draw[->] (-0.2,0) -- (10.4,0);
        \draw[->] (0,-0.2) -- (0,5.3);
        \draw[thin,gray!55] (0,0) -- (10,5);
        \draw[densely dashed,gray] (5,-0.1) -- (5,5.15)
            node[above] {\scriptsize $x=N/2$};
        \draw[line width=1.2pt,RoyalBlue] (0,0) -- (2,0.5);
        \draw[line width=1.2pt,BrickRed] (2,0.5) -- (4,1);
        \draw[line width=1.2pt,RoyalBlue] (4,1) -- (6,2);
        \draw[line width=1.2pt,BrickRed] (6,2) -- (8,3.5);
        \draw[line width=1.2pt,RoyalBlue] (8,3.5) -- (10,5);
        \foreach \x/\y in {0/0,2/0.5,4/1,6/2,8/3.5,10/5}
            \filldraw[black] (\x,\y) circle (1.7pt);
        \node[align=center] at (5,-1.05)
            {\scriptsize $f^{\mathcal H}(P_1,P_2)$\\[-1pt]
             \scriptsize slopes $(1/4)^4,(1/2)^2,(3/4)^4$};
    \end{scope}
\end{tikzpicture}
\end{minipage}
\end{center}

\subsubsection{An intermediate stratum with only the trivial contribution}
The preceding computations reduce the condition defining $\Eell(G)_b$ to a decomposition problem for Newton polygons. We now give an example for which that condition excludes every non-trivial elliptic endoscopic datum.

Retain $N=r+s$, and suppose that, after interchanging $r$ and $s$ if necessary, $0<r<s$, and that
$$
\gcd(s-r,\,s+r)=2.
$$
This forces $r$ and $s$ to have the same parity, so that $N$ is even. It holds in particular when $r$ and $s$ are odd and coprime, but also for instance when $(r,s)=(2,4)$. Put
$$
\lambda=\frac{s-r}{2N}.
$$
With the increasing-slope convention used above, the Hodge polygon $P_{\mu_h}$---that is, the polygon attached to the $\Gamma_F$-average of the Hodge cocharacter---has slopes
$$
0^r,\ (1/2)^{s-r},\ 1^r.
$$
It is also the $\mu$-ordinary polygon displayed below. By Remark \ref{BGmudef}, for a class $b$ with the Kottwitz invariant determined by $\mu_h^{-1}$, membership in $B(G_F,\mu_h^{-1})$ is equivalent to its Newton polygon having the same endpoints as that of $\mu_h^{-1}$ and lying on or above it; in the displayed, negated, picture this says that the polygon shown for $b$ has the same endpoints as $P_{\mu_h}$ and lies on or above it.

Consider the intermediate polygon $P_b$ with slopes
$$
\lambda^{N/2},\ (1-\lambda)^{N/2}.
$$
Its unique interior breakpoint is
$$
\left(\frac N2,\frac{s-r}{4}\right).
$$
It is integral since $N\lambda=(s-r)/2\in\mathbb Z$, and it lies above $P_{\mu_h}$, meeting it at $x=N/2$. Together with the Kottwitz invariant determined by $\mu_h^{-1}$, its negative therefore defines a class $b\in B(G_F,\mu_h^{-1})$.

\begin{center}
\begin{minipage}{0.95\textwidth}
\centering
\begin{tikzpicture}[x=0.9cm,y=0.9cm]
    \draw[->] (-0.35,0) -- (8.45,0);
    \draw[->] (0,-0.35) -- (0,4.4);

    \draw[line width=0.9pt,RoyalBlue] (0.35,4.7) -- (0.95,4.7);
    \node[anchor=west] at (1.05,4.7) {\small basic};
    \draw[line width=0.9pt,Violet] (2.65,4.7) -- (3.25,4.7);
    \node[anchor=west] at (3.35,4.7) {\small intermediate};
    \draw[line width=0.9pt,BrickRed] (6.05,4.7) -- (6.65,4.7);
    \node[anchor=west] at (6.75,4.7) {\small ordinary};

    \draw[line width=0.9pt,RoyalBlue] (0,0) -- (8,4);
    \draw[line width=1.1pt,Violet] (0,0) -- (4,0.5) -- (8,4);
    \draw[line width=0.9pt,BrickRed] (0,0) -- (3,0) -- (5,1) -- (8,4);

    \filldraw[black] (0,0) circle (2pt) node[below left] {\scriptsize $(0,0)$};
    \filldraw[black] (3,0) circle (2pt) node[below] {\scriptsize $(r,0)$};
    \filldraw[black] (4,0.5) circle (2pt)
        node[above left] {\scriptsize $\left(\frac N2,\frac{s-r}{4}\right)$};
    \filldraw[black] (5,1) circle (2pt)
        node[below right] {\scriptsize $\left(s,\frac{s-r}{2}\right)$};
    \filldraw[black] (8,4) circle (2pt)
        node[above left] {\scriptsize $\left(N,\frac N2\right)$};
\end{tikzpicture}

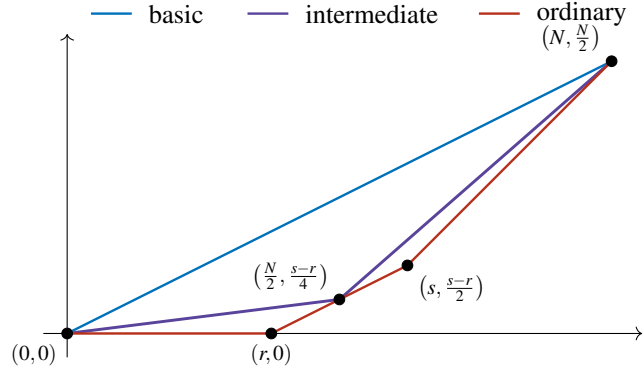
\small The basic, selected intermediate, and ordinary Newton polygons, drawn to scale for $(r,s)=(3,5)$. The displayed formulas for the breakpoints hold for general $r$ and $s$.
\end{minipage}
\end{center}

We claim that
$$
\Eell(G)_b=\{\mathcal H_{\mathrm{triv}}\}.
$$
Indeed, suppose that $\mathcal H\in\Eell(G)_b$ were non-trivial, with underlying group $H_{N_1,N_2}$. By the description of $f^{\mathcal H}$ above, there would be a pair $(P_1,P_2)\in\mathcal N(H_F)_{\mathbb Z}$ whose transferred polygon is $P_b$. Since $P_b$ has only the two slopes $\lambda$ and $1-\lambda$, the unitary symmetry of each $P_a$ forces
$$
P_a:\quad \lambda^{k_a},\ (1-\lambda)^{k_a},
\qquad N_a=2k_a,
\qquad k_1+k_2=N/2.
$$
Since $N_1,N_2>0$, we have $0<k_a<N/2$. The breakpoint $(k_a,k_a\lambda)$ of $P_a$ must lie in $\mathbb Z\times\frac12\mathbb Z$, and hence
$$
2k_a\lambda=\frac{k_a(s-r)}{N}\in\mathbb Z.
$$
Writing $d=\gcd(s-r,N)$, this says that $N/d$ divides $k_a$, so that the smallest admissible value of $k_a$ is $N/d$. By hypothesis $d=2$, so $N/d=N/2$, contradicting $0<k_a<N/2$. Equivalently, the lower segment of $P_b$ contains no point of $\mathbb Z\times\frac12\mathbb Z$ at which it can be split into two smaller integral unitary polygons. Thus no non-trivial datum belongs to $\Eell(G)_b$; the trivial datum belongs to it by taking $b$ itself.

The hypothesis is exactly what the argument requires: if $d>2$, then $k_a=N/d$ satisfies $0<k_a<N/2$, the polygon $P_b$ does admit a splitting of the above shape, and the exclusion argument breaks down.

\subsection*{Structure of the text}
A reader interested mainly in the main result and its proof can read this introduction and then turn directly to Section \ref{mainproof}, where the modified Langlands-Kottwitz method is carried out and the Main Theorem is proved.

Section \ref{prelim} collects the preliminaries: isocrystals, endoscopy and base change, together with the notation used throughout. An expert reader may skip it and refer back to it only as needed, guided by the cross-references in Sections \ref{2.4} and \ref{mainproof}.

Section \ref{2.4} contains the proofs of Lemma A and Lemma B stated in the introduction. A reader who wishes to proceed directly to Section \ref{mainproof} may safely take these two lemmas as black boxes.

\subsection*{Acknowledgements}
This paper is a part of the author's PhD research at the University of Amsterdam, under the supervision of Arno Kret. I thank Arno Kret for suggesting this direction of research to extend his results to the endoscopic setting and for his comments on earlier drafts of this paper. I also thank Kaan Bilgin, Yachen Liu and Reinier Sorgdrager for helpful discussions.

\subsection*{AI use disclosure}
AI was used as an assistant in writing some parts of this paper, which consisted of polishing earlier drafts and outlines written by the author. In particular, the Complements in Section \ref{mainproof} were written with AI assistance. Additionally, the version of Hensel lemma, due to Rychl\'{\i}k, used in Lemma \ref{continuityofroots} was following a suggestion made by Claude Opus 5.

\section{Preliminaries}\label{prelim}
To explain the proofs of Lemma A and Lemma B as stated in the introduction, we start by reviewing isocrystals with $G$-structure in Section \ref{2.1}, the theory of endoscopy in Section \ref{2.2}, and cyclic base change in Section \ref{2.3} which are essential ingredients in the proofs of these lemmas and will also be used later on in the proof of the main theorem. Finally, in Section \ref{2.4} we will present new results and proofs of the lemmas. The results obtained there will be used in Section \ref{mainproof} in the proof of the main theorem.

\subsection{Review of isocrystals}\label{2.1}
\subsubsection{Setup}
We will use the following notation for the rest of this subsection.
\begin{itemize}
    \item Let $F$ denote a $p$-adic field, with residue field $k_F$ of cardinality $q_F$, and let $F^{\text{un}}$ denote a maximal unramified extension of $F$, whose residue field is an algebraic closure $\overline{k_F}$. We write $\sigma$ for the \textit{arithmetic} Frobenius automorphism of $F^{\text{un}}$ over $F$, that is, for the topological generator of $\Gal(F^{\text{un}}/F)$ inducing $x \mapsto x^{q_F}$ on $\overline{k_F}$. Let $L$ denote the completion of $F^{\text{un}}$. Then $\sigma$ admits a continuous extension to an automorphism of $L$. There is a valuation map $\val: L^{\times} \xrightarrow[]{} \mathbb{Z}$ normalized by the condition $\val(\pi) = 1$ for a uniformizer $\pi$ of $F$, which is also a uniformizer of $L$ since $L/F$ is unramified.
    \item Let $G$ denote a connected linear algebraic group over $F$ and let $\Rep(G)$ denote the category of representations of $G$ on a finite dimensional $F$-vector space. 
\end{itemize}

\begin{D}[Isocrystals] 
    An isocrystal is a pair $(V,\Phi)$ consisting of a finite dimensional vector space $V$ over $L$ and a bijective $\sigma$-linear endomorphism $\Phi: V \xrightarrow[]{} V$. Here a map $\Phi: V \xrightarrow[]{} V$ is called \textit{$\sigma$-linear} if it is additive and satisfies
    $$\Phi(\lambda v)=\sigma(\lambda)\Phi(v), \qquad \lambda\in L,\; v\in V.$$ A morphism of isocrystals is a linear map on the underlying vector spaces that is compatible with the $\sigma$-linear operators. The category of isocrystals is denoted by $\Isoc$.
\end{D}

By a result of Dieudonné-Manin, it follows that $\Isoc$ is semisimple. The simple objects are parametrized by the set of rational numbers. For $\lambda \in \mathbb{Q}$, with $\lambda = r/s$ such that $r,s \in \mathbb{Z}$ with $s > 0$ and $(r,s) = 1$, the corresponding simple object is $E_{\lambda} = (L^s, \Phi)$, where $\Phi: L^s \xrightarrow[]{} L^s$ is given by

$$ \Phi = \begin{pmatrix}
    0 & & & \pi^r \\
    1 & & & \\
    & \ddots & & \\
    & & 1 & 0
\end{pmatrix} \cdot \sigma. $$

Note the property that $\Phi^s = \pi^r \cdot \sigma^s$. An isocrystal which is isotypic of slope $r/s$, and in particular the simple object $E_{r/s}$, can be defined over the fixed field $F_s$ of $\sigma^s$ on $L$, that is, over the unramified extension of $F$ of degree $s$ contained in $L$. Explicitly, let $(V_L, \Phi_L)$ be isotypic of slope $r/s$. The operator $\pi^{-r}\Phi_L^s$ is $\sigma^s$-linear and therefore $F_s$-linear, so its fixed points
    $$V := V_L^{\pi^{-r}\Phi_L^s} = \{v \in V_L \mid \Phi_L^s(v) = \pi^r v\}$$
    form an $F_s$-vector space. This space generates $V_L$ over $L$: the natural map $V \otimes_{F_s} L \xrightarrow[]{} V_L$ is an isomorphism. Since $\Phi_L$ commutes with $\pi^{-r}\Phi_L^s$, it preserves $V$, and $\Phi := \Phi_L|_V$ is $\sigma$-linear over $F_s$, the automorphism $\sigma$ restricting to a generator of $\Gal(F_s/F)$. Thus $\Phi_L$ is the $\sigma$-linear extension of $\Phi$ to $V_L$. See $\S 3$ of \cite{MR809866}, where this is proved for an arbitrary $p$-adic field $F$, and $\S 1$ of \cite{MR2074057} for further details. The category $\Isoc$ is Tannakian over $F$ and admits a fiber functor $\omega$ over $L$ which maps an isocrystal to the underlying vector space.

\begin{D}[$B(G)$]\label{B(G)def}
    Let $B(G)$ denote the set of $\sigma$-conjugacy classes in $G(L)$, where $b, b' \in G(L)$ are said to be \textit{$\sigma$-conjugate} if there exists a $g \in G(L)$ such that $b' = g b \sigma(g)^{-1}$. This is equivalent to the elements $b\sigma$ and $b'\sigma$ of $G(L) \rtimes \langle \sigma \rangle$ being conjugate under $G(L)$. Hence $B(G)$ can also be identified with the pointed set $H^1(\langle \sigma \rangle, G(L))$.
\end{D}

\begin{D}[Isocrystals with $G$-structure]\label{isocGstructure}
     Let $\omega_G: \Rep(G) \xrightarrow[]{} \Vect_L$ denote the fiber functor over $L$ which maps $V$ to $L \otimes_F V$. An isocrystal with $G$-structure is a pair $(\alpha, \beta)$ consisting of an exact tensor functor $\beta: \Rep(G) \xrightarrow[]{} \Isoc$ together with a tensor isomorphism $\alpha: \omega \circ \beta \xrightarrow[]{\sim} \omega_G$, where $\omega$ denotes the fiber functor on $\Isoc$ introduced above. Such pairs are in bijection with $G(L)$. Moreover, the tensor functors $\beta$ which admit such an $\alpha$, taken up to tensor isomorphism, are in bijection with the set $B(G)$ of \ref{B(G)def}. See $\S 3.1$ of \cite{MR1485921} for more details.
\end{D}

\begin{const}\label{cohomisocrystal}
    There is an injective map $H^1(F, G) \xrightarrow[]{} B(G)$ given as follows. The set $H^1(F, G)$ is in bijection with isomorphism classes of fiber functors $\Rep(G) \xrightarrow[]{} \Vect_F$. Given such a fiber functor, one can construct an exact tensor functor $\Rep(G) \xrightarrow[]{} \Isoc$ by composing it with the functor $\Vect_F \xrightarrow[]{} \Isoc$ which maps $V$ to $L \otimes_F V$, equipped with the $\sigma$-linear operator $\Phi$ determined by $l \otimes v \mapsto \sigma(l) \otimes v$.
\end{const}

\begin{const}[$J_b$]\label{J}
    Given $b \in B(G)$, there is an algebraic group $J_b$ over $F$ defined as follows. For an $F$-algebra $R$, denote by $\Isoc^R$ the category with the same objects as $\Isoc$ but with an enlarged set of morphisms, namely
    $$\Hom_{\Isoc^R}((V_1,\Phi_1),(V_2,\Phi_2)) = \Hom_{\Isoc}((V_1, \Phi_1), (V_2, \Phi_2)) \otimes_F R.$$
    This is a tensor category and there is a natural tensor functor $\Isoc \xrightarrow[]{} \Isoc^R$. Let $\beta^R$ denote the composition of $\beta$ with this natural functor, where $\beta: \Rep(G) \xrightarrow[]{} \Isoc$ is associated to $b$ as in \ref{isocGstructure}. Then $J_b$ is the algebraic group defined by the condition that $J_b(R)$ is the group of tensor automorphisms of the tensor functor $\beta^R$. See Appendix A of \cite{MR1485921} for more details.
\end{const}

\subsubsection{Slope morphism}
Let $\mathbb{D}$ be the pro-torus over $F$ with character group $X^*(\mathbb{D}) = \mathbb{Q}$, equipped with the trivial action of $\Gamma_F$. These two conditions determine $\mathbb{D}$ uniquely; in particular $\mathbb{D}$ is split over $F$. Given an isocrystal with $G$-structure $b \in G(L)$ (which corresponds to a pair $(\alpha, \beta)$ as in \ref{isocGstructure}) and a finite dimensional representation $(V,\rho) \in \Rep(G)$, we obtain using $(\alpha, \beta)$ an isocrystal $V \otimes_F L$ with a $\sigma$-linear operator $\Phi$. Using the semisimplicity of $\Isoc$ and the fact that simple isocrystals are parametrized by rational numbers, the vector space $V \otimes_F L$ acquires a $\mathbb{Q}$-grading which encodes the decomposition into simple objects. The $\mathbb{Q}$-grading on the vector space $V \otimes_F L$ corresponds to a morphism $\mathbb{D}_L \xrightarrow{} \GL(V)_L$ over $L$.

\begin{D}[Slope morphism]\label{slopedef}
    Slope morphism of an isocrystal with $G$-structure $b \in G(L)$ is the morphism $\nu_G(b) \in \Hom_L(\mathbb{D}, G)$ characterized by the property that for any representation $(V, \rho) \in \Rep(G)$, the $\mathbb{Q}$-grading on $V \otimes_F L$ obtained from $b$, as described above, corresponds to the $\rho_L \circ \nu_G(b)$. See $\S$ 4.2 of \cite{MR809866} for more details.
\end{D}

\begin{Le}[Alternative characterization of the slope morphism]\label{slopealternativedef}
    The element $\nu_G(b)$ obtained from $b \in G(L)$ can also be characterized as the unique element in $\Hom_L(\mathbb{D}, G)$ for which there exists an integer $n > 0$, an element $c \in G(L)$ and a uniformizer $\pi$ of $F$ such that the following three conditions hold:
\begin{enumerate}
    \item $n\nu_G(b) \in \Hom_L(\mathbb{G}_m, G)$
    \item $c (n\nu_G(b)) c^{-1}$ is defined over the fixed field of $\sigma^n$ on $L$
    \item $c (b\sigma)^nc^{-1} = c (n\nu_G(b))(\pi)c^{-1} \cdot  \sigma^n$, viewed as elements of $G(L) \rtimes \langle \sigma \rangle $
\end{enumerate}
\end{Le}
\begin{proof}
    See $\S$ 4.3 of \cite{MR809866}.
\end{proof}

\begin{C}\label{slopebg}
    The $G(L)$ conjugacy class of the slope morphism $\nu_G(b) \in \Hom_L(\mathbb{D}, G)$ of an isocrystal with $G$-structure $b \in G(L)$ depends only on the $\sigma$-conjugacy class of $b$. Moreover, the $G(L)$-conjugacy class of $\nu_G(b)$ is fixed by $\langle \sigma \rangle$. Hence, there is a map $\nu_G:  B(G) \xrightarrow[]{} (\Hom_L(\mathbb{D}, G)/G(L))^{\langle \sigma \rangle}$.
\end{C}

\begin{D}[$\mathcal{N}(G)$]
    For a connected reductive group $G$ over $F$, the set $\mathcal{N}(G) := (\Hom_L(\mathbb{D}, G)/G(L))^{\langle \sigma \rangle}$ denotes the Newton points of $G$.
\end{D}

\begin{Le}\label{slopefunctoriality}
    The map $b \mapsto \nu_G(b)$ induces a natural transformation of set valued functors $B(-) \xrightarrow[]{} \mathcal{N}(-)$ from the category of connected linear algebraic groups
\end{Le}
\begin{proof}
    See $\S$ 1.9 of \cite{MR1411570}.
\end{proof}

\begin{Le}\label{slopesoverclosure}
    $\mathcal{N}(G) = (\Hom_L(\mathbb{D}, G)/G(L))^{\langle \sigma \rangle} \simeq (\Hom_{\overline{F}}(\mathbb{D}, G)/G(\overline{F}))^{\Gamma_F}$
\end{Le}
\begin{proof}
    This follows using the fact (due to Steinberg) that $H^1(L,G)$ is trivial for connected linear algebraic groups $G$.
\end{proof}

\subsubsection{Newton map}\label{newton}
For the rest of the section on isocrystals we will assume (unless otherwise stated) that 
\begin{itemize}
    \item The group $G$ is connected reductive over a $p$-adic field $F$.
    \item Let $T$ be a maximal torus of $G$ over $\overline{F}$ and let $N_G(T)$ denote its normalizer in $G$. The \textit{absolute Weyl group} of $G$ is $\Omega_G(\overline{F}) := (N_G(T)/T)(\overline{F})$. Any two maximal tori of $G$ over $\overline{F}$ are conjugate under $G(\overline{F})$, so this group is independent of the choice of $T$ up to canonical isomorphism.
\end{itemize}

We will use the following notation in the context of the Newton map.
\begin{itemize}
    \item For a group of multiplicative type $D$ over a field $F$, denote by $X^*(D)$ and $X_*(D)$ the groups of characters and cocharacters of $D$ over $\overline{F}$, respectively. These have a natural action of $\Gamma_F$.
    \item Let $G_0$ denote the quasi-split inner form of $G$ over $F$. Let $A$ be a maximal split torus in $G_0$ over $F$. The centralizer $T$ of $A$ in $G_0$ is a maximal torus and since $G_0$ is quasi-split, there is a Borel subgroup $B$ of $G_0$ over $F$ which contains $T$. Let $B = TN$, where $N$ is the unipotent radical of $B$.
    \item The \textit{relative Weyl group} of $G_0$ is $\Omega_{G_0}(F) := (N_{G_0}(A)/Z_{G_0}(A))(F)$.
    \item There is a canonical identification of the absolute Weyl groups $\Omega_{G_0}(\overline{F}) \simeq \Omega_G(\overline{F})$.
    \item Let $\mathfrak{U}_{\mathbb{Q}} = X_*(A) \otimes \mathbb{Q}$ and $\mathfrak{U} = X_*(A) \otimes \mathbb{R}$. Let $\overline{C}$ denote the closed relative Weyl chamber $\{ x \in \mathfrak{U} \mid \langle \alpha, x \rangle \geq 0 \text{ for every root } \alpha \text{ of } A \text{ in } \Lie(N) \}$ and let $\overline{C}_{\mathbb{Q}} = \overline{C} \cap \mathfrak{U}_{\mathbb{Q}}$.
\end{itemize}

\begin{D}[Newton map]\label{newtonmap}
    Let $b \in G(L)$ and let $\nu_G(b) \in \Hom_L(\mathbb{D}, G)$ be the corresponding slope morphism. Choose an inner twisting $\psi: G_{\overline{F}} \xrightarrow[]{\sim} G_{0,\overline{F}}$, an isomorphism over $\overline{F}$. Then $\psi \circ \nu_G(b)$ lies in $\Hom_{\overline{F}}(\mathbb{D}, G_0)$. Note that, in view of \ref{slopebg} and \ref{slopesoverclosure}, the $G_0(\overline{F})$-conjugacy class of this morphism depends only on the $\sigma$-conjugacy class of $b$ and this conjugacy class is $\Gamma_F$-invariant. We have the identifications $$(\Hom_{\overline{F}}(\mathbb{D}, G_0)/G_0(\overline{F}))^{\Gamma_F} \simeq ((X_*(T) \otimes \mathbb{Q}) / \Omega_{G_0}(\overline{F}))^{\Gamma_F} \simeq (X_*(A) \otimes \mathbb{Q})/{\Omega_{G_0}(F)} \simeq \overline{C}_{\mathbb{Q}}.$$
    Using this, we obtain a map $\overline{\nu}_G: B(G) \xrightarrow[]{}\overline{C}_{\mathbb{Q}} \subset \mathfrak{U}$, called the Newton map. The map does not depend on the choice of $\psi$.
\end{D}

\begin{Le}\label{slopetori}
    Let $T$ be a torus over $F$. The slope morphism of $b \in T(L)$ is the unique element $\nu_T(b) \in X_*(T)^{\Gamma_F} \otimes \mathbb{Q}$ satisfying the property that $\val(\lambda(b)) = \langle \nu_T(b), \lambda \rangle$ for all $\lambda \in X^*(T)^{\Gamma_F}$.
\end{Le}
\begin{proof}
    See $\S$ 2.8 of \cite{MR809866}.
\end{proof}

\begin{C}\label{slopetorispecialcase}
    Let $T$ be a torus over $F$ and let $b \in T(F)$. Then the slope morphism $\nu_T(b) \in X_*(T)^{\Gamma_F} \otimes \mathbb{Q}$ of $b$ (viewing it as an element of $T(L)$) satisfies $\val(\lambda (b)) = \langle \nu_T(b), \lambda \rangle$ for all $\lambda \in X^*(T) \otimes \mathbb{Q}$.
\end{C}
\begin{proof}
    Note that the pairing $\langle \cdot, \cdot \rangle: (X_*(T)\otimes \mathbb{Q}) \times (X^*(T) \otimes \mathbb{Q}) \xrightarrow[]{} \mathbb{Q}$ is $\Gamma_F$-equivariant. Since $\nu_T(b) \in X_*(T)^{\Gamma_F} \otimes  \mathbb{Q}$, it follows that $$\langle \nu_T(b), \lambda \rangle = \langle \nu_T(b), \tau(\lambda) \rangle$$ for all $\lambda \in X^*(T) \otimes \mathbb{Q}$ and $\tau \in \Gamma_F$. Let $\lambda \in X^*(T) \otimes \mathbb{Q}$ be a $\mathbb{Q}$-linear combination of characters defined over a finite Galois extension $E$ of $F$, and let $$\overline{\lambda} = \frac{1}{\mid \Gamma_{E/F} \mid}\cdot (\sum_{\tau \in \Gamma_{E/F}} \tau  (\lambda) ) \in X^*(T)^{\Gamma_F}\otimes \mathbb{Q}$$ denote the $\Gamma_F$-average of $\lambda$. Then it follows that $$\langle \nu_T(b) , \overline{\lambda} \rangle = \langle \nu_T(b), \lambda \rangle.$$ 
    Using the fact that $b \in T(F)$, it follows that $(\tau(\lambda))(b) = \tau(\lambda(b))$ for all $\tau \in \Gamma_F$ and hence $$ \val((\tau(\lambda))(b)) = \val(\lambda(b))$$ for all $\lambda \in X^*(T) \otimes \mathbb{Q}$ and $\tau \in \Gamma_F$. Hence, we can conclude that $\val(\overline{\lambda}(b)) = \val (\lambda(b))$. Since $\overline{\lambda} \in X^*(T)^{\Gamma_F} \otimes \mathbb{Q}$, it is a $\mathbb{Q}$-linear combination of characters of $T$ which are $\Gamma_F$-invariant and using Lemma \ref{slopetori} it follows that $\val(\overline{\lambda}(b)) = \langle \nu_T(b), \overline{\lambda} \rangle$. Combining all this, it now follows that $$\val(\lambda (b)) = \langle \nu_T(b), \lambda \rangle$$ for all $\lambda \in X^*(T) \otimes \mathbb{Q}$.
\end{proof}

\begin{R}\label{valQextension}
    In \ref{slopetorispecialcase}, the quantity $\val(\lambda(b))$ only makes sense for $\lambda \in X^*(T)$. It is extended $\mathbb{Q}$-linearly to all $\lambda \in X^*(T) \otimes \mathbb{Q}$, and it is this extension which appears in the statement.
\end{R}

\subsubsection{Basic isocrystals}
We will need the following additional notation.
\begin{itemize}
    \item Let $P$ denote a standard parabolic subgroup of $G_0$ over $F$. Let $M_P$ be the standard Levi subgroup of $P$ and let $N_P$ denote the unipotent radical of $P$. Let $A_P$ be a maximal split torus in the center of $M_P$.
    \item Let $X(M_P)$ denote the group of characters of $M_P$ defined over $F$.
\end{itemize}

\begin{const}\label{paraboliclemma}
     For standard parabolic subgroups $P_1 \subset P_2$, there are maps $X^*(A_{P_1}) \xrightarrow[]{} X^*(A_{P_2})$ and $X^*(A_{P_2}) \otimes \mathbb{Q} \xrightarrow[]{} X^*(A_{P_1}) \otimes \mathbb{Q}$ constructed as follows. The map $X^*(A_{P_1}) \xrightarrow[]{} X^*(A_{P_2})$ is obtained as a restriction map with respect to the inclusion $A_{P_2} \subset A_{P_1}$ over $F$. Note that the restriction map $X(M_P) \xrightarrow[]{} X^*(A_P)$ is injective and has a finite cokernel. Hence, there is an identification $X(M_P) \otimes \mathbb{Q} \simeq X^*(A_P)\otimes \mathbb{Q}$. The inclusion $M_{P_1} \subset M_{P_2}$ over $F$ induces the restriction map $X^*(A_{P_2}) \otimes \mathbb{Q} \xrightarrow[]{} X^*(A_{P_1}) \otimes \mathbb{Q}$, using the above identification.
\end{const}

\begin{D}[Parabolic contracted by an isocrystal]\label{paracontractediso}
    With the notation as above, let $\mathfrak{U}_P := X_*(A_P) \otimes \mathbb{R}$. For $P_1 \subset P_2$, consider the map $\mathfrak{U}_{P_2} \xrightarrow[]{} \mathfrak{U}_{P_1}$, obtained by dualizing the corresponding map in \ref{paraboliclemma}. From the inclusion $A_P \subset A$, we obtain the inclusion $\mathfrak{U}_P \xrightarrow[]{} \mathfrak{U}$. We identify $\mathfrak{U}_P$ with its image under this map and thus we view $\mathfrak{U}_P$ as a subspace of $\mathfrak{U}$. Let $\mathfrak{U}_P^+ = \{ x \in \mathfrak{U}_P \mid \langle \alpha , x \rangle > 0 \text{ for every positive root } \alpha \text{ of } A_P \text{ in } \Lie(N_P)\}$. Then
    $$ \overline{C} = \coprod_P \mathfrak{U}_P^+.$$
    The parabolic subgroup $P_b$ contracted by $b \in B(G)$ is the unique standard parabolic subgroup of $G_0$ over $F$ such that the Newton point $\overline{\nu}_G(b)$ lies in $\mathfrak{U}_{P_b}^+$.
\end{D}

\begin{D}[Basic isocrystals]\label{basicisocrystaldef}
    The set $B(G)_b$ of \textit{basic} isocrystals is the subset of those elements of $B(G)$ for which the slope morphism $\nu_G(b) \in \Hom_L(\mathbb{D}, G)$ of a representative $b \in G(L)$ factors through the center $Z(G)$ of $G$. This condition does not depend on the choice of the representative. It implies that $\nu_G(b) \in \Hom_F(\mathbb{D}, Z(G))$, and hence basic isocrystals are precisely the elements $b \in B(G)$ for which $P_b = G_0$.
\end{D}

\begin{Le}\label{basicandcohomlemma}
    The composition
    $$H^1(F,G) \xrightarrow[]{} B(G) \xrightarrow[]{} (\Hom_L(\mathbb{D}, G)/G(L))^{\langle \sigma \rangle}$$
    is trivial. Hence, there is a map $H^1(F,G) \xrightarrow[]{} B(G)_b$.
\end{Le}
\begin{proof}
    See $\S$ 3.2 of \cite{MR1485921}.
\end{proof}

\begin{Le}\label{leviofcontracted}
\begin{enumerate}
    \item For a $G$-isocrystal $b \in G(L)$, there is an isomorphism $u$ over $L$ from $J_{b,L}$ to the centralizer in $G_L$ of $\nu_G(b) \in \Hom_L(\mathbb{D}, G)$, where the group $J_b$ was constructed in \ref{J}.
    \item Let $\psi: G_{\overline{F}} \xrightarrow[]{\sim} G_{0,\overline{F}}$ be an inner twisting such that $\psi \circ \nu_G(b) \in \Hom_{\overline{F}}(\mathbb{D}, G_0)$ factors through $A$, and is thus defined over $F$. The centralizer of $\psi \circ \nu_G(b)$ in $G_0$ is $M_{P_b}$, where $P_b$ is the parabolic contracted by $b$. Then $J_b$ is an inner form of $M_{P_b}$, with $\psi \circ u$ an inner twisting from $J_b$ to $M_{P_b}$.
\end{enumerate} 
\end{Le}
\begin{proof}
    See $\S$ 3.3 of \cite{MR1485921} for the first part and $\S$ 4.3 of \cite{MR1485921} for the second part.
\end{proof}

\begin{Le}\label{characterizationJ}
    For a basic $G$-isocrystal $b \in G(L)$, the group $J_b$ and the isomorphism $u$ are characterized uniquely by the following conditions:
    \begin{enumerate}
        \item $J_b$ is an inner form of $G$
        \item $u$ is an isomorphism $J_{b,L} \simeq G_L$ over $L$
        \item $u(\sigma(x)) = b \cdot \sigma(u(x)) \cdot b^{-1}$
    \end{enumerate}
\end{Le}
\begin{proof}
    See $\S$ 5.2 of \cite{MR809866}.
\end{proof}

\subsubsection{Tate-Nakayama duality}
Tate-Nakayama duality plays an important role in the discussion of the Kottwitz map in the theory of isocrystals, and it will also come up later in the theory of endoscopy. Before stating this duality, we recall the definition of a dual torus.
\begin{D}[Dual torus]\label{dualtorusdef}
    For a torus $T$ over $F$, denote by $\Hat{T}$ the dual torus over $\mathbb{C}$ defined by the property $X^*(\Hat{T}) = X_*(T)$ and $X_*(\Hat{T}) = X^*(T)$. The $\Gamma_F$-action on $X_*(T) = X^*(\Hat{T})$ induces an action of $\Gamma_F$ on $\Hat{T}$, acting via automorphisms over $\mathbb{C}$.
\end{D}

We will use the following notation.
\begin{itemize}
    \item Write $H^0(F, X^*(T))^c$ for the completion of $H^0(F,X^*(T))$ in the topology of subgroups of finite index.
    \item Write $H^0(F,T)^c$ for the completion of $H^0(F,T) = T(F)$ in the topology of open subgroups of finite index.
\end{itemize}

The pairing $X^*(T) \times T(\overline{F}) \xrightarrow[]{} \overline{F}^{\times}$ induces, via cup products, the following pairing for $0 \leq r \leq 2$:
$$H^r(F,X^*(T)) \times H^{2-r}(F,T) \xrightarrow[]{\cup} H^2(F, \mathbb{G}_m) = \mathbb{Q}/\mathbb{Z}.$$
Note that we use the following facts: $F$ has cohomological dimension $2$, that is, for any $\Gamma_F$-module $A$, we have $H^r(F,A) = 0$ for all $r \geq 3$, and $H^2(F,\mathbb{G}_m) = \Br(F) = \mathbb{Q}/\mathbb{Z}$, where $\Br(F)$ denotes the Brauer group of $F$.

\begin{Le}[Tate-Nakayama duality for tori]\label{tatenakayamatoriduality}
With respect to this pairing, there is a duality of topological groups as follows.
    \begin{enumerate}
        \item The compact group $H^0(F, X^*(T))^c$ and the discrete group $H^2(F,T)$ are dual.
        \item The finite groups $H^1(F,X^*(T))$ and $H^1(F,T)$ are dual.
        \item The discrete group $H^2(F,X^*(T))$ and the compact group $H^0(F,T)^c$ are dual.
    \end{enumerate}
\end{Le}

\begin{Le}[Tate-Nakayama duality reformulation]\label{Tatenakayamatori}
    There is a canonical isomorphism of abelian groups $H^1(F,T) \simeq X^*(\pi_0(\Hat{T}^{\Gamma_F}))$, where $\pi_0(-)$ denotes the group of connected components. This isomorphism extends to an isomorphism of the functors $H^1(F,-)$ and $X^*(\pi_0(\Hat{(-)}^{\Gamma_F}))$ from the category of tori over $F$ to abelian groups.
\end{Le}
\begin{proof}
    Note that $H^1(F,X^*(T)) = H^1(F,X_*(\Hat{T}))$. There is a short exact sequence
$$ 0 \xrightarrow[]{} X_*(\Hat{T}) \xrightarrow[]{} X_*(\Hat{T}) \otimes \mathbb{C} \xrightarrow{} X_*(\Hat{T}) \otimes \mathbb{C}/\mathbb{Z} \xrightarrow[]{} 0.$$
The corresponding long exact sequence in Galois cohomology is as follows:
$$ \cdots \xrightarrow[]{} (X_*(\Hat{T}) \otimes \mathbb{C})^{\Gamma_F} \xrightarrow[]{} (X_*(\Hat{T}) \otimes \mathbb{C}/\mathbb{Z})^{\Gamma_F} \xrightarrow[]{} H^1(F, X_*(\Hat{T})) \xrightarrow[]{} 0 \xrightarrow[]{} \cdots$$
using the fact that $H^1(F, X_*(\Hat{T}) \otimes \mathbb{C}) = 0$. Here $M := X_*(\Hat{T}) \otimes \mathbb{C}$ is a $\mathbb{Q}$-vector space on which $\Gamma_F$ acts through the finite quotient $\Delta = \Gal(E/F)$, for $E/F$ a finite extension splitting $T$. The inflation-restriction sequence for the normal subgroup $\Gamma_E \subset \Gamma_F$ reads
$$0 \xrightarrow[]{} H^1(\Delta, M^{\Gamma_E}) \xrightarrow[]{\mathrm{inf}} H^1(F, M) \xrightarrow[]{\mathrm{res}} H^1(\Gamma_E, M)^{\Delta}.$$
Since $\Gamma_E$ acts trivially on $M$, the left hand term is $H^1(\Delta, M)$, and the right hand term is $\Hom_{\text{cont}}(\Gamma_E, M)^{\Delta} = 0$, because a continuous homomorphism from a profinite group to the torsion free group $M$ has finite image and is therefore trivial. Hence inflation identifies $H^1(F, M)$ with $H^1(\Delta, M)$, and the latter vanishes because taking $\Delta$-invariants is exact on $\mathbb{Q}[\Delta]$-modules: the averaging operator $\mid \Delta \mid^{-1} \sum_{\delta \in \Delta} \delta$ is a functorial projector onto the invariants, so $(-)^{\Delta}$ is a direct summand of the identity functor and its higher derived functors vanish. Hence
$$H^1(F, X^*(T)) = \coker ((X_*(\Hat{T}) \otimes \mathbb{C})^{\Gamma_F} \xrightarrow[]{} (X_*(\Hat{T}) \otimes \mathbb{C}/\mathbb{Z})^{\Gamma_F}).$$
Note the following.
\begin{enumerate}
    \item We have $\Lie(\Hat{T}) = X_*(\Hat{T}) \otimes \mathbb{C}$, and hence $\Lie(\Hat{T}^{\Gamma_F}) = (X_*(\Hat{T}) \otimes \mathbb{C})^{\Gamma_F}$. The second identity says that passing to invariants commutes with $\Lie$, which holds because the action factors through the finite group $\Delta$ and taking $\Delta$-invariants is exact, as above.
    \item We have $\Hat{T}(\mathbb{C}) = X_*(\Hat{T}) \otimes \mathbb{C}^{\times}$. Using the exponential map $\exp: z \mapsto e^z$, we get an isomorphism $\mathbb{C}/\mathbb{Z} \xrightarrow[]{\sim} \mathbb{C}^{\times}$, where $\mathbb{Z} \xrightarrow[]{} \mathbb{C}$ is $n \mapsto 2\pi ni$. Hence $\Hat{T}^{\Gamma_F}(\mathbb{C}) = (X_*(\Hat{T}) \otimes \mathbb{C}/\mathbb{Z})^{\Gamma_F}$.
\end{enumerate}
Using this, it follows that
$$\coker ((X_*(\Hat{T}) \otimes \mathbb{C})^{\Gamma_F} \xrightarrow[]{} (X_*(\Hat{T}) \otimes \mathbb{C}/\mathbb{Z})^{\Gamma_F}) = \coker (\Lie(\Hat{T}^{\Gamma_F}) \xrightarrow[]{\exp} \Hat{T}^{\Gamma_F}).$$
Hence, we have shown that $H^1(F, X^*(T)) = \pi_0(\Hat{T}^{\Gamma_F})$, and the result now follows from part (2) of \ref{tatenakayamatoriduality}.
\end{proof}

\subsubsection{Kottwitz map}
To define the Kottwitz map, we first need to discuss the characterization of $B(T)$ for a torus $T$. We will also discuss the compatibility of the Kottwitz map with the Tate-Nakayama isomorphism.

\begin{Le}\label{isocrystalsfortori}
    There is a canonical isomorphism of pointed sets $B(T) \simeq X_*(T)_{\Gamma_F}$, satisfying the conditions:
    \begin{enumerate}
        \item The map $B(\mathbb{G}_m) \xrightarrow[]{} X_*(\mathbb{G}_m)_{\Gamma_F} = \mathbb{Z}$ sends a $\sigma$-conjugacy class in $L^{\times}$ to the normalized valuation of any of its elements.
        \item The isomorphism extends to an isomorphism of the functors $B(-)$ and $X_*(-)_{\Gamma_F}$ from the category of tori over $F$ to pointed sets.
    \end{enumerate}  
\end{Le}
\begin{proof}
    See $\S$ 2.4 of \cite{MR809866}.
\end{proof}

\begin{C}\label{galoisaverage}
    The slope morphism for tori, discussed in \ref{slopetori}, admits the alternative description as the map $B(T) \simeq X_*(T)_{\Gamma_F} \xrightarrow[]{} X_*(T)^{\Gamma_F}\otimes {\mathbb{Q}}$ given by $\overline{\mu} \mapsto \mid \Gamma_F \cdot \mu \mid ^{-1} \sum_{\mu' \in \Gamma_F \cdot \mu} \mu'$.
\end{C}
\begin{proof}
    Let $b \in B(T)$ correspond to $\overline{\mu} \in X_*(T)_{\Gamma_F}$ and let $\lambda \in X^*(T)^{\Gamma_F}$, so that $\lambda: T \xrightarrow[]{} \mathbb{G}_m$ is defined over $F$. Applying the functoriality in \ref{isocrystalsfortori} to $\lambda$, together with the normalization for $\mathbb{G}_m$ given there, yields $\val(\lambda(b)) = \langle \mu, \lambda \rangle$. Let $\nu$ denote the average of $\mu$ over its $\Gamma_F$-orbit; it lies in $X_*(T)^{\Gamma_F} \otimes \mathbb{Q}$ and depends only on $\overline{\mu}$, since averaging annihilates $(\tau - 1)X_*(T)$ for every $\tau \in \Gamma_F$. As the pairing is $\Gamma_F$-equivariant and $\lambda$ is $\Gamma_F$-invariant, we get $\langle \nu, \lambda \rangle = \langle \mu, \lambda \rangle = \val(\lambda(b))$, so $\nu$ has the property characterizing $\nu_T(b)$ in \ref{slopetori}. The uniqueness asserted there gives $\nu = \nu_T(b)$.
\end{proof}

We need the following notions associated with dual groups for the rest of the discussion.
\begin{itemize}
    \item For a connected reductive group $G$ over $F$, we denote by $\Hat{G}$ the connected reductive group over $\mathbb{C}$ whose root datum is dual to the root datum of $G$.
    \item Choosing a pinning of $\Hat{G}$ gives an action of $\Gamma_F$ on $\Hat{G}$, acting via automorphisms over $\mathbb{C}$.
    \item The restriction of the action of $\Gamma_F$ on $\Hat{G}$ to the center $Z(\Hat{G})$ is independent of the choice of the pinning on $\Hat{G}$ and hence the diagonalizable group $Z(\Hat{G})^{\Gamma_F}$ is well defined.
\end{itemize}

We will also introduce the algebraic fundamental group for a connected reductive group $G$, which will later be used in relation with the Kottwitz map.

\begin{D}[$\pi_1(G)$]
    The algebraic fundamental group $\pi_1(G)$ of a connected reductive group $G$ is the quotient of the cocharacter lattice by the coroot lattice, that is, $\pi_1(G) = X_*(T) / \sum_{\alpha \in \Phi(G,T)} \mathbb{Z} \Check{\alpha}$, where $\Phi(G,T)$ denotes the set of roots of $G$ with respect to a maximal torus $T$. Then $\pi_1(G)$ is an abelian group with a $\Gamma_F$-action, and is independent of the choice of $T$.
\end{D}

\begin{R}\label{observation}
    There are canonical identifications between the following groups:  
    \begin{itemize}
        \item $\pi_1(G) = X^*(Z(\Hat{G}))$.
        \item $\pi_1(G)_{\Gamma_F} = X^*(Z(\Hat{G})^{\Gamma_F})$, for the $\Gamma_F$-coinvariants.
        \item $(\pi_1(G)_{\Gamma_F})_{\text{tors}} \simeq X^*(\pi_0(Z(\Hat{G})^{\Gamma_F}))$, for the torsion subgroup.
    \end{itemize}
\end{R}

\begin{Le}[Kottwitz map]\label{kottwitzmap}
    There is a canonical map of pointed sets
    $$\kappa_G: B(G) \xrightarrow[]{} X^*(Z(\Hat{G})^{\Gamma_F}),$$
    which extends to a morphism between the functors $B(-)$ and $X^*(Z(\Hat{(-)})^{\Gamma_F})$ from the category of connected reductive groups to pointed sets. This morphism extends the isomorphism \ref{isocrystalsfortori} of functors for tori.
\end{Le}
\begin{proof}
    The map is constructed in Lemma 6.1 of \cite{MR1044820}. However, only functoriality with respect to normal homomorphisms of algebraic groups is discussed there, since the functoriality of $G \mapsto X^*(Z(\Hat{G})^{\Gamma_F})$ beyond this case is not clear, a priori. To obtain functoriality for all homomorphisms, we need to use the algebraic fundamental group $\pi_1(G)$ of $G$. The assignment $G \mapsto \pi_1(G)$ is functorial for all homomorphisms of connected reductive groups, and combining this with Remark \ref{observation}, the functoriality now follows. For more details see $\S$ 1.13 and $\S$ 1.14 of \cite{MR1411570}.
\end{proof}

\begin{C}\label{classificationofbasic}
    The map $\kappa_G$ of \ref{kottwitzmap} restricts to an isomorphism
    $$\kappa_G: B(G)_b \xrightarrow[]{\sim} X^*(Z(\Hat{G})^{\Gamma_F}).$$
\end{C}
\begin{proof}
    See Proposition 5.6 of \cite{MR809866}.
\end{proof}

\begin{Le}[Kottwitz's generalization of Tate-Nakayama duality]\label{KottwitzTateNakayama}
    There is a canonical isomorphism of pointed sets $\kappa_G: H^1(F,G) \simeq X^*(\pi_0(Z(\Hat{G})^{\Gamma_F}))$ which is a restriction of the isomorphism $\kappa_G: B(G)_b \xrightarrow[]{} X^*(Z(\Hat{G})^{\Gamma_F})$ to $H^1(F,G)$ with respect to the inclusion discussed in \ref{cohomisocrystal}. This extends to a functorial isomorphism from the category of connected reductive groups to pointed sets. This extends the functorial isomorphism \ref{Tatenakayamatori} for tori.
\end{Le}
\begin{proof}
    See Proposition 6.4 of \cite{MR0757954} and $\S$ 5.7 of \cite{MR809866}.
\end{proof}

Note that via the isomorphism $H^1(F,G) \simeq X^*(\pi_0(Z(\Hat{G})^{\Gamma_F}))$ one can view $H^1(F,G)$ as an abelian group, which is a priori just a pointed set.

\subsubsection{Compatibility between the Newton map and Kottwitz map}
The Newton point $\overline{\nu}_G(b)$ and the Kottwitz point $\kappa_G(b)$ characterize an isocrystal $b \in B(G)$, and the following result also gives a relation between $\overline{\nu}_G(b)$ and $\kappa_G(b)$.

\begin{const}[$\delta_G: \mathcal{N}(G) \rightarrow \pi_1(G)^{\Gamma_F} \otimes \mathbb{Q}$]
    There is a map $\delta_G: \mathcal{N}(G) \simeq \overline{C}_{\mathbb{Q}} \subset \mathfrak{U}_{\mathbb{Q}} \xrightarrow[]{} \pi_1(G)^{\Gamma_F}\otimes \mathbb{Q}$ which is constructed as follows. Recall that $A_G$ denotes the maximal $F$-split torus in the center of $G$, and that $A$ denotes the maximal split torus in $G_0$. Note that the center of $G$ is isomorphic over $F$ to the center of $G_0$. A cocharacter of $A_G$ determines a homomorphism $\Hat{G} \xrightarrow[]{} \mathbb{C}^{\times}$, which can be restricted to $Z(\Hat{G})^{\Gamma_F}$. This induces an isomorphism $X_*(A_G) \otimes \mathbb{Q} \xrightarrow[]{} X^*(Z(\Hat{G})^{\Gamma_F}) \otimes \mathbb{Q}$. Precomposing this with the map $X_*(A) \otimes \mathbb{Q} \xrightarrow[]{} X_*(A_G) \otimes \mathbb{Q}$ of \ref{paraboliclemma}, we obtain the map $\delta_G: \mathcal{N}(G) \xrightarrow[]{} \pi_1(G)^{\Gamma_F} \otimes \mathbb{Q}$.
\end{const}

\begin{Le}\label{Newonkottwitzcharisoc}
    The map $\overline{\nu}_G \times \kappa_G:  B(G) \xrightarrow[]{} \mathcal{N}(G) \times X^*(Z(\Hat{G})^{\Gamma_F})$ is injective. The following square commutes:
\begin{center}
\begin{tikzcd}
    B(G) \arrow{r}{\overline{\nu}_G} \arrow{d}{\kappa_G} & \mathcal{N}(G) \arrow{d}{\delta_G} \\
    \pi_1(G)_{\Gamma_F} \arrow{r} & \pi_1(G)^{\Gamma_F} \otimes \mathbb{Q}
\end{tikzcd}
\end{center}
Note that the bottom map is the natural map $\pi_1(G)_{\Gamma_F} \xrightarrow[]{} \pi_1(G)_{\Gamma_F} \otimes \mathbb{Q} \simeq \pi_1(G)^{\Gamma_F} \otimes \mathbb{Q}$.
\end{Le}
\begin{proof}
    See $\S$ 4.13 of \cite{MR1485921} for the first part. For the second part, see Theorem 1.15 (iii) of \cite{MR1411570}.
\end{proof}

\subsubsection{Isocrystals and inner forms}
If $G$ and $G'$ are inner forms of each other, there are natural questions which explore the relationship between $B(G)$ and $B(G')$. The following is a result in this direction in the special case $G' = J_b$, for $b \in B(G)_b$.

\begin{Le}\label{isocrystalsandinnerforms}
    Let $b \in G(L)$ be a basic $G$-isocrystal. Let $J_b$ denote the group constructed in \ref{J} and let $u: J_{b,L} \simeq G_L$ be the isomorphism of \ref{characterizationJ}. Then the assignment $b' \mapsto u(b')b$, for $b' \in J_b(L)$, induces a map $B(J_b) \xrightarrow[]{\times b} B(G)$, since $$h b' \sigma(h)^{-1} \mapsto u(h) u(b')u(\sigma(h)^{-1})b = u(h)u(b')b\sigma(u(h)^{-1})$$ for all $h \in J_b(L)$, using condition (3) of \ref{characterizationJ}. Since $J_b$ is an inner form of $G$, the two groups share the quasi-split form $G_0$ and the dual group $\Hat{G}$, so the vertical maps below have the same targets for $J_b$ as for $G$. This map has the following property with respect to the Newton map:
    \begin{center}
        \begin{tikzcd}
            B(J_b) \arrow{rr}{\times b} \arrow{d}{\overline{\nu}_{J_b}} && B(G) \arrow{d}{\overline{\nu}_G} \\
            \mathfrak{U} \arrow{rr}{+ \overline{\nu}_G(b)} && \mathfrak{U}
        \end{tikzcd}
    \end{center}
    and the following property with respect to the Kottwitz map:
    \begin{center}
        \begin{tikzcd}
            B(J_b) \arrow{rr}{\times b} \arrow{d}{\kappa_{J_b}} && B(G) \arrow{d}{\kappa_G} \\
            X^*(Z(\Hat{G})^{\Gamma_F}) \arrow{rr}{+ \kappa_G(b)} && X^*(Z(\Hat{G})^{\Gamma_F})
        \end{tikzcd}
    \end{center}
\end{Le}
\begin{proof}
    See $\S$ 4.18 of \cite{MR1485921} for the property with respect to the Newton map and $\S$ 4.12 of \textit{loc. cit.} for the property with respect to the Kottwitz map. Note that $\overline{\nu}_G(b) \in X_*(A_G) \otimes \mathbb{Q}$, since $b$ is basic, and translation by this vector on $\mathfrak{U} = X_*(A)\otimes \mathbb{R}$ preserves the cone $\overline{C}_\mathbb{Q}$.
\end{proof}

\subsubsection{Isocrystals and unipotent groups}

Before we discuss isocrystals in the context of unipotent groups, we will first need to define the set $B(G)$ in greater generality than for connected linear algebraic groups. Hence, for this subsubsection, we assume that
\begin{itemize}
    \item Let $G$ be a linear algebraic group over $F$.
\end{itemize}
The theory of isocrystals in this generality is worked out in \cite{MR1485921}. We recall the setup.

\begin{itemize}
    \item Let $F$ denote a $p$-adic field and let $\sigma$ denote the arithmetic Frobenius automorphism of $F^{\text{un}}$ over $F$.
    \item Let $W_F$ denote the Weil group of $\overline{F}$ over $F$, that is, the subgroup of $\Gamma_F$ consisting of elements whose restriction to $F^{\text{un}}$ is an integral power of $\sigma$.
\end{itemize}

\begin{D}[$B(G)$ for linear algebraic groups $G$]\label{generalBG}
    Let $G$ be a linear algebraic group over a $p$-adic field $F$. Let $L$ be the completion of $F^{\text{un}}$ as before. We can regard elements of $W_F$ as automorphisms of $\overline{L}$ over $F$. We define
    $$ B(G) = H^1(W_F, G(\overline{L})).$$
\end{D}

\begin{Le}\label{generalBGagrees}
    For a connected linear algebraic group $G$ over $F$, the set $B(G)$ of \ref{generalBG} agrees with the set of $\sigma$-conjugacy classes in $G(L)$ of \ref{B(G)def}.
\end{Le}
\begin{proof}
    See $\S$ 1.4 of \cite{MR1485921}.
\end{proof}

Let $P$ be a linear algebraic group over $F$, that is, $P$ admits an embedding $P \hookrightarrow \GL_n$ over $F$ for some $n$ (note that $P$ is not necessarily reductive). Let $U$ be the unipotent radical of $P$ and let $M$ be a Levi factor in $P$, so that $P=MU$. Then the following holds.

\begin{Le}\label{unipotentlemma}
    The natural map $B(P) \xrightarrow[]{} B(P/U)$ is a bijection.
\end{Le}
\begin{proof}
    See $\S$ 3.6 of \cite{MR1485921}.
\end{proof}

\subsubsection{Image of the Newton map in the quasi-split case}\label{integralnewtonpoints}
Let $G$ be quasi-split (hence $G = G_0$ in our previous notation). By definition, $\mathcal{N}(G)_{\mathbb{Z}}$ denotes the image of the map $\overline{\nu}_G: B(G) \xrightarrow[]{} \mathcal{N}(G) = \overline{C}_{\mathbb{Q}} \subset \mathfrak{U}$. The following result gives an explicit description of $\mathcal{N}(G)_{\mathbb{Z}}$ as a subset of $\mathfrak{U}$ in terms of the root datum of $G$.

\begin{Le}\label{chailemma}
    For a quasi-split group $G$, we have the following description of $\mathcal{N}(G)_{\mathbb{Z}}$. We start by using Proposition 6.2 of \cite{MR809866}, which states that every element of $B(G)$ lies in the image of $B(M)_b \xrightarrow[]{} B(G)$ for some standard Levi subgroup $M$. Recall from \ref{classificationofbasic} the isomorphism $B(M)_b \simeq X^*(Z(\Hat{M})^{\Gamma_F})$. From the proof of \ref{Newonkottwitzcharisoc}, it follows that there is a map 
    \begin{equation*}
        \phi_M: X^*(Z(\Hat{M})^{\Gamma_F}) \xrightarrow[]{} X^*(Z(\Hat{M})^{\Gamma_F}) \otimes \mathbb{R} \simeq X_*(A_P) \otimes \mathbb{R} = \mathfrak{U}_P
    \end{equation*}
    where $P$ denotes the standard parabolic subgroup with Levi factor $M$, and we are using the notation introduced in \ref{paracontractediso}. Hence, combining these observations, we are reduced to understanding $\im(\phi_M) \subset \mathfrak{U}_P \subset \mathfrak{U}$.
    
    We have the identifications $$X^*(Z(\Hat{M})^{\Gamma_F}) = \pi_1(M)_{\Gamma_F} = \left( X_*(T) \bigg/ \sum_{\alpha \in \Phi(M,T)} \mathbb{Z} \Check{\alpha} \right)_{\Gamma_F}$$
    where we recall that $T$ is a maximal torus of $G$ defined over $F$, namely the centralizer of a maximal split torus $A$ of $G$, and in particular is contained in a Borel subgroup $B$ over $F$. The image $\im(\phi_M)$ can be described as the image of 
    $$\pr^{\Gamma_F} \circ \pr^{\Omega_M(\overline{F})} = \pr^{\Omega_M(\overline{F})} \circ \pr^{\Gamma_F}: X_*(T) \xrightarrow[]{} X_*(A_P) \otimes \mathbb{Q}$$
    where $\pr^{\Gamma_F}$ and $\pr^{\Omega_M(\overline{F})}$ are the averaging operators on $X_*(T)$ with respect to $\Gamma_F$ and $\Omega_M(\overline{F})$, for the natural actions of these groups on $X_*(T)$, the Galois group acting through a finite quotient. This is the same as the composition 
    $$X_*(T) \xrightarrow{\pr^{\Gamma_F}} X_*(A) \otimes \mathbb{Q} \xrightarrow[]{} X_*(A_P) \otimes \mathbb{Q}$$
    where the second map is the one constructed in \ref{paraboliclemma}. Let $\overline{C}_P^+$ denote the intersection of $\im(\phi_M)$ with $\mathfrak{U}_P^+$. Then $$\mathcal{N}(G)_{\mathbb{Z}} = \bigcup_P \overline{C}_P^+.$$
\end{Le}
\begin{proof}
    For more details, see $\S$ 3 of \cite{MR1781927}.
\end{proof}

\subsection{Review of endoscopy}\label{2.2}
\subsubsection{Setup}
We will use the following notation for the rest of this subsection.
\begin{itemize}
    \item Let $F$ be a $p$-adic field or a number field.
    \item Let $G$ be a connected reductive group over $F$, with a maximal torus $T$ and a Borel subgroup $B$, both defined over $\overline{F}$.
    \item The adjoint action of $T$ on $\mathfrak{g} = \Lie (G)$ is completely reducible. The (finite) subset of roots $\Phi(G,T) \subset X^*(T)$ is defined by the condition
    $$\det (xI - \Ad_G(t) \mid  \mathfrak{g}) = (x-1)^r \prod_{a \in \Phi(G,T)} (x - a(t))$$
    for all $t \in T$, where $r = \dim T$ is the rank of $G$.
    \item There is a weight space decomposition $\mathfrak{g} = \mathfrak{g}_0 \oplus \left( \bigoplus_{a \in \Phi(G,T)} \mathfrak{g}_a \right)$, where $\mathfrak{g}_0 = \mathfrak{t} = \Lie(T)$ and the weight space $\mathfrak{g}_a$ is the one dimensional subspace of $\mathfrak{g}$ on which $T$ acts by $a$.
    \item For any $a \in \Phi(G,T)$ there exists a homomorphism $\phi_a: \SL_2 \xrightarrow{} G$ carrying the diagonal torus $D$ into $T$ and the strictly upper triangular and strictly lower triangular unipotent subgroups $U^{\pm}$ isomorphically onto the respective root groups $U_{\pm a}$; here the root group $U_a$ is the unique subgroup of $G$ with $U_a \simeq \mathbb{G}_a$ and $\Lie(U_a) = \mathfrak{g}_a$. Moreover, such a homomorphism is unique up to $T$-conjugation (see Theorem 1.2.7 of \cite{ConradSGA3}).
    \item For $a \in \Phi(G,T)$ we define the associated coroot $\Check{a} \in X_*(T)$ by
    $$\Check{a}(c) = \phi_a \left( \begin{pmatrix} c & 0 \\ 0 & c^{-1} \end{pmatrix} \right)$$
    which is independent of the choice of $\phi_a$ and hence is well defined. This defines the finite subset of coroots $\Check{\Phi}(G,T) \subset X_*(T)$.
    \item We write $\Psi(G,T) = (X^*(T), \Phi(G,T), X_*(T), \Check{\Phi}(G,T))$ for the root datum of $G$.
    \item We write $\Psi_0(G,B,T) = (X^*(T), \Phi_0(G,B,T), X_*(T), \Check{\Phi}_0(G,B, T))$ for the based root datum of $G$.
    \item A pinning of a connected reductive group $G$ defined over an algebraically closed field $k$ is a triple $(B,T, \{ X_a \}_{a \in \Phi_0(G,B,T)})$, where $(B,T)$ is a Borel pair and $X_a \in \mathfrak{g}_a$ is a non-zero vector for each $a \in \Phi_0(G,B,T)$. Giving $X_a$ is equivalent to giving an isomorphism $p_a: U_a \simeq \mathbb{G}_a$, and in turn to giving a homomorphism $\phi_a: \SL_2 \xrightarrow{} G$ as above.
\end{itemize}

\subsubsection{Langlands dual groups} 
There are different ways to define Langlands dual groups in the literature, and we start by specifying the definition we will be using in this paper.

\begin{D}[Langlands dual group]\label{Lgroupdef}
    A Langlands dual group for a connected reductive group $G$ over a field $F$ is the data of
    \begin{itemize}
        \item A connected reductive group $\Hat{G}$ over $\mathbb{C}$ with maximal torus $\mathscr{T}$, whose root datum $\Psi(\Hat{G}, \mathscr{T})$ is dual to $\Psi(G,T)$.
        \item An action $\Gamma_F \curvearrowright \Hat{G}$, known as the $L$-action.
    \end{itemize}
    such that the following conditions hold:
    \begin{enumerate}
    \item The action of $\Gamma_F$ on $\Hat{G}$ preserves a pinning of $\Hat{G}$ whose maximal torus is $\mathscr{T}$.
    \item The isomorphism from $\Psi(\Hat{G}, \mathscr{T})$ to the dual of $\Psi(G,T)$ is $\Gamma_F$-equivariant, where the $\Gamma_F$ action on $\Psi(\Hat{G}, \mathscr{T})$ is induced from $\Gamma_F \curvearrowright \Hat{G}$.
    \end{enumerate}
    The $L$-group ${}^L G = \Hat{G} \rtimes \Gamma_F$ is independent, up to isomorphism, of the choice of $\Gamma_F \curvearrowright \Hat{G}$. Recall, once again, that the action of $\Gamma_F$ on $Z(\Hat{G})$ is well defined, independent of the choice of the action of $\Gamma_F$ on $\Hat{G}$.
\end{D}

The following construction will be used later in the definition of admissible embeddings.

\begin{const}\label{firstconst}
    Let $T \subset G$ be a maximal torus. Denote by $\Hat{T}$ the torus over $\mathbb{C}$ such that $X^*(\Hat{T}) = X_*(T)$. There is a canonical $\Hat{G}$-conjugacy class of embeddings of complex algebraic groups $\iota_T: \Hat{T} \xrightarrow[]{} \Hat{G}$. This is constructed as follows. Let $(B_0, T_0)$ be a Borel pair of $G$, let $(\mathcal{B}_0, \mathcal{T}_0)$ be a Borel pair of $\Hat{G}$ and let $\iota: \Hat{T_0} \xrightarrow[]{} \mathcal{T}_0$ be an isomorphism such that $\Psi_0(G,B_0,T_0)$ is the dual of $\Psi_0(\Hat{G}, \mathcal{B}_0, \mathcal{T}_0)$ under the identifications $X^*(T_0) \simeq X_*(\mathcal{T}_0)$ and $X_*(T_0) \simeq X^*(\mathcal{T}_0)$ induced by $\iota$. Choose $g \in G(\overline{F})$ such that $g T_0 g^{-1} = T$, which gives the isomorphism $\Ad(g): T_0 \xrightarrow[]{} T$. Dualizing this map and composing it with $\iota$ gives an embedding $\iota_T: \Hat{T} \xrightarrow[]{} \Hat{G}$. It can be checked that the $\Hat{G}$-conjugacy class of $\iota_T$ is independent of the choice of $g$.
\end{const}

\begin{R}\label{secondconst}
    In the special case when the maximal torus $T \subset G$ is defined over $F$, the above construction has the additional property that the canonical $\Hat{G}$-conjugacy class $\iota_T: \Hat{T} \xrightarrow[]{} \Hat{G}$ is $\Gamma_F$-invariant. The group $\Gamma_F$ has an action on $\Hat{T}$ via automorphisms over $\mathbb{C}$ induced from the action of $\Gamma_F$ on $X_*(T)$, coming from the $F$-structure on $T$. The $\Gamma_F$-action on the embedding $\iota_T: \Hat{T} \xrightarrow[]{} \Hat{G}$ is given by $\sigma_G \circ \iota_T \circ \sigma_T^{-1}$, where $\sigma_T$ denotes the action of $\Gamma_F$ on $\Hat{T}$ and $\sigma_G$ denotes the action on $\Hat{G}$, well defined up to $\Int(\Hat{G})$. The $\Hat{G}$-conjugacy class of $\iota_T$ is invariant under this action.
\end{R}

\subsubsection{Endoscopic datum}
We need the following lemma before we can state the definition of an endoscopic datum (equivalently, endoscopic triple) for a connected reductive group $G$ over $F$.

\begin{Le}\label{keylemmaendo}
    Let $1 \xrightarrow[]{} D_1 \xrightarrow[]{} D_2 \xrightarrow[]{} D_3 \xrightarrow[]{} 1$ be an exact sequence of diagonalizable groups over $\mathbb{C}$ with a $\Gamma_F$-action. Then there is a long exact sequence
\begin{center}
\begin{tikzcd} 
  0 \arrow{r}
  & X_*(D_1)^{\Gamma_F} \arrow{r} 
  & X_*(D_2)^{\Gamma_F} \arrow{r} 
  & X_*(D_3)^{\Gamma_F} \\
  \arrow{r} 
  & \pi_0(D_1^{\Gamma_F}) \arrow{r} 
  & \pi_0(D_2^{\Gamma_F}) \arrow{r} 
  & \pi_0(D_3^{\Gamma_F}) \\
  \arrow{r}
  & H^1(\Gamma_F, D_1) \arrow{r}
  & H^1(\Gamma_F, D_2) \arrow{r}
  & H^1(\Gamma_F, D_3) \\
  \arrow{r}
  & H^2(\Gamma_F, D_1) \arrow{r}
  & H^2(\Gamma_F, D_2) \arrow{r}
  & H^2(\Gamma_F, D_3) \arrow{r}
  & \cdots 
\end{tikzcd}
\end{center}
\end{Le}
\begin{proof}
    See Corollary 2.3 of \cite{MR0757954}.
\end{proof}

For a connected reductive group $G$, let $\Int(G)$ denote the group of inner automorphisms of $G$ and let $\Out(G)$ denote the group of outer automorphisms of $G$.

\begin{D}[Endoscopic datum]\label{endodef}
    An endoscopic datum for $G$ is a pair $(s, \rho)$ consisting of an element $s \in \Hat{G}/Z(\Hat{G})$ and a homomorphism $\rho: \Gamma_F \xrightarrow[]{} \Out(\Hat{H})$, where $\Hat{H} = \Hat{G}_s^0$ denotes the identity component of the centralizer in $\Hat{G}$ of any lift of $s$, which does not depend on the lift, such that
    \begin{enumerate}
        \item For $\sigma \in \Gamma_F$, the element $\rho(\sigma)$ is induced by an element of the normalizer $N_{{}^L G}(\Hat{H})$ whose image under the canonical map ${}^L G \xrightarrow[]{} \Gamma_F$ is $\sigma$.
        \item The element $s \in Z(\Hat{H})/Z(\Hat{G})$ is fixed by $\Gamma_F$, and its image in $\pi_0([Z(\Hat{H})/Z(\Hat{G})]^{\Gamma_F})$ maps to a locally trivial element of $H^1(F, Z(\Hat{G}))$ in the long exact sequence of \ref{keylemmaendo} associated with $1 \xrightarrow[]{} Z(\Hat{G}) \xrightarrow[]{} Z(\Hat{H}) \xrightarrow[]{} Z(\Hat{H}) / Z(\Hat{G}) \xrightarrow[]{} 1$.
    \end{enumerate}
    There is a notion of isomorphism for endoscopic data $(s_1, \rho_1)$ and $(s_2, \rho_2)$ as given in $\S$ 7.2 of \cite{MR0757954}.
\end{D}

\begin{D}[$\mathfrak{K}(H/F)$]
The set of such elements of $\pi_0([Z(\Hat{H})/Z(\Hat{G})]^{\Gamma_F})$ mapping to a locally trivial element of $H^1(F, Z(\Hat{G}))$ in the long exact sequence of \ref{keylemmaendo} associated with $1 \xrightarrow[]{} Z(\Hat{G}) \xrightarrow[]{} Z(\Hat{H}) \xrightarrow[]{} Z(\Hat{H}) / Z(\Hat{G}) \xrightarrow[]{} 1$ is denoted by $\mathfrak{K}(H/F)$.
\end{D}

\begin{R}
    The choice of a pinning on $\Hat{H}$, together with $\rho$, gives an action of $\Gamma_F$ on $\Hat{H}$. Note that although $\Hat{H} \subset \Hat{G}$, this $\Gamma_F$-action on $\Hat{H}$ is not necessarily the restriction to $\Hat{H}$ of the $\Gamma_F$-action on $\Hat{G}$. However, it can be checked that on the common subgroup $Z(\Hat{G})$, the two $\Gamma_F$-actions obtained by restriction from $\Hat{G}$ and from $\Hat{H}$ coincide. This is what makes the map $Z(\Hat{G}) \xrightarrow[]{} Z(\Hat{H})$ used in \ref{endodef} $\Gamma_F$-equivariant.
\end{R}

\begin{Le}[Endoscopic triples]\label{endotriples}
    An endoscopic triple for $G$ is a datum $(H,s,\eta)$ consisting of a quasi-split connected reductive group $H$ over $F$ (referred to as the endoscopic group), an element $s \in Z(\Hat{H})$ and an embedding $\eta: \Hat{H} \xrightarrow[]{} \Hat{G}$ over $\mathbb{C}$ such that
    \begin{enumerate}
        \item $\eta (\Hat{H}) = \Hat{G}^0_{\eta(s)}$.
        \item The $\Hat{G}$-conjugacy class of $\eta$ is fixed by $\Gamma_F$, where the $\Gamma_F$-action on the $\Hat{G}$-conjugacy class of $\eta: \Hat{H} \xrightarrow[]{} \Hat{G}$ is given by $\sigma(\eta) = \sigma_G\circ \eta \circ \sigma_H^{-1}$; here $\sigma_H$ and $\sigma_G$ denote the actions of $\sigma \in \Gamma_F$ on $\Hat{H}$ and $\Hat{G}$, well defined up to $\Int(\Hat{H})$ and $\Int(\Hat{G})$.
        \item The image of $s$ in $Z(\Hat{H})/Z(\Hat{G})$ is fixed by $\Gamma_F$ and its image in $\pi_0([Z(\Hat{H})/Z(\Hat{G})]^{\Gamma_F})$ belongs to $\mathfrak{K}(H/F)$.
    \end{enumerate}
    There is a notion of isomorphism of endoscopic triples as given in $\S$ 7.5 of \cite{MR0757954}. There is a bijection between the set of isomorphism classes of endoscopic data and the set of isomorphism classes of endoscopic triples.
\end{Le}
\begin{proof}
    See $\S$ 7.6 of \cite{MR0757954}.
\end{proof}

\begin{R}\label{rootdataendogrp}
    For an endoscopic datum $(s, \rho)$ of $G$, let $\mathscr{T}$ be a maximal torus of $\Hat{G}$ which contains $s$. Let $\Psi(\Hat{G}, \mathscr{T}) = (X^*(\mathscr{T}), \Phi( \Hat{G}, \mathscr{T}), X_*(\mathscr{T}), \Check{\Phi}( \Hat{G}, \mathscr{T}))$. Then the root datum of $\Hat{H} = \Hat{G}_s^0$ is $\Psi(\Hat{H}, \mathscr{T}) = (X^*(\mathscr{T}), \Phi( \Hat{H}, \mathscr{T}), X_*(\mathscr{T}), \Check{\Phi}( \Hat{H}, \mathscr{T}))$ such that $\Phi(\Hat{H}, \mathscr{T}) = \{ \alpha \in \Phi(\Hat{G}, \mathscr{T}) \mid \alpha(s) = 1 \}$. Choose pinnings on $\Hat{H}$ and $\Hat{G}$ that contain $\mathscr{T}$. This gives $\Gamma_F$ actions on $\Hat{H}$ and $\Hat{G}$, denoted by $\rho_H$ and $\rho_G$, which preserve $\mathscr{T}$. The actions of $\rho_H$ and $\rho_G$ on $\mathscr{T}$ differ by an element of the Weyl group $\Omega_{\Hat{G}}$ depending on $\sigma \in \Gamma_F$, and this defines a cocycle in $Z^1(\Gamma_F, \Omega_{\Hat{G}})$.
\end{R}

\subsubsection{Stable conjugacy} 
We introduce the notion of stable conjugacy, which is important when one wants to relate conjugacy classes in the endoscopic group $H$ with conjugacy classes in $G$. The notion of stable conjugacy is a refinement of the notion of $\overline{F}$-conjugacy.

\begin{D}[$\overline{F}$-conjugacy]\label{Fbarconjugacy}
    Two elements $x, y \in G(F)$ are $\overline{F}$-conjugate if there exists $g \in G(\overline{F})$ such that $y = gxg^{-1}$.
\end{D}

\begin{D}[Stable conjugacy]\label{StableConjugacy}
    Two elements $x, y \in G(F)$ are stably conjugate if there exists $g \in G(\overline{F})$ such that $y = gxg^{-1}$ and $g^{-1}\sigma(g) \in G_s^0(\overline{F})$ for all $\sigma \in \Gamma_F$, where $s \in G(F)$ is the semisimple part of $x$ in the Jordan decomposition and $G_s^0$ is the identity component of the centralizer of $s$.
\end{D}

\begin{R}\label{rmkaboutstableconj}
    Note that it follows from the definitions that an $F$-homomorphism $\alpha: G_1 \xrightarrow[]{} G_2$ induces the following maps, which we also denote by $\alpha$:
    $$ \{ \overline{F} \text{-conjugacy classes in } G_1(F) \} \xrightarrow[]{\alpha} \{ \overline{F} \text{-conjugacy classes in } G_2(F) \},$$
    $$ \{ \text{stable conjugacy classes in } G_1(F) \} \xrightarrow[]{\alpha} \{ \text{stable conjugacy classes in } G_2(F) \}.$$
\end{R}

\begin{D}[$G_x^*$]\label{stargroupdef}
    For an element $x \in G(F)$, let $G_x^* = G_s^0 \cap G_x$, where $s \in G(F)$ is the semisimple part of $x$. Then $G_x^*$ is a group over $F$.
\end{D}

\begin{Le}\label{simplyconnectedlemma}
    If $G^{\text{der}}$ is simply connected, then $G_x = G_x^*$ for every $x$, and hence stable conjugacy in $G(F)$ is the same as $\overline{F}$-conjugacy. 
\end{Le}
\begin{proof}
    Use the fact, due to Steinberg, that in a simply connected semisimple group the centralizer of any semisimple element is connected.
\end{proof}

\begin{Le}\label{stabconjlem1}
    Let $x,y \in G(F)$ be stably conjugate and let $g \in G(\overline{F})$ be such that $y = gxg^{-1}$. Then $\Int(g): G_x^* \xrightarrow[]{} G_y^*$ is an inner twisting.
\end{Le}
\begin{proof}
    See Lemma 3.2 of \cite{MR683003}.
\end{proof}

\begin{Le}
    Assume that $G$ is quasi-split. Then, for a semisimple $\gamma \in G(F)$, there exists a stable conjugate $\gamma_0 \in G(F)$ such that $G_{\gamma_0}^0 = G_{\gamma_0}^*$ is quasi-split.
\end{Le}
\begin{proof}
    See Lemma 3.3 of \cite{MR683003}.
\end{proof}

\begin{const}\label{stableconjconst}
    For a semisimple $\gamma \in G(F)$, let $I = G_{\gamma}^0$. Choose a maximal torus $T$ over $F$ in $I$ (this can be done since $I$ is reductive). Then $T$ contains $\gamma$, since $\gamma$ lies in $Z(I)$. Moreover, $T$ is also a maximal torus of $G$. Using \ref{secondconst}, we get a canonical $\Gamma_F$-invariant $\Hat{I}$-conjugacy (resp. $\Hat{G}$-conjugacy) class of embeddings $\Hat{T} \xrightarrow[]{} \Hat{I}$ (resp. $\Hat{T} \xrightarrow[]{} \Hat{G}$). Using these embeddings, identify $Z(\Hat{I})$ and $Z(\Hat{G})$ with subgroups of $\Hat{T}$, which is independent of the choice of the particular embedding in the $\Hat{I}$-conjugacy (resp. $\Hat{G}$-conjugacy) class. Via these identifications, we obtain that $Z(\Hat{G}) \subset Z(\Hat{I})$ and that this is a $\Gamma_F$-equivariant embedding.
\end{const}

\begin{Le}\label{conjclassinsidestable}
    Let $\gamma \in G(F)$ be semisimple. Let $I = G_{\gamma}^0$. Then the set of conjugacy classes within the stable conjugacy class of $\gamma$ is parametrized by the pointed set $\ker(H^1(F,I) \xrightarrow[]{} H^1(F,G))$. Further, if $F$ is a $p$-adic field, then the following diagram commutes
    \begin{center}
        \begin{tikzcd}
            H^1(F,I) \arrow{r} \arrow{d}{\simeq} & H^1(F,G) \arrow{d}{\simeq}\\
            X^*(\pi_0(Z(\Hat{I})^{\Gamma_F})) \arrow{r} & X^*(\pi_0(Z(\Hat{G})^{\Gamma_F}))
        \end{tikzcd}
    \end{center}
    where the bottom map is induced by the $\Gamma_F$-equivariant injection $Z(\Hat{G}) \xrightarrow[]{} Z(\Hat{I})$.
\end{Le}
\begin{proof}
    See $\S$ 4.3 of \cite{MR858284}. See also \ref{KottwitzTateNakayama}.
\end{proof}

\begin{D}[$\mathcal{D}(I/F)$]
    The pointed set $\ker(H^1(F,I) \xrightarrow[]{} H^1(F,G))$ parameterizing conjugacy classes with the stable conjugacy class of $\gamma$ has the structure of an abelian group by \ref{conjclassinsidestable}. We will denote this abelian group by $\mathcal{D}(I/F)$.
\end{D}

\begin{D}[$\mathfrak{K}(I/F)$]\label{kappadef}
    Consider the long exact sequence of \ref{keylemmaendo} associated with $1 \xrightarrow[]{} Z(\Hat{G}) \xrightarrow[]{} Z(\Hat{I}) \xrightarrow[]{} Z(\Hat{I}) / Z(\Hat{G}) \xrightarrow[]{} 1$. Let $\mathfrak{K}(I/F)$ denote the subgroup of $\pi_0((Z(\Hat{I})/Z(\Hat{G}))^{\Gamma_F})$ consisting of elements which map to a locally trivial element of $H^1(F, Z(\Hat{G}))$, in this long exact sequence. Note that if $F$ is a $p$-adic field, then $\mathfrak{K}(I/F) =X^*(\mathcal{D}(I/F))$
\end{D}

The following lemma regarding the $\Gamma_F$-invariant $\overline{F}$-conjugacy class of embeddings of tori into a quasi-split group will be applied later on in the context of admissible embeddings.

\begin{Le}\label{quasisplitlemma}
    Let $G$ be a quasi-split group over $F$. Let $\iota: T \xrightarrow[]{} G$ be an embedding over $\overline{F}$ of a torus $T$ defined over $F$ such that $\iota(T)$ is a maximal torus of $G$. Assume that this embedding has the property that the $G(\overline{F})$-conjugacy class of $\iota$ is $\Gamma_F$-invariant. Then there is a $G(\overline{F})$-conjugate of $\iota$ that is defined over $F$.
\end{Le}
\begin{proof}
    See Corollary 2.2 of \cite{MR683003}.
\end{proof}

\subsubsection{Endoscopic transfer of conjugacy classes}
Now we will explain how to relate conjugacy classes in an endoscopic group $H$ with conjugacy classes in $G$. For the rest of the section, let $(H,s,\eta)$ denote an endoscopic triple for $G$.

\begin{D}[Admissible embeddings]\label{admissibleembedding}
    Let $T_H$ be a maximal torus of $H$. An embedding $\iota_{T_H}^G: T_H \xrightarrow[]{} G$ is admissible if the $\Hat{G}$-conjugacy class of embeddings $\Hat{T}_H \xrightarrow[]{} \Hat{G}$ determined by $\iota_{T_H}^G$ is the composition of the canonical $\Hat{H}$-conjugacy class $\Hat{T}_H \xrightarrow[]{} \Hat{H}$ of \ref{firstconst} with the canonical $\Gamma_F$-invariant $\Hat{G}$-conjugacy class of $\eta$. Note here that the $\Hat{G}$-conjugacy class $\Hat{T}_H \xrightarrow[]{} \Hat{G}$ is constructed as in \ref{firstconst}. Further, it can be checked that any two admissible embeddings $T_H \xrightarrow[]{} G$ are conjugate under $G(\overline{F})$. If $T_H$ is defined over $F$, then the $G(\overline{F})$-conjugacy class of admissible embeddings $T_H \xrightarrow[]{} G$ is $\Gamma_F$-invariant, which follows by using \ref{secondconst}.
\end{D}

\begin{Le}\label{quasisplit}
    If $G$ is quasi-split and $T_H$ is a maximal torus of $H$ defined over $F$ for an endoscopic triple $(H,s,\eta)$ of $G$, then there exists an admissible embedding $\iota_{T_H}^G: T_H \xrightarrow[]{} G$ defined over $F$.
\end{Le}
\begin{proof}
    Follows from \ref{quasisplitlemma}.
\end{proof}

\begin{D}[Transfer of tori]
    If there is an admissible embedding $\iota_{T_H}^G$ over $F$ with $T_G = \iota_{T_H}^G(T_H)$ a maximal torus in $G$ over $F$, then we say that $T_H$ transfers to $T_G$.
\end{D}

\begin{D}[Related conjugacy classes]\label{relatedelementsdef}
    Let $\gamma_H \in H(F)$ and $\gamma \in G(F)$ be semisimple. We say that $\gamma_H$ is related to $\gamma$ if there exists an admissible embedding $\iota_{T_H}^G: T_H \xrightarrow[]{} G$ such that $\iota_{T_H}^G(\gamma_H) = \gamma$. We say that an $\overline{F}$-conjugacy class in $H$ is related to an $\overline{F}$-conjugacy class in $G$, if there exist representatives $\gamma_H \in H(F)$ and $\gamma \in G(F)$ of the corresponding $\overline{F}$-conjugacy classes that are related. Note that the condition of two $\overline{F}$-conjugacy classes being related is independent of the choice of the representatives $\gamma_H$ and $\gamma$.
\end{D}

\begin{D}[$(G,H)$-regular conjugacy classes]\label{GH-regular}
    Let $\gamma_H \in H(F)$ be semisimple. Assume that there exists a $\gamma \in G(F)$ related to $\gamma_H$ with respect to an admissible embedding $\iota_{T_H}^G: T_H \xrightarrow[]{} G$. Let $T = \iota_{T_H}^G(T_H)$. Let $\Phi(G,T)$ and $\Phi(H,T_H)$ denote the sets of roots of $G$ and $H$ respectively. Then, with the identification of $T_H$ with $T$ via $\iota_{T_H}^G$, we can view $\Phi(H,T_H) \subset \Phi(G,T) \subset X^*(T)$. We say that $\gamma_H$ is $(G,H)$-regular if $\alpha(\gamma_H) \neq 1$ for every $\alpha \in \Phi(G,T) \setminus \Phi(H,T_H)$. Note that this condition does not depend on the choice of the admissible embedding and is a property of the $\overline{F}$-conjugacy class of $\gamma_H$.
\end{D}

\begin{const}\label{GHregular}
    Let $\gamma_H \in H(F)$ be related to $\gamma \in G(F)$ such that $\gamma_H$ is $(G,H)$-regular. Let $I = G_{\gamma}^0$ and $I_H = H_{\gamma_H}^0$ be the identity components of the respective centralizers, taken over $F$. Then, there is an inner twisting from $I_H$ to $I$ constructed as follows. Let $T_H \xrightarrow[]{} T \subset G$ be an admissible embedding, as before. Then the sets of roots are $\Phi(I,T) = \{\alpha \in \Phi(G,T) \mid \alpha(\gamma) = 1 \}$ and $\Phi(I_H, T_H) = \{ \alpha \in \Phi(H,T_H) \mid \alpha(\gamma_H) = 1 \}$. Using that $\gamma_H$ is $(G,H)$-regular, it follows that $\Phi(I,T) = \Phi (I_H, T_H)$, under the identification of $T$ and $T_H$. Hence, $I$ and $I_H$ are inner forms of each other, and there exists an inner twisting $j_{I_H}^G: I_{H,\overline{F}} \xrightarrow[]{\sim} I_{\overline{F}}$ which extends $\iota_{T_H}^G$ and is defined uniquely up to conjugation by elements of $T$.
\end{const}

\begin{Le}\label{connectedlemmaGHreg}
    Let $\gamma_H \in H(F)$ and $\gamma \in G(F)$ be related with $\gamma_H$ being $(G,H)$-regular. If the centralizer $I$ of $\gamma$ in $G$ is connected, then so is the centralizer $I_H$ of $\gamma_H$ in $H$.
\end{Le}
\begin{proof}
    See Lemma 3.2 of \cite{MR858284}.
\end{proof}

\subsubsection{Ellipticity}
We will now define the important related notions of elliptic tori, elliptic conjugacy classes and elliptic endoscopic groups, which will be used in a crucial lemma. Recall the following setup regarding the cocenter of a connected reductive group $G$, which we will use in a later argument.

\begin{itemize}
    \item Let $G$ be a connected reductive group over a field $F$. Let $G^{\text{der}} = (G,G)$ be the derived subgroup, which is semisimple and let $D(G) = G / G^{\text{der}}$ be the cocenter of $G$.
    \item For a dual group $\Hat{G}$ of $G$, there are canonical isomorphisms $X^*(Z(\Hat{G})^0) \simeq X_*(D(G))$ and $X_*(Z(\Hat{G})^0) \simeq X^*(D(G)) $. These isomorphisms are induced from the isomorphisms $X^*(\mathscr{T}) \simeq X_*(T)$ and $X_*(\mathscr{T}) \simeq X^*(T)$, for maximal tori $T$ and $\mathscr{T}$ in $G$ and $\Hat{G}$ respectively. Hence, $D(G)$ and $Z(\Hat{G})^0$ are dual tori.
    \item The natural map $Z(G) \xrightarrow[]{} D(G)$ is surjective with finite kernel.
\end{itemize}

\begin{R}
    Note that one should not confuse the above isomorphism
    $$X^*(Z(\Hat{G})^0) = X_*(D(G))$$
    with the one in \ref{kottwitzmap}, which was $X^*(Z(\Hat{G})) = \pi_1(G)$.
\end{R}

\begin{D}[Elliptic tori]
    A torus $T$ of $G$ over $F$ is elliptic if $T/Z(G)$ is anisotropic, that is, if it has no non-trivial characters defined over $F$.
\end{D}

The following lemma allows us to express the condition for a maximal torus to be elliptic in terms of dual groups.

\begin{Le}\label{ellipticmain}
    A maximal torus $T$ of $G$ over $F$ is elliptic iff $(\Hat{T}/Z(\Hat{G}))^{\Gamma_F}$ is finite.
\end{Le}
\begin{proof}
    The canonical $\Gamma_F$-invariant $\Hat{G}$-conjugacy class of embeddings $\Hat{T} \xrightarrow[]{} \Hat{G}$, obtained from \ref{secondconst} applied to $T$, allows us to view $Z(\Hat{G})$ as a subgroup of $\Hat{T}$. Note that since $\Hat{T}/Z(\Hat{G})$ is a torus, one has that $(\Hat{T}/Z(\Hat{G}))^{\Gamma_F}$ is finite iff $$X_*((\Hat{T}/Z(\Hat{G}))^{\Gamma_F}) = X_*(\Hat{T}/ Z(\Hat{G}))^{\Gamma_F} = X_*((\Hat{T}/ Z(\Hat{G}))^0)^{\Gamma_F}$$ is trivial. Further, this is equivalent to $X_*(\Hat{T}/Z(\Hat{G}))^{\Gamma_F} \otimes \mathbb{Q}$ being trivial since $X_*(\Hat{T}/Z(\Hat{G}))^{\Gamma_F}$ is a free $\mathbb{Z}$-module for any torus $T$. We have the exact sequence (with the canonical identification $\mathscr{T} = \Hat{T}$ for a maximal torus $\mathscr{T}$ of $\Hat{G}$):
    $$ 1 \xrightarrow[]{} Z(\Hat{G}) \xrightarrow[]{} \Hat{T} \xrightarrow[]{} \Hat{T}/Z(\Hat{G}) \xrightarrow[]{} 1.$$
    Applying the functors $X_*(-)$, $(-) \otimes \mathbb{Q}$ and $(-)^{\Gamma_F}$ in this order to the above exact sequence gives
    $$ 1 \xrightarrow[]{} (X_*(Z(\Hat{G}))\otimes \mathbb{Q})^{\Gamma_F} \xrightarrow[]{} (X_*(\Hat{T})\otimes \mathbb{Q})^{\Gamma_F} \xrightarrow[]{} (X_*(\Hat{T}/Z(\Hat{G}))\otimes \mathbb{Q})^{\Gamma_F} \xrightarrow[]{} 1. $$
    Note that the resulting sequence is exact because: $X_*(-)$ is left exact in the category of reductive algebraic groups and exact in the category of tori; $(-) \otimes \mathbb{Q}$ is exact in the category of $\mathbb{Z}$-modules and $(-)^{\Gamma_F}$ is exact in the category of $\mathbb{Q}$-vector spaces with $\Gamma_F$-action. Thus, it follows that $(\Hat{T}/Z(\Hat{G}))^{\Gamma_F}$ is finite iff $(X_*(Z(\Hat{G}))\otimes \mathbb{Q})^{\Gamma_F} \xrightarrow[]{\sim} (X_*(\Hat{T})\otimes \mathbb{Q})^{\Gamma_F}$ is an isomorphism.

    Next, note that $T$ is an elliptic maximal torus iff $X^*(T/Z(G))^{\Gamma_F}$ is trivial. For any torus $T$, $X^*(T/Z(G))^{\Gamma_F}$ is a free $\mathbb{Z}$-module and hence the condition for $T$ to be elliptic is further equivalent to $X^*(T/Z(G))^{\Gamma_F} \otimes \mathbb{Q}$ being trivial. Consider the exact sequence
    $$ 1 \xrightarrow[]{} Z(G) \xrightarrow[]{} T \xrightarrow[]{} T / Z(G) \xrightarrow[]{} 1.$$
    Applying the functors $X^*(-)$, $(-) \otimes \mathbb{Q}$ and $(-)^{\Gamma_F}$ in that order to the above exact sequence gives
    $$ 1 \xrightarrow[]{} (X^*(Z(G))\otimes \mathbb{Q})^{\Gamma_F} \xrightarrow[]{} (X^*(T) \otimes \mathbb{Q})^{\Gamma_F} \xrightarrow[]{} (X^*(T / Z(G)) \otimes \mathbb{Q})^{\Gamma_F} \xrightarrow[]{} 1.$$
    Once again, the resulting sequence is exact because: $X^*(-)$ is exact in the category of groups of multiplicative type, $(-) \otimes \mathbb{Q}$ is exact in the category of $\mathbb{Z}$-modules and $(-)^{\Gamma_F}$ is exact in the category of $\mathbb{Q}$-vector spaces with $\Gamma_F$-action. Hence, $T$ is elliptic iff $(X^*(Z(G))\otimes \mathbb{Q})^{\Gamma_F} \xrightarrow[]{\sim} (X^*(T) \otimes \mathbb{Q})^{\Gamma_F}$ is an isomorphism.

    Combining all this, we are reduced to showing that $(X_*(Z(\Hat{G}))\otimes \mathbb{Q})^{\Gamma_F} \xrightarrow[]{\sim} (X_*(\Hat{T})\otimes \mathbb{Q})^{\Gamma_F}$ is an isomorphism iff $(X^*(Z(G))\otimes \mathbb{Q})^{\Gamma_F} \xrightarrow[]{\sim} (X^*(T) \otimes \mathbb{Q})^{\Gamma_F}$ is an isomorphism. Note that 
    $$(X_*(Z(\Hat{G}))\otimes \mathbb{Q})^{\Gamma_F} = (X_*(Z(\Hat{G})^0)\otimes \mathbb{Q})^{\Gamma_F} = (X^*(D(G))\otimes \mathbb{Q})^{\Gamma_F}$$
    and hence $(X_*(Z(\Hat{G}))\otimes \mathbb{Q})^{\Gamma_F} \xrightarrow[]{\sim} (X_*(\Hat{T})\otimes \mathbb{Q})^{\Gamma_F}$ is an isomorphism iff $(X^*(D(G))\otimes \mathbb{Q})^{\Gamma_F} \xrightarrow[]{\sim} (X^*(T)\otimes \mathbb{Q})^{\Gamma_F}$ is an isomorphism.

    Finally, note that since the composition $Z(G) \xrightarrow[]{} T \xrightarrow[]{} D(G)$ is surjective with finite kernel, it follows that the composition
    $$(X^*(Z(G))\otimes \mathbb{Q})^{\Gamma_F} \xrightarrow[]{\sim} (X^*(T) \otimes \mathbb{Q})^{\Gamma_F} \xrightarrow[]{} (X^*(D(G))\otimes \mathbb{Q})^{\Gamma_F}$$
    is an isomorphism. The result now follows.
\end{proof}

\begin{Le}\label{ellipticlem1}
    The quotient $T/ Z(G)$ is anisotropic iff $T/Z(G)^0$ is anisotropic.
\end{Le}
\begin{proof}
    Note that for a torus $T$, $X^*(T/Z(G))^{\Gamma_F}$ and $X^*(T/Z(G)^0)^{\Gamma_F}$ are free $\mathbb{Z}$-modules and hence, they are trivial iff $X^*(T/Z(G))^{\Gamma_F} \otimes \mathbb{Q}$ and $X^*(T/Z(G)^0)^{\Gamma_F} \otimes \mathbb{Q}$ respectively are trivial. There is an exact sequence
    $$ 1 \xrightarrow[]{} Z(G)/Z(G)^0 \xrightarrow[]{} T/Z(G)^0 \xrightarrow[]{} T/Z(G) \xrightarrow[]{} 1.$$
    Applying the functors $X^*(-)$, $(-) \otimes \mathbb{Q}$ and $(-)^{\Gamma_F}$ preserves exactness (using the same reasoning as in \ref{ellipticmain}) and hence the following map
    $$ (X^*(T/Z(G)^0) \otimes \mathbb{Q})^{\Gamma_F} \xrightarrow[]{\sim} (X^*(T/Z(G)) \otimes \mathbb{Q})^{\Gamma_F} $$
    is an isomorphism. The result now follows.
\end{proof}

\begin{Le}\label{elliptictemp}
For an elliptic maximal torus $T$ of $G$ over $F$, the map $H^1(F,T) \xrightarrow[]{} H^1(F,G)$ is surjective.
\end{Le}
\begin{proof}
    Using \ref{KottwitzTateNakayama}, it reduces to showing that the map $\pi_0(Z(\Hat{G})^{\Gamma_F}) \xrightarrow[]{} \pi_0(\Hat{T}^{\Gamma_F})$ is injective. Consider the long exact sequence of \ref{keylemmaendo} associated with
    $$ 1 \xrightarrow[]{} Z(\Hat{G}) \xrightarrow[]{} \Hat{T} \xrightarrow[]{} \Hat{T}/Z(\Hat{G}) \xrightarrow[]{} 1.$$
    Since \ref{ellipticmain} implies that $X_*(\Hat{T}/Z(\Hat{G}))^{\Gamma_F}$ is trivial, the exactness of 
    $$ \cdots \xrightarrow[]{} X_*(\Hat{T}/Z(\Hat{G}))^{\Gamma_F} \xrightarrow[]{} \pi_0(Z(\Hat{G})^{\Gamma_F}) \xrightarrow[]{} \pi_0(\Hat{T}^{\Gamma_F}) \xrightarrow[]{} \cdots$$
    gives the desired result.
\end{proof}

\begin{Le}
    If $F$ is a number field or a $p$-adic field, then any connected reductive group $G$ over $F$ admits an elliptic maximal torus over $F$.
\end{Le}
\begin{proof}
    See Theorem 6.21 of \cite{MR4615820}.
\end{proof}

For the rest of the discussion of elliptic conjugacy classes and elliptic endoscopic groups, assume from now on that
\begin{itemize}
    \item Let $F$ be a number field or a $p$-adic field.
\end{itemize}

\begin{D}[Elliptic conjugacy classes]\label{ellipticdef}
    Let $\gamma \in G(F)$ be semisimple. Then $\gamma$ is elliptic if $Z(I) / Z(G)$ is anisotropic, that is, if it has no non-trivial characters defined over $F$, where $I = G_{\gamma}^0$ denotes the identity component of the centralizer of $\gamma$.
\end{D}

\begin{Le}\label{ellipticlem2}
    For a semisimple $\gamma \in G(F)$, the condition that $\gamma$ is elliptic is equivalent to $\gamma$ being contained in an elliptic maximal torus.
\end{Le}
\begin{proof}
    Since $\gamma$ is semisimple, it follows that $I$ is reductive. Let $T$ be an elliptic maximal torus of $I$ over $F$. Consider the following exact sequence
    $$ 1 \xrightarrow[]{} Z(I)/Z(G) \xrightarrow[]{} T/Z(G) \xrightarrow[]{} T/Z(I) \xrightarrow[]{} 1.$$
    Applying the functors $X^*(-)$, $(-) \otimes \mathbb{Q}$ and $(-)^{\Gamma_F}$ preserves exactness (using the same reasoning as in \ref{ellipticmain}) and hence the following map
    $$ (X^*(T/Z(G)) \otimes \mathbb{Q})^{\Gamma_F} \xrightarrow[]{\sim} (X^*(T/Z(I))\otimes \mathbb{Q})^{\Gamma_F} $$
    is an isomorphism. Using that $T$ is elliptic in $I$, it follows that $X^*(T/Z(G))^{\Gamma_F}$ is trivial, and hence $\gamma$ is contained in an elliptic maximal torus $T$ of $G$.
\end{proof}

\begin{R}\label{ellipticremarkequiv}
    Note that the following groups being anisotropic are all equivalent conditions: $Z(I)/Z(G)$, $Z(I)^0/Z(G)$, $Z(I) / Z(G)^0$ and $Z(I)^0/Z(G)^0$.
    The proof follows by arguments similar to those in \ref{ellipticlem1} and \ref{ellipticlem2}.
\end{R}

Now we state the analogue of \ref{ellipticmain} for elliptic conjugacy classes. In the formulation of this analogue, note that the embedding $Z(\Hat{G}) \xrightarrow[]{} Z(\Hat{I})$ is the one constructed in \ref{stableconjconst}.

\begin{Le}\label{ellipticlem3}
    For a semisimple $\gamma \in G(F)$, the condition that $\gamma$ is elliptic is equivalent to $(Z(\Hat{I})/Z(\Hat{G}))^{\Gamma_F}$ being finite.
\end{Le}
\begin{proof}
    Assume that $\gamma$ is elliptic. Let $T$ be an elliptic maximal torus of $G$ over $F$ containing $\gamma$ (using \ref{ellipticlem2}). As in construction \ref{stableconjconst}, there are inclusions $Z(\Hat{G}) \xrightarrow[]{} Z(\Hat{I}) \xrightarrow[]{} \Hat{T}$. Since $(\Hat{T}/Z(\Hat{G}))^{\Gamma_F}$ is finite using \ref{ellipticmain}, it follows that $(Z(\Hat{I})/Z(\Hat{G}))^{\Gamma_F}$ is finite.
    
    For the other direction, assume that $(Z(\Hat{I})/Z(\Hat{G}))^{\Gamma_F}$ is finite. Let $T$ be an elliptic maximal torus of $I$. There is the exact sequence
    $$ 1 \xrightarrow[]{} Z(\Hat{I})/Z(\Hat{G}) \xrightarrow{} \Hat{T}/Z(\Hat{G}) \xrightarrow[]{} \Hat{T} / Z(\Hat{I}) \xrightarrow[]{} 1. $$
    Applying the left exact functor $(-)^{\Gamma_F}$ gives
    $$ 1 \xrightarrow[]{} (Z(\Hat{I})/Z(\Hat{G}))^{\Gamma_F} \xrightarrow[]{} (\Hat{T}/Z(\Hat{G}))^{\Gamma_F} \xrightarrow[]{} (\Hat{T} / Z(\Hat{I}))^{\Gamma_F}.$$
    From this it follows that $(\Hat{T}/Z(\Hat{G}))^{\Gamma_F}$ is finite and hence $T$ is an elliptic maximal torus in $G$ which contains $\gamma$. This completes the proof.
\end{proof}

\begin{Le}
    For an elliptic semisimple $\gamma \in G(F)$, the map $H^1(F,I) \xrightarrow[]{} H^1(F,G)$ is surjective.
\end{Le}
\begin{proof}
    Similar to the proof of \ref{elliptictemp}.
\end{proof}

Finally, we define the notion of ellipticity for endoscopic triples.

\begin{D}[Elliptic endoscopic triple]
    An endoscopic triple $(H,s,\eta)$ for $G$ over $F$ is elliptic if $(Z(\Hat{H})/Z(\Hat{G}))^{\Gamma_F}$ is a finite group.
\end{D}

\begin{Le}
    If $(H,s,\eta)$ is elliptic, then $\pi_0(Z(\Hat{G})^{\Gamma_F}) \xrightarrow[]{} \pi_0(Z(\Hat{H})^{\Gamma_F})$ is injective and the induced map $H^1(F,H) \xrightarrow[]{} H^1(F,G)$ is surjective.
\end{Le}
\begin{proof}
    Similar to the proof of \ref{elliptictemp}.
\end{proof}

\begin{Le}
    Assume that $G$ is quasi-split. Then $(H,s,\eta)$ is elliptic if and only if the transfer to $G$, via an admissible embedding, of one (hence every) elliptic maximal torus of $H$ is elliptic.
\end{Le}
\begin{proof}
    Assume $(H,s,\eta)$ is elliptic. Let $T_H$ be an elliptic maximal torus in $H$ over $F$. Suppose that $T_H$ transfers to $T_G$, a maximal torus of $G$ over $F$. We can canonically identify $\Hat{T_H}$ and $\Hat{T_G}$ and denote both of them by $\Hat{T}$. Since $T_H$ is elliptic, $(\Hat{T}/Z(\Hat{H}))^{\Gamma_F}$ is finite by \ref{ellipticmain}. By an argument similar to \ref{ellipticlem3}, it follows now that $(\Hat{T} / Z(\Hat{G}))^{\Gamma_F}$ is finite.

    For the other direction, assume that there is an elliptic maximal torus $T_H$ of $H$ that transfers to an elliptic maximal torus $T_G$ of $G$. Then, once again, we can set $\Hat{T} = \Hat{T_H} = \Hat{T_G}$. Since $(Z(\Hat{H})/Z(\Hat{G}))^{\Gamma_F}$ is a subgroup of $(\Hat{T} / Z(\Hat{G}))^{\Gamma_F}$, it follows that $(H,s,\eta)$ is elliptic.
\end{proof}

\subsubsection{A key construction}
We will explain the statement of a lemma that will play an important role in the proof of the main theorem. We will first need to explain a construction before we can state the lemma. We will now make the following additional assumption.
\begin{itemize}
    \item The derived group $G^{\text{der}}$ is simply connected.
\end{itemize}

\begin{const}\label{prestabconst}
    Let $(H,s,\eta)$ be an endoscopic triple for a connected reductive group $G$ over $F$. Let $\gamma_H \in H(F)$ be a $(G,H)$-regular element. Assume that $G$ is quasi-split. Then by \ref{quasisplit} there exists $\gamma_0 \in G(F)$ such that $\gamma_H$ is related to $\gamma_0$. Let $I_H$ and $I_0$ denote the identity components of the centralizers of $\gamma_H$ and $\gamma_0$ respectively. By \ref{GHregular}, $I_H$ and $I_0$ are inner forms of each other, and in particular there is a canonical $\Gamma_F$-invariant isomorphism between $Z(\Hat{I_H})$ and $Z(\Hat{I_0})$. Let $\kappa$ denote the image of $s \in Z(\Hat{H})$ under the $\Gamma_F$-invariant map
    $$ Z(\Hat{H}) \hookrightarrow Z(\Hat{I_H}) \xrightarrow[]{\sim} Z(\Hat{I_0}).$$
    Since $s$ lies in $\mathfrak{K}(H/F)$, it follows that $\kappa$ lies in $\mathfrak{K}(I_0/F)$. We will denote a pair $(\gamma_0, \kappa)$ consisting of $\gamma_0 \in G(F)$ and $\kappa \in \mathfrak{K}(I_0/F)$ constructed in this way starting from an endoscopic triple $(H,s,\eta)$ and a $(G,H)$-regular $\gamma_H \in H(F)$ by the notation $((H,s,\eta), \gamma_H) \rightarrow (\gamma_0, \kappa)$.
\end{const}

\begin{Le}\label{keylemma}
    Let $G$ be a quasi-split group over a number field or a $p$-adic field $F$. Let $\gamma_0 \in G(F)$ be an elliptic semisimple element and let $\kappa \in \mathfrak{K}(I_0/F)$. Then there exists an elliptic endoscopic triple $(H,s,\eta)$ and a $(G,H)$-regular semisimple $\gamma_H \in H(F)$ such that $((H,s,\eta), \gamma_H) \rightarrow (\gamma_0, \kappa)$, in the sense of \ref{prestabconst}. Moreover, $((H',s',\eta'),\gamma'_H) \rightarrow (\gamma_0, \kappa)$ also holds if and only if there is an isomorphism of endoscopic triples $(H,s,\eta) \rightarrow (H',s',\eta')$ such that the induced $F$-isomorphism $H \rightarrow H'$ maps $\gamma_H$ to a stable conjugate of $\gamma'_H$.
\end{Le}
\begin{proof}
    This is Lemma 9.7 in \cite{MR858284} and we refer to it for more details. We will only outline the construction of $((H,s,\eta), \gamma_H)$ with the property $((H,s,\eta), \gamma_H) \rightarrow (\gamma_0, \kappa)$ here.
    
    \textit{Construction of $s$.} Since $\gamma_0$ is elliptic, there exists an elliptic maximal $F$-torus $T$ of $G$ containing $\gamma_0$, by \ref{ellipticlem2}. Write $I = I_0 = G_{\gamma_0}^0$. It also follows, using \ref{ellipticlem3}, that $\pi_0((Z(\Hat{I})/Z(\Hat{G}))^{\Gamma_F}) = (Z(\Hat{I})/Z(\Hat{G}))^{\Gamma_F}$. Hence, $\kappa \in \mathfrak{K}(I_0/F)$ can be viewed as an element of $(Z(\Hat{I})/Z(\Hat{G}))^{\Gamma_F}$. Since $\gamma_0 \in T$, $T$ is also an elliptic maximal torus of $I_0$. Using \ref{stableconjconst}, there is a $\Gamma_F$-invariant embedding $Z(\Hat{I}) \xrightarrow[]{} \Hat{T}$. Denote by $s$ the image of $\kappa$ under the induced embedding $(Z(\Hat{I})/ Z(\Hat{G}))^{\Gamma_F} \xrightarrow[]{} (\Hat{T}/Z(\Hat{G}))^{\Gamma_F}$. Denote by $\Hat{H}$ the identity component of the centralizer of $s$ in $\Hat{G}$.
    
    \textit{Construction of endoscopic triple $(H,s,\eta)$.} For $\sigma \in \Gamma_F$, let $\sigma_T$ denote the action of $\sigma$ on $\Hat{T}$ induced by the $F$-structure on $T$. Let $\sigma_G$ denote the action of $\sigma$ on $\Hat{G}$ induced by the $F$-structure on $G$ and the choice of a pinning on $\Hat{G}$ which includes $\Hat{T}$. Then, since the $\Hat{G}$-conjugacy class of $\Hat{T} \xrightarrow[]{} \Hat{G}$ is $\Gamma_F$-invariant by \ref{secondconst}, the restriction of $\sigma_G$ to $\Hat{T}$ differs from $\sigma_T$ by an element of the Weyl group $\Omega_G(\overline{F}) = \Omega_{\Hat{G}}(\overline{F})$. Hence, since $\sigma_G$ preserves the roots of $\Hat{T}$ in $\Hat{G}$, it follows also that $\sigma_T$ preserves the roots of $\Hat{T}$ in $\Hat{G}$. Using the description of roots of $\Hat{T}$ in $\Hat{H}$ in \ref{rootdataendogrp} and the fact that $s \in (\Hat{T}/Z(\Hat{G}))^{\Gamma_F}$, it can be checked that $\sigma_T$ preserves the roots of $\Hat{T}$ in $\Hat{H}$. This gives a map $\rho: \Gamma_F \xrightarrow[]{} \Out(\Hat{H})$ sending $\sigma$ to the action induced by $\sigma_T$. The conditions in \ref{endodef} can be checked to conclude that $(s,\rho)$ is an endoscopic datum for $G$. Let $(H,s,\eta)$ be an endoscopic triple associated to $(s,\rho)$, as in \ref{endotriples}.
    
    \textit{Construction of $\gamma_H$.} The $F$-structure on $H$ is induced from $\rho: \Gamma_F \xrightarrow[]{} \Out(\Hat{H})$. Since $\rho$ is induced from the action of $\Gamma_F$ on $\Hat{T}$ coming from the $F$-structure of $T$, it follows that $T$ can be viewed as a maximal torus in $H$ that is defined over $F$. Let $\gamma_H \in T(F) \subset H(F)$ be such that $\gamma_H$ is related to $\gamma_0 \in T(F) \subset G(F)$ (see \ref{relatedelementsdef}). We will now show that $\gamma_H$ is $(G,H)$-regular. Let $I_H$ denote the identity component of the centralizer of $\gamma_H$ in $H$. It suffices to check that the set of roots $R(\gamma_0)$ of $T$ in $I$ is equal to the set of roots $R_H(\gamma_0)$ of $T$ in $I_H$ (see discussion in \ref{GHregular}). Let $R$ denote the set of roots of $T$ in $G$ and $R_H$ denote the set of roots of $T$ in $H$. Then,
    $$ R_H = \{ \alpha \in R \mid \check{\alpha}(s) = 1 \} $$
    $$ R(\gamma_0) = \{ \alpha \in R \mid \alpha(\gamma_0) = 1 \}. $$
    Hence,
    $$ R_H(\gamma_0) = \{ \alpha \in R \mid \alpha(\gamma_0) = 1 \text{ and } \check{\alpha}(s) = 1\}.$$
    However, $\alpha(\gamma_0) = 1$ implies that $\check{\alpha} (s) = 1$, since $s$ lies in $Z(\Hat{I})$. Hence $R_H(\gamma_0) = R(\gamma_0)$.
    
    This shows the existence of $((H,s,\eta), \gamma_H)$ with the property that $((H,s,\eta),\gamma_H) \rightarrow (\gamma_0, \kappa)$.
\end{proof}

\begin{R}
    Since $G^{\text{der}}$ is simply connected, stable conjugacy in $G(F)$ is the same as $\overline{F}$-conjugacy, since the centralizer of $\gamma$ in $G$ is connected. Now, using \ref{connectedlemmaGHreg}, it follows that the centralizer of $\gamma_H$ in $H$ is also connected. Hence, the stable conjugacy class of $\gamma_H$ in $H(F)$ is the same as the $\overline{F}$-conjugacy class in $H(F)$.
\end{R}

\subsection{Review of base change}\label{2.3}
In this subsection, we will review the alternate versions of the norm mapping and its properties. We will also review the restriction of scalars group which appears in connection with the norm mapping and is also used later in defining twisted orbital integrals and in the setting of twisted endoscopy.
\subsubsection{Restriction of scalars}\label{restrictionofscalars}
We will recall some basic facts about restriction of scalars and set up some notation.
\begin{itemize}
    \item Let $F$ denote a field of characteristic zero and let $E$ be a finite extension of $F$. 
    \item Let $G$ be an affine algebraic group over $E$.
    \item Let $\Res_{E/F}G$ denote the affine algebraic group over $F$ with the property that
    $$\Res_{E/F}G(R) = G(R \otimes_F E)$$
    for any $F$-algebra $R$.
\end{itemize}

In the following lemma, we compute the action of $\Gamma_F$ on $\Res_{E/F}G(\overline{F})$.

\begin{Le}\label{twistedlem1}
    There is an isomorphism of groups $\Res_{E/F}G(\overline{F}) \simeq \prod_{\iota: E \hookrightarrow \overline{F}} G(\overline{F})$, where the product on the right is over all $F$-embeddings of $E$ into a fixed algebraic closure $\overline{F}$ of $F$. There is a $\Gamma_F$-action on $\Res_{E/F}G(\overline{F})$, since $\Res_{E/F}G$ is a group over $F$, and there is a $\Gamma_F$-action on the set of embeddings $\iota: E \hookrightarrow \overline{F}$ which sends $\iota \mapsto \gamma \circ \iota$ for $\gamma \in \Gamma_F$. For $(x_{\iota})_{\iota} \in \prod_{\iota: E \xrightarrow[]{} \overline{F}} G(\overline{F})$ and $\gamma \in \Gamma_F$, we define $\gamma \cdot (x_{\iota})_{\iota} := (\gamma(x_{\gamma^{-1} \circ \iota}))_{\iota}$. The isomorphism is $\Gamma_F$-equivariant for this action.
\end{Le}
\begin{proof}
    The result follows from the isomorphism $\overline{F} \otimes_F E \simeq \prod_{\iota: E \hookrightarrow \overline{F}} \overline{F}$ mapping
    $$x \otimes y \mapsto (x\iota(y))_{\iota} = (x \sigma_1(y), x \sigma_2(y), \cdots, x \sigma_d(y)),$$
    where $\{ \sigma_1, \sigma_2, \cdots, \sigma_d\}$ denotes the set of embeddings $\iota: E \hookrightarrow \overline{F}$ and the following commutative square
\begin{center}
\begin{tikzcd}
\overline{F} \otimes_F E \arrow{r}{\simeq} \arrow{d}{\gamma \cdot()} & \prod_{\iota: E \hookrightarrow \overline{F}} \overline{F} \arrow{d}{\gamma \cdot()} \\
\overline{F} \otimes_F E \arrow{r}{\simeq} & \prod_{\iota: E \hookrightarrow \overline{F}} \overline{F}
\end{tikzcd}
\quad
\begin{tikzcd}
x \otimes y \arrow{r}{} \arrow{d}{} & (x\iota(y))_{\iota} \arrow{d}{} \\
\gamma(x) \otimes y \arrow{r}{} & (\gamma(x(\gamma^{-1} \circ \iota)(y)))_{\iota}
\end{tikzcd}
\end{center}
\end{proof}

\begin{C}\label{twistedcor1}
    There is an isomorphism over $E$: $(\Res_{E/F}G)_E \xrightarrow[]{\sim} G_E^l$, where $l = [E:F]$ and the factors of $G_E$ are indexed by the set of $F$-embeddings of $E$ into $\overline{F}$ as in \ref{twistedlem1}.
\end{C}

Now we make the following additional assumptions:
\begin{itemize}
    \item Let $E$ be a Galois extension of $F$.
    \item Let $G$ be a group defined over $F$ (instead of $E$).
\end{itemize}
Then for every $\gamma \in \Gal(E/F)$, there is an automorphism $\theta_{\gamma}: \Res_{E/F}G_E \xrightarrow[]{} \Res_{E/F}G_E$ defined over $F$ which on $R$-points is induced by the $F$-algebra map $R \otimes_F E \xrightarrow[]{} R \otimes_F E$ sending $x \otimes y \mapsto x \otimes \gamma(y)$, for every $F$-algebra $R$. In the following lemma, we compute the action of $\theta_{\gamma}$ on $\Res_{E/F}G_E(\overline{F})$, for any $\gamma \in \Gal(E/F)$.

\begin{Le}
    The action of $\theta_{\gamma}$ on $\Res_{E/F}G_E(\overline{F}) \simeq \prod_{\iota: E \hookrightarrow \overline{F}} G(\overline{F})$ is given by $(x_{\iota})_{\iota} \mapsto (x_{\iota \circ \gamma})_{\iota}$.
\end{Le}
\begin{proof}
    The proof follows from the commutative square
\begin{center}
\begin{tikzcd}
\overline{F} \otimes_F E \arrow{r}{\simeq} \arrow{d}{\theta_{\gamma}} & \prod_{\iota: E \hookrightarrow \overline{F}} \overline{F} \arrow{d}{\theta_{\gamma}} \\
\overline{F} \otimes_F E \arrow{r}{\simeq} & \prod_{\iota: E \hookrightarrow \overline{F}} \overline{F}
\end{tikzcd}
\quad
\begin{tikzcd}
x \otimes y \arrow{r}{} \arrow{d}{} & (x\iota(y))_{\iota} \arrow{d}{} \\
x \otimes \gamma(y) \arrow{r}{} & (x (\iota \circ \gamma)(y))_{\iota}
\end{tikzcd}
\end{center}
\end{proof}

\begin{C}\label{twistedcor2}
    There is an inclusion $G \hookrightarrow \Res_{E/F}G_E$ over $F$ induced by the $F$-algebra map $R \xrightarrow[]{} R \otimes_F E$ sending $r \mapsto r \otimes 1$, for an $F$-algebra $R$. Then, we have that $G = (\Res_{E/F}G_E)^{\Gamma_{E/F}}$.
\end{C}

\subsubsection{Restriction of scalars and isocrystals}
We will use the same notation as in Section \ref{2.1}. Recall the setup.
\begin{itemize}
    \item Let $F$ denote a $p$-adic field and let $F'$ be a finite extension of $F$.
    \item Let $L'$ and $\sigma'$ be the analogues, for $F'$, of $L$ and $\sigma$ introduced there.
    \item Let $G$ be an affine algebraic group defined over $F'$.
\end{itemize}

We can view $L$ and $L'$ as subfields of $\overline{F}$ and since the compositum of unramified extensions is unramified, it follows that $L$ is contained in $L'$. Hence, automorphisms of $L'$ over $F'$ can be restricted to automorphisms of $L$ over $F$. This gives a map between the groups $\langle\sigma'\rangle \xrightarrow[]{} \langle \sigma \rangle$.

\begin{itemize}
    \item Let $G_1$ denote a locally compact topological group and let $G_2 \subset G_1$ be a closed subgroup.
    \item Let $\Ind_{G_2}^{G_1}$ denote the induction functor from the category of (continuous) $G_2$-modules to (continuous) $G_1$-modules.
\end{itemize}

With respect to the map $\langle\sigma'\rangle \xrightarrow[]{} \langle \sigma \rangle$, we have the following.

\begin{Le}\label{Shapiro}
\begin{enumerate}
    \item The $\langle \sigma \rangle$-module $\Res_{F'/F}G(L)$ is $\Ind_{\langle \sigma' \rangle}^{\langle \sigma \rangle} G(L')$.
    \item There is a bijection $B(\Res_{F'/F}G) \xrightarrow[]{} B(G)$.
\end{enumerate}
\end{Le}
\begin{proof}
    This is $\S 1.10$ of \cite{MR809866} and $\S 1.6$ of \cite{MR1485921}. Using Lemma \ref{twistedlem1}, it can be checked that the $W_F$-module $\Res_{F'/F}G(\overline{L}) = \Ind_{W_{F'}}^{W_F} G(\overline{L'})$. Note that $\overline{L} = \overline{L'}$. By Shapiro's lemma, it follows that $H^1(W_F, \Res_{F'/F}G(\overline{L})) = H^1(W_{F'}, G(\overline{L'}))$. Hence, by \ref{generalBG}, the isomorphism $B(\Res_{F'/F}G) \xrightarrow[]{} B(G)$ follows.
\end{proof}

\subsubsection{Norm map}
We need the following additional notation before introducing the norm map.
\begin{itemize}
    \item Let $F$ be a field of characteristic $0$ and let $E$ be a cyclic extension of degree $l$.
    \item Let $\sigma$ denote a generator of $\Gal(E/F)$.
    \item As before, let $G$ be a connected reductive group over $F$ with quasi-split form $G_0$ and an inner twisting $\psi: G_{\overline{F}} \xrightarrow[]{\sim} G_{0,\overline{F}}$.
    \item For $g \in G(E)$ and $\tau \in \Gal(E/F)$, $g^{\tau}$ denotes the action of $\tau$ on $g$.
\end{itemize}

\begin{D}
    A conjugacy class $C$ in $G(\overline{F})$ is said to be defined over $F$ if for every $x \in C$ and for every $\tau \in \Gamma_F$, we have that $x^{\tau} \in C$.
\end{D}

\begin{Le}\label{conjugacyclassdefoverFlemma}
    If $G^{\text{der}}$ is simply connected, then every semisimple conjugacy class in $G_0(\overline{F})$ that is defined over $F$, contains an element of $G_0(F)$. Hence, to such a semisimple conjugacy class in $G_0(\overline{F})$, one can associate an $\overline{F}$-conjugacy class in $G_0(F)$. Since $G^{\text{der}}$ is simply connected, an $\overline{F}$-conjugacy class in $G_0(F)$ is the same as a stable conjugacy class in $G_0(F)$ (see \ref{Fbarconjugacy} and \ref{StableConjugacy}).
\end{Le}
\begin{proof}
    See Theorem 4.1 of \cite{MR683003}.
\end{proof}

\begin{D}[Norm mapping]\label{normmap}
    We define two versions of the norm mapping, $N: G(E) \xrightarrow[]{} G(E)$ and $\mathfrak{N}: G(E) \xrightarrow[]{} \{ \text{stable conjugacy classes in }G_0(F) \}$.
    \begin{enumerate}
        \item Consider the map $N: G(E) \xrightarrow[]{} G(E)$ defined as $Nx = x \cdot x^{\sigma} \cdots x^{\sigma^{l-1}}$. This has the property that $(Nx)^{\sigma} = x^{-1} (Nx) x$. Hence, it follows that the $G(\overline{F})$-conjugacy class of $Nx$ is defined over $F$.
        \item We also have that the $G_0(\overline{F})$-conjugacy class of $\psi(Nx)$ is defined over $F$. Assume now that $G^{\text{der}}$ is simply connected. Hence, by \ref{conjugacyclassdefoverFlemma}, there is a stable conjugacy class in $G_0(F)$ corresponding to this $G_0(\overline{F})$-conjugacy class, which will be denoted by $\mathfrak{N}x$.
    \end{enumerate}
\end{D}

We can extend the definition of $\mathfrak{N}$ to groups $G$ beyond the $G^{\text{der}}$ simply connected case. This is done in $\S$ 5 of \cite{MR683003} using $z$-extensions.

\begin{D}[$z$-extension]
    An $F$-homomorphism $\alpha:G' \xrightarrow[]{} G$ is a $z$-extension of $G$ if
    \begin{enumerate}
        \item The group $G'$ is connected reductive over $F$ and its derived group is simply connected.
        \item The map $\alpha$ is surjective.
        \item The kernel $\ker(\alpha)$ is central in $G'$ and is a product of tori of the form $\Res_{K/F} \mathbb{G}_m$, where each $K$ is a finite extension of $F$.
    \end{enumerate}
\end{D}

\begin{Le}\label{zextensionlemma0}
    Let $\alpha: G' \xrightarrow[]{} G$ be a $z$-extension. Then for any subfield $E$ of $\overline{F}$, the map $G'(E) \xrightarrow[]{} G(E)$ is surjective.
\end{Le}
\begin{proof}
    See Lemma 1.1 of \cite{MR683003}.
\end{proof}

\begin{Le}\label{zextensionlemma}
    Let $\alpha: G' \xrightarrow[]{} G$ be a $z$-extension. Let $G'_0$ denote the quasi-split form of $G'$. Then there exists an inner twisting $\psi': G'_{\overline{F}} \xrightarrow[]{\sim} G'_{0,\overline{F}}$ and a $z$-extension $\beta: G'_0 \xrightarrow[]{} G_0$ such that the following diagram of groups over $\overline{F}$ commutes
    \begin{center}
        \begin{tikzcd}
        G'_{\overline{F}} \arrow{r}{\psi'} \arrow{d}{\alpha} & G'_{0,\overline{F}} \arrow{d}{\beta} \\
        G_{\overline{F}} \arrow{r}{\psi} & G_{0,\overline{F}}
    \end{tikzcd}
    \end{center}    
\end{Le}
\begin{proof}
    See Lemma 5.1 of \cite{MR683003}.
\end{proof}

\begin{Le}\label{normmapdef2}
    There is a map $\mathfrak{N}: G(E) \xrightarrow[]{} \{ \text{stable conjugacy classes in } G_0(F) \}$ such that for any $z$-extension $\alpha: G' \xrightarrow[]{} G$ with $\psi': G'_{\overline{F}} \xrightarrow{\sim} G'_{0,\overline{F}}$ and $\beta: G'_0 \xrightarrow[]{} G_0$ as in \ref{zextensionlemma}, the following square commutes
    \begin{center}
        \begin{tikzcd}
        G'(E) \arrow{r}{\mathfrak{N}} \arrow{d}{\alpha} & \{ \text{stable conjugacy classes in } G'_0(F) \} \arrow{d}{\beta} \\
        G(E) \arrow{r}{\mathfrak{N}} & \{ \text{stable conjugacy classes in } G_0(F) \}
    \end{tikzcd}
    \end{center}
    where here the horizontal map $\mathfrak{N}$ on the top is the one discussed in \ref{normmap}, since $G'$ has a simply connected derived subgroup. See \ref{rmkaboutstableconj} for the vertical map on the right. Note that, using \ref{zextensionlemma0}, it follows that the above property characterizes the map $\mathfrak{N}$ uniquely.
\end{Le}
\begin{proof}
    See $\S$ 5 of \cite{MR683003}.
\end{proof}

\subsubsection{Twisted conjugacy}\label{twistedconjsection}

We need to define additional notions of conjugacy to be able to describe the fibres of the map $\mathfrak{N}$. This is what we will do in this subsubsection. Recall that $G$ is a group over $F$ and there is an $F$-automorphism $\theta_{\sigma}:\Res_{E/F}G \xrightarrow[]{} \Res_{E/F}G$, for $\sigma \in \Gal(E/F)$ a generator of the cyclic group. We will use the following shorthand notation.

\begin{itemize}
    \item Let $R = \Res_{E/F}G$ and $s = \theta_{\sigma} \in \Aut_F(R)$.
\end{itemize}

In \ref{normmap}, we defined $N: G(E) \xrightarrow[]{} G(E)$. Now, we can upgrade $N$ to an $F$-homomorphism $R \xrightarrow[]{} R$, as follows.

\begin{D}[Norm mapping]
    The map $N: R \xrightarrow[]{} R$ is the $F$-homomorphism defined by the condition $Nx = x \cdot x^s \cdots x^{s^{l-1}}$. On the $F$-points, this coincides with $N: R(F) \xrightarrow[]{} R(F)$ defined in \ref{normmap}.
\end{D}

Now, we will define the notion of $\sigma$-conjugacy and $\sigma$-centralizers.

\begin{D}[$\sigma$-conjugacy]\label{sigmaconjdef}
    Two elements $x , y \in G(E) = R(F)$ are $\sigma$-conjugate if there exists $g \in G(E) = R(F)$ such that $y = gx\sigma(g^{-1}) = gxs(g^{-1})$.
\end{D}

\begin{D}[$\sigma$-centralizers]\label{sigmacentralizerdef}
    For $x \in G(E) = R(F)$, the $\sigma$-centralizer of $x$ is the $F$-group $I_{sx} = \{ g \in R \mid gxs(g^{-1}) = x \}$.
\end{D}

Using \ref{twistedcor1}, we can write $R_E \simeq G_E \times \cdots \times G_E$, where the factors are indexed by elements of $\Gal(E/F)$ as in \ref{twistedlem1}. Let $p: R_E \xrightarrow{} G_E$ be the projection onto the factor indexed by the identity element of $\Gal(E/F)$.

\begin{Le}\label{twistedlemma}
    For $x \in G(E) = R(F)$, the map $p$ induces an isomorphism $(I_{sx})_E \simeq (G_E)_{Nx}$.
\end{Le}
\begin{proof}
    See Lemma 5.4 of \cite{MR683003}.
\end{proof}

Now, we will define some important notions of conjugacy which are twisted analogues of $\overline{F}$-conjugacy and stable conjugacy defined in \ref{Fbarconjugacy} and \ref{StableConjugacy} respectively.

\begin{D}[$\overline{F}$-$\sigma$-conjugacy]
    Two elements $x,y \in G(E) = R(F)$ are $\overline{F}$-$\sigma$-conjugate if there exists $g \in R(\overline{F})$ such that $y = gxs(g^{-1})$.
\end{D}

We need to define the group $I_{sx}^*$ which will be the analogue of $G_x^*$ of \ref{stargroupdef} before we can define stable $\sigma$-conjugacy.

\begin{D}[$I_{sx}^*$]
    The group $I_{sx}^*$ is the inverse image under the $E$-isomorphism $p: I_{sx} \xrightarrow[]{} G_{Nx}$ (\ref{twistedlemma}) of the subgroup $G_{Nx}^* \subset G_{Nx}$.
\end{D}

This defines $I_{sx}^*$ as a group over $E$. However, it admits an $F$-structure, which we will describe below. Note that, using \ref{simplyconnectedlemma}, it follows that if $G^{\text{der}}$ is simply connected, then $G_{Nx}^* = G_{Nx}$ and hence $I_{sx}^* = I_{sx}$. Hence, there is an $F$-structure in this case. For the general case, we choose a $z$-extension $\alpha: G' \xrightarrow[]{} G$. Let $R' = \Res_{E/F}G'$. Then, it follows that the induced map $\gamma: R' \xrightarrow[]{} R$ is a $z$-extension.

\begin{Le}
    Let $x \in G(E)$. Then, by \ref{zextensionlemma0}, we can choose a $y \in G'(E) = R'(F)$ such that $\alpha(y) = x$. Let $I'_{sy}$ denote the $\sigma$-centralizer of $y$. Then, we have $\gamma(I'_{sy}) = I_{sx}^*$.
\end{Le}
\begin{proof}
    See Lemma 5.5 of \cite{MR683003}.
\end{proof}

Hence, with this lemma, we have an alternate description of $I_{sx}^*$ as the subgroup of $I_{sx}$ which is the image of the map $\gamma: I'_{sy} \xrightarrow[]{} I_{sx}$. This description gives an $F$-structure on $I_{sx}^*$.

\begin{C}
    The group $I_{sx}^*$ is an $F$-subgroup of $I_{sx}$.
\end{C}

\begin{D}[Stable $\sigma$-conjugacy]
    Two elements $x,y \in G(E) = R(F)$ are stably $\sigma$-conjugate if there exists $g \in R(\overline{F})$ such that $y = gxs(g^{-1})$ and $g^{-1} \cdot g^{\tau} \in I_{sx}^*(\overline{F})$ for all $\tau \in \Gamma_F$. For $E = F$, so that $s$ is the identity, this recovers \ref{StableConjugacy}.
\end{D}

Now, we can finally describe the fibres of the map $\mathfrak{N}$.

\begin{Le}
    Let $x,y \in G(E)$. Then $x,y$ are stably $\sigma$-conjugate if and only if $\mathfrak{N}x = \mathfrak{N}y$.
\end{Le}
\begin{proof}
    See Proposition 5.7 of \cite{MR683003}.
\end{proof}

\section{Proofs of Lemma A and Lemma B}\label{2.4}
\subsection{Isocrystal attached to a stable conjugacy class}
We will continue using the notation introduced in the previous sections.
\begin{itemize}
    \item Let $G$ be a connected reductive group over a $p$-adic field $F$.
    \item Let $\gamma$ be a semisimple conjugacy class in $G(F)$, with identity connected component of the centralizer $I = G_{\gamma}^0$.
    \item Using the map $G(F) \xrightarrow[]{} G(L)$, we can associate to $\gamma$ an element of $B(G)$, we denote by $b_{\gamma}^G \in B(G)$.
\end{itemize}

\begin{Le}\label{isocrystalstableconj}
    If $\gamma' \in G(F)$ is stably conjugate to $\gamma \in G(F)$, then $b_{\gamma}^G = b_{\gamma'}^G$.
\end{Le}
\begin{proof}
    Since $\gamma, \gamma'$ are stably conjugate, there exists $g \in G(\overline{F})$ such that $\gamma' = g \gamma g^{-1}$ and the cocycle $\tau \mapsto c_{\tau} = g^{-1} \tau (g)$ for $\tau \in \Gamma_F$ lies in $H^1(F,I)$ (see \ref{StableConjugacy}).
    
    \textit{Centralizers $I, I'$ are inner forms.} Let $I' = G_{\gamma'}^0$. Semisimplicity of $\gamma$ and $\gamma'$ implies that $I$ and $I'$ are connected reductive groups over $F$, while stable conjugacy implies that $\Int(g):I \xrightarrow[]{} I'$ is an inner twisting by \ref{stabconjlem1}. Moreover, $gG_{\gamma}g^{-1}=G_{\gamma'}$ and $gIg^{-1}=I'$.

    \textit{The stable conjugacy cocycle determines an element $b_{\tau}^I \in B(I)$.} The coset $X=gI$ is defined over $F$, since $g^{-1}\tau(g)=c_\tau\in I$ for every $\tau\in\Gamma_F$, and is an $I$-torsor whose class is represented by the cocycle $\tau\mapsto c_\tau$. Now $H^1(L,I)$ is trivial: indeed, $L$ is complete with algebraically closed residue field $\overline{k_F}$ and has cohomological dimension at most $1$, so Steinberg's theorem applies. Hence $X(L)$ is non-empty, and after choosing an $L$-point of $X$ we may assume that $g\in G(L)$. In the description of the map $H^1(F,I)\xrightarrow[]{}B(I)$ of \ref{cohomisocrystal} obtained from this trivialization over $L$, the class of $X$ maps to
    $$b_\tau^I=g^{-1}\sigma(g)\in B(I).$$
    This element is basic by Lemma \ref{basicandcohomlemma}.
    
    \textit{Transfer of isocrystals for inner forms $I,I'$.} Now, we will apply Lemma \ref{isocrystalsandinnerforms} for the element $b_{\tau}^I \in B(I)_b$. First, note that $J_{b_{\tau}^I} = I'$, by using Lemma \ref{characterizationJ} (in the application of this lemma, we have $u = \Int(g^{-1})$, where $\Int(g^{-1}): (-) \mapsto g^{-1}(-)g$). Next, we can view $\gamma'$ as an element of $B(I') = B(J_{b_{\tau}^I})$ using the map $I'(F) \xrightarrow[]{} I'(L)$. Denote this element by $b_{\gamma'}^{I'}$ (we use superscripts to distinguish it from $b_{\gamma'}^G$). Similarly, we have the element $b_{\gamma}^I \in B(I)$. Thus, under the map $B(J_{b_{\tau}}) \xrightarrow[]{} B(I)$ of lemma \ref{isocrystalsandinnerforms}, the element $b_{\gamma'}^{I'}$ is mapped to
    $$ u(b_{\gamma'}^{I'}) b_{\tau}^I = \Int(g^{-1}) (\gamma') b_{\tau}^I = (g^{-1} \gamma' g) (g^{-1} \sigma(g)) = g^{-1} \gamma' \sigma(g) = \gamma g^{-1} \sigma(g) = b_{\gamma}^Ib_{\tau}^I.$$

    \textit{Relationship between Newton points of $b_{\gamma'}^{I'} \in B(I')$ and $b_{\gamma}^Ib_{\tau}^I \in B(I)$}. Let $I_0$ denote the quasi split inner form of $I$ (and also $I'$) with a maximal split torus $A_{I_0} \subset I_0$. Let $\mathfrak{U}_{A_{I_0}} = X_*(A_{I_0}) \otimes \mathbb{R}$ (similar to the notation for the Newton map in Section \ref{2.1}). Let $\overline{\nu}_I$ and $\overline{\nu}_{I'}$ denote the Newton maps of $I$ and $I'$ respectively. Then, we have the following commutative square from Lemma \ref{isocrystalsandinnerforms}
    \begin{center}
        \begin{tikzcd}
            B(I') \arrow{r}{\times b_{\tau}^I} \arrow{d}{\overline{\nu}_{I'}} & B(I) \arrow{d}{\overline{\nu}_I} \\
            \mathfrak{U}_{A_{I_0}} \arrow{r}{+ \overline{\nu}_I(b_{\tau}^I)} & \mathfrak{U}_{A_{I_0}}
        \end{tikzcd}
    \end{center}
    Here the bottom map is translation by $\overline{\nu}_I(b_{\tau}^I)$, which is trivial by Lemma \ref{basicandcohomlemma}, so the bottom map is the identity. Hence the Newton points of $b_{\gamma'}^{I'} \in B(I')$ and $b_{\gamma}^I b_{\tau}^I \in B(I)$ in $\mathfrak{U}_{A_{I_0}}$ are the same.

    \textit{Relationship between Kottwitz points of $b_{\gamma'}^{I'} \in B(I')$ and $b_{\gamma}^I b_{\tau}^I \in B(I)$}.
    Let $\kappa_I$ and $\kappa_{I'}$ denote the Kottwitz maps of $I$ and $I'$ respectively. Note that, since $I$ and $I'$ are inner forms, $X^*(Z(\Hat{I_0})^{\Gamma_F}) = X^*(Z(\Hat{I})^{\Gamma_F}) = X^*(Z(\Hat{I'})^{\Gamma_F})$. We have the following commutative square from Lemma \ref{isocrystalsandinnerforms}
    \begin{center}
        \begin{tikzcd}
            B(I') \arrow{r}{\times b_{\tau}^I} \arrow{d}{\kappa_{I'}} & B(I) \arrow{d}{\kappa_I} \\
            X^*(Z(\Hat{I_0})^{\Gamma_F}) \arrow{r}{+ \kappa_I(b_{\tau}^I)} & X^*(Z(\Hat{I_0})^{\Gamma_F})
        \end{tikzcd}
    \end{center}
    Hence, the Kottwitz point of $b_{\gamma}^I b_{\tau}^I \in B(I)$ is the sum of the Kottwitz point of $b_{\gamma'}^{I'} \in B(I')$ and the Kottwitz point of $b_{\tau}^{I} \in B(I)$.

    \textit{Property of the Kottwitz point of $b_{\tau}^I \in B(I)$}. Combining Lemma \ref{KottwitzTateNakayama} and Lemma \ref{conjclassinsidestable} we get the diagram
    \begin{center}
        \begin{tikzcd}
            B(I) \arrow{d}{\gamma} & H^1(F,I) \arrow{r}{} \arrow{l}{} \arrow{d}{\sim} & H^1(F,G) \arrow{d}{\sim} \\
            X^*(Z(\Hat{I})^{\Gamma_F}) & X^*(\pi_0(Z(\Hat{I})^{\Gamma_F})) \arrow{l}{} \arrow{r}{} & X^*(\pi_0(Z(\Hat{G})^{\Gamma_F}))
        \end{tikzcd}
    \end{center}
    Since $b_{\tau}^I \in \ker(H^1(F,I) \xrightarrow[]{} H^1(F,G))$, it follows that $\kappa_I(b_{\tau}^I)$ is a character of $Z(\Hat{I})^{\Gamma_F}$ that is constant on the connected components and trivial on the image of $Z(\Hat{G})^{\Gamma_F}$ under the embedding $Z(\Hat{G}) \xrightarrow[]{} Z(\Hat{I})$ of \ref{stableconjconst}.

    \textit{The isocrystal $b_{\gamma}^I \in B(I)$ is basic}. Since $\gamma \in Z(I)(F)$, choose a maximal torus $T_I$ of $I$ defined over $F$. Then $\gamma \in T_I$, and we denote by $b_{\gamma}^{T_I}$ the corresponding element of $B(T_I)$. For every root $\alpha$ of $I$ with respect to $T_I$, Corollary \ref{slopetorispecialcase} gives
    $$\langle \nu_{T_I}(b_{\gamma}^{T_I}),\alpha\rangle=\val(\alpha(\gamma))=0,$$
    since $\gamma \in Z(I)$. Thus $\nu_{T_I}(b_{\gamma}^{T_I})$ factors through $Z(I)$, or equivalently is invariant under $\Omega_I(\overline{F})$. By the functoriality of the slope morphism for $T_I \xrightarrow[]{} I$ (see \ref{slopefunctoriality}), the slope morphism of $b_{\gamma}^I$ also factors through $Z(I)$, so $b_{\gamma}^I$ is basic.

    \textit{Transfer of isocrystals on $I$.} We will apply Lemma \ref{isocrystalsandinnerforms} for the element $b_{\gamma}^I \in B(I)_b$. Since $\gamma \in Z(I)(F)$, it follows from \ref{characterizationJ} that $J_{b_{\gamma}^I} = I$ and $u$ is the identity map. Thus, under the map $B(I) \xrightarrow[]{} B(I)$ of Lemma \ref{isocrystalsandinnerforms}, the element $b_{\tau}^I \in B(I)$ is mapped to $b_{\tau}^I b_{\gamma}^I = b_{\gamma}^Ib_{\tau}^I$ (the commutativity holds since $\gamma \in Z(I)$).
    
    \textit{Relationship between Newton points of $b_{\gamma}^I \in B(I)$ and $b_{\gamma}^I b_{\tau}^I \in B(I)$}. Application of Lemma \ref{isocrystalsandinnerforms} in the above setup gives
    \begin{center}
        \begin{tikzcd}
            B(I) \arrow{r}{\times b_{\gamma}^I} \arrow{d}{\overline{\nu}_I} & B(I) \arrow{d}{\overline{\nu}_I} \\
            \mathfrak{U}_{A_{I_0}} \arrow{r}{+ \overline{\nu}_I(b_{\gamma}^I)} & \mathfrak{U}_{A_{I_0}}
        \end{tikzcd}
    \end{center}
    Applying this square to $b_{\tau}^I$ and using $\overline{\nu}_I(b_{\tau}^I)=0$ by \ref{basicandcohomlemma}, we obtain
    $$\overline{\nu}_I(b_{\gamma}^I b_{\tau}^I)=\overline{\nu}_I(b_{\tau}^I)+\overline{\nu}_I(b_{\gamma}^I)=\overline{\nu}_I(b_{\gamma}^I).$$
    Thus $b_{\gamma}^I$ and $b_{\gamma}^I b_{\tau}^I$ have the same Newton point.

    \textit{Relationship between Kottwitz points of $b_{\gamma}^I \in B(I)$ and $b_{\gamma}^I b_{\tau}^I \in B(I)$}. Continuing with the application of Lemma \ref{isocrystalsandinnerforms} we get
    \begin{center}
        \begin{tikzcd}
            B(I) \arrow{r}{\times b_{\gamma}^I} \arrow{d}{\kappa_I} & B(I) \arrow{d}{\kappa_I} \\
            X^*(Z(\Hat{I})^{\Gamma_F}) \arrow{r}{+ \kappa_I(b_{\gamma}^I)} & X^*(Z(\Hat{I})^{\Gamma_F})
        \end{tikzcd}
    \end{center}
    Hence, the Kottwitz point of $b_{\gamma}^I b_{\tau}^I \in B(I)$ is the sum of the Kottwitz point of $b_{\gamma}^I$ and the Kottwitz point of $b_{\tau}^I$.

    \textit{Relation between Newton points of $b_{\gamma}^G, b_{\gamma'}^G \in B(G)$}. We want to show that the isocrystals $b_{\gamma}^G, b_{\gamma'}^G \in B(G)$ have the same Newton point in $\mathcal{N}(G)$. It suffices to show this for $\gamma$ and $g^{-1}\gamma'\sigma(g)$. Note that, $g^{-1} \gamma' \sigma(g) = b_{\gamma}^I b_{\tau}^I$ in $B(I)$. Since $b_{\gamma}^I$ and $b_{\gamma}^I b_{\tau}^I$ have the same Newton point, they also have the same slope morphism: $\nu_I(b_{\gamma}^I) = \nu_I(b_{\gamma}^I b_{\tau}^I) \in \mathcal{N}(I)$. Using the functoriality of the slope morphism \ref{slopefunctoriality}, it follows that $\gamma$ and $g^{-1}\gamma'\sigma(g)$ have the same slope morphism, and hence the same Newton point in $\mathcal{N}(G)$. Equality of the isocrystals themselves is deduced only in the conclusion below, from the Newton and Kottwitz points together.

    \textit{Relation between Kottwitz points of $b_{\gamma}^G, b_{\gamma'}^G \in B(G)$}. By the preceding discussion, the character $\kappa_I(b_{\tau}^I)$ restricts trivially to $Z(\Hat{G})^{\Gamma_F}$. Functoriality of the Kottwitz map for $I \xrightarrow[]{} G$, together with
    $$\kappa_I(b_{\gamma}^I b_{\tau}^I)=\kappa_I(b_{\gamma}^I)+\kappa_I(b_{\tau}^I),$$
    therefore shows that $b_{\gamma}^I$ and $b_{\gamma}^I b_{\tau}^I$ have the same Kottwitz point after mapping to $B(G)$. Since $g^{-1}\gamma'\sigma(g)=b_{\gamma}^I b_{\tau}^I$ is $\sigma$-conjugate to $\gamma'$ in $G(L)$, it follows that $b_{\gamma}^G$ and $b_{\gamma'}^G$ have the same Kottwitz point.

    \textit{Conclusion}. We have shown that the Newton point and Kottwitz point of $\gamma, \gamma' \in B(G)$ are equal. Hence, by Lemma \ref{Newonkottwitzcharisoc}, it follows that $b_{\gamma}^G = b_{\gamma'}^G \in B(G)$.
\end{proof}

\subsection{Alternative construction of slope morphism}
We will work with the following setup in this subsection.
\begin{itemize}
    \item Let $G$ be a linear algebraic group over a $p$-adic field $F$.
    \item Let $\gamma \in G(F)$, which we can naturally view as a $G$-isocrystal under the embedding $G(F) \xrightarrow[]{} G(L)$.
\end{itemize}

In this subsection, we want to give an alternative construction of the slope morphism $\nu_G(\gamma) \in \Hom_L(\mathbb{D},G)$ associated to $\gamma$.

\begin{D}[$\langle \gamma \rangle$]
    For $\gamma \in G(F)$, let $\langle \gamma \rangle$ denote the Zariski closure of the group generated by $\gamma$. This is a closed commutative $F$-subgroup of $G$. It need not be connected: for a $\gamma$ of finite order in $\mathbb{G}_m$ it is a finite group. By \ref{delignelemma} below, it is the direct product of a group of multiplicative type with a unipotent group.
\end{D}

\begin{Le}\label{delignelemma}
    Let $\gamma = \gamma_m\gamma_u$ be the Jordan decomposition of $\gamma$ with $\gamma_m \in G(F)$ the semisimple part and $\gamma_u \in G(F)$ the unipotent part. The group $\langle \gamma \rangle$ is the direct product of a group $\langle \gamma \rangle_m$ of multiplicative type with a unipotent group $\langle \gamma \rangle_u$, where $\langle \gamma \rangle_m = \langle \gamma_m \rangle$ and $\langle \gamma \rangle_u = \langle \gamma_u \rangle$.
\end{Le}
\begin{proof}
    See $\S$ 1 of \cite{MR425033}.
\end{proof}

In the case when $G = \GL_n$, there is a description of the multiplicative group $\langle \gamma \rangle_m = \langle \gamma_m \rangle$ as follows.

\begin{Le}\label{deligneGLn}
    Let $G = \GL_n$. Evaluation at $\gamma_m$ identifies the group $X^*(\langle \gamma \rangle_m)$ of characters over $\overline{F}$ of $\langle \gamma \rangle_m$ with the subgroup of $\overline{F}^{\times}$ generated by the eigenvalues of $\gamma_m \in \GL_n(F)$. For a maximal torus $T$ such that $\gamma_m \in T$, the group $\langle \gamma \rangle_m$ is precisely the diagonalizable subgroup of $T$ defined by the multiplicative relations between the eigenvalues.
\end{Le}
\begin{proof}
    See $\S$ 1 of \cite{MR425033}.
\end{proof}

\begin{D}[Alternative construction of slope morphism]\label{delignecocharacter}
    For $\gamma \in G(F)$, we construct a map $m_G(\gamma): \mathbb{D} \xrightarrow[]{} \langle \gamma \rangle_m \subset G$ over $F$. Since $\langle \gamma \rangle_m$ is a group of multiplicative type, giving such a map is the same as giving a $\Gamma_F$-equivariant homomorphism $X^*(\langle \gamma \rangle_m) \xrightarrow[]{} X^*(\mathbb{D}) = \mathbb{Q}$, and we take it to be
    $$\chi \mapsto \val(\chi(\gamma_m)), \qquad \chi \in X^*(\langle \gamma \rangle_m).$$
    This is $\Gamma_F$-equivariant because $\val$ is Galois-invariant, so $m_G(\gamma)$ is defined over $F$. The construction involves no choice of embedding, and $m_G(\gamma)$ takes values in $\langle \gamma \rangle_m \subset G$ by definition. It has the following two properties.
    \begin{enumerate}
        \item Let $G = \GL_n$. Under the identification of \ref{deligneGLn}, which is evaluation at $\gamma_m$, the map on character groups is described by the commutative square
        \begin{center}
            \begin{tikzcd}
            X^*(\langle \gamma \rangle_m) \arrow{r}{m_G(\gamma)} \arrow{d}{\sim} & X^*(\mathbb{D}) \arrow{d}{\sim} \\
            \langle \text{eigenvalues of } \gamma \rangle \subset \overline{F}^{\times} \arrow{r}{\val_{\pi}} & \mathbb{Q}
        \end{tikzcd}
        \end{center}
        so that $m_G(\gamma)$ records the $\pi$-adic valuations of the eigenvalues of $\gamma$.
        \item The construction is functorial: for any homomorphism $f: G \xrightarrow[]{} H$ of linear algebraic groups over $F$ one has $f \circ m_G(\gamma) = m_H(f(\gamma))$. Indeed, since $F$ has characteristic zero, $f$ preserves Jordan decompositions, so $f(\gamma_m) = f(\gamma)_m$ and $f$ carries $\langle \gamma \rangle_m$ into $\langle f(\gamma) \rangle_m$; the induced map on character groups is compatible with $\chi \mapsto \val(\chi(\gamma_m))$, because $\val(\chi(f(\gamma)_m)) = \val((\chi \circ f)(\gamma_m))$. In particular this applies to a representation $\rho: G \xrightarrow[]{} \GL(W)$ and to the inclusion of a subgroup of $G$.
    \end{enumerate}
    This cocharacter is due to Deligne; see $\S$ 1 of \cite{MR425033}.
\end{D}

\begin{R}\label{deligneslopedecomposition}
    Choose a positive integer $n$ such that $n \cdot \val(\chi(\gamma_m)) \in \mathbb{Z}$ for every $\chi \in X^*(\langle \gamma \rangle_m)$. Then, the mapping $n \cdot m_G(\gamma): \mathbb{D} \xrightarrow[]{} G$ factors through $\mathbb{G}_m$, where $\mathbb{D} \xrightarrow[]{} \mathbb{G}_m$ corresponds to the inclusion $\mathbb{Z} \subset \mathbb{Q}$ on character lattices. Hence, we can view $n \cdot m_G(\gamma): \mathbb{G}_m \xrightarrow[]{} G$. For a representation $\rho: G \xrightarrow[]{} \GL(W)$, using $n \cdot m_G(\gamma)$ we get an action of $\mathbb{G}_m$ on the $F$-vector space $W$. This gives a $\mathbb{Z}$-grading on $W$, by decomposing it as a sum of characters of $\mathbb{G}_m$. The degree $d$ graded part of $W$ is the sum of generalized eigenspaces of $\rho(\gamma): W \xrightarrow[]{} W$ corresponding to eigenvalues $\alpha \in \overline{F}^{\times}$ such that $n \cdot \val(\alpha) = d$. Although the eigenvalues themselves lie in $\overline{F}$, this condition on $\alpha$ is $\Gamma_F$-stable, since $\val$ is Galois-invariant, so each graded piece is defined over $F$.
\end{R}

For the rest of this subsection, we have the setup
    \begin{itemize}
        \item Let $G$ be a connected reductive group over a $p$-adic field $F$ with uniformizer $\pi$.
        \item Let $m: \mathbb{G}_m \xrightarrow[]{} G$ be a cocharacter of $G$ defined over $F$.
    \end{itemize}

Such an $m$ determines a pair of opposite parabolic subgroups of $G$ over $F$, denoted by $P_m^+$ and $P_m^-$ and referred to as the parabolics contracted and dilated by $m$, with unipotent radicals $U_m^+$ and $U_m^-$. The groups $P_m^+$ and $U_m^+$ are defined as follows; the groups $P_m^-$ and $U_m^-$ are defined by the same recipe applied to $m^{-1}$, equivalently by using $\lim_{\lambda \to \infty}$ in place of $\lim_{\lambda \to 0}$.

\begin{D}\label{contracteddilated}
    For $x \in G(\overline{F})$, write $c_x$ for the morphism
    $$c_x: \mathbb{G}_m \xrightarrow[]{} G, \qquad \lambda \mapsto m(\lambda) \, x \, m(\lambda)^{-1}.$$
    \begin{enumerate}
        \item The element $x$ lies in $P_m^+$ if and only if $c_x$ extends to a morphism $\mathbb{A}^1 \xrightarrow[]{} G$, that is, if and only if $\lim_{\lambda \to 0} c_x(\lambda)$ exists.
        \item The element $x$ lies in the unipotent part $U_m^+$ if and only if, in addition, $\lim_{\lambda \to 0} c_x(\lambda) = e$.
    \end{enumerate}
\end{D}

\begin{R}\label{parabolicsofgamma}
    The morphism $m_G(\gamma): \mathbb{D} \xrightarrow[]{} G$ is only a rational cocharacter, so \ref{contracteddilated} does not apply to it directly. We therefore set $P_{\gamma}^{\pm} := P_{n \cdot m_G(\gamma)}^{\pm}$ and $U_{\gamma}^{\pm} := U_{n \cdot m_G(\gamma)}^{\pm}$, for a positive integer $n$ as in \ref{deligneslopedecomposition}, and refer to these as the parabolics contracted and dilated by $\gamma$. The result does not depend on $n$: replacing $n$ by $kn$ replaces the cocharacter $\mu$ by $\lambda \mapsto \mu(\lambda^k)$, and $\lambda \to 0$ if and only if $\lambda^k \to 0$, so the limit conditions above are unchanged. The intuitive idea is that the homomorphism $m_G(\gamma)$ records the $\pi$-adic valuations of the eigenvalues of $\gamma$.
\end{R}

\begin{R}\label{liealgebraparabolic}
    The Lie algebra of $P_{\gamma}^+$ (resp. $U_{\gamma}^+$) is the sum of the generalized eigenspaces of $\Ad(\gamma)$ on $\Lie(G)$ corresponding to eigenvalues of valuation $\geq 0$ (resp. $>0$).
\end{R}

Note that both $\nu_G(\gamma)$ and $m_G(\gamma)$ are honest elements of $\Hom_L(\mathbb{D}, G)$ attached to the element $\gamma$, so that the equality in the following lemma is an equality of morphisms, not merely of $G(L)$-conjugacy classes. It is only the induced map on $B(G)$ that takes values in
$$(\Hom_L(\mathbb{D}, G)/G(L))^{\langle \sigma \rangle},$$
by \ref{slopebg}.

\begin{Le}\label{equalityofslopes}
    Let $G$ be a connected reductive group over $F$. For $\gamma \in G(F)$, we have that $\nu_G(\gamma) = m_G(\gamma) \in \Hom_L(\mathbb{D}, G)$.
\end{Le}
\begin{proof}
    \textit{Case of a torus}. We start with the case that $G = T$ is a torus over $F$ and $\gamma \in T(F)$. The slope morphism of $\gamma$ has a characterization using Lemma \ref{slopetori}. Hence, it suffices to show that $m_G(\gamma)$ satisfies the property that $$\val(\lambda(\gamma)) = \langle m_G(\gamma), \lambda \rangle$$ for all $\lambda \in X^*(T)^{\Gamma_F}$, that is, for all characters $\lambda$ of $T$ defined over $F$. Viewing a character of $T$ as a 1-dimensional representation, this property follows from the characterization of $m_G(\gamma)$ in \ref{delignecocharacter}. Hence $m_G(\gamma) = \nu_G(\gamma)$ in this case.
    
    \textit{Case of a semisimple $\gamma \in G(F)$}. Assume now that $\gamma \in G(F)$ is semisimple. Let $I$ denote the identity component of the centralizer of $\gamma$ in $G$. Then, $I$ is a reductive group over $F$ and hence admits a maximal torus $T$ over $F$. Then, $\gamma \in T(F)$ and $T$ is also a maximal torus in $G$. The slope morphism is functorial for connected linear algebraic groups (see \ref{slopefunctoriality}), and $m_G(\gamma)$ is functorial for all homomorphisms of linear algebraic groups by (2) of \ref{delignecocharacter}. Applying both to the inclusion $T \hookrightarrow G$, the equality $m_G(\gamma) = \nu_G(\gamma)$ now follows from the case of the torus.

    \textit{General case}. Decompose $\gamma \in G(F)$ as $\gamma = \gamma_m \cdot \gamma_u$ where $\gamma_m \in G(F)$ is semisimple and $\gamma_u \in G(F)$ is unipotent. It follows from \ref{delignecocharacter} that $m_G(\gamma) = m_G(\gamma_m)$. We now want to show that $\nu_G(\gamma) = \nu_G(\gamma_m)$. Note that the unipotent radical of $\langle \gamma \rangle$ is $\langle \gamma_u \rangle$ and $\langle \gamma_m \rangle$ is a Levi factor in $\langle \gamma \rangle$, by using Lemma \ref{delignelemma}. Applying Lemma \ref{unipotentlemma} to the group $P = \langle \gamma \rangle$, which does not require $\langle \gamma \rangle$ to be connected, we get that $$ B(\langle \gamma \rangle) \xrightarrow[]{} B(\langle \gamma \rangle/\langle \gamma_u \rangle)$$ is a bijection. Since $\gamma \gamma_m^{-1} = \gamma_u$, the elements $\gamma$ and $\gamma_m$ of $\langle \gamma \rangle(F)$ have the same image in $(\langle \gamma \rangle/\langle \gamma_u \rangle)(F)$, and hence define the same class in $B(\langle \gamma \rangle)$. By the construction of slope morphisms for arbitrary linear algebraic groups in $\S$ 3.2 of \cite{MR1485921}, cohomologous cocycles have slope morphisms related by inner conjugation. Since $\langle \gamma \rangle$ is commutative, this conjugation is trivial, and therefore $\nu_{\langle \gamma \rangle}(\gamma) = \nu_{\langle \gamma \rangle}(\gamma_m)$ as morphisms. Functoriality along $\langle \gamma \rangle \hookrightarrow G$ then gives the equality of morphisms $\nu_G(\gamma) = \nu_G(\gamma_m)$ in $\Hom_L(\mathbb{D}, G)$. Hence, the general case is now reduced to the semisimple case and the proof is complete.
\end{proof}

\begin{R}
    Note that Lemma \ref{equalityofslopes} gives an independent proof of the fact that if $\gamma, \gamma' \in G(F)$ are stably conjugate, then $\nu_G(b^G_\gamma) = \nu_G(b^G_{\gamma'}) \in (\Hom_L(\mathbb{D},G)/G(L))^{\langle \sigma \rangle}$, which also follows from Lemma \ref{isocrystalstableconj}. Recall that $b_{\gamma}^G \in B(G)$ and $b_{\gamma'}^G \in B(G)$ denote the $\sigma$-conjugacy class of $\gamma \in G(L)$ and $\gamma' \in G(L)$ resp. In fact, for the equality $\nu_G(b^G_\gamma) = \nu_G(b_{\gamma'}^G)$, it is sufficient to assume that $\gamma, \gamma'$ are $\overline{F}$-conjugate. This is because it follows from \ref{delignecocharacter} that, if $\gamma, \gamma'$ are $\overline{F}$-conjugate then the rational cocharacters $m_G(\gamma): \mathbb{D} \xrightarrow[]{} \langle \gamma \rangle_m \subset G$ and $m_G(\gamma'): \mathbb{D} \xrightarrow[]{} \langle \gamma' \rangle_m \subset G$ over $F$ are conjugate by an element of $G(\overline{F})$. Using Lemma \ref{equalityofslopes}, it follows that $\nu_G(\gamma)$ and $\nu_G(\gamma')$ are $G(\overline{F})$-conjugate. Hence, by using Lemma \ref{slopesoverclosure} it follows that $\nu_G(b^G_\gamma) = \nu_G(b_{\gamma'}^G) \in (\Hom_{\overline{F}}(\mathbb{D},G)/G(\overline{F}))^{\Gamma_F}$. 
\end{R}

\begin{R}
    Using Lemma \ref{equalityofslopes}, it follows that the decomposition of $F$-vector space $W$ for a representation $\rho: G \xrightarrow[]{} \GL(W)$ with respect to the cocharacter $n \cdot m_G(\gamma): \mathbb{G}_m \xrightarrow[]{} G$ explained in Remark \ref{deligneslopedecomposition} coincides with the slope decomposition of the isocrystal $(W \otimes L, \rho(\gamma)\circ\sigma )$.
\end{R}

\begin{R}
    In the case when $G$ is a quasi-split reductive group over $F$, every element of $B(G)$ is obtained as the image of $B(M)_b \xrightarrow[]{} B(G)$ for some standard Levi subgroup (see Lemma \ref{chailemma}). Let $b \in B(G)$ denote the class of $\gamma \in G(F)$ and let $M = M_{P_b}$ be the Levi factor of the parabolic subgroup contracted by $b$, as in \ref{paracontractediso}; this is the minimal standard Levi for which $b$ is basic. Then the Levi factor $M_{\gamma} = P_{\gamma}^+ \cap P_{\gamma}^-$ of the parabolic subgroup $P_{\gamma}^+$ contracted by $\gamma$ is a $G(\overline{F})$-conjugate of $M$.

    Both groups are centralizers of the Newton cocharacter. On the one hand, $P_{\gamma}^{\pm} = P_{n \cdot m_G(\gamma)}^{\pm}$ by \ref{parabolicsofgamma}, and for an honest cocharacter $\mu$ the dynamic description of these subgroups gives $P_{\mu}^+ \cap P_{\mu}^- = Z_G(\mu)$, see \cite{ConradSGA3}; on Lie algebras this is visible from \ref{liealgebraparabolic}, since the eigenvalues of $\Ad(\gamma)$ of valuation $\geq 0$ and those of valuation $\leq 0$ meet in the part of valuation $0$. As $Z_G(n\mu) = Z_G(\mu)$ for $n > 0$, and $m_G(\gamma) = \nu_G(\gamma)$ by \ref{equalityofslopes}, we obtain $M_{\gamma} = Z_G(\nu_G(\gamma))$. On the other hand, choose $g \in G(\overline{F})$ such that $\Int(g) \circ \nu_G(\gamma)$ factors through $A$. Applying part (2) of \ref{leviofcontracted} with the inner twisting $\psi = \Int(g)$ gives
    $$M_{P_b} = Z_G(\Int(g) \circ \nu_G(\gamma)) = g Z_G(\nu_G(\gamma))g^{-1} = gM_{\gamma}g^{-1},$$
    which proves the claim.
\end{R}

\subsection{Isocrystals and Norm map}
In this subsection, we will formulate an alternative version of the Norm map which allows us to relate it to isocrystals. We recall the setup:
\begin{itemize}
    \item Let $F$ denote a $p$-adic field and let $E$ denote the unramified extension of $F$ of degree $l$.
    \item Let $G$ be a connected reductive group over $F$.
    \item Let $L$ denote the completion of the maximal unramified extension of $F$, which is also the completion of the maximal unramified extension of $E$.
    \item Let $\sigma$ denote the Frobenius automorphism of $L$ over $F$, that is, the continuous extension of the Frobenius automorphism of $F^{\text{un}}$ over $F$. Then the Frobenius automorphism of $L$ over $E$ is $\sigma^l$.
\end{itemize}

\begin{D}[$B_l(G)$]
    For $l \in \mathbb{Z}_{>0}$, denote by $B_l(G)$ the set of $\sigma^l$-conjugacy classes in $G(L)$. This is a generalization of \ref{B(G)def}. Note that $B_1(G) = B(G)$ and $B_l(G) = B(G_E)$.
\end{D}

\begin{D}[Norm mapping on isocrystals]\label{normisocrystal}
    Consider the map $N: G(L) \xrightarrow[]{} G(L)$ given by $Nx = x \cdot x^{\sigma} \cdots x^{\sigma^{l-1}}$. This induces a map $B(G) \xrightarrow[]{} B_l(G)$ which we will also denote by $N$. Indeed, if $x' = gx\sigma(g)^{-1}$ as in \ref{B(G)def}, then $$Nx' = \prod_{i=0}^{l-1} \sigma^i(g) \, x^{\sigma^i} \, \sigma^{i+1}(g)^{-1} = g \, (Nx) \, \sigma^l(g)^{-1},$$ the interior factors telescoping, so $N$ carries the $\sigma$-conjugacy class of $x$ to the $\sigma^l$-conjugacy class of $Nx$.
\end{D}

\begin{D}[Slope morphism for $B_l(G)$]
    For $b \in G_E(L)$, viewing it as an isocrystal with $G_E$-structure, we get a slope morphism $\nu_{G_E}(b) \in \Hom_L(\mathbb{D}, G)$ using \ref{slopedef}. We will denote this by $\nu_{G,l}(b)$. Passing to $\sigma^l$-conjugacy classes, this induces the Newton map $B_l(G) \xrightarrow[]{} (\Hom_L(\mathbb{D}, G)/G(L))^{\langle \sigma^l \rangle}$.
\end{D}

The following result gives a relationship between $\nu_G(x)$ and $\nu_{G,l}(Nx)$. As in \ref{equalityofslopes}, both sides below are elements of $\Hom_L(\mathbb{D},G)$ attached to the elements $x$ and $Nx$, and the asserted equality is one of morphisms, not of $G(L)$-conjugacy classes; this is what the uniqueness in \ref{slopealternativedef} provides.

\begin{Le}\label{normisolem1}
    Let $x \in G(L)$. Then we have $\nu_{G,l}(Nx) = l \cdot \nu_G (x) \in \Hom_L(\mathbb{D},G)$.
\end{Le}
\begin{proof}
    Using Lemma \ref{slopealternativedef}, there exists an integer $n > 0$, an element $c \in G(L)$ and a uniformizer $\pi$ of $F$ such that
    \begin{enumerate}
    \item $n\nu_G(x) \in \Hom_L(\mathbb{G}_m, G)$
    \item $c (n\nu_G(x)) c^{-1}$ is defined over the fixed field of $\sigma^n$ on $L$
    \item $c (x\sigma)^nc^{-1} = c (n\nu_G(x))(\pi)c^{-1} \cdot  \sigma^n$, viewed as elements of $G(L) \rtimes \langle \sigma \rangle $
\end{enumerate}
From 1. it follows that
\begin{itemize}
    \item $n (l \cdot \nu_G(x)) = l (n\nu_G(x)) \in \Hom_L(\mathbb{G}_m, G)$
\end{itemize}
Note that $c (n(l \cdot \nu_G(x))) c^{-1} = l \cdot (c(n\nu_G(x))c^{-1})$ and hence $c (n(l \cdot \nu_G(x))) c^{-1}$ is defined over the fixed field of $\sigma^n$ on $L$. In particular, we also have the weaker statement
\begin{itemize}
    \item $c (n(l \cdot \nu_G(x))) c^{-1}$ is defined over the fixed field of $(\sigma^l)^n$ on $L$
\end{itemize}
From 3., it follows that
$$c(x \cdot x^{\sigma} \cdots x^{\sigma^{n-1}})\sigma^n(c^{-1}) =  c (n\nu_G(x))(\pi)c^{-1}.$$
Note that $c (n\nu_G(x))(\pi)c^{-1}$ is fixed by $\sigma^n$ (using 2.). For any integer $j \geq 0$ we can apply $\sigma^{nj}$ to both sides of the equation above to obtain 
$$\sigma^{nj}(c)(x^{\sigma^{nj}} \cdot x^{\sigma^{nj+1}} \cdots x^{\sigma^{n(j+1)-1}})\sigma^{n(j+1)}(c^{-1}) =  c (n\nu_G(x))(\pi)c^{-1}.$$
Taking the product of all these equations for $j = 0$ to $l-1$ we obtain
$$ c((Nx) \cdot (Nx)^{(\sigma^l)} \cdots (Nx)^{(\sigma^l)^{(n-1)}}) (\sigma^l)^n(c^{-1}) =  c (n(l \cdot \nu_G(x)))(\pi)c^{-1}.$$
From this it follows that
\begin{itemize}
    \item $c (Nx \cdot (\sigma^l))^nc^{-1} = c (n(l\cdot \nu_G(x)))(\pi)c^{-1} \cdot  (\sigma^l)^n$, viewed as elements of $G(L) \rtimes \langle \sigma^l \rangle$
\end{itemize}
Thus the three conditions of Lemma \ref{slopealternativedef} hold for the element $Nx \in G_E(L)$, with the same integer $n$ and element $c$, and with $\pi$, which is also a uniformizer of $E$ because $E/F$ is unramified. The result follows.
\end{proof}

\begin{Le}
    Under the identification $B(\Res_{E/F}G_E) \xrightarrow[]{} B(G_E)$ of Lemma \ref{Shapiro}, the map $B(G) \xrightarrow[]{} B(\Res_{E/F}G_E)$ induced by Corollary \ref{twistedcor2} corresponds to the map $B(G) \xrightarrow[]{} B(G_E)$ defined in \ref{normisocrystal}.
\end{Le}
\begin{proof}
    We start with the bijection $$B(\Res_{E/F}G_E) \xrightarrow[]{} B(G_E) = B_l(G) = G(L)/\{\sigma^l\text{-conjugacy}\}$$   Further, using Lemma \ref{twistedlem1} it follows that $$B(\Res_{E/F}G_E) =  (\Res_{E/F}G_E)(L)/\{\sigma \text{-conjugacy} \} = (G(L)\times \cdots \times G(L))/\{ \sigma\text{-conjugacy} \} $$ where the product on the right is indexed by the set of $F$-embeddings $\iota: E \xrightarrow[]{} L$ and the $\sigma$-action on the product is given by $\sigma \cdot (x_{\iota})_{\iota} = (\sigma (x_{\sigma^{-1} \circ \iota}))_{\iota}$. Explicitly, the set of $F$-embeddings $\iota: E \xrightarrow[]{} L$ is given by $\{ \id, \sigma, \cdots, \sigma^{l-1} \}$ and we order the factors in the product $\prod_{\iota: E \xrightarrow[]{} L} G(L)$ such that the $i$-th factor corresponds to the embedding $\sigma^{l - i}$. With this setup, the map $$ (G(L)\times \cdots \times G(L))/\{ \sigma\text{-conjugacy} \} \xrightarrow[]{} G(L)/\{\sigma^l\text{-conjugacy}\}$$ is given by $$ (b_1, b_2, \cdots, b_l) \in G(L) \times G(L) \times \cdots G(L) \mapsto b_1 b_2^{\sigma} \cdots b_l^{\sigma^{l-1}} \in G(L)$$ Finally, we have the $F$-embedding $G \xrightarrow[]{} \Res_{E/F}G_E$, using Corollary \ref{twistedcor2}. Using the notation of subsection \ref{restrictionofscalars}, the image of this embedding is $(\Res_{E/F}G_E)^{\theta_{\sigma}}$. Hence, the induced map $B(G) \xrightarrow[]{} B(\Res_{E/F}G_E)$ is given by $$ b \in G(L) \mapsto (b,b,\cdots,b) \in (\Res_{E/F}G_E)(L) = G(L) \times G(L) \times \cdots \times G(L). $$
    Combining this with the earlier computation, we now get the result.
\end{proof}

\begin{R}
    The set of Newton points of $\Res_{E/F}G_E$ is $$\mathcal{N}(\Res_{E/F} G_E) = (\Hom_L(\mathbb{D},\Res_{E/F}G_E)/\Res_{E/F}G_E(L))^{\langle \sigma \rangle},$$ which can be written as $$\mathcal{N}(\Res_{E/F} G_E) =(\Hom_L(\mathbb{D}, (G\times \cdots \times G))/(G \times \cdots \times G)(L))^{\langle \sigma \rangle}$$ where the expression on the right simplifies to $$(\Hom_L(\mathbb{D},G)/G(L) \times \Hom_L(\mathbb{D},G)/G(L) \times \cdots \times \Hom_L(\mathbb{D},G)/G(L))^{\langle \sigma \rangle}$$ Again, using Lemma \ref{twistedlem1}, projection onto the last factor gives an isomorphism with $$(\Hom_L(\mathbb{D}, G)/G(L))^{\langle \sigma^l \rangle},$$ whose inverse is given by $$\chi \in (\Hom_L(\mathbb{D}, G)/G(L))^{\langle \sigma^l \rangle} \mapsto (\chi^{\sigma^{l-1}}, \cdots, \chi^{\sigma}, \chi) \in \mathcal{N}(\Res_{E/F}G_E).$$ The set of Newton points of $G_E$ is $$ \mathcal{N}(G_E) = (\Hom_L(\mathbb{D}, G)/G(L))^{\langle \sigma^l \rangle}.$$ Under the Shapiro bijection $B(\Res_{E/F}G_E) \xrightarrow[]{} B(G_E)$, the Newton point of the image is $l$ times the first component of the Newton point of the original element: indeed, the product $b_1b_2^{\sigma}\cdots b_l^{\sigma^{l-1}}$ is the coefficient of the $l$-th iterate of $b\sigma$ on the first factor, and taking the $l$-th iterate multiplies slopes by $l$. With the above parametrization, the induced map on Newton points is therefore $$\chi \mapsto l\chi^{\sigma^{l-1}}.$$ In particular, on the $\sigma$-fixed Newton points arising from $B(G)$ under the diagonal inclusion $G \xrightarrow[]{} \Res_{E/F}G_E$, this is simply the multiplication by $l$ map.
\end{R}

The setup for the rest of the subsection will be as follows.
\begin{itemize}
    \item Let $G$ be quasi-split.
    \item Let $x \in G(E)$.
    \item Let $\mathfrak{N}x$ be the associated stable conjugacy class in $G(F)$ (see \ref{normmapdef2}), defined whenever $Nx$ is semisimple.
\end{itemize}

Next, we want to understand the relationship between the slopes $\nu_G(x)$ and $\nu_G(\mathfrak{N}x)$. Note here that
\begin{itemize}
    \item $\nu_G(x)$ denotes the slope of $x \in G(E)$ when we view it as an element of $G(L)$.
    \item $\nu_G(\mathfrak{N}x)$ denotes slope of the element of $B(G)$ attached to the stable conjugacy class $\mathfrak{N}x$ in $G(F)$, as constructed in Lemma \ref{isocrystalstableconj}.
\end{itemize}

\begin{Le}\label{normisolem2}
    Let $x \in G(E)$ be such that $Nx$ is semisimple, so that $\mathfrak{N}x$ is defined by \ref{normmapdef2} and carries an isocrystal by \ref{isocrystalstableconj}. Then $\nu_{G}(\mathfrak{N}x) = l \cdot \nu_G(x)$ in $\mathcal{N}(G)$.
\end{Le}
\begin{proof}
    \textit{Relation between $Nx$ and $\mathfrak{N}x$}. For $x \in G(E)$, the element $Nx$ again lies in $G(E)$. Let $\alpha: G' \xrightarrow[]{} G$ be a $z$-extension of $G$ and let $x' \in G'(E)$ be such that $\alpha(x') = x$ (using \ref{zextensionlemma0}). Since $(G')^{\text{der}}$ is simply connected, it follows from \ref{normmap} that $Nx'$ lies in the $G'(\overline{F})$-conjugacy class corresponding to $\mathfrak{N}x'$, which is a stable conjugacy class in $G'(F)$; here stable conjugacy and $\overline{F}$-conjugacy agree in $G'(F)$, by \ref{simplyconnectedlemma}. Since $N$ commutes with $\alpha$, that is $N(\alpha(x')) = \alpha(Nx')$, we obtain $Nx = \alpha(Nx')$, and the commutative square of \ref{normmapdef2} identifies $\mathfrak{N}x$ with the image of $\mathfrak{N}x'$ under the map on stable conjugacy classes induced by $\alpha$. Finally, the natural map $G(F) \xrightarrow[]{} G_E(E)$ induces a map from stable conjugacy classes in $G(F)$ to stable conjugacy classes in $G_E(E)$, where $G_E$ denotes the extension of scalars of $G$ to $E$, a connected reductive group over $E$. It now follows that $Nx$ lies in the stable conjugacy class of $G_E(E)$ obtained from the stable conjugacy class $\mathfrak{N}x$ in $G(F)$.

    \textit{Slope of the isocrystal attached to the stable conjugacy class of $Nx$ in $G_E(E)$.} If we view $Nx$ as a stable conjugacy class in $G_E(E)$, we have an isocrystal $b_{Nx}$ attached to it in $B_l(G) = B(G_E)$, by using Lemma \ref{isocrystalstableconj}. By combining this with Lemma \ref{normisolem1}, it follows that $$\nu_{G_E}(b_{Nx}) = \nu_{G_E}(Nx) = \nu_{G,l}(Nx) = l \cdot \nu_G(x).$$
    \textit{Slope of the isocrystal attached to the stable conjugacy class $\mathfrak{N}x$}. Let $y \in G(F)$ denote an element of the stable conjugacy class $\mathfrak{N}x$. The isocrystal $b_{\mathfrak{N}x} \in B(G)$ attached to $\mathfrak{N}x$ by Lemma \ref{isocrystalstableconj} is the isocrystal corresponding to $y$ when we view it as an element of $G(L)$. Using the above relation between $Nx$ and $\mathfrak{N}x$, it follows that the isocrystal $b_{Nx} \in B(G_E)$ attached to the stable conjugacy class of $Nx$ in $G_E(E)$ is the isocrystal corresponding to $y$ when we view it as an element of $G_E(L)$. Let $\nu_G(y)$ and $\nu_{G_E}(y)$ denote the slopes of $y$ when we view it as an element of $G(L)$ and $G_E(L)$ respectively. We claim that $\nu_G(y) = \nu_{G_E}(y)$. Using Lemma \ref{equalityofslopes}, it suffices to show that $m_G(y) = m_{G_E}(y)$. Choose an embedding $\iota: G \hookrightarrow \GL_n$ over $F$; by \ref{delignecocharacter} the composite $\iota \circ m_G(y) = m_{\GL_n}(\iota(y))$ takes values in $\langle \iota(y) \rangle$. We get the following diagram:
    \begin{center}
        \begin{tikzcd}
            \mathbb{D} \arrow{r}{m_G(y)} \arrow{d}{\sim} & G \arrow{r}{\iota} \arrow{d}{} & \GL_n \arrow{d}{} \\
            \mathbb{D} \arrow{r}{m_{G_E}(y)} & G_E \arrow{r}{\iota} & (\GL_n)_E
        \end{tikzcd}
    \end{center}
    The outer square in the diagram is seen to commute because the corresponding map on the character lattices commutes. The commutativity of the maps on character lattices follows from the fact that since $E$ is an unramified extension of $F$ and a uniformizer of $F$ is a uniformizer of $E$:
    \begin{center}
        \begin{tikzcd}
            \mathbb{Q} = X^*(\mathbb{D}) \arrow{d}{\sim} && \langle \text{eigenvalues of } \iota(y) \rangle \subset \overline{F}^{\times} \arrow{ll}{\val_{\pi}} \arrow{d}{\sim} \\
            \mathbb{Q} = X^*(\mathbb{D}) && \langle \text{eigenvalues of } \iota(y) \rangle \subset \overline{E}^{\times} = \overline{F}^{\times} \arrow{ll}{\val_{\pi_E} =\val_{\pi}}
        \end{tikzcd}
    \end{center}
    This allows us to conclude that $\nu_G(y) = m_G(y) = m_{G_E}(y) = \nu_{G_E}(y)$ and hence it follows that $$ \nu_{G_E}(b_{Nx}) = \nu_{G_E}(y) = \nu_G(y) = \nu_G(b_{\mathfrak{N}x})$$ which implies our result by combining it with our earlier computations.    
\end{proof}

\subsection{Truncation functions}
In this subsection we will construct certain truncation functions on $G(F)$, that is, continuous functions on $G(F)$ valued in $\{ 0,1 \}$. We have the following setup:
\begin{itemize}
    \item Let $F$ denote a $p$-adic field with a uniformizer $\pi$.
    \item Let $M_n$ denote the $F$-scheme such that $M_n(R)$ is the set of $n \times n$ matrices with entries in $R$, for any $F$-algebra $R$. Hence, $M_n$ is isomorphic to the affine space $\mathbb{A}^{n^2}_F$, with the isomorphism being the map that sends a matrix to its coefficients.
\end{itemize}

We will first define these truncation functions in the case of $\GL_n$ before generalizing them to a general connected reductive group $G$.

\begin{D}[charpoly]
    Let $\charpoly: M_n \xrightarrow[]{} P_n$ denote the affine morphism sending a matrix to its characteristic polynomial. Here, $P_n$ denotes the $F$-scheme such that $P_n(R)$ is the set of degree $n$ monic polynomials with coefficients in $R$, for any $F$-algebra $R$. Hence, $P_n$ is isomorphic to the affine space $\mathbb{A}^n_F$ with the isomorphism sending a degree $n$ monic polynomial with coefficients in $R$ $$x^n + a_{n-1}x^{n-1} + \cdots + a_0$$ to the element $(a_0, a_1, \cdots, a_{n-1}) \in \mathbb{A}_F^{n}(R)$.
\end{D}

\begin{D}[Quotient $\mathbb{A}_F^n / S_n$]\label{GIT}
    The symmetric group $S_n$ acts on the affine space $\mathbb{A}_F^n = \Spec (F[x_1, x_2, \cdots, x_n])$ as $$\sigma \cdot (x_1, x_2, \cdots, x_n) = (x_{\sigma(1)}, x_{\sigma(2)}, \cdots, x_{\sigma(n)})$$ for any $\sigma \in S_n$. This is an action over $F$. The quotient $\mathbb{A}_F^n/S_n$ is the $F$-scheme represented by the algebra of functions on $\mathbb{A}^n_F$ invariant under the $S_n$ action: $$\Spec(F[x_1, x_2, \cdots x_n]^{S_n}) = \Spec(F[s_1, s_2, \cdots s_n ])$$ where $s_i$ denotes the degree $i$ elementary symmetric polynomial. This naturally gives an isomorphism of $\mathbb{A}_F^n/S_n$ with the affine space $\mathbb{A}^n_F$. The $F$-scheme $\mathbb{A}^n_F/S_n$ can also be characterized by the following properties, using classical invariant theory (invariant theory in the case of a finite group action on an affine variety):
    \begin{enumerate}
        \item The points over the closure $\overline{F}$ are given by $(\mathbb{A}^n_F/S_n)(\overline{F}) = \mathbb{A}^n_F(\overline{F})/S_n$.
        \item The points over a field extension $E$ of $F$ are given by $(\mathbb{A}^n_F/S_n)(E) = (\mathbb{A}^n_F(\overline{F})/S_n)^{\Gamma_E}$.
    \end{enumerate}
\end{D}

\begin{D}[roots]
    Let $\roots: P_n \xrightarrow[]{} \mathbb{A}^n_F/S_n$ denote the affine morphism sending a characteristic polynomial to its roots. This map is given as follows on the points over a field extension $E$ of $F$: Let $x^n + a_{n-1}x^{n-1} + \cdots + a_0 \in P_n(E)$ be factorized as $$x^n + a_{n-1}x^{n-1} + \cdots + a_0 = (x-\alpha_1)(x-\alpha_2)\cdots(x - \alpha_n) \in \overline{F}[x]$$ Then, we map $$ x^n + a_{n-1}x^{n-1} + \cdots + a_0 \mapsto [(\alpha_1, \alpha_2, \cdots ,\alpha_n)] \in (\overline{F}^n/S_n)^{\Gamma_E}.$$
\end{D}

We will need below the following factorization form of Hensel's lemma. Although slightly non-standard, it is a very classical result: the $p$-adic case is due to Hensel, and the case of complete discretely valued fields was proved by Rychl\'{\i}k in a Czech paper in 1919. We state the result and supply its proof here because it will be used in the proof of the root-continuity lemma below.

\begin{Le}[Hensel--Rychl\'{\i}k]\label{henselrychlik}
    Let $K$ be a complete discretely valued field, with ring of integers $\mathcal{O}_K$ and maximal ideal $\mathfrak{m}$. Let $f(x),g_0(x),h_0(x)\in\mathcal{O}_K[x]$ satisfy
    $$\deg f=\deg g_0+\deg h_0,$$
    and suppose that the leading coefficient of $f$ is the product of the leading coefficients of $g_0$ and $h_0$. Let $R(g_0,h_0)$ denote the resultant of $g_0$ and $h_0$. If, for some integer $s\geq 0$,
    $$R(g_0,h_0)\notin\mathfrak{m}^{s+1}\qquad\text{and}\qquad f\equiv g_0h_0\mod \mathfrak{m}^{2s+1},$$
    then there exist $g(x),h(x)\in\mathcal{O}_K[x]$ such that
    $$f=gh,\qquad g\equiv g_0\mod\mathfrak{m}^{s+1},\qquad h\equiv h_0\mod\mathfrak{m}^{s+1},$$
    with $\deg g=\deg g_0$ and $\deg h=\deg h_0$. Moreover, the leading coefficients of $g$ and $h$ are those of $g_0$ and $h_0$, respectively.
\end{Le}
\begin{proof}
    Put $m=\deg g_0$ and $n=\deg h_0$, and choose a uniformizer $\varpi$ of $K$. We first recall an elementary consequence of the definition of the resultant. For $a,b\in\mathcal{O}_K[x]$ of degrees $m,n$, the determinant of the map
    $$\mathcal{O}_K[x]_{<n}\oplus\mathcal{O}_K[x]_{<m}\longrightarrow\mathcal{O}_K[x]_{<m+n},\qquad (u,v)\longmapsto au+bv,$$
    is, up to sign, $R(a,b)$. The adjugate matrix therefore shows that $R(a,b)q$ lies in the image for every $q\in\mathcal{O}_K[x]_{<m+n}$. Consequently, if $R(a,b)\notin\mathfrak{m}^{s+1}$, then
    \begin{equation}\label{resultantcontainment}
        \mathfrak{m}^s\mathcal{O}_K[x]_{<m+n}\subset a\mathcal{O}_K[x]_{<n}+b\mathcal{O}_K[x]_{<m}.
    \end{equation}
    Indeed, writing $R(a,b)=\varpi^r u$ with $u\in\mathcal{O}_K^{\times}$, we have $r\leq s$, and the assertion follows after multiplying the adjugate identity by $\varpi^{s-r}u^{-1}$.

    We now construct inductively polynomials $g^{(i)},h^{(i)}\in\mathcal{O}_K[x]$, starting with $g^{(0)}=g_0$ and $h^{(0)}=h_0$, which retain the leading coefficients of $g_0,h_0$ and satisfy
    $$f-g^{(i)}h^{(i)}\in\mathfrak{m}^{2s+i+1}[x],\qquad g^{(i)}\equiv g^{(i-1)}\mod\mathfrak{m}^{s+i},\qquad h^{(i)}\equiv h^{(i-1)}\mod\mathfrak{m}^{s+i}.$$
    Suppose that $g^{(i-1)},h^{(i-1)}$ have been constructed. We may write
    $$f-g^{(i-1)}h^{(i-1)}=\varpi^{2s+i}q_i$$
    with $\deg q_i<m+n$. Since $g^{(i-1)}\equiv g_0\mod\mathfrak{m}^{s+1}$ and $h^{(i-1)}\equiv h_0\mod\mathfrak{m}^{s+1}$, we have
    $$R(g^{(i-1)},h^{(i-1)})\equiv R(g_0,h_0)\mod\mathfrak{m}^{s+1}.$$
    Applying \eqref{resultantcontainment}, there are polynomials $G_i,H_i\in\mathcal{O}_K[x]$, with $\deg G_i<m$ and $\deg H_i<n$, such that
    $$\varpi^s q_i=g^{(i-1)}H_i+h^{(i-1)}G_i.$$
    Set
    $$g^{(i)}=g^{(i-1)}+\varpi^{s+i}G_i,\qquad h^{(i)}=h^{(i-1)}+\varpi^{s+i}H_i.$$
    Then
    $$f-g^{(i)}h^{(i)}=-\varpi^{2s+2i}G_iH_i\in\mathfrak{m}^{2s+i+1}[x],$$
    since $i\geq1$. This proves the induction step.

    The sequences $g^{(i)}$ and $h^{(i)}$ converge coefficientwise, since $\mathcal{O}_K$ is complete. Let their limits be $g$ and $h$. The construction gives $f=gh$ and the required congruences. Since all the corrections have degree strictly less than $m$ and $n$, respectively, the degrees and leading coefficients are unchanged.
\end{proof}

\begin{R}[Topologies on $\mathbb{A}_F^n/S_n(\overline{F})$]\label{topologiesaresame}
    \begin{enumerate}
        \item The isomorphism $\mathbb{A}_F^n/S_n \simeq \mathbb{A}_F^n$ induced by the isomorphism $F[x_1, x_2, \cdots, x_n]^{S_n} \simeq F[s_1, s_2, \cdots, s_n]$ as explained in \ref{GIT} defines a topology on $(\mathbb{A}_F^n/S_n)(\overline{F})$ via the identifications $(\mathbb{A}_F^n/S_n)(\overline{F}) \simeq \mathbb{A}_F^n(\overline{F}) = \overline{F}^n$, which is the product topology.
        \item As explained in \ref{GIT}, we also have $(\mathbb{A}_F^n/S_n)(\overline{F}) = \mathbb{A}_F^n(\overline{F})/S_n = \overline{F}^n/S_n$. This also defines a topology on $(\mathbb{A}_F^n/S_n)(\overline{F})$, which is the quotient topology.
    \end{enumerate}
    We claim that these two topologies are the same. Let us denote the topologies in 1. and 2. by $\tau_1$ and $\tau_2$ respectively. Note that the map which sends an $n$-tuple to the coefficients of the degree $n$ polynomial whose roots correspond to the $n$-tuple: $$(\alpha_1, \cdots, \alpha_n) \in \overline{F}^n \xrightarrow[]{} (s_1(\alpha_1,\cdots,\alpha_n), s_2(\alpha_1, \cdots, \alpha_n), \cdots, s_n(\alpha_1, \cdots, \alpha_n)) \in \overline{F}^n$$ is a continuous map. Hence, it follows that the map $$((\mathbb{A}_F^n/S_n)(\overline{F}), \tau_2) \xrightarrow[]{} ((\mathbb{A}_F^n/S_n)(\overline{F}), \tau_1)$$ is continuous. The fact that this is actually a homeomorphism follows from the following lemma
\end{R}

\begin{Le}\label{continuityofroots}
    Let $p(x) \in \overline{F}[x]$ be fixed with $$p(x) = x^n + a_{n-1}x^{n-1} + \cdots + a_0 = (x - \alpha_1)(x - \alpha_2) \cdots (x - \alpha_n)$$ For every $\epsilon > 0$, there exists $\delta > 0$ (depending on $p(x)$ and $\epsilon$) such that: 
    
    For any $p'(x) \in \overline{F}[x]$ (not to be confused with the derivative of $p(x)$) with $$p'(x) = x^n + a_{n-1}'x^{n-1}+ \cdots + a_0' = (x - \alpha'_1)(x - \alpha'_2) \cdots (x - \alpha'_n)$$  satisfying $\| a_i - a'_i \|_p < \delta$ for all $0 \leq  i \leq n-1$, there exists a permutation $\sigma \in S_n$ satisfying $\| \alpha_j - \alpha'_{\sigma(j)} \|_p < \epsilon$ for all $1 \leq j\leq n$.
\end{Le}
\begin{proof}

    \textit{Reduce to the integral case.} We start with the observation that one can reduce to the case when $p(x)$ is integral i.e. $p(x) \in \mathcal{O}_{\overline{F}}[x]$. Choose $u \in \mathcal{O}_{\overline{F}} \neq 0$ such that $\beta_i := u \alpha_i \in \mathcal{O}_{\overline{F}}$. Consider the polynomials $$p_0(x) = x^n + b_{n-1}x^{n-1} + \cdots + b_0 := (x - \beta_1)(x-\beta_2) \cdots (x-\beta_n) \in \mathcal{O}_{\overline{F}}[x]$$  $$p_0'(x) = (x - \beta_1')(x - \beta_2') \cdots (x - \beta_n') := (x - u \alpha_1')(x-u\alpha_2') \cdots (x-u\alpha_n') \in \overline{F}[x].$$
    Let $p_0'(x) = x^n + b_{n-1}' x^{n-1} + \cdots + b_0'$ and $c := \|u \|_p$. Then $b_i = u^{n-i}a_i$ and $b_i'=u^{n-i}a_i'$. If the lemma for $p_0(x)$ and $c\epsilon$ supplies a coefficient tolerance $\eta$, then $\|a_i-a_i'\|_p<\eta$ for all $i$ implies $\|b_i-b_i'\|_p<\eta$, since $c\leq 1$. The resulting inequalities $\|\beta_j-\beta'_{\sigma(j)}\|_p<c\epsilon$ are equivalent to $\|\alpha_j-\alpha'_{\sigma(j)}\|_p<\epsilon$. Hence, proving the lemma for $p_0(x) \in \mathcal{O}_{\overline{F}}[x]$ implies it for $p(x)$, so without loss of generality we can assume that $p(x) \in \mathcal{O}_{\overline{F}}[x]$.
    
    \textit{Constructing $\delta$ given $p(x)$ and $\epsilon$.} Replacing $\epsilon$ by a smaller positive number if necessary, assume that $\epsilon<1$. Let $\gamma_1,\ldots,\gamma_r$ be the distinct roots of $p(x)$, with multiplicities $m_1,\ldots,m_r$, so that $$p(x)=\prod_{i=1}^r(x-\gamma_i)^{m_i}\in\mathcal{O}_{\overline{F}}[x].$$ We may also assume that $\epsilon<\|\gamma_i-\gamma_j\|_p$ for every $1\leq i<j\leq r$. For $1\leq i\leq r$, put
    $$q_i(x)=(x-\gamma_i)^{m_i},\qquad Q_i(x)=\prod_{j\neq i}(x-\gamma_j)^{m_j},\qquad R_i=R(q_i,Q_i).$$
    Since $q_i$ and $Q_i$ have no common root, $R_i\neq0$. Let
    $$c_0=\min_{1\leq i\leq r}\{\|R_i\|_p,\epsilon^{m_i}\}>0,$$
    and choose $\delta>0$ such that
    $$\delta<\min\{1,\|\pi\|_p c_0^2\}.$$
    Let $$p'(x) = x^n + a_{n-1}'x^{n-1} + \cdots + a_0' = (x - \alpha_1')(x-\alpha_2')\cdots (x-\alpha_n') \in \overline{F}[x]$$ satisfy $\| a_i - a_i' \|_p < \delta$ for all $0 \leq i \leq n-1$. Let $L$ be a finite extension of $F$ containing all $\gamma_i$ and $\alpha_j'$. Since $p(x)\in\mathcal{O}_L[x]$ and $\delta<1$, it follows that $p'(x)\in\mathcal{O}_L[x]$.

    \textit{Applying Hensel's lemma.} Let $\mathfrak{m}_L$ be the maximal ideal of $\mathcal{O}_L$ and let $\varpi_L$ be a uniformizer of $L$. Choose $s\geq0$ such that
    $$\|\varpi_L\|_p^{s+1}<c_0\leq\|\varpi_L\|_p^s.$$
    Since $\|\pi\|_p\leq\|\varpi_L\|_p$, the choice of $\delta$ gives
    $$\|a_i-a_i'\|_p<\delta<\|\varpi_L\|_p^{2s+1},$$
    and hence $p'\equiv p\mod\mathfrak{m}_L^{2s+1}$. Moreover, $\|R_i\|_p\geq c_0>\|\varpi_L\|_p^{s+1}$, so $R_i\notin\mathfrak{m}_L^{s+1}$. Applying Lemma \ref{henselrychlik} to $f=p'$, $g_0=q_i$ and $h_0=Q_i$, for every $1\leq i\leq r$ we obtain a monic factor $h_i(x)$ of $p'(x)$ such that
    $$h_i(x)\equiv q_i(x)\mod\mathfrak{m}_L^{s+1},\qquad \deg h_i=m_i.$$
    Write $q_i(x)=h_i(x)+e_i(x)$. Every coefficient of $e_i(x)$ has absolute value at most $\|\varpi_L\|_p^{s+1}<c_0\leq\epsilon^{m_i}$. If $\alpha'$ is a root of $h_i(x)$, then $\alpha'\in\mathcal{O}_L$, since it is a root of the monic integral polynomial $p'(x)$. Thus
    $$\|\alpha'-\gamma_i\|_p^{m_i}=\|e_i(\alpha')\|_p<\epsilon^{m_i},$$
    and hence $\|\alpha'-\gamma_i\|_p<\epsilon$. The $\epsilon$-balls around the distinct $\gamma_i$ are disjoint, so the factors $h_i(x)$ are pairwise coprime. Since they all divide $p'(x)$ and their degrees sum to $n$, their product is $p'(x)$. Matching the roots of $h_i(x)$ with the $m_i$ copies of $\gamma_i$ defines the required permutation $\sigma$.
    \end{proof}

\begin{D}[val]
    We have the valuation map $\val_{\pi}: \overline{F} \xrightarrow[]{} \mathbb{Q}$ normalized by the condition that $\val(\pi) = 1$. Since $\val_{\pi}$ is not defined at $0$, we restrict to the open subscheme $\mathbb{G}_m^n \subset \mathbb{A}_F^n$; on $\mathbb{G}_m^n(\overline{F})$ we obtain a map $\val: \mathbb{G}_m^n(\overline{F}) \xrightarrow[]{} \mathbb{Q}^n$ which is continuous for the discrete topology on $\mathbb{Q}^n$, since each fibre $\{ x \mid \val_{\pi}(x) = q \}$ is open. We have the following diagram, where the vertical maps are quotient maps:
    \begin{center}
        \begin{tikzcd}
            \mathbb{G}_m^n(\overline{F}) \arrow{r}{\val} \arrow{d}{} & \mathbb{Q}^n \arrow{d}{} \\
            \mathbb{G}_m^n(\overline{F})/S_n \arrow{r}{\val} & \mathbb{Q}^n/S_n
        \end{tikzcd}
    \end{center}
    Since the quotients in the above diagram are induced with the quotient topology (see Remark \ref{topologiesaresame}), it follows that the map $\val: \mathbb{G}_m^n(\overline{F})/S_n \xrightarrow[]{} \mathbb{Q}^n/S_n$ is continuous, again with the discrete topology on $\mathbb{Q}^n/S_n$.
\end{D}

\begin{const}\label{constructionGLn}
Consider the $F$ sub-schemes $\GL_n \subset M_n$ and $T = \mathbb{G}_m^n \subset \mathbb{A}_F^n$. Identifying $T$ as a maximal torus in $\GL_n$ consisting of diagonal matrices with non-zero entries, we obtain that $S_n = N_G(T)/T = \Omega_{\GL_n}(\overline{F})$ is the Weyl group of $\GL_n$. Since the eigenvalues of an invertible matrix are non-zero, $\roots \circ \charpoly$ carries $\GL_n(\overline{F})$ into $(\mathbb{G}_m^n/S_n)(\overline{F})$, which is where $\val$ is defined. Combining all the above constructions, we get the following diagram.
\begin{center}
    \begin{tikzcd}
        \GL_n(\overline{F}) \arrow{rr}{\roots \circ \charpoly} & &(\mathbb{G}_m^n/S_n)(\overline{F}) \arrow{r}{\val} & \mathbb{Q}^n/S_n\\
        \GL_n(\overline{F}) \arrow{rr}{} \arrow[equal]{u} & & T(\overline{F})/\Omega_{\GL_n}(\overline{F}) \arrow{r}{} \arrow{u}{} & (X_*(T)\otimes \mathbb{Q}) / \Omega_{\GL_n}(\overline{F}) \arrow{u}{}
    \end{tikzcd}
\end{center}
where the vertical arrows are the natural inclusions. Since we have shown that maps in the top row are continuous, it follows now that the maps in the bottom row are also continuous.
\end{const}

The setup for the rest of the subsection is as follows:
\begin{itemize}
    \item $G$ denotes a connected reductive group over $F$
    \item $T$ denotes a maximal torus in $G$ defined over $F$
    \item $\Omega_G$ denotes the Weyl group scheme over $F$
\end{itemize}

\begin{const}\label{sspartG}
The horizontal maps in the bottom row of the diagram from \ref{constructionGLn} can be generalized to the setting of a general connected reductive group $G$ over $F$ as follows:
$$\semisimplepart: G(\overline{F}) \xrightarrow[]{} T(\overline{F})/\Omega_G(\overline{F}) $$
is the map that sends $g \in G(\overline{F})$ to the conjugacy class of its semisimple part $g_m$. Semisimple conjugacy classes in $G(\overline{F})$ are parametrized by $T(\overline{F})/\Omega_G(\overline{F})$. Note that the map $\semisimplepart$ coincides with $\roots \circ \charpoly$ in the case $G = \GL_n$. Note also that if $T'$ is another maximal torus of $G$ over $F$, then there is a canonical bijection of the sets $T(\overline{F}) / \Omega_G(\overline{F}) = T'(\overline{F})/ \Omega_G(\overline{F})$, where the bijection is given by conjugation by any element of $G(\overline{F})$ that maps $T$ to $T'$.

Fix $x \in T(\overline{F})$. The assignment $\lambda \mapsto \val(\lambda(x))$ is defined for $\lambda \in X^*(T)$ and is extended $\mathbb{Q}$-linearly to all of $X^*(T) \otimes \mathbb{Q}$, as in \ref{valQextension}. With this convention there is a unique $\mu \in X_*(T) \otimes \mathbb{Q}$ such that
$$ \val(\lambda(x)) =  \langle \mu, \lambda \rangle $$
for all $\lambda \in X^*(T) \otimes \mathbb{Q}$, since the pairing $\langle \cdot, \cdot \rangle: (X_*(T)\otimes{\mathbb{Q}}) \times (X^*(T) \otimes \mathbb{Q} )\xrightarrow[]{} \mathbb{Q}$ is perfect. Moreover, this pairing is $\Omega_G(\overline{F})$-equivariant and hence it follows that this induces a map $$ \val: T(\overline{F})/\Omega_G(\overline{F}) \xrightarrow[]{} (X_*(T) \otimes \mathbb{Q})/ \Omega_G(\overline{F}).$$
Composing the two maps above we get $$\val \circ \semisimplepart: G(\overline{F}) \xrightarrow[]{} (X_*(T) \otimes \mathbb{Q})/ \Omega_G(\overline{F}). $$
\end{const}

\begin{Le}\label{slopeconstruction3}
    Recall that for $\gamma \in G(F)$, there is the associated $m_G(\gamma) \in \Hom_F(\mathbb{D}, G)$ constructed in \ref{delignecocharacter}. This is the same as the slope morphism $\nu_G(\gamma) \in \Hom_L(\mathbb{D},G)$ associated to $\gamma$, by Lemma \ref{equalityofslopes}. Further, there are the maps
    $$ \Hom_F(\mathbb{D},G) \xrightarrow[]{} (\Hom_{\overline{F}}(\mathbb{D},G)/G(\overline{F}))^{\Gamma_F} = ((X_*(T) \otimes \mathbb{Q})/\Omega_G(\overline{F}))^{\Gamma_F}. $$
    Then, we claim that the following diagram commutes
    \begin{center}
        \begin{tikzcd}
            G(F) \arrow{r}{m_G / \nu_G} \arrow{d}{} & \Hom_F(\mathbb{D}, G) \arrow{r}{} & ((X_*(T) \otimes \mathbb{Q})/\Omega_G(\overline{F}))^{\Gamma_F} \arrow{d}{} \\
            G(\overline{F}) \arrow{rr}{\val \circ \semisimplepart} & & (X_*(T) \otimes \mathbb{Q})/ \Omega_G(\overline{F}) 
        \end{tikzcd}
    \end{center}
\end{Le}
\begin{proof}
    Let $\gamma = \gamma_m \gamma_u$ be the Jordan decomposition of $\gamma$ with $\gamma_m \in G(F)$ the semisimple part of $\gamma$. By construction, $m_G(\gamma) = m_G(\gamma_m)$. Let $T$ be a maximal torus of $G$ over $F$ such that $\gamma_m \in T(F)$ (this can be arranged, for example by choosing a maximal torus defined over $F$ in the identity component of the centralizer of $\gamma_m$ in $G$).
    
    Now, we have that $m_G(\gamma_m) = \nu_G(\gamma_m)$ by Lemma \ref{equalityofslopes}. To show that $\val \circ \semisimplepart (\gamma) = \nu_G(\gamma)$, it suffices to show that $\val(\gamma_m) = \nu_T(\gamma_m)$. Here $\val(\gamma_m)$ is, by \ref{sspartG}, the unique element of $X_*(T)\otimes \mathbb{Q}$ satisfying $\val(\lambda(\gamma_m)) = \langle \val(\gamma_m), \lambda \rangle$ for all $\lambda \in X^*(T)$, the integral characters sufficing by \ref{valQextension}. On the other hand, $\gamma_m \in T(F)$, so Corollary \ref{slopetorispecialcase} gives $\val(\lambda(\gamma_m)) = \langle \nu_T(\gamma_m), \lambda \rangle$ for the same $\lambda$. Since the pairing is perfect, $\val(\gamma_m) = \nu_T(\gamma_m)$ in $X_*(T) \otimes \mathbb{Q}$, and passing to $\Omega_G(\overline{F})$-orbits gives the asserted equality in $(X_*(T)\otimes \mathbb{Q})/\Omega_G(\overline{F})$. Finally, $\nu_G(\gamma) = \nu_G(\gamma_m)$ by the general case of \ref{equalityofslopes}, so the result follows.
\end{proof}

\begin{D}[Trace characters]
    Let $\tr: M_n \xrightarrow[]{} \mathbb{A}^1_F$ be the affine morphism sending a matrix to its trace. Let $\rho: G \xrightarrow[]{} \GL_n$ denote a representation of $G$ over $\overline{F}$. Then, there is an affine morphism $\tr \circ \rho: G_{\overline{F}} \xrightarrow[]{} \mathbb{A}^1_{\overline{F}}$. Restricting this map to $T(\overline{F})$ and noting that the map is invariant under conjugacy classes we obtain a map
    $$(\tr \circ \rho)|_T: (T/\Omega_G)(\overline{F}) = T(\overline{F})/\Omega_G(\overline{F}) \xrightarrow[]{} \mathbb{A}^1_{\overline{F}}$$
    which is known as the trace character of the representation $\rho$. Let $\mathcal{O}((T/\Omega_G)_{\overline{F}})$ denote the ring of global sections of the affine variety $(T/ \Omega_G)_{\overline{F}}$. The trace character can be viewed as an element of $\mathcal{O}((T / \Omega_G)_{\overline{F}})$.
\end{D}

\begin{D}[Fundamental set of representations]
    A set $S$ of representations of $G$ over $\overline{F}$ is called a fundamental set if the trace characters of all exterior powers of the representations in $S$ generate $\mathcal{O}((T / \Omega_G)_{\overline{F}})$ as an $\overline{F}$-algebra.
\end{D}

\begin{Le}\label{fundamentalsetexists}
    A finite fundamental set of representations of $G$ over $\overline{F}$ exists.
\end{Le}
\begin{proof}
    The variety $(T/\Omega_G)_{\overline{F}}$ is affine of finite type, so $\mathcal{O}((T/\Omega_G)_{\overline{F}}) = \mathcal{O}(T_{\overline{F}})^{\Omega_G}$ is a finitely generated $\overline{F}$-algebra; fix generators $f_1, \cdots, f_m$. By highest weight theory, the trace characters of the irreducible representations of $G_{\overline{F}}$ form an $\overline{F}$-basis of $\mathcal{O}(T_{\overline{F}})^{\Omega_G}$, so each $f_j$ is a finite linear combination of such characters. Let $S$ be the finite set of irreducible representations occurring in these expressions. The subalgebra generated by their trace characters contains every $f_j$, hence is all of $\mathcal{O}((T/\Omega_G)_{\overline{F}})$, and since $\wedge^1 \rho = \rho$ these trace characters occur among those of exterior powers of members of $S$. Hence $S$ is a fundamental set.
\end{proof}

The above argument uses only the representations themselves, and not their exterior powers. Allowing exterior powers in the definition makes it possible to work with much smaller fundamental sets. Kret and Shin give explicit examples for the classical groups, see $\S$ 1 of \cite{MR4556781}.

\begin{Le}\label{arnolemma}
    Let $S$ be a fundamental set of representations of $G$ over $\overline{F}$. Let $g_1, g_2 \in G(\overline{F})$ be semisimple. Then, $g_1, g_2$ are conjugate in $G(\overline{F})$ iff $\rho(g_1), \rho(g_2)$ are conjugate in $\GL_n(\overline{F})$ for all $\rho: G \xrightarrow[]{} \GL_n$ in $S$.
\end{Le}
\begin{proof}
    One direction is clear. For the other, view the $G(\overline{F})$-conjugacy classes of $g_1$ and $g_2$ as elements $x_1, x_2$ of $(T / \Omega_G)(\overline{F})$. If $\rho(g_1)$ and $\rho(g_2)$ are conjugate for every $\rho \in S$, then so are $\wedge^k \rho(g_1)$ and $\wedge^k \rho(g_2)$ for every $k$, and the trace is a class function, so the trace characters of all exterior powers of members of $S$ take the same value at $x_1$ and at $x_2$. These functions generate $\mathcal{O}((T/\Omega_G)_{\overline{F}})$, so $f(x_1) = f(x_2)$ for every $f \in \mathcal{O}((T/\Omega_G)_{\overline{F}})$. Since $(T/\Omega_G)_{\overline{F}}$ is an affine variety, two of its $\overline{F}$-points agree if and only if all global functions take the same value on them, so $x_1 = x_2$. This is Lemma 1.1 of \cite{MR4556781}.
\end{proof}

\begin{Le}\label{truncationcontinuous}
    Recall the following map from construction \ref{sspartG}
    $$ \val \circ \semisimplepart: G(\overline{F}) \xrightarrow[]{} (X_*(T) \otimes \mathbb{Q})/\Omega_G(\overline{F}).$$
    We claim that the map is continuous under the natural topology on $G(\overline{F})$ and the discrete topology on $(X_*(T) \otimes \mathbb{Q})/\Omega_G(\overline{F})$.
\end{Le}
\begin{proof}
    The lemma has already been established in the case $G = \GL_n$, see construction \ref{constructionGLn}. For a general connected reductive group $G$, let $S = \{\rho_1, \rho_2, \cdots \rho_k \}$ denote a set of fundamental representations of $G$, where $\rho_i: G \xrightarrow[]{} \GL_{n_i}$ for all $1 \leq i \leq k$. Let $T_i \subset \GL_{n_i}$ denote the subset of diagonal matrices for all $1 \leq i \leq k$. Consider the following diagram
    \begin{center}
        \begin{tikzcd}
            G(\overline{F}) \arrow{r}{\rho_1 \times \cdots \times \rho_k} \arrow{d}{\semisimplepart_G} \arrow{rdd}{\theta} &\GL_{n_1}(\overline{F}) \times \cdots \times \GL_{n_k}(\overline{F}) \arrow{d}{\semisimplepart_{\GL_n}} \\
            T(\overline{F}) / \Omega_G(\overline{F}) \arrow{d}{\val_G} & T_1(\overline{F}) / \Omega_{\GL_{n_1}}(\overline{F}) \times \cdots \times T_k(\overline{F}) / \Omega_{\GL_{n_k}}(\overline{F}) \arrow{d}{\val_{\GL_n}} \\
            (X_*(T) \otimes \mathbb{Q})/\Omega_G(\overline{F}) \arrow[hook]{r}{} & (X_*(T_1) \otimes \mathbb{Q})/\Omega_{\GL_{n_1}}(\overline{F}) \times \cdots \times (X_*(T_k) \otimes \mathbb{Q})/\Omega_{\GL_{n_k}}(\overline{F})
        \end{tikzcd}
    \end{center}
    Note that the bottom arrow is injective. Indeed, let $\nu, \nu' \in X_*(T) \otimes \mathbb{Q}$ have the same image, and choose an integer $d > 0$ such that $d\nu, d\nu' \in X_*(T)$. Put $x = (d\nu)(\pi)$ and $x' = (d\nu')(\pi)$, which are semisimple elements of $T(\overline{F})$. Writing $\lambda_1, \cdots, \lambda_{n_i}$ for the weights of $\rho_i$ with multiplicity, the matrix $\rho_i(x)$ is diagonal with entries $\pi^{\langle \lambda_j, d\nu \rangle}$, and similarly for $x'$. By hypothesis these two families of exponents agree up to a permutation, so $\rho_i(x)$ and $\rho_i(x')$ are conjugate in $\GL_{n_i}(\overline{F})$ for every $i$. Lemma \ref{arnolemma} then gives that $x$ and $x'$ are conjugate in $G(\overline{F})$, and since semisimple conjugacy classes are parametrized by $T(\overline{F})/\Omega_G(\overline{F})$, we get $(d\nu')(\pi) = (w(d\nu))(\pi)$ for some $w \in \Omega_G(\overline{F})$. Evaluating $\lambda \in X^*(T)$ on this identity gives $\pi^{\langle \lambda, d\nu' \rangle} = \pi^{\langle \lambda, w(d\nu) \rangle}$, and since $\pi$ is not a root of unity it follows that $\langle \lambda, d\nu' \rangle = \langle \lambda, w(d\nu) \rangle$ for every $\lambda \in X^*(T)$. Hence $d\nu' = w(d\nu)$ and $\nu' = w(\nu)$. Note that this argument applies \ref{arnolemma} to the elements $(d\nu)(\pi)$, and not to the cocharacters themselves; the lemma separates semisimple conjugacy classes, not $\Omega_G$-orbits of rational cocharacters. 

    The arrows on the right are continuous from the $\GL_n$ case. Since the topology on $(X_*(T) \otimes \mathbb{Q})/\Omega_G(\overline{F})$ is discrete, it follows that showing $(\val \circ \semisimplepart)_G$ is continuous is equivalent to showing that the diagonal map $\theta$ is continuous. But this is clear after viewing $\theta$ as the composition of continuous maps $\rho_1 \times \cdots \times \rho_k$ and $(\val \circ \semisimplepart)_{\GL_n}$. The proof is now complete.
\end{proof}

\begin{C}\label{continuousnewtonpoint}
    Let $\gamma \in G(F)$. Recall the slope morphism $\nu_G(\gamma) \in \Hom_L(\mathbb{D}, G)$ associated to $\gamma$ by viewing it as an element of $G(L)$. In Definition \ref{delignecocharacter}, we gave an alternate construction for this map which we denoted by $m_G(\gamma)$. We claim that the following composition is continuous:
    $$ G(F) \xrightarrow{m_G / \nu_G} \Hom_F(\mathbb{D}, G) \xrightarrow[]{} ((X_*(T) \otimes \mathbb{Q})/\Omega_G(\overline{F}))^{\Gamma_F}. $$
\end{C}
\begin{proof}
    This follows from combining Lemma \ref{slopeconstruction3} and Lemma \ref{truncationcontinuous}. 
\end{proof}

\subsection{Completion of the proofs}
\begin{itemize}
    \item Let $(H,s,\eta)$ denote an endoscopic triple for a connected reductive group $G$ over a $p$-adic field $F$.
\end{itemize}

We first recall the statement of Lemma A from Section \ref{intro2}.

\begin{LemA}
For an endoscopic triple $\mathcal{H} =(H,s,\eta)$ for a connected reductive group $G$ over a $p$-adic field there are canonical maps $f^{\mathcal{H}}: \mathcal{N}(H) \xrightarrow[]{} \mathcal{N}(G)$ and a $\Gamma_F$-equivariant map $\psi^{\mathcal{H}}: \pi_1(H) \xrightarrow[]{} \pi_1(G)$. We abuse notation by also using $\psi^{\mathcal{H}}$ to denote the induced maps on the invariants and the coinvariants. The following diagram commutes
\begin{center}
\begin{tikzcd}
    \mathcal{N}(H) \arrow{r}{\delta_H} \arrow{d}{f^{\mathcal{H}}} & \pi_1(H)^{\Gamma_F} \otimes \mathbb{Q} \arrow{d}{\psi^{\mathcal{H}}} & \pi_1(H)_{\Gamma_F} \arrow{l} \arrow{d}{\psi^{\mathcal{H}}} \\
    \mathcal{N}(G) \arrow{r}{\delta_G} & \pi_1(G)^{\Gamma_F} \otimes \mathbb{Q} & \pi_1(G)_{\Gamma_F} \arrow{l}
\end{tikzcd}
\end{center}

The map $f^{\mathcal{H}}:\mathcal{N}(H) \xrightarrow[]{} \mathcal{N}(G)$ has finite fibres.
\end{LemA}

The four lemmas below construct the two maps occurring in this statement, establish the commutativity of the diagram, and show that $f^{\mathcal H}$ has finite fibres. The proof of Lemma A then consists of assembling them.

\begin{Le}\label{endoscopicnewtonmap}
    The endoscopic triple $(H,s,\eta)$ determines a canonical map
    $$f^{\mathcal H}: \mathcal{N}(H) \xrightarrow[]{} \mathcal{N}(G).$$
    More precisely, let $G_0$ be the quasi-split inner form of $G$, let $T_H$ be a maximal torus of $H$ over $F$, and choose an admissible embedding $\iota:T_H \xrightarrow[]{} G_0$ defined over $F$. If $T=\iota(T_H)$, then $f^{\mathcal H}$ is induced by
    $$[\nu]_{\Omega_H(\overline F)} \longmapsto [\iota_*(\nu)]_{\Omega_{G_0}(\overline F)}.$$
\end{Le}
\begin{proof}
    \textit{Well-definedness.} By Lemma \ref{endotriples}, $(H,s,\eta)$ is also an endoscopic triple for $G_0$. The admissible embedding $\iota$ exists by Lemma \ref{quasisplit}. Identify $\Hat{H}$ with its image under $\eta$, and choose a common maximal torus $\mathscr T \subset \Hat{H} \subset \Hat{G_0}$. Then
    $$N_{\Hat{H}}(\mathscr T) \subset N_{\Hat{G_0}}(\mathscr T),$$
    and hence
    $$\Omega_{\Hat{H}} \subset \Omega_{\Hat{G_0}}.$$
    Equivalently, by Remark \ref{rootdataendogrp}, the Weyl group of $\Hat{H}$ is generated by the reflections belonging to the roots $\alpha$ of $\Hat{G_0}$ satisfying $\alpha(s)=1$, and is therefore a subgroup of the Weyl group of $\Hat{G_0}$. Under the identifications of root data determined by the admissible embedding $\iota$, this gives
    $$\Omega_H(\overline F) \subset \Omega_{G_0}(\overline F)$$
    as groups acting on $X_*(T_H)\otimes\mathbb Q=X_*(T)\otimes\mathbb Q$. Thus $\iota_*$ induces a map on the Weyl-group quotients. Since $\iota$ is defined over $F$, this map is $\Gamma_F$-equivariant, and hence restricts to a map
    $$\left((X_*(T_H)\otimes\mathbb Q)/\Omega_H(\overline F)\right)^{\Gamma_F}
    \longrightarrow
    \left((X_*(T)\otimes\mathbb Q)/\Omega_{G_0}(\overline F)\right)^{\Gamma_F}.$$
    By Definition \ref{newtonmap}, these two sets are $\mathcal N(H)$ and $\mathcal N(G_0)$, respectively, and the latter is identified with $\mathcal N(G)$ through the inner twisting, as in \ref{newtonmap}.

    \textit{Independence of the choices.} For a fixed $T_H$, any two admissible embeddings are conjugate under $G_0(\overline F)$, by \ref{admissibleembedding}. Conjugation identifies their image tori and the corresponding cocharacter spaces; changing the conjugating element by an element of the normalizer changes this identification by $\Omega_{G_0}(\overline F)$, which is invisible in the target quotient. If $T_H'$ is another maximal torus of $H$, choose $h\in H(\overline F)$ with $T_H'=hT_Hh^{-1}$; such an $h$ exists since any two maximal tori of $H$ are conjugate over $\overline{F}$. The embedding $\iota\circ\Int(h^{-1})$ of $T_H'$ is again admissible, since by \ref{firstconst} the canonical $\Hat{H}$-conjugacy class of embeddings $\Hat{T}_{H}' \xrightarrow[]{} \Hat{H}$ is the one obtained from that of $\Hat{T}_H$ by transport along $\Int(h)$; any other admissible embedding of $T_H'$ is then covered by the preceding argument. Changing $h$ by an element of the normalizer changes the source identification by $\Omega_H(\overline F)$, and the Weyl-group inclusion above shows that the displayed formula is again unchanged. Finally, replacing the common dual torus, or replacing $\eta$ within its $\Hat{G}$-conjugacy class, which is $\Gamma_F$-invariant by condition (2) of \ref{endotriples}, only conjugates the construction, and by \ref{secondconst} the canonical $\Gamma_F$-invariant $\Hat{G}$-conjugacy class of embeddings $\Hat{T} \xrightarrow[]{} \Hat{G}$ is unaffected; the resulting cocharacter identification therefore changes by $\Omega_{G_0}(\overline F)$ only. This proves that $f^{\mathcal H}$ is canonical.
\end{proof}

\begin{Le}[The map $\psi^{\mathcal H}$]\label{endoscopicpimap}
    The endoscopic triple $(H,s,\eta)$ determines a canonical $\Gamma_F$-equivariant map
    $$\psi^{\mathcal H}: \pi_1(H) \xrightarrow[]{} \pi_1(G).$$
\end{Le}
\begin{proof}
    After identifying $\Hat{H}$ with its image under $\eta$, there is an inclusion $Z(\Hat{G})\subset Z(\Hat{H})$. It is $\Gamma_F$-equivariant, as observed after Definition \ref{endodef}, and therefore restriction of characters, together with Remark \ref{observation}, gives a $\Gamma_F$-equivariant map
    $$\psi^{\mathcal H}:\pi_1(H)=X^*(Z(\Hat{H}))\longrightarrow X^*(Z(\Hat{G}))=\pi_1(G).$$
    This map is unchanged if $\eta$ is replaced by a $\Hat{G}$-conjugate, since inner automorphisms act trivially on $Z(\Hat{G})$.
\end{proof}

\begin{Le}[Compatibility of $f^{\mathcal H}$ and $\psi^{\mathcal H}$]\label{endoscopiccompatibility}
    The diagram of Lemma A commutes. In particular $\delta_G\circ f^{\mathcal H}=\psi^{\mathcal H}\circ\delta_H$.
\end{Le}
\begin{proof}
    \textit{The square involving the maps $\delta$.} Retain the tori and the admissible embedding $\iota:T_H\xrightarrow[]{}T$ used in Lemma \ref{endoscopicnewtonmap}, and let $q_H$ and $q_{G_0}$ denote the quotient maps from the corresponding cocharacter lattices to the algebraic fundamental groups, using the canonical identification $\pi_1(G_0)=\pi_1(G)$. The inclusion of dual root systems in Remark \ref{rootdataendogrp} gives an inclusion of coroot lattices for $H$ and $G_0$, and hence
    $$q_{G_0}\circ\iota_*=\psi^{\mathcal H}\circ q_H.$$
    The Weyl groups act trivially on the algebraic fundamental groups, so the displayed equality descends to the Weyl-group quotients. We then tensor with $\mathbb{Q}$ and pass to $\Gamma_F$-invariants. Here the identification
    $$\pi_1(-)_{\Gamma_F} \otimes \mathbb{Q} \xrightarrow[]{\sim} \pi_1(-)^{\Gamma_F} \otimes \mathbb{Q}$$
    is the Galois average, sending the class of $x$ to $\mid \Gamma_F \cdot x \mid^{-1} \sum_{x' \in \Gamma_F \cdot x} x'$, normalized exactly as in \ref{galoisaverage}. This normalization is the one that enters the maps $\delta$ of \ref{sspartG}, and it matters: a different scaling would change $\delta_H$ and $\delta_G$ by a constant and so alter the identity below. Being an average, it is functorial for $\Gamma_F$-equivariant homomorphisms, and $\psi^{\mathcal H}$ is one by \ref{endoscopicpimap}. With this convention the equality above becomes
    $$\delta_G\circ f^{\mathcal H}=\psi^{\mathcal H}\circ\delta_H,$$
    by the constructions of the Newton points and the maps $\delta$ in Section \ref{2.1}.

    \textit{The square relating invariants and coinvariants.} This commutes because Galois averaging from coinvariants to invariants tensored with $\mathbb Q$ is functorial for $\Gamma_F$-equivariant homomorphisms, and $\psi^{\mathcal H}$ is such a homomorphism by \ref{endoscopicpimap}.
\end{proof}

\begin{Le}[Finiteness of the fibres]\label{endoscopicfinitefibres}
    The map $f^{\mathcal H}: \mathcal{N}(H) \xrightarrow[]{} \mathcal{N}(G)$ has finite fibres.
\end{Le}
\begin{proof}
    Put $V=X_*(T)\otimes\mathbb Q$, using $\iota$ to identify the two cocharacter spaces. The map $f^{\mathcal H}$ is the restriction to $\Gamma_F$-fixed points of
    $$V/\Omega_H(\overline F)\longrightarrow V/\Omega_{G_0}(\overline F).$$
    Since $\Omega_{G_0}(\overline F)$ is finite, each of its orbits is a finite set, and is therefore a finite union of $\Omega_H(\overline F)$-orbits. Hence the fibres of the displayed map, and so also those of $f^{\mathcal H}$, are finite.
\end{proof}

\begin{proof}[Proof of Lemma A]
    The two maps are constructed in \ref{endoscopicnewtonmap} and \ref{endoscopicpimap}, the diagram commutes by \ref{endoscopiccompatibility}, and the fibres of $f^{\mathcal H}$ are finite by \ref{endoscopicfinitefibres}.
\end{proof}

The definition of the set $\Eell(G)_b$ in Section \ref{details} carries a finiteness assertion, which we prove here since it uses the same circle of ideas.

\begin{Le}\label{finitelymanybH}
    Let $b \in B(G)$. Then only finitely many $b_H \in B(H)$ satisfy
    $$f^{\mathcal H}\bigl(\overline{\nu}_H(b_H)\bigr) = \overline{\nu}_G(b).$$
    In particular only finitely many $b_H$ satisfy the two conditions defining $\Eell(G)_b$.
\end{Le}
\begin{proof}
    By \ref{endoscopicfinitefibres} the set $S = (f^{\mathcal H})^{-1}(\overline{\nu}_G(b)) \subset \mathcal{N}(H)$ is finite, and the displayed condition says that $\overline{\nu}_H(b_H)$ lies in $S$. By \ref{Newonkottwitzcharisoc} an element of $B(H)$ is determined by its Newton point together with its Kottwitz point, so it suffices to show that, for each fixed $\nu \in S$, only finitely many values of $\kappa_H(b_H)$ are possible.

    Fix such a $\nu$. The commutative square of \ref{Newonkottwitzcharisoc} says that the image of $\kappa_H(b_H) \in \pi_1(H)_{\Gamma_F}$ in $\pi_1(H)^{\Gamma_F} \otimes \mathbb{Q}$ equals $\delta_H(\nu)$, which depends only on $\nu$. Hence $\kappa_H(b_H)$ is determined modulo the kernel of $\pi_1(H)_{\Gamma_F} \xrightarrow[]{} \pi_1(H)_{\Gamma_F} \otimes \mathbb{Q}$, that is, modulo the torsion subgroup of $\pi_1(H)_{\Gamma_F}$. Now $\pi_1(H)$ is a quotient of the cocharacter lattice of a maximal torus of $H$, hence finitely generated, and therefore so is $\pi_1(H)_{\Gamma_F}$; its torsion subgroup is thus finite. There are consequently finitely many possibilities for $\kappa_H(b_H)$ for each of the finitely many $\nu \in S$, and the lemma follows.
\end{proof}

We recall the statement of Lemma B from Section \ref{intro2}. Here $F = \mathbb{Q}_p$, the integer $j$ is the one fixed there by $p^j = q^n$, and $\chi_{b,n}^{\mathcal H}: H(\mathbb{Q}_p) \xrightarrow[]{} \{0,1\}$ is the characteristic function of the set of $x_H \in H(\mathbb{Q}_p)$ with
$$
f^{\mathcal H}\bigl(\overline{\nu}_H([x_H])\bigr)=j\cdot\overline{\nu}_G(b),
$$
where $[x_H]\in B(H_{\mathbb Q_p})$ denotes the $\sigma$-conjugacy class of the image of $x_H$ in $H(L)$.

\begin{LemB}
    The function $\chi_{b,n}^{\mathcal{H}}$ is smooth and constant on stable conjugacy classes of semisimple elements in $H(\mathbb{Q}_p)$.
\end{LemB}

\begin{proof}[Proof of Lemma B]
    \textit{Finiteness of the relevant set of Newton points.} Let
    $$S_b^{\mathcal H}=(f^{\mathcal H})^{-1}(j\cdot\overline{\nu}_G(b))\subset\mathcal N(H).$$
    This set is finite by Lemma A.

    \textit{Smoothness.} The map $\overline{\nu}_H:H(F)\rightarrow\mathcal N(H)$ is continuous by Corollary \ref{continuousnewtonpoint}, which in turn uses Lemmas \ref{slopeconstruction3} and \ref{truncationcontinuous}. Since $\mathcal N(H)$ has the discrete topology, the characteristic function of $\overline{\nu}_H^{-1}(\nu)$ is smooth for every $\nu\in\mathcal N(H)$. The function $\chi_{b,n}^{\mathcal H}$ is the finite sum of these characteristic functions over $\nu\in S_b^{\mathcal H}$, and is therefore smooth.

    \textit{Constancy on stable conjugacy classes.} If two semisimple elements of $H(F)$ are stably conjugate, their associated elements of $B(H)$ agree by Lemma \ref{isocrystalstableconj}. Their Newton points are therefore equal, so $\chi_{b,n}^{\mathcal H}$ takes the same value on them.
\end{proof}

\section{Proof of Main Theorem and the modified Langlands-Kottwitz method}\label{mainproof}

We now turn to the proof of the main theorem, Theorem \ref{maintheorem}. Recall from section \ref{informalintro} that our goal is to describe the Lefschetz number
$$c_b(n, f^p, \xi) = \sum_i (-1)^i \tr(\Frob_q^n \times f^p \mid H^i_{c,\mathrm{et}}(\overline{S}_{K^p}(b)_{\overline{F}_q}, \mathcal{F}_{\xi}))$$
of the Frobenius--Hecke correspondence on the compactly supported cohomology of the Newton stratum $\overline{S}_{K^p}(b)$ (see Section \ref{newtonstratdetails}) in a form suitable for comparison with automorphic spectral data.

As mentioned in sections \ref{informalintro} and \ref{intro2}, the analogous question for the whole Shimura variety $S_K$ was answered by Kottwitz \cite{MR1044820} in the PEL case, and later by Kisin--Shin--Zhu \cite{kisin2021stabletraceformulashimura} for Shimura varieties of abelian type. Conditional on Arthur's conjectures, Kottwitz obtained from this a complete automorphic description of the cohomology.

His method is known as the Langlands-Kottwitz method, and proceeds in three steps. The Lefschetz-Verdier trace formula is first used to rewrite the Lefschetz number $c(n,f,\xi)$ of $\Frob_q^n \times f$ on $H_c^*(S_K(\mathbb{C}), \mathcal{F}_{\xi})$ as a sum of orbital and twisted orbital integrals. This expression is then stabilized, so that it becomes a sum of elliptic terms on the geometric sides of the stable trace formulas of the elliptic endoscopic groups of $G$. Finally one invokes the stabilized Arthur-Selberg trace formula to pass from these geometric terms to their spectral counterparts, which is what yields the automorphic description; it is this last step that is conditional on Arthur's conjectures.

Only the first two steps are carried out in Part I of \cite{MR1044820}, and their outcome is a stable trace formula rather than a spectral description. We follow Part I closely and explain how to modify each of its steps so as to treat the Newton stratum $\overline{S}_{K^p}(b)$ in place of $S_K$; we call this the modified Langlands-Kottwitz method. The results so obtained are unconditional. We will first review Kottwitz's argument for $c(n,f,\xi)$ in sections 4.1-4.3 below, explaining along the way the modifications needed to treat $c_b(n,f^p,\xi)$ instead.

We will always assume $(G,X)$ satisfies the following two hypotheses, which are the same as the hypotheses of the Main Theorem (see \ref{maintheorem}):

\vspace{1mm}\hspace{2mm} (\textbf{Hyp 1}) $G^{der}$ is simply connected

\vspace{1mm}\hspace{2mm} (\textbf{Hyp 2}) The maximal $\mathbb{Q}$-split torus coincides with the maximal $\mathbb{R}$-split torus in the center of $G$

For Shimura varieties of abelian type, the point-count formula needed below is supplied by Kisin--Shin--Zhu: their Theorem 6.3.6 proves it without either of the hypotheses imposed here. Under (\HypOne{}) and (\HypTwo{}), their formula reduces to the formula conjectured by Kottwitz in \cite{MR1044820}, which is the one reviewed below.

We impose (\HypOne{}) and (\HypTwo{}) as simplifying hypotheses, in order to keep the discussion at the same level of technicality as \cite{MR1044820}. We expect that a similar result should hold in the generality of Kisin--Shin--Zhu. In this section we point out at various places where (\HypOne{}) is used, to indicate to the reader how it simplifies some technical difficulties. A similar discussion regarding (\HypTwo{}) is carried out in Appendix \ref{hyp2appendix}.

\textbf{Structure of this section.} A word on how this section is organized. The material here is technical and follows Part I of \cite{MR1044820} closely, the point being to adapt it to the Newton stratum $\overline{S}_{K^p}(b)$. The formulas involved carry a large number of terms, and not all of them matter for that adaptation: several are needed only to make the statements meaningful, and play no role in the passage from $S_K$ to $\overline{S}_{K^p}(b)$. Accordingly, each of the following subsections is arranged in two parts. It opens with the formula in question and with those of its ingredients that are used in the modification, and it closes with a series of complements explaining the remaining terms in more detail. The closing complements are included for completeness only and have no bearing on the modification to Newton strata; a reader interested only in the proof of the Main Theorem may skip them and refer back as needed.

\subsection{Lefschetz-Verdier trace formula}\label{lvformula}
In addition to the notation introduced in \ref{details}, we will use the following notation for the rest of this section.
\begin{itemize}
    \item $\mathbb{Q}_{p^j}$ denotes the unramified extension of $\mathbb{Q}_p$ of degree $j$.
    \item $\sigma$ denotes the Frobenius automorphism of $\mathbb{Q}_{p^j}$ over $\mathbb{Q}_p$, which lifts $x \mapsto x^p$ on $\mathbb{F}_{p^j}$ (we also use $\sigma$ to denote the induced automorphism of $G(\mathbb{Q}_{p^j})$).
    \item $\delta, \delta' \in G(\mathbb{Q}_{p^j})$ are $\sigma$-conjugate if $\delta = g \delta' \sigma(g^{-1})$ for some $g \in G(\mathbb{Q}_{p^j})$. This is the notion of $\sigma$-conjugacy defined in section \ref{twistedconjsection} (taking $E = \mathbb{Q}_{p^j}$).
\end{itemize}

For $f \in \mathcal{H}(G(\mathbb{A}_f)//K)$ and $n \in \mathbb{Z}_{>0}$, let
$$c(n,f,\xi) = \sum_i (-1)^i \tr(\Frob_q^n \times f \mid H^i_{c,\mathrm{et}}((S_K)_{\overline E}, \mathcal{F}_{\xi}))$$
denote the Lefschetz number of $\Frob_q^n \times f$ on the cohomology of $S_K$. We will always take $f$ to be of the special form $f = f^p \times \mathbf{1}_{K_{p,0}}$, where $f^p \in \mathcal{H}(G(\mathbb{A}_f^p)//K^p)$ and $K_{p,0}$ is a hyperspecial subgroup, so that $c(n,f,\xi)$ is directly comparable with $c_b(n,f^p,\xi)$.

The Lefschetz-Verdier trace formula relates the trace of the operator $\Frob_q^n \times f$ on the cohomology
$$H^i_{c,\mathrm{et}}((S_K)_{\overline E}, \overline{\mathbb{Q}}_\ell)$$
with the fixed points of the associated correspondence on $S_K$. Using an adaptation of this formula by Fujiwara \cite{MR1431137}, in the case of PEL type Shimura varieties, Kottwitz \cite{MR1124982} obtained the following formula, for a sufficiently large $n$, later generalized to Shimura varieties of abelian type by Kisin--Shin--Zhu \cite{kisin2021stabletraceformulashimura}:
$$
    c(n,f,\xi) = \sum_{\gamma_0 \in G(\mathbb{Q})} \tr\,\xi(\gamma_0) \Biggl( \sum_{(\gamma, \delta) \text{ with } \alpha(\gamma_0,\gamma, \delta) = 0} c(\gamma_0, \gamma, \delta) \Or_{\gamma} (f^p) \TO_{\delta} (\phi_j) \Biggr).
$$
Here, and for the rest of this section, $j$ denotes the positive integer with $p^j = q^n$.

\textbf{Adaptation to Newton strata.}
Recall from section \ref{newtonstratdetails} that the Newton stratum $\overline{S}_{K^p}(b)$ is indexed by an isocrystal $b \in B(G_{\mathbb{Q}_p}, \mu_h^{-1})$ (see Remark \ref{BGmudef} for $B(G_{\mathbb{Q}_p}, \mu_h^{-1})$). Every $\delta$ occurring above lies in $G(\mathbb{Q}_{p^j}) \subset G(L)$ and therefore has a well-defined $\sigma$-conjugacy class in $G(L)$, that is, an image under $G(\mathbb{Q}_{p^j}) \xrightarrow[]{} B(G_{\mathbb{Q}_p})$; we write $\delta \mapsto b$ when this image is the fixed isocrystal $b$. We claim that the modification needed for the Newton stratum is precisely to keep only those $\delta$:
$$
    c_b(n,f^p,\xi) = \sum_{\gamma_0 \in G(\mathbb{Q})} \tr\,\xi(\gamma_0) \Biggl( \sum_{\substack{(\gamma, \delta) \text{ with } \alpha(\gamma_0,\gamma, \delta) = 0 \\ \delta \mapsto b}} c(\gamma_0, \gamma, \delta) \Or_{\gamma} (f^p) \TO_{\delta} (\phi_j) \Biggr).
$$
To justify the claim, return briefly to the proof of the formula of Kisin--Shin--Zhu. In their Lefschetz--Verdier calculation, fixed points are grouped according to admissible Kottwitz parameters $\mathfrak c=(\gamma_0,a,[b_I])$. By \cite[\S\S 2.2.7, 3.5.1]{kisin2021stabletraceformulashimura}, the image $[b_I]_G\in B(G_{\mathbb Q_p})$ is the isocrystal attached to a point represented by $\mathfrak c$. The $\delta$-component of the associated classical Kottwitz triple is represented by
$$
\delta=u^{-1}b_I\sigma(u)
$$
for some $u\in G(L)$, by \cite[Definition 1.6.5 and Lemma 1.6.7]{kisin2021stabletraceformulashimura}. Hence $[\delta]=[b_I]_G$ in $B(G_{\mathbb Q_p})$.

By Section \ref{newtonstratdetails}, the Frobenius--Hecke correspondence restricts to $\overline{S}_{K^p}(b)$. Applying the same Lefschetz--Verdier argument to this restricted correspondence therefore gives a sum over precisely those fixed points of the global correspondence that lie in $\overline{S}_{K^p}(b)$. By the preceding identification, a fixed point represented by $\mathfrak c$ lies in this stratum if and only if $[b_I]_G=b$, equivalently if and only if $[\delta]=b$. The local-term calculation is unchanged, and this proves the displayed restricted formula.

\textbf{Further discussion.}
We now describe the objects entering these formulas, following $\S$ 2 and $\S$ 3 of \cite{MR1044820} and pointing out the features relevant to the modification addressing Question \ref{lknumberquestion}.

\begin{itemize}
    \item The outer sum is over semisimple $\gamma_0 \in G(\mathbb{Q})$ which are $\mathbb{R}$-elliptic (Definition \ref{ellipticdef} with $F = \mathbb{R}$), taken up to $G(\overline{\mathbb{Q}})$-conjugacy. We write $I_0$ for the centralizer of $\gamma_0$ in $G$, which is connected because $G^{der}$ is simply connected by (\HypOne{}).
    \item The inner sum is over pairs $(\gamma, \delta)$ with $\gamma = (\gamma_l)_{l \neq p} \in G(\mathbb{A}_f^p)$ and $\delta \in G(\mathbb{Q}_{p^j})$, taken up to $G(\mathbb{A}_f^p)$-conjugacy and $\sigma$-conjugacy in $G(\mathbb{Q}_{p^j})$ respectively, such that $(\gamma_0, \gamma, \delta)$ is a \textit{Kottwitz triple}. This means:
    \begin{enumerate}
        \item $\gamma_l$ is stably conjugate to $\gamma_0$ (\ref{StableConjugacy}) for every prime $l \neq p$;
        \item the norm $N\delta$ (\ref{normmap}) is conjugate to $\gamma_0$ under $G(\overline{\mathbb{Q}}_p)$;
        \item a further condition on $\delta$ involving $\mu_h$, stated in Remark \ref{kottwitztripledetails} below.
    \end{enumerate}
    \item $\alpha(\gamma_0, \gamma, \delta)$ is the \textit{Kottwitz invariant} of the triple, an element of the finite abelian group $\mathfrak{K}(I_0)^D$; the inner sum runs only over triples whose invariant is trivial. Here $\mathfrak{K}(I_0) := \mathfrak{K}(I_0/\mathbb{Q})$ is the group of Definition \ref{kappadef} taken with $F = \mathbb{Q}$ and $I = I_0$, and $\mathfrak{K}(I_0)^D = \Hom(\mathfrak{K}(I_0), \mathbb{C}^{\times})$ is its Pontryagin dual.
    \item $\tr\,\xi(\gamma_0)$ is the trace of $\gamma_0$ on the representation $\xi$ defining the local system (Section \ref{detailssetup}). Since $\xi$ is algebraic, this trace depends only on the $G(\overline{\mathbb{Q}})$-conjugacy class of $\gamma_0$, which is exactly how $\gamma_0$ is indexed in the outer sum. It is the only place where the local system enters, and it is constant on the inner sum; for $\xi$ trivial it equals $1$ and the formula reduces to the one for constant coefficients.
\end{itemize}

The element $\gamma_0$ is taken up to $G(\overline{\mathbb{Q}})$-conjugacy, whereas its centralizer $I_0$ depends on the $G(\mathbb{Q})$-conjugacy class of $\gamma_0$. If $\gamma_0, \gamma'_0 \in G(\mathbb{Q})$ are $G(\overline{\mathbb{Q}})$-conjugate, then their centralizers $I_0$ and $I'_0$ are in general only \textit{inner forms} of one another. Since $\mathfrak{K}(I_0)$ and $\mathfrak{K}(I'_0)$ are isomorphic for inner forms, the group $\mathfrak{K}(I_0)$ nevertheless depends only on the $G(\overline{\mathbb{Q}})$-conjugacy class of $\gamma_0$, as required for the formula to make sense.
    
The invariant $\alpha(\gamma_0, \gamma, \delta)$ is constructed in $\S$ 2 of \cite{MR1044820} from local invariants $\alpha_v(\gamma_0,\gamma,\delta)$ indexed by the places $v$ of $\mathbb{Q}$. Each $\alpha_v$ is a character of $Z(\Hat{I_0})^{\Gamma_{\mathbb{Q}_v}}$ and is trivial for all but finitely many $v$. By Remark \ref{kappasimplification} the group $\mathfrak{K}(I_0)$ is contained in $Z(\Hat{I_0})^{\Gamma_{\mathbb{Q}_v}}Z(\Hat{G})/Z(\Hat{G})$ for every $v$. Remark \ref{kottwitztripledetails} constructs $\alpha_p$ and $\alpha_{\infty}$, and $\alpha_l$ for a finite place $l \neq p$ is constructed below; here we only explain how the $\alpha_v$ are assembled into a character of $\mathfrak{K}(I_0)$.

Consider a lift of an element $x \in \mathfrak{K}(I_0)$ to
$$\Tilde{x} \in \bigcap_v Z(\Hat{I_0})^{\Gamma_{\mathbb{Q}_v}} Z(\Hat{G}),$$
well defined up to $Z(\Hat{G})$. In order to evaluate the local invariants on such a lift, each $\alpha_v$ is first extended to a character $\beta_v$ of the larger group $Z(\Hat{I_0})^{\Gamma_{\mathbb{Q}_v}} Z(\Hat{G})$. Since a character of a product of two abelian groups is the same thing as a pair of characters agreeing on their intersection, such an extension amounts to prescribing $\beta_v$ on $Z(\Hat{G})$ compatibly with $\alpha_v$ on
$$Z(\Hat{G})^{\Gamma_{\mathbb{Q}_v}} = Z(\Hat{G}) \cap Z(\Hat{I_0})^{\Gamma_{\mathbb{Q}_v}}.$$
Kottwitz takes $\beta_v$ to be trivial on $Z(\Hat{G})$ for $v \neq p, \infty$, which is consistent because $\alpha_v$ is trivial on $Z(\Hat{G})^{\Gamma_{\mathbb{Q}_v}}$ at these places, and takes $\beta_p = -\mu_h^{\natural}$ and $\beta_{\infty} = \mu_h^{\natural}$ on $Z(\Hat{G})$, in the notation of Remark \ref{BGmudef}; this is consistent because the restrictions of $\alpha_p$ and $\alpha_{\infty}$ to $Z(\Hat{G})^{\Gamma_{\mathbb{Q}_p}}$ and $Z(\Hat{G})^{\Gamma_{\mathbb{R}}}$ are $-\mu_h^{\natural}$ and $\mu_h^{\natural}$ respectively, which at $p$ is exactly what condition (3) on $\delta$ asserts.

The $\beta_v$ are again trivial for almost all $v$, so the product $\prod_v \beta_v(\Tilde{x})$ has only finitely many factors different from $1$ and makes sense. Moreover $\prod_v \beta_v$ is trivial on $Z(\Hat{G})$, the contributions at $p$ and $\infty$ cancelling, so the product is unchanged if $\Tilde{x}$ is altered by an element of $Z(\Hat{G})$. Setting
$$\langle \alpha(\gamma_0, \gamma, \delta), x \rangle = \prod_v \beta_v(\Tilde{x})$$
therefore gives a well-defined element $\alpha(\gamma_0, \gamma, \delta) \in \mathfrak{K}(I_0)^D$.

By condition (1), $\gamma_0$ and $\gamma_l$ are stably conjugate, which under (\HypOne{}) is the same as being conjugate under $G(\overline{\mathbb{Q}}_\ell)$ (\ref{conjugacyclassdefoverFlemma}). A priori, the set
$$\{ G(\mathbb{Q}_l)\text{-conj. classes within the stable conj. class of }\gamma_0 \}$$
is only a pointed set. However, Lemma \ref{conjclassinsidestable} (with $F = \mathbb{Q}_l$ and $I = I_0$) identifies it with $\ker(H^1(\mathbb{Q}_l, I_0) \to H^1(\mathbb{Q}_l, G))$ and, since $\mathbb{Q}_l$ is $p$-adic, with a group of characters of $Z(\Hat{I_0})^{\Gamma_{\mathbb{Q}_l}}$ whose restriction to $Z(\Hat{G})^{\Gamma_{\mathbb{Q}_l}}$ is the trivial. The invariant $\alpha_l$ is the image of $\gamma_l$ under this identification. The fact that $\alpha_l$ is trivial for almost all places $l$ follows from Proposition 7.1 of \cite{MR858284}. It is precisely this upgrade of a pointed set to an abelian group that makes the prestabilization of Section \ref{prestab} possible, since the argument there pairs $\alpha(\gamma_0,\gamma,\delta)$ against characters $\kappa \in \mathfrak{K}(I_0)$ and sums.

\begin{itemize}
    \item $c(\gamma_0, \gamma, \delta)$ is a constant attached to the Kottwitz triple, and factors as
    $$c(\gamma_0, \gamma, \delta) = c_1 ( \gamma_0, \gamma, \delta) \cdot c_2 (\gamma_0).$$
    The first factor is the volume term $c_1(\gamma_0,\gamma,\delta) = \vol (I(\mathbb{Q}) \backslash I(\mathbb{A}_f))$, taken with respect to a certain inner form $I$ of $I_0$ constructed below; the second is the cardinality $c_2(\gamma_0) = \mid \ker[\ker^1(\mathbb{Q}, I_0) \xrightarrow{} \ker^1(\mathbb{Q}, G)] \mid$ of a finite cohomological set.
\end{itemize}

Following $\S$ 2 of \cite{MR1044820}, the group $I$ is obtained from local data. Attached to the Kottwitz triple $(\gamma_0, \gamma, \delta)$ and the Shimura datum $(G,X)$ there are groups $I(v)$ over $\mathbb{Q}_v$, one for each place $v$ of $\mathbb{Q}$, each an inner form of $I_{0,\mathbb{Q}_v}$. The group $I$ over $\mathbb{Q}$ is then the group with $I_{\mathbb{Q}_v} \simeq I(v)$ for every $v$, whenever such a group exists; its existence is equivalent to $\alpha(\gamma_0, \gamma, \delta) = 0$. The $I(v)$ are constructed as follows.

\begin{itemize}
    \item At a finite place $v = l \neq p$: $I(l)$ is the centralizer of $\gamma_l$ in $G_{\mathbb{Q}_l}$. By condition (1) and (\HypOne{}), $\gamma_l$ is conjugate to $\gamma_0$ under $G(\overline{\mathbb{Q}}_\ell)$ (\ref{conjugacyclassdefoverFlemma}), so any $g_l \in G(\overline{\mathbb{Q}}_\ell)$ with $g_l \gamma_0 g_l^{-1} = \gamma_l$ gives an inner twisting $\Int(g_l) : I_{0,\mathbb{Q}_l} \xrightarrow{\sim} I(l)$, well defined up to inner automorphisms of $I_{0,\mathbb{Q}_l}$.
    \item At $v = p$: recall $R = \Res_{\mathbb{Q}_{p^j}/\mathbb{Q}_p}(G)$ and $s = \sigma \in \Aut_{\mathbb{Q}_p}(R)$ from \ref{twistedconjsection}, so that $\delta \in G(\mathbb{Q}_{p^j}) = R(\mathbb{Q}_p)$. The group $I(p)$ is the $\sigma$-centralizer $I_{s\delta} = \{ g \in R \mid g \delta \sigma(g^{-1}) = \delta \}$ of $\delta$ (\ref{sigmacentralizerdef}), which is the same group that appears in the twisted orbital integral $\TO_{\delta}(\phi_j)$ below. It is an inner form of $I_{0,\mathbb{Q}_p}$ by Lemma \ref{twistedlemma}, which identifies $(I_{s\delta})_{\mathbb{Q}_{p^j}}$ with the centralizer of the norm $N\delta$ (\ref{normmap}) in $G_{\mathbb{Q}_{p^j}}$: choosing $c$ with $c \gamma_0 c^{-1} = N\delta$ as in Remark \ref{kottwitztripledetails}, which is possible by condition (2), the map $\Int(c)$ carries $I_0 = Z_G(\gamma_0)$ to $Z_G(N\delta)$ and so gives an inner twisting $I_{0,\mathbb{Q}_p} \xrightarrow{\sim} I(p)$, well defined up to inner automorphisms.
    \item At $v = \infty$: we refer the reader to $\S$ 2 of \cite{MR1044820} for the construction of $I(\infty)$. We need only the properties that $I(\infty)$ is an inner form of $I_{0,\mathbb{R}}$, so that $I_{\mathbb{R}} \simeq I(\infty)$, and that $I(\infty)/Z(G)$ is anisotropic over $\mathbb{R}$.

    These two properties give the finiteness of the volume terms occurring in the formula. Write $A_G$ for the maximal $\mathbb{Q}$-split torus in $Z(G)$ and $A_G(\mathbb{R})^0$ for the identity component of $A_G(\mathbb{R})$. By (\HypTwo{}), $A_G$ is also the maximal $\mathbb{R}$-split torus of $Z(G)$, so $Z(G)(\mathbb{R})$ is compact modulo $A_G(\mathbb{R})^0$; and by the anisotropy of $I(\infty)/Z(G)$ so is $I(\mathbb{R}) \simeq I(\infty)(\mathbb{R})$. Since invariant measures are finite on compact sets, the quotient $A_G(\mathbb{R})^0 \backslash I(\infty)(\mathbb{R})$ has finite non-zero volume; this is what makes the term $c_{\infty}$ of Section \ref{prestab} well defined. The same anisotropy forces every $\mathbb{Q}$-split subtorus of $I$, being $\mathbb{R}$-split, to lie in $Z(G)$ and hence in $A_G$, so $I/A_G$ is anisotropic over $\mathbb{Q}$ and $I(\mathbb{Q}) \backslash I(\mathbb{A})^1$ is compact by Appendix \ref{hyp2appendix}. Separating off the compact archimedean factor leaves $c_1(\gamma_0, \gamma, \delta) = \vol(I(\mathbb{Q}) \backslash I(\mathbb{A}_f))$ finite.
\end{itemize}

\begin{R}\label{alphaIv}
    The local invariants $\alpha_v$ and the groups $I(v)$ record the same local data. Write $I_0^{ad}$ for the adjoint group of $I_0$. The inner twisting $\psi_v : I_{0,\mathbb{Q}_v} \xrightarrow{\sim} I(v)$ determines a class $x_v \in H^1(\mathbb{Q}_v, I_0^{ad})$, inner forms being classified by the adjoint group, and by $\S$ 2 of \cite{MR1044820} the image of $x_v$ in $\pi_0(Z(\Hat{I_0^{ad}})^{\Gamma_{\mathbb{Q}_v}})^D$ is the image of $\alpha_v(\gamma_0, \gamma, \delta)$ under
    $$X^*(Z(\Hat{I_0})^{\Gamma_{\mathbb{Q}_v}}) \xrightarrow[]{} X^*(Z(\Hat{I_0^{ad}})^{\Gamma_{\mathbb{Q}_v}}).$$
    So $\alpha_v$ determines $I(v)$, and it is shown there that $\alpha(\gamma_0, \gamma, \delta) = 0$ implies the existence of the global group $I$: the invariant is the obstruction to gluing the $I(v)$ together. The factor $c_2(\gamma_0)$ measures the complementary ambiguity. Applying \ref{conjclassinsidestable} over $\mathbb{Q}$ and over each $\mathbb{Q}_v$, the elements $\gamma' \in G(\mathbb{Q})$ that are $G(\mathbb{Q}_v)$-conjugate to $\gamma_0$ at every place are parametrized by $\ker[\ker^1(\mathbb{Q}, I_0) \xrightarrow[]{} \ker^1(\mathbb{Q}, G)]$, so $c_2(\gamma_0)$ is the number of $G(\mathbb{Q})$-conjugacy classes inside the $G(\mathbb{A})$-conjugacy class of $\gamma_0$, which is the interpretation used in Section \ref{prestab}. Such $\gamma'$ have the same $I(v)$ and the same $\alpha_v$ as $\gamma_0$, so $\alpha$ decides whether the local data glue, while $c_2$ counts the rational classes carrying the same local data.
\end{R}

\begin{itemize}
    \item The orbital integral $\Or_{\gamma}(-)$ is the distribution $\mathcal{C}_c^{\infty}(G(\mathbb{A}_f^p)) \xrightarrow[]{} \mathbb{C}$ given by
    $$\Or_{\gamma}(f^p) = \int_{G(\mathbb{A}_f^p)_{\gamma} \backslash G(\mathbb{A}_f^p)} f^p(x^{-1} \gamma x) \, d\overline{x},$$
    where $G(\mathbb{A}_f^p)_{\gamma}$ is the centralizer of $\gamma$ in $G(\mathbb{A}_f^p)$ and $d\overline{x}$ is the quotient of a chosen Haar measure on $G(\mathbb{A}_f^p)$ by one on $G(\mathbb{A}_f^p)_{\gamma}$.
    \item The twisted orbital integral $\TO_{\delta}(-)$ is the distribution $\mathcal{C}^{\infty}_c(G(\mathbb{Q}_{p^j})) \xrightarrow[]{} \mathbb{C}$ taken with respect to the $\sigma$-conjugacy of \ref{sigmaconjdef}, given by
    $$\TO_{\delta}(\phi_j) = \int_{I_{s\delta}(\mathbb{Q}_p) \backslash G(\mathbb{Q}_{p^j})} \phi_j(y^{-1} \delta \sigma(y)) \, d\overline{y},$$
    where $I_{s\delta}$ is the $\sigma$-centralizer of $\delta$ (\ref{sigmacentralizerdef}), that is, the group $I(p)$ above, and $d\overline{y}$ is defined similarly, using a Haar measure on $I_{s\delta}(\mathbb{Q}_p)$.
\end{itemize}

Each of $\Or_{\gamma}(-)$, $\TO_{\delta}(-)$ and the volume $c_1(\gamma_0, \gamma, \delta)$
depends on a choice of Haar measure, so none of these terms is well defined on its own. On
the groups $G(\mathbb{A}_f^p)$ and $G(\mathbb{Q}_{p^j})$ we fix measures once and for all,
normalized by $\vol(K^p) = 1$ and $\vol(\mathcal{G}(\mathbb{Z}_{p^j})) = 1$; these do not
vary with the Kottwitz triple. The remaining choices are on the centralizers, and here the
three terms are tied together by the group $I$: since $I_{\mathbb{Q}_v} \simeq I(v)$ for
every $v$, the constructions of the $I(v)$ above give
$$I(\mathbb{A}_f) \simeq G(\mathbb{A}_f^p)_{\gamma} \times I_{s\delta}(\mathbb{Q}_p),$$
the first factor being the centralizer occurring in $\Or_{\gamma}$ and the second the
$\sigma$-centralizer occurring in $\TO_{\delta}$. The compatible choice is therefore to fix
a single Haar measure on $I(\mathbb{A}_f)$ and use its prime-to-$p$ part in $\Or_{\gamma}$,
its $p$-part in $\TO_{\delta}$, and the measure itself in $c_1$. This is unambiguous even
though the inner twistings $I_{\mathbb{Q}_v} \xrightarrow{\sim} I(v)$ are only well defined
up to inner automorphisms: on p.\ 632 of \cite{MR942522} Kottwitz defines what it means for
Haar measures on two inner forms to be \textit{compatible}, and shows that the notion does
not depend on the choice of inner twisting nor on the invariant differential form used to
define it. Throughout what follows, measures on a group and on an inner form of it are
understood to be chosen compatibly in this sense.

With this choice the dependence cancels. Rescaling the measure on $I(\mathbb{A}_f)$ by
$t > 0$ multiplies $c_1(\gamma_0,\gamma,\delta) = \vol(I(\mathbb{Q}) \backslash
I(\mathbb{A}_f))$ by $t$, while $\Or_{\gamma}(f^p)$ and $\TO_{\delta}(\phi_j)$ are computed
against quotient measures and so scale by the inverse of the corresponding factor of $t$.
Hence the product $c(\gamma_0, \gamma, \delta) \Or_{\gamma}(f^p) \TO_{\delta}(\phi_j)$
appearing in the formula is independent of all choices, even though no individual factor
is.

\begin{itemize}
    \item The function $\phi_j$ lies in the Hecke algebra $\mathcal{H}(G(\mathbb{Q}_{p^j})//\mathcal{G}(\mathbb{Z}_{p^j}))$, where $\mathcal{G}$ is the hyperspecial integral model of section \ref{newtonstratdetails} (see section \ref{detailssetup} for the Hecke algebra notation), and is constructed as follows. Since $p$ is good for $G$, the completion $E_{\mathfrak{p}}$ of the reflex field at the place above $p$ fixed in section \ref{newtonstratdetails} is unramified over $\mathbb{Q}_p$ with residue field $\mathbb{F}_q$; as $p^j = q^n$, the field $\mathbb{Q}_{p^j}$ is precisely the unramified extension of $E_{\mathfrak{p}}$ of degree $n$ inside $\overline{\mathbb{Q}_p}$, which is the field denoted $F$ in $\S$ 2 of \cite{MR1044820}. In particular $E_{\mathfrak{p}} \subset \mathbb{Q}_{p^j}$, so the $G(\mathbb{C})$-conjugacy class of $\mu_h$ gives a $G(\overline{\mathbb{Q}_p})$-conjugacy class of cocharacters $\mu : \mathbb{G}_m \xrightarrow[]{} G_{\overline{\mathbb{Q}_p}}$ fixed by $\Gal(\overline{\mathbb{Q}_p}/\mathbb{Q}_{p^j})$.

    To define $\phi_j$ we need a representative of the class of $\mu$ that is defined over
    $\mathbb{Q}_{p^j}$, equivalently one factoring through a maximal $\mathbb{Q}_{p^j}$-split
    torus; only then is $\mu(\pi^{-1})$ a point of $G(\mathbb{Q}_{p^j})$ and the double coset
    below a subset of $G(\mathbb{Q}_{p^j})$. Let $S$ be the maximal $\mathbb{Q}_{p^j}$-split subtorus of $T$, for $T \subset B = TN$ as in Remark \ref{BGmudef}; it extends to a maximal $\mathbb{Z}_{p^j}$-split torus of $\mathcal{G}$ over $\mathbb{Z}_{p^j}$. The dominant representative $\mu_h \in X_*(T)$ of the class of $\mu$ factors through $S$. Indeed, $T$ and $B = TN$ are defined over $\mathbb{Q}_p$, so $\Gamma_{\mathbb{Q}_{p^j}}$ acts on $X_*(T)$ and permutes the roots of $T$ in $\Lie(N)$; since the pairing is Galois-equivariant, $\langle \alpha, \tau\mu \rangle = \langle \tau^{-1}\alpha, \mu \rangle$ for $\tau \in \Gamma_{\mathbb{Q}_{p^j}}$, so the action preserves dominance. As the class of $\mu$ is $\Gamma_{\mathbb{Q}_{p^j}}$-stable and contains a unique dominant element of $X_*(T)$ (Remark \ref{BGmudef}), that element is fixed by $\Gamma_{\mathbb{Q}_{p^j}}$, so
    $$\mu_h \in X_*(T)^{\Gamma_{\mathbb{Q}_{p^j}}} = X_*(S).$$
    Then $\phi_j$ is the characteristic function of the double coset $\mathcal{G}(\mathbb{Z}_{p^j}) \, a \, \mathcal{G}(\mathbb{Z}_{p^j})$, where
    $$a = \mu_h(\pi^{-1})$$
    for a uniformizer $\pi$ of $\mathbb{Q}_{p^j}$, for instance $\pi = p$ since $\mathbb{Q}_{p^j}/\mathbb{Q}_p$ is unramified. This double coset is independent of the choices of $S$, of the representative and of $\pi$.
\end{itemize}

\begin{R}\label{interpretation}
    When the Shimura variety can be interpreted as a moduli space of abelian varieties of dimension $g$ with extra structure, the Kottwitz triple admits the following interpretation. The fixed points of the correspondence $\Frob_q^n \times f$ are the fixed points of the correspondence $f$ on the moduli of abelian varieties over $\mathbb{F}_{q^n} = \mathbb{F}_{p^j}$ with the extra structure prescribed by the moduli problem. Let $A$ over $\mathbb{F}_{p^j}$ be the abelian variety corresponding to such a fixed point. Then $\gamma_0 \in G(\mathbb{Q})$ corresponds to the \textit{Frobenius endomorphism} of $A$, the element $\gamma \in G(\mathbb{A}_f^p)$ to the map it induces on the \textit{Tate module} of $A$, and $\delta$ to the map it induces on the \textit{Dieudonn\'e module} of $A$; in these terms, $\delta \mapsto b$ holds exactly when $A$ lies in the Newton stratum $\overline{S}_{K^p}(b)$. If $(\gamma_0,\gamma,\delta)$ is a Kottwitz triple arising in this way from an abelian variety $A$, then $\alpha(\gamma_0, \gamma, \delta) = 0$ by Theorem 7.9 of \cite{MR1155229}.
\end{R}

\begin{R}
The term $c(\gamma_0, \gamma, \delta) \Or_{\gamma}(f^p) \TO_{\delta}(\phi_j)$ is non-zero for only finitely many Kottwitz triples $(\gamma_0, \gamma, \delta)$. In the setting of Remark \ref{interpretation}, the finitely many $\gamma_0$ contributing non-zero terms can be classified by Honda-Tate theory.
\end{R}

\subsubsection*{Complements}

\begin{Rc}[$B(G,\mu)$]\label{BGmudef}
    We use the notation $T$, $B = TN$, $A$ and $\overline{C}_{\mathbb{Q}}$ fixed in the discussion of the Newton map in Section \ref{2.1}, taking $F = \mathbb{Q}_p$. Since $p$ is good for $G$, the group $G_{\mathbb{Q}_p}$ is quasi-split, so we may take $G_0 = G_{\mathbb{Q}_p}$ there: thus $A$ is a maximal split torus of $G_{\mathbb{Q}_p}$, $T$ is its centralizer, and $B = TN$ is a Borel subgroup containing $T$, all defined over $\mathbb{Q}_p$.

    We recall two standard notions attached to this data.
    \begin{itemize}
        \item A cocharacter $\mu \in X_*(T)$ is \textit{dominant} if $\langle \alpha, \mu \rangle \geq 0$ for every root $\alpha$ of $T$ in $\Lie(N)$.
        \item For $\nu_1, \nu_2 \in \overline{C}_{\mathbb{Q}}$, the \textit{dominance order} $\nu_1 \leq \nu_2$ holds if $\nu_2 - \nu_1$ is a non-negative $\mathbb{Q}$-linear combination of positive coroots of $A$.
    \end{itemize}

    Recall from section \ref{detailssetup} that the Shimura datum $(G,X)$ determines a $G(\mathbb{C})$-conjugacy class of Hodge cocharacters of $G_{\mathbb{C}}$. Since $G$ is defined over $\mathbb{Q}$, every cocharacter of $G_{\mathbb{C}}$ is already defined over $\overline{\mathbb{Q}}$, so this class is the base change of a unique $G(\overline{\mathbb{Q}})$-conjugacy class of cocharacters of $G_{\overline{\mathbb{Q}}}$. Fix an embedding $\overline{\mathbb{Q}} \hookrightarrow \overline{\mathbb{Q}_p}$ inducing the place of the reflex field above $p$ chosen in section \ref{newtonstratdetails}; base change along it carries the class to a $G(\overline{\mathbb{Q}_p})$-conjugacy class of cocharacters of $G_{\overline{\mathbb{Q}_p}}$. Any such class meets $X_*(T)$ in a single orbit of the absolute Weyl group $\Omega_{G_{\mathbb{Q}_p}}(\overline{\mathbb{Q}_p})$ of Section \ref{2.1}, and this group permutes the Weyl chambers simply transitively, so the orbit contains exactly one dominant element. This is the dominant cocharacter $\mu_h \in X_*(T)$ referred to above.

    We will repeatedly use the image of $\mu_h$ under the surjection
    $$X_*(T) = X^*(\Hat{T}) \twoheadrightarrow X^*(Z(\Hat{G})),$$
    where the equality identifies cocharacters of $T$ with characters of the dual torus $\Hat{T}$ (\ref{dualtorusdef}) and the surjection is restriction of characters along $Z(\Hat{G}) \subset \Hat{T}$ (using the embedding of $\Hat{T}$ in $\Hat{G}$ as discussed in Remark \ref{secondconst}). We write $\mu_h^{\natural} \in X^*(Z(\Hat{G}))$ for this image. Note that $\mu_h^{\natural}$ is a character of $Z(\Hat{G})$ itself, so it may be restricted to $Z(\Hat{G})^{\Gamma_{\mathbb{Q}_v}}$ for any place $v$; this is how it is used at $v = p$ and $v = \infty$ above and in Remark \ref{kottwitztripledetails}.

    Following $\S$ 6 of \cite{MR1485921}, this construction is carried out for an arbitrary $G(\overline{\mathbb{Q}_p})$-conjugacy class of cocharacters. Let $\mu$ be such a class, and write $\mu^{\natural} \in X^*(Z(\Hat{G}))$ for the image of any representative of $\mu$ lying in $X_*(T)$ under the surjection above; since $Z(\Hat{G}) \subset \Hat{T}$ is fixed pointwise by the Weyl group, this does not depend on the representative chosen. The subset $B(G_{\mathbb{Q}_p}, \mu) \subset B(G_{\mathbb{Q}_p})$ consists of those $b \in B(G_{\mathbb{Q}_p})$ satisfying two conditions:
    \begin{itemize}
        \item The Kottwitz point $\kappa_G(b)$ (\ref{kottwitzmap}) equals the restriction of $\mu^{\natural}$ to $Z(\Hat{G})^{\Gamma_{\mathbb{Q}_p}}$, that is, the image of $\mu$ under the composite (using $\pi_1(G) = X^*(Z(\Hat{G}))$ as in \ref{observation})
        $$X_*(T) = X^*(\Hat{T}) \twoheadrightarrow X^*(Z(\Hat{G})) \twoheadrightarrow X^*(Z(\Hat{G})^{\Gamma_{\mathbb{Q}_p}}),$$
        the last map being passage to $\Gamma_{\mathbb{Q}_p}$-coinvariants.
        \item The Newton point $\overline{\nu}_G(b)$ (\ref{newtonmap}) is $\leq$ the image of $\mu$ under the $\Gamma_{\mathbb{Q}_p}$-averaging map of \ref{galoisaverage}, in the dominance order above.
    \end{itemize}
    By $\S$ 6.4 of \cite{MR1485921}, $B(G_{\mathbb{Q}_p}, \mu)$ is a finite set.

    We apply this with $\mu = \mu_h^{-1}$, the class of the inverse of the Hodge cocharacter, whose dominant representative is $-w_0\mu_h$ for $w_0$ the longest element of $\Omega_{G_{\mathbb{Q}_p}}(\overline{\mathbb{Q}_p})$. Thus $b \in B(G_{\mathbb{Q}_p}, \mu_h^{-1})$ if and only if $\kappa_G(b)$ is the restriction of $-\mu_h^{\natural}$ and $\overline{\nu}_G(b)$ is bounded above by the $\Gamma_{\mathbb{Q}_p}$-average of $\mu_h^{-1}$.
\end{Rc}

\begin{Rc}[Normalisation of the Newton map]\label{newtonnormalisation}
    We use the covariant normalisation of the map $x \mapsto b_x$ of Section \ref{newtonstratdetails}, as announced there. It is the one forced by the function $\phi_j$, which is supported on the double coset of $\mu_h(\pi^{-1})$: the class of $\delta$ then has Kottwitz point the restriction of $-\mu_h^{\natural}$ (\ref{hconstruction}), which is what condition (3) on a Kottwitz triple asserts (\ref{kottwitztripledetails}) and is the constant value of $\kappa_G$ on $B(G_{\mathbb{Q}_p}, \mu_h^{-1})$ by the preceding complement. Accordingly the class of $\delta$ is $b$ itself, which is what is used in Section \ref{completionmainproof}.
\end{Rc}

\begin{Rc}\label{kappasimplification}
    The group $\mathfrak{K}(I_0)$ simplifies using ellipticity of $\gamma_0$. By Definition \ref{ellipticdef} (with $F = \mathbb{R}$), $\gamma_0$ being $\mathbb{R}$-elliptic means precisely that $(Z(\Hat{I_0})/Z(\Hat{G}))^{\Gamma_{\mathbb{R}}}$ is finite (\ref{ellipticlem3}). Since $\Gamma_{\mathbb{R}} \subset \Gamma_{\mathbb{Q}}$ (via the embedding $\overline{\mathbb{Q}} \hookrightarrow \mathbb{C}$ fixing the real place), the group $(Z(\Hat{I_0})/Z(\Hat{G}))^{\Gamma_{\mathbb{Q}}}$ is a subgroup of $(Z(\Hat{I_0})/Z(\Hat{G}))^{\Gamma_{\mathbb{R}}}$, hence finite too; by \ref{ellipticlem3} again, now with $F = \mathbb{Q}$, this says that $\gamma_0$ is also $\mathbb{Q}$-elliptic. Being finite, $(Z(\Hat{I_0})/Z(\Hat{G}))^{\Gamma_{\mathbb{Q}}}$ equals its own group of connected components, so the group $\pi_0((Z(\Hat{I_0})/Z(\Hat{G}))^{\Gamma_{\mathbb{Q}}})$ appearing in Definition \ref{kappadef} (with $F = \mathbb{Q}$, $I = I_0$) is simply $(Z(\Hat{I_0})/Z(\Hat{G}))^{\Gamma_{\mathbb{Q}}}$ itself.

    Next we restate the ``locally trivial'' condition of Definition \ref{kappadef} concretely. Fix an element
    $$x \in (Z(\Hat{I_0})/Z(\Hat{G}))^{\Gamma_{\mathbb{Q}}}.$$
    Applying \ref{keylemmaendo} to the exact sequence $$1 \to Z(\Hat{G}) \to Z(\Hat{I_0}) \to Z(\Hat{I_0})/Z(\Hat{G}) \to 1$$ both over $\mathbb{Q}$ and over each completion $\mathbb{Q}_v$, gives connecting maps $\delta: (Z(\Hat{I_0})/Z(\Hat{G}))^{\Gamma_{\mathbb{Q}}} \to H^1(\mathbb{Q}, Z(\Hat{G}))$ and $\delta_v: (Z(\Hat{I_0})/Z(\Hat{G}))^{\Gamma_{\mathbb{Q}_v}} \to H^1(\mathbb{Q}_v, Z(\Hat{G}))$, compatible with restriction to $\mathbb{Q}_v$. By exactness of the sequence in \ref{keylemmaendo}, $\ker(\delta_v)$ is exactly the image of $Z(\Hat{I_0})^{\Gamma_{\mathbb{Q}_v}}$ in $(Z(\Hat{I_0})/Z(\Hat{G}))^{\Gamma_{\mathbb{Q}_v}}$, namely $Z(\Hat{I_0})^{\Gamma_{\mathbb{Q}_v}}Z(\Hat{G})/Z(\Hat{G})$. So the restriction of $\delta(x)$ to $\mathbb{Q}_v$ vanishes exactly when $x \in Z(\Hat{I_0})^{\Gamma_{\mathbb{Q}_v}}Z(\Hat{G})/Z(\Hat{G})$, and since ``$\delta(x)$ is locally trivial'' means that this holds for every $v$, Definition \ref{kappadef} gives
    $$\mathfrak{K}(I_0) = (Z(\Hat{I_0})/Z(\Hat{G}))^{\Gamma_{\mathbb{Q}}} \cap \Bigl( \bigcap_v Z(\Hat{I_0})^{\Gamma_{\mathbb{Q}_v}} Z(\Hat{G})/Z(\Hat{G}) \Bigr).$$

    Now we claim that the first factor is redundant: it is implied by the second, by the Chebotarev density theorem. Indeed, $Z(\Hat{I_0})/Z(\Hat{G})$ is of multiplicative type, so the $\Gamma_{\mathbb{Q}}$-action on it factors through $\Gal(M/\mathbb{Q})$ for some finite Galois extension $M/\mathbb{Q}$. For $v$ unramified in $M$, the image of $\Gamma_{\mathbb{Q}_v}$ in $\Gal(M/\mathbb{Q})$ is generated by $\Frob_v$, so an element fixed by $\Gamma_{\mathbb{Q}_v}$ is in particular fixed by $\Frob_v$. By Chebotarev density theorem, every element of $\Gal(M/\mathbb{Q})$ arises as $\Frob_v$ for some unramified $v$. Hence an element of the second factor is fixed by $\Frob_v$ for every unramified $v$, therefore by all of $\Gal(M/\mathbb{Q})$, therefore by $\Gamma_{\mathbb{Q}}$; that is, it already lies in the first factor. We conclude
    $$\mathfrak{K}(I_0) = \Bigl( \bigcap_v Z(\Hat{I_0})^{\Gamma_{\mathbb{Q}_v}} Z(\Hat{G}) \Bigr)/Z(\Hat{G}),$$
    the intersection being over all places $v$ of $\mathbb{Q}$. In particular $\mathfrak{K}(I_0) \subset Z(\Hat{I_0})^{\Gamma_{\mathbb{Q}_v}}Z(\Hat{G})/Z(\Hat{G})$ for each $v$, which is the containment used above in restricting the local invariants $\alpha_v$ to $\mathfrak{K}(I_0)$.
\end{Rc}

\begin{Rc}\label{kottwitztripledetails}
    The notion of a Kottwitz triple is defined relative to the Shimura datum $(G,X)$ and depends in particular on the Hodge cocharacter $\mu_h$; so are the components $\alpha_p$ and $\alpha_{\infty}$ of the invariant. Following $\S$ 2 of \cite{MR1044820}, we recall the construction of $\alpha_p$, and briefly that of $\alpha_{\infty}$, which plays little role for Newton strata. Throughout, $\mu_h^{\natural} \in X^*(Z(\Hat{G}))$ is as in Remark \ref{BGmudef}.

    Write $\Hat{I_0}$ for a dual group of $I_0$, so that the embedding $Z(\Hat{G}) \hookrightarrow Z(\Hat{I_0})$ dual to $I_0 \subset G$ is one of $\Gamma_{\mathbb{Q}_p}$-modules. Since $H^1(L, I_0)$ is trivial by a theorem of Steinberg, condition (2) above, that $N\delta$ is conjugate to $\gamma_0$ under $G(\overline{\mathbb{Q}}_p)$, upgrades to conjugacy under $G(L)$: choose $c \in G(L)$ with $c \gamma_0 c^{-1} = N\delta$ (\ref{normmap}) and set $\delta_0 = c^{-1} \delta \sigma(c)$. Applying $\sigma$ to this equation shows $\delta_0 \in I_0(L)$, and $\delta_0$ is well defined up to $\sigma$-conjugacy (\ref{sigmaconjdef}) in $I_0(L)$, hence determines a class in $B((I_0)_{\mathbb{Q}_p})$. Applying the Kottwitz map (\ref{kottwitzmap}) for $I_0$ to this class gives $\alpha_p \in X^*(Z(\Hat{I_0})^{\Gamma_{\mathbb{Q}_p}})$. Condition (3) above, the extra condition on $\delta$ beyond (2), is that the restriction of $\alpha_p$ to $Z(\Hat{G})^{\Gamma_{\mathbb{Q}_p}}$ equals $-\mu_h^{\natural}$ restricted to that subgroup.

    At the infinite place, choose an elliptic maximal $\mathbb{R}$-torus $T$ of $G$ containing $\gamma_0$, which is then also a maximal torus of $I_0$, and $h \in X_{\infty}$ factoring through $T$, so that $\mu_h \in X_*(T) = X^*(\Hat{T})$. Restricting along $Z(\Hat{I_0}) \hookrightarrow \Hat{T}$ gives $\alpha_{\infty} \in X^*(Z(\Hat{I_0})^{\Gamma_{\mathbb{R}}})$, independent of the choices of $T$ and $h$; its restriction to $Z(\Hat{G})^{\Gamma_{\mathbb{R}}}$ is $+\mu_h^{\natural}$ restricted to that subgroup, of sign opposite to that of $\alpha_p$.
\end{Rc}

\subsection{Prestabilization of Lefschetz-Verdier trace formula}\label{prestab}
An important intermediate step before \textit{stabilizing} the Lefschetz-Verdier formula is known as prestabilization. It consists of rewriting the formula of Section \ref{lvformula} as a sum of terms, each of which can be matched with a stable distribution on an \textit{endoscopic group} of $G$. Following $\S$ 4 of \cite{MR1044820}, the prestabilized expression for $c(n,f)$ is
$$
    c(n,f,\xi) = \tau(G) \sum_{\gamma_0} \tr\,\xi(\gamma_0) \sum_{\kappa} \sum_{(\gamma, \delta)} \langle \alpha(\gamma_0, \gamma, \delta), \kappa \rangle \, e(\gamma, \delta) \, \Or_\gamma (f^p) \, \TO_{\delta} (\phi_j) \, c_{\infty}.
$$

\textbf{Adaptation to Newton strata.}
For the Newton stratum indexed by $b$, prestabilization gives the same expression with the innermost sum restricted to the pairs for which $\delta \mapsto b$:
\begin{equation}\label{prestabnewton}
    c_b(n,f^p,\xi) = \tau(G) \sum_{\gamma_0} \tr\,\xi(\gamma_0) \sum_{\kappa} \sum_{\substack{(\gamma, \delta) \\ \delta \mapsto b}} \langle \alpha(\gamma_0, \gamma, \delta), \kappa \rangle \, e(\gamma, \delta) \, \Or_\gamma (f^p) \, \TO_{\delta} (\phi_j) \, c_{\infty}.
\end{equation}
The factor $\tr\,\xi(\gamma_0)$ plays no role in prestabilization: it depends only on $\gamma_0$, so it is constant on the sums over $\kappa$ and over $(\gamma,\delta)$ that are rearranged below, and it passes through the argument unchanged, exactly as the volume factor $c_{\infty}$ does.

Indeed, fix $\gamma_0$ and restrict throughout the argument of Complement \ref{prestabkey} to the set of Kottwitz triples with $\delta \mapsto b$. For a triple in this set, character orthogonality gives
$$
    \sum_{\kappa \in \mathfrak{K}(I_0)} \langle \alpha(\gamma_0,\gamma,\delta),\kappa \rangle
    =
    \begin{cases}
        \lvert \mathfrak{K}(I_0) \rvert, & \alpha(\gamma_0,\gamma,\delta)=0,\\
        0, & \alpha(\gamma_0,\gamma,\delta)\neq 0.
    \end{cases}
$$
For the triples with trivial invariant, Complement \ref{prestabkey} identifies the resulting coefficient $\tau(G)\,\lvert \mathfrak{K}(I_0) \rvert\,c_{\infty}$ with $c(\gamma_0,\gamma,\delta)$; the sign $e(\gamma,\delta)$ is then $1$ by Remark \ref{sign}. Thus the surviving terms are exactly those in the Newton-restricted Lefschetz--Verdier formula of Section \ref{lvformula}. Triples with $\delta \mapsto b$ and non-trivial invariant occur in \eqref{prestabnewton}, but cancel after summing over $\kappa$. Triples with the wrong Newton point are absent; those with trivial invariant must be absent, while those with non-trivial invariant would cancel in the same way if they were included. This proves \eqref{prestabnewton}.

In the stabilization step of Section \ref{stab}, the outer sums $\sum_{\gamma_0} \sum_{\kappa}$ will be replaced by a sum $\sum_{\Eell(G)} \sum_{\gamma_H}$, where $\Eell(G)$ is the set of elliptic endoscopic data associated with $G$ (\ref{endodef}) and $\gamma_H$ runs over $\mathbb{Q}$-elliptic conjugacy classes in $H(\mathbb{Q})$. The innermost sum is referred to as a \textit{$\kappa$-orbital integral}, and the goal of the stabilization can be stated roughly as an identity of the form
$$ \sum_{\gamma_0} \sum_{\kappa} \text{``$\kappa$-orbital integral on } G \text{''} = \sum_{\Eell(G)} \sum_{\gamma_H} \text{``stable orbital integral on }H \text{''}.$$

\textbf{Further discussion.}
We now explain the terms in the prestabilized expression.

\begin{itemize}
    \item The outermost sum $\sum_{\gamma_0}$ is the same as the outer sum in the Lefschetz-Verdier formula of Section \ref{lvformula}: it runs over the $G(\overline{\mathbb{Q}})$-conjugacy classes of semisimple $\gamma_0 \in G(\mathbb{Q})$ that are $\mathbb{R}$-elliptic (Definition \ref{ellipticdef}). As there, $I_0$ denotes the centralizer of $\gamma_0$ in $G$.
    \item The sum $\sum_{\kappa}$ is over all $\kappa \in \mathfrak{K}(I_0)$, the finite abelian group of Definition \ref{kappadef} described concretely in Remark \ref{kappasimplification}.
    \item The sum $\sum_{(\gamma, \delta)}$ is over pairs $(\gamma, \delta)$ such that $(\gamma_0, \gamma, \delta)$ is a Kottwitz triple, taken up to the equivalence used in Section \ref{lvformula}. In \eqref{prestabnewton}, it is further restricted by $\delta \mapsto b$.
\end{itemize}

The last of these differs from the corresponding sum in the Lefschetz-Verdier formula in an important way: it is \emph{not} restricted to the Kottwitz triples with trivial invariant $\alpha(\gamma_0, \gamma, \delta) = 0$. The restriction is instead enforced by the character sum, since
$$\lvert \mathfrak{K}(I_0) \rvert^{-1} \sum_{\kappa \in \mathfrak{K}(I_0)} \langle \alpha(\gamma_0, \gamma, \delta), \kappa \rangle$$
is $1$ if $\alpha(\gamma_0, \gamma, \delta) = 0$ and $0$ otherwise, by orthogonality of characters of the finite abelian group $\mathfrak{K}(I_0)$.

\begin{itemize}
    \item $\tau(G)$ denotes the \textit{Tamagawa number} of $G$.
\end{itemize}

Tamagawa numbers are invariants attached to a connected reductive group over a number field. They are defined by means of the \textit{Tamagawa measure}, a Haar measure on $G(\mathbb{A})$ which is canonical, in the sense that it involves no choices: it is built from an invariant differential form of top degree on $G$ defined over the base field, and the dependence on that form disappears by the product formula. We shall not reproduce the construction, for which we refer to $\S$ 3.5 and $\S$ 5.3 of \cite{MR4615820}. Writing $A_G$ for the maximal $\mathbb{Q}$-split torus in the center of $G$, as in Section \ref{lvformula}, and $A_G(\mathbb{R})^0$ for the identity component of $A_G(\mathbb{R})$, the Tamagawa number of $G$ is
$$\tau(G) = \vol \bigl(G(\mathbb{Q})A_G(\mathbb{R})^0 \backslash G(\mathbb{A})\bigr),$$
the volume being taken with respect to the Tamagawa measure. Since that measure is canonical, no choice of compatible measures is required here, and $\tau(G)$ depends on nothing but $G$.

\begin{R}\label{Tamagawa}
    Two properties of Tamagawa numbers are used repeatedly in what follows. Both are difficult results, proved in \cite{MR942522}.
    \begin{itemize}
        \item[(T1)] $\tau(G) = 1$ for a semisimple, simply connected group $G$.
        \item[(T2)] $\tau(G) = \tau(G_0)$ whenever $G_0$ is an inner form of $G$.
    \end{itemize}
    Property (T2) is what lets the Tamagawa number of the group $I$, which is the one appearing through $c_1(\gamma_0,\gamma,\delta)$, be replaced by that of $I_0$: since $I$ is an inner form of $I_0$ we have $\tau(I) = \tau(I_0)$. Property (T1) enters through (\HypOne{}), which makes $G^{der}$ semisimple and simply connected.
\end{R}

\begin{itemize}
    \item $c_{\infty} := \vol(A_G(\mathbb{R})^0 \backslash I(\infty)(\mathbb{R}))^{-1}$, where $I(\infty)$ is the archimedean group constructed in Section \ref{lvformula}.
\end{itemize}

This volume is finite and non-zero by the discussion of $I(\infty)$ in Section \ref{lvformula} and Appendix \ref{hyp2appendix}, so $c_{\infty}$ is well defined; note that it depends only on $\gamma_0$ and not on the pair $(\gamma, \delta)$. The term arises through the Tamagawa number of $I$, in the case $\alpha(\gamma_0, \gamma, \delta) = 0$ where the group $I$ exists. Indeed, with the measures normalized as above,
$$\tau (I) = c_1(\gamma_0, \gamma, \delta) \cdot \vol (A_G(\mathbb{R})^0 \backslash I(\mathbb{R})), \qquad \text{that is,} \qquad c_1(\gamma_0, \gamma, \delta) = \tau(I) \, c_{\infty},$$
and since $I$ is an inner form of $I_0$ we have $\tau(I) = \tau(I_0)$ by (T2) of Remark \ref{Tamagawa}.

\begin{itemize}
    \item $e(\gamma,\delta) := \prod_v e(I(v))$ is the sign attached to the groups $I(v)$ of Section \ref{lvformula}, the product being over all places $v$ of $\mathbb{Q}$.
\end{itemize}

In \cite{MR697075}, Kottwitz constructs, for a connected reductive group $G$ over a field $F$, an element $e(G)$ in the Brauer group of $F$; when $F$ is a local field this may be regarded as a sign.

\begin{R}\label{sign}
    The three properties of this sign that matter here are the following.
    \begin{itemize}
        \item[(S1)] $e(G) = 1$ for a quasi-split $G$.
        \item[(S2)] (Product formula) For a group $G$ over a global field, $\prod_v e(G_v) = 1$, the product being over all places $v$. Since $G_v$ is quasi-split for all but finitely many $v$, this infinite product is well defined by (S1).
        \item[(S3)] $e(T) = 1$ for any torus $T$. This is used in Section \ref{stab}, where the signs enter the definition of the stable orbital integral through the centralizers of the elements being integrated over.
    \end{itemize}
    Together these give $e(\gamma,\delta) = 1$ whenever $\alpha(\gamma_0, \gamma, \delta) = 0$: the $I(v)$ are then the localizations of the single group $I$ (Remark \ref{alphaIv}), so (S2) applies to $I$. Consequently prestabilization holds even with the sign dropped from the formula, and its presence does not affect the prestabilization step. We nevertheless carry it, because in Section \ref{stab} it plays an essential role in the matching of orbital integrals on different groups: unlike the prestabilization, the transfer identity \eqref{endo} is false without it (Remark \ref{signessential}).
\end{R}

\subsubsection*{Complements}

The following complement explains how the prestabilized expression is derived from the Lefschetz-Verdier formula. The Newton-stratum formula \eqref{prestabnewton} uses the same argument after restricting the indexing set by $\delta \mapsto b$.

\begin{Rc}\label{prestabkey}
    The crucial step in passing from the Lefschetz-Verdier formula to prestabilization is the identity
    $$c(\gamma_0, \gamma, \delta) = \tau(G)\,\lvert \mathfrak{K}(I_0) \rvert\,c_{\infty}$$
    for a Kottwitz triple $(\gamma_0, \gamma, \delta)$ with $\alpha(\gamma_0, \gamma, \delta) = 0$. Granting it, the two formulas agree: writing $c(\gamma_0,\gamma,\delta) \Or_{\gamma}(f^p) \TO_{\delta}(\phi_j)$ in this form and using the orthogonality relation above to replace the restriction to triples with $\alpha(\gamma_0, \gamma, \delta) = 0$ by a sum over all Kottwitz triples weighted by $\lvert \mathfrak{K}(I_0) \rvert^{-1} \sum_{\kappa} \langle \alpha(\gamma_0, \gamma, \delta), \kappa \rangle$, one obtains for each fixed $\gamma_0$
    $$\sum_{\substack{(\gamma, \delta) \text{ with } \\ \alpha(\gamma_0,\gamma,\delta) = 0}} c(\gamma_0, \gamma, \delta) \Or_{\gamma}(f^p) \TO_{\delta}(\phi_j) = \tau(G) \sum_{\kappa \in \mathfrak{K}(I_0)} \sum_{(\gamma, \delta)} \langle \alpha(\gamma_0, \gamma, \delta), \kappa \rangle \Or_{\gamma}(f^p) \TO_{\delta}(\phi_j) \, c_{\infty},$$
    the inner sum on the right being over all $(\gamma,\delta)$ with $(\gamma_0, \gamma, \delta)$ a Kottwitz triple. Inserting the sign $e(\gamma,\delta)$, which is $1$ on the triples that contribute, gives the prestabilized expression.

    Since $c(\gamma_0, \gamma, \delta) = c_1(\gamma_0,\gamma,\delta) \cdot c_2(\gamma_0)$ and $c_1(\gamma_0,\gamma,\delta) = \tau(I_0) \, c_{\infty}$ as above, the identity to be proved is equivalent to
    $$c_2(\gamma_0) = \lvert \mathfrak{K}(I_0) \rvert\,\tau(G)\,\tau(I_0)^{-1},$$
    which is the difficult part. To prove this, the first step is to interpret $c_2(\gamma_0) = \mid \ker (\ker^1 (\mathbb{Q}, I_0) \xrightarrow[]{} \ker^1(\mathbb{Q},G)) \mid$ as the number of $G(\mathbb{Q})$-conjugacy classes in the $G(\mathbb{A})$-conjugacy class of $\gamma_0$, as in Remark \ref{alphaIv}; here the $G(\mathbb{A})$-conjugacy class of $\gamma_0$ is the set of all $\gamma' \in G(\mathbb{Q})$ such that $\gamma' = g^{-1} \gamma_0 g$ for some $g \in G(\mathbb{A})$. After this, the result follows using mainly the properties in Remark \ref{Tamagawa}. An important consequence is that the number of $G(\mathbb{Q})$-conjugacy classes in the $G(\mathbb{A})$-conjugacy class of $\gamma_0$ depends only on the stable conjugacy class of $\gamma_0$ (\ref{StableConjugacy}).
\end{Rc}

\subsection{Stabilization of the Lefschetz-Verdier trace formula}\label{stab}
Throughout this section we use the Weyl group scheme $\Omega_G = N_G(T)/T$ of a reductive group $G$ over a field $F$, together with its absolute Weyl group $\Omega_G(\overline{F})$ and relative Weyl group $\Omega_G(F)$, as recalled in Section \ref{2.1}.

The final stabilized expression for $c(n,f,\xi)$ is obtained by transferring the $\kappa$-orbital integrals on $G$ appearing in the prestabilization of Section \ref{prestab} to stable orbital integrals on endoscopic groups of $G$. Following $\S$ 7 of \cite{MR1044820}, it reads
\begin{equation}\label{final}
    c(n,f,\xi) = \sum_{\mathcal{H} \in \Eell(G)} \iota (G,H) \sum_{\gamma_H} \mid (H_{\gamma_H}/H_{\gamma_H}^{0})(\mathbb{Q}) \mid ^{-1} \tau(H) \SO_{\gamma_H}(h).
\end{equation}

\begin{itemize}
    \item $\Eell(G)$ denotes the set of elliptic endoscopic data associated with $G$ (Definition \ref{endodef}), taken up to equivalence.
\end{itemize}

An endoscopic datum $\mathcal{H}$ associated with $G$ was defined in \ref{endodef}; equivalently, under (\HypOne{}), it may be described by an endoscopic triple $(H,s,\eta)$ as in \ref{endotriples}. It consists of a quasi-split group $H$ over $\mathbb{Q}$, known as the endoscopic group, together with additional data. Some consequences of these data which are relevant for stabilization are the following.
\begin{enumerate}
    \item For a maximal torus $T_H$ of $H$ and a maximal torus $T_G$ of $G$, both defined over $\mathbb{Q}$, there is an \textit{admissible homomorphism} $T_H \xrightarrow[]{} T_G$ (see \ref{quasisplitlemma}) such that the isomorphism $T_H(\overline{\mathbb{Q}}) \simeq T_G(\overline{\mathbb{Q}})$ transports the action of an element of $\Omega_H(\overline{\mathbb{Q}})$ on $T_H(\overline{\mathbb{Q}})$ to the action of an element of $\Omega_G(\overline{\mathbb{Q}})$ on $T_G(\overline{\mathbb{Q}})$. The induced map $T_H(\overline{\mathbb{Q}}) / \Omega_H(\overline{\mathbb{Q}}) \xrightarrow[]{} T_G(\overline{\mathbb{Q}}) / \Omega_G(\overline{\mathbb{Q}})$ is $\Gamma_{\mathbb{Q}}$-equivariant.
    \item Given a semisimple element $\gamma_H \in H(\mathbb{Q})$, taken up to $H(\overline{\mathbb{Q}})$-conjugacy, the induced map produces a \textit{transferred conjugacy class}: a $\Gamma_{\mathbb{Q}}$-invariant $G(\overline{\mathbb{Q}})$-conjugacy class in $G(\overline{\mathbb{Q}})$. Suppose this class contains an element $\gamma_0 \in G(\mathbb{Q})$. Then the $H(\overline{\mathbb{Q}})$-conjugacy class of $\gamma_H$ and the $G(\overline{\mathbb{Q}})$-conjugacy class of $\gamma_0$ are said to be \textit{matching}; this is the notion of related conjugacy classes of \ref{relatedelementsdef}.
    \item A further notion relevant for stabilization is that of a \textit{$(G,H)$-regular element} $\gamma_H \in H(\mathbb{Q})$ for an endoscopic group $H$ of $G$ (Definition \ref{GH-regular}). If an element is $(G,H)$-regular then so is every element of its stable conjugacy class (\ref{StableConjugacy}). Let $\gamma_H \in H(\mathbb{Q})$ be $(G,H)$-regular with a matching $\gamma_0 \in G(\mathbb{Q})$. From the pair $(\mathcal{H}, \gamma_H)$ one obtains an element $\kappa \in \mathfrak{K}(I_0)$, using the part of the data of $\mathcal{H}$ beyond $H$ itself. When the $G(\overline{\mathbb{Q}})$-conjugacy class of $\gamma_0$ and the element $\kappa \in \mathfrak{K}(I_0)$ arise in this way from an endoscopic datum $\mathcal{H}$ and a $(G,H)$-regular $\gamma_H \in H(\mathbb{Q})$, we write $(\mathcal{H}, \gamma_H) \xrightarrow[]{} (\gamma_0, \kappa)$.
\end{enumerate}

Using these properties of endoscopic data we can now explain the remaining terms in \eqref{final}.

\begin{itemize}
    \item The \textit{stable orbital integral} $\SO_{\gamma_H}(-)$ is a \textit{stable} distribution, defined below; the definition applies to $h$ despite its compact-mod-centre support, as explained there.
    \item The function $h=h^ph_ph_{\infty}$ on $H(\mathbb{A})$ has compactly supported finite factors $h^p$ and $h_p$, while $h_{\infty}$ is compactly supported modulo $A_G(\mathbb{R})^0$; it is constructed as follows.
\end{itemize}

The elements $\gamma_H$ that occur in \eqref{final} are $(G,H)$-regular and the function $h$ is decomposable, and in this situation the stable orbital integral takes a simple form. Locally, let $F$ be a local field and let $\gamma_H \in H(F)$ be $(G,H)$-regular semisimple, with centralizer $I_H = H_{\gamma_H}^0$. Fix Haar measures $dg$ on $H(F)$ and $di$ on $I_H(F)$, and let
$$\Or_{\gamma_H}(f) = \int_{I_H(F) \backslash H(F)} f(g^{-1} \gamma_H g) \, \frac{dg}{di}, \qquad f \in \mathcal{C}_c^{\infty}(H(F)),$$
be the orbital integral. At the archimedean place the same integral is defined for $h_{\infty}$, since its support is compact modulo the central subgroup $A_G(\mathbb{R})^0$, which is contained in $I_H(\mathbb{R})$. For a stable conjugate $\gamma_H'$ of $\gamma_H$ (\ref{StableConjugacy}) the centralizer of $\gamma_H'$ is an inner form of $I_H$, so that $\Or_{\gamma_H'}$ requires a Haar measure on a group other than $I_H(F)$. We adopt the following convention, in force here and in all the complements below.
\begin{itemize}
    \item \textbf{Measures for stable orbital integrals.} The centralizers of the stable conjugates of a semisimple element are inner forms of one another, and the Haar measures used on them are always taken to be compatible in the sense of $\S$ 1 of \cite{MR942522}, as recalled in Section \ref{lvformula}. Without this the individual terms below could be rescaled independently and their sum would carry no information.
\end{itemize} The \textit{stable orbital integral} is
$$\SO_{\gamma_H}(f) = \sum_{\gamma_H'} e(H_{\gamma_H'}) \, \Or_{\gamma_H'}(f),$$
where $\gamma_H'$ runs over a set of representatives for the $H(F)$-conjugacy classes inside the stable conjugacy class of $\gamma_H$, a finite set by \ref{conjclassinsidestable}, and $e(-)$ is the sign of Remark \ref{sign}. These signs are all $1$ exactly when the centralizers occurring are tori, by (S3) of that remark; this is the case for $G$-regular $\gamma_H$, but not in general for the $(G,H)$-regular elements considered here, so the signs cannot be dropped. Since $\gamma_H$ is $(G,H)$-regular these centralizers are connected, which is what makes the formula this simple; the definition for a general semisimple element carries an extra factor and is recalled in Complement \ref{SOgeneral}.

Globally, the function $h = h^p h_p h_{\infty}$ is a product of local functions, so its stable orbital integral is the corresponding product of local stable orbital integrals:
$$\SO_{\gamma_H}(h) = \SO_{\gamma_H}(h^p) \cdot \SO_{\gamma_H}(h_p) \cdot \SO_{\gamma_H}(h_{\infty}),$$
where $\SO_{\gamma_H}(h^p) = \prod_{v \neq p, \infty} \SO_{\gamma_H}(h_v)$, almost all of whose factors are $1$. The three factors are computed separately below.

Let $\mathcal{H} \in \Eell(G)$ and let $H$ denote the associated endoscopic group. Using the Fundamental Lemma, Kottwitz \cite{MR1044820} shows that there exists a function $h$ on $H(\mathbb{A})$, decomposing locally as a product $h^p h_p h_{\infty}$ of functions on $H(\mathbb{A}_f^p)$, $H(\mathbb{Q}_p)$ and $H(\mathbb{R})$ respectively, such that the following identity holds for every $(G,H)$-regular semisimple $\gamma_H \in H(\mathbb{Q})$ with $(\mathcal{H}, \gamma_H) \xrightarrow[]{} (\gamma_0, \kappa)$. The three factors, and the identities that characterize their stable orbital integrals, are recorded in Complement \ref{hconstruction}; the way they combine to give \eqref{final} is explained in Complement \ref{route72}.
\begin{equation}\label{endo}
    \SO_{\gamma_H}(h) = \tr\,\xi(\gamma_0) \sum_{(\gamma, \delta)} \langle \alpha(\gamma_0, \gamma, \delta), \kappa \rangle e(\gamma, \delta)\Or_{\gamma}(f^p) \TO_{\delta}(\phi_j) c_{\infty}.
\end{equation}
Here the sum $\sum_{(\gamma, \delta)}$ is over pairs $(\gamma, \delta)$ such that $(\gamma_0, \gamma, \delta)$ is a Kottwitz triple, taken up to the equivalence of Section \ref{lvformula} and, as in the prestabilization, not restricted to those with $\alpha(\gamma_0,\gamma,\delta) = 0$. The right-hand side is exactly the $\kappa$-orbital integral attached to $(\gamma_0, \kappa)$ in Section \ref{prestab}.

\begin{R}\label{signessential}
    Unlike prestabilization, which holds even without the sign term $e(\gamma, \delta)$ (see Remark \ref{sign}), equation \eqref{endo} does not hold without it. This sign is essential for the existence of a function $h$ on $H(\mathbb{A})$ whose stable orbital integrals match the $\kappa$-orbital integrals on the right.
\end{R}

\begin{R}\label{kappaisreached}
    Since $\gamma_0 \in G(\mathbb{Q})$ in the outermost sum of the prestabilization is $\mathbb{R}$-elliptic (\ref{ellipticdef}), it follows from Lemma 9.7 of \cite{MR858284} that for every $\kappa \in \mathfrak{K}(I_0)$ there exist $\mathcal{H} \in \Eell(G)$ and a $(G,H)$-regular $\gamma_H \in H(\mathbb{Q})$ with $(\mathcal{H},\gamma_H) \xrightarrow[]{} (\gamma_0, \kappa)$. Combining this with \eqref{endo}, the terms of the prestabilization can be replaced by stable orbital integrals.
\end{R}

\begin{itemize}
    \item $\iota(G,H) = \tau(G) \tau(H)^{-1} \mid \Aut(\mathcal{H})/H_{ad}(\mathbb{Q}) \mid^{-1}$, where $\tau$ denotes the Tamagawa number of Section \ref{prestab}.
\end{itemize}

This term accounts for the fibres of the map $(\mathcal{H}, \gamma_H) \mapsto (\gamma_0, \kappa)$. Two pairs $(\mathcal{H}, \gamma'_H)$ and $(\mathcal{H}, \gamma_H)$ map to the same $(\gamma_0, \kappa)$ if and only if there is an automorphism of $\mathcal{H}$ carrying the stable conjugacy class of $\gamma'_H$ to that of $\gamma_H$. An automorphism of $\mathcal{H}$ consists of an automorphism $\alpha: H \xrightarrow[]{} H$ over $\mathbb{Q}$ of the endoscopic group satisfying compatibility conditions with the remaining data of $\mathcal{H}$; the group $H_{ad}(\mathbb{Q})$ is a normal subgroup of $\Aut(\mathcal{H})$. Thus $\mid \Aut(\mathcal{H}) / H_{ad}(\mathbb{Q}) \mid$ is the number of stable conjugacy classes $\gamma'_H$ in $H(\mathbb{Q})$ with $(\mathcal{H}, \gamma'_H) \xrightarrow[]{} (\gamma_0, \kappa)$.

\begin{itemize}
    \item The inner sum $\sum_{\gamma_H}$ in \eqref{final} runs over a set of representatives for the $\mathbb{Q}$-elliptic semisimple stable conjugacy classes in $H(\mathbb{Q})$. Note that the range of this sum is not restricted to the $(G,H)$-regular classes, so the stable orbital integrals occurring in it are in general those of Complement \ref{SOgeneral} rather than the simplified ones above; it is only the terms that survive, by Remark \ref{SOvanishing}, for which the two definitions agree.
    \item $H_{\gamma_H}^0$ denotes the identity component of the centralizer of $\gamma_H$ in $H$; the order of the finite group $(H_{\gamma_H}/H_{\gamma_H}^0)(\mathbb{Q})$ depends only on the stable conjugacy class of $\gamma_H$, so that the summand is well defined.
\end{itemize}

\begin{R}\label{SOvanishing}
    For $\mathcal{H} \in \Eell(G)$, the function $h$ has the following additional property. One has $\SO_{\gamma_H}(h) = 0$ unless
    \begin{enumerate}
        \item[(i)] $\gamma_H$ is elliptic in $H(\mathbb{R})$,
        \item[(ii)] $\gamma_H$ is $(G,H)$-regular, and
        \item[(iii)] there exists a pair $(\gamma_0, \kappa)$ with $(\mathcal{H}, \gamma_H) \xrightarrow[]{} (\gamma_0, \kappa)$.
    \end{enumerate}
    Parts (i) and (iii) are proved by Kottwitz in \cite{MR1044820}; part (ii) is due to Morel; see \cite[Prop.~3.3.4, Rem.~3.3.5]{MR2567740}. Condition (i) matches the outermost sum of the prestabilization being taken over $\mathbb{R}$-elliptic $\gamma_0$: by $\S$ 7 of \cite{MR1044820}, $\gamma_0$ is elliptic in $G(\mathbb{R})$ if and only if $\gamma_H$ is elliptic in $H(\mathbb{R})$. Such a $\gamma_H$ is in particular $\mathbb{Q}$-elliptic, by the argument of Remark \ref{kappasimplification} applied to $H$, which is why the inner sum in \eqref{final} may be taken over $\mathbb{Q}$-elliptic stable conjugacy classes.
\end{R}

It follows from Remark \ref{SOvanishing} that the non-zero terms of the inner sum in \eqref{final} correspond to $(G,H)$-regular $\gamma_H$. For such $\gamma_H$ the centralizer $H_{\gamma_H}$ is connected, so $\mid (H_{\gamma_H}/H_{\gamma_H}^0)(\mathbb{Q}) \mid = 1$ and the factor may be omitted. It is retained in \eqref{final} for later comparison with the stabilized Arthur-Selberg trace formula, where the sum is over all $\mathbb{Q}$-elliptic stable conjugacy classes and the factor is not identically $1$.

\subsubsection*{Complements}

\begin{Rc}[Stable orbital integrals in general]\label{SOgeneral}
    The general definition is worth recording because it is the one that occurs in the elliptic part of the geometric side of the Arthur-Selberg trace formula, as stated in Section \ref{intro2}: there the sum defining $\ST_e$ runs over \emph{all} elliptic stable conjugacy classes in $H(F)$, with no $(G,H)$-regularity imposed, so the centralizers occurring need not be connected. For a semisimple element that is not assumed $(G,H)$-regular the centralizer may be disconnected, and the definition of the stable orbital integral carries a correction factor. Following $\S$ 5 of \cite{MR858284}, let $F$ be a local field, $H$ a connected reductive group over $F$ and $\gamma \in H(F)$ semisimple with $I = H_{\gamma}^0$. For a stable conjugate $\gamma'$ the group $I' = H_{\gamma'}^0$ is an inner twist of $I$, so a Haar measure $di$ on $I(F)$ determines one on $I'(F)$.
    \begin{itemize}
        \item \textbf{Measures for stable orbital integrals.} Recall the convention of Section \ref{stab}: the Haar measures on the centralizers of the stable conjugates are always taken compatible in the sense of $\S$ 1 of \cite{MR942522}.
    \end{itemize}
    With this understood, one sets
    $$\SO_{\gamma}(f) = \sum_{\gamma'} \mid \ker[H^1(F, I') \xrightarrow[]{} H^1(F, H_{\gamma'})] \mid \cdot \; e(I') \cdot \Or_{\gamma'}(f),$$
    the sum being over representatives for the $H(F)$-conjugacy classes in the stable conjugacy class of $\gamma$. The cardinality is $1$ exactly when $H_{\gamma'}$ is connected, which is automatic if $H^{der}$ is simply connected and, more to the point here, for $(G,H)$-regular elements; so the formula used above is the special case of this one. That $\SO_{\gamma}$ is a stable distribution is Conjecture 5.3 of \cite{MR858284}.
\end{Rc}

\begin{Rc}[Transfer, transfer factors and the fundamental lemma]\label{transferFL}
    The function $h$ is produced by transferring functions from $G$ to $H$. We collect the statements that give the terms in \eqref{endo} their meaning, following $\S\S$ 5--7 of \cite{MR858284}. Nothing is proved here. The following notation is used here and in the complements that follow.
    \begin{itemize}
        \item $(H,s,\eta)$ is an endoscopic triple for $G$ (\ref{endotriples}), and $\eta'$ denotes a fixed extension of $\eta : \Hat{H} \xrightarrow[]{} \Hat{G}$ to an $L$-homomorphism ${}^L H \xrightarrow[]{} {}^L G$.
        \item $\Delta$ denotes a transfer factor. Over a local field it is written $\Delta$, and in the global setting $\Delta_v$ denotes the factor at the place $v$ and $\Delta = \prod_v \Delta_v$ the adelic one.
        \item $\Or^{\kappa}_{\gamma}$ denotes the $\kappa$-orbital integral, defined below, which is the innermost sum of the prestabilization of Section \ref{prestab}.
    \end{itemize}

    \textbf{Transfer.} Let $F$ be a local field and let $G$ be a connected reductive group over $F$ with $G^{der}$ simply connected. An extension $\eta'$ as above exists because $G^{der}$ is simply connected \cite{MR540901}. Here, as in \cite{MR1044820}, the $L$-groups are formed with the Weil group $W_F$ rather than with $\Gamma_F$ as in Definition \ref{Lgroupdef}; this is what the existence statement requires. There is then a correspondence $f \mapsto f^H$ between functions $f \in \mathcal{C}_c^{\infty}(G(F))$ and $f^H \in \mathcal{C}_c^{\infty}(H(F))$ such that
    $$\SO_{\gamma_H}(f^H) = \sum_{\gamma} \Delta(\gamma_H, \gamma) \, e(G_{\gamma}) \, \Or_{\gamma}(f)$$
    for every $(G,H)$-regular semisimple $\gamma_H \in H(F)$, the sum being over the $G(F)$-conjugacy classes of semisimple $\gamma \in G(F)$ related to $\gamma_H$ in the sense of \ref{relatedelementsdef}. If no element of $G(F)$ is related to $\gamma_H$ the sum is empty and the right-hand side is $0$.
    \begin{itemize}
        \item \textbf{Measures for matching orbital integrals.} In an identity of this kind the left-hand side is formed with a Haar measure on $H_{\gamma_H}^0(F)$ and the right-hand side with one on $G_{\gamma}(F)$. These two groups are inner forms of one another precisely because $\gamma_H$ is $(G,H)$-regular (\ref{GHregular}), and the two measures are always taken compatible in the sense of $\S$ 1 of \cite{MR942522}. Otherwise the two sides could be scaled independently and the identity would say nothing.
    \end{itemize}

    The force of this statement is easily mistaken. The quantifiers run: $\Delta$ is fixed first, independently of any test function, and $f^H$ is then required to work at every $\gamma_H$ at once. So what is asserted is that the function of $\gamma_H$ given by the right-hand side, which $f$ and $\Delta$ determine completely, lies in the image of $f^H \mapsto (\gamma_H \mapsto \SO_{\gamma_H}(f^H))$.

    \textbf{Transfer factors.} The complex numbers $\Delta(\gamma_H, \gamma)$ are the \textit{transfer factors}. Their first property involves the following invariant. For stably conjugate $\gamma, \gamma' \in G(F)$, choose $g \in G(\overline{F})$ with $g \gamma g^{-1} = \gamma'$; then $\tau \mapsto g^{-1}\tau(g)$ is a $1$-cocycle of $\Gamma_F$ with values in $I = G_{\gamma}$, and its class $\inv(\gamma,\gamma') \in H^1(F,I)$ lies in $\ker[H^1(F,I) \to H^1(F,G)]$, hence by \ref{conjclassinsidestable} is a character of $\pi_0(Z(\Hat{I})^{\Gamma_F})$. This is the construction of the local invariants of Section \ref{lvformula}: for a finite place $l \neq p$ one has $\alpha_l = \inv(\gamma_0, \gamma_l)$. Two properties of $\Delta$ are used repeatedly.
    \begin{itemize}
        \item[(D1)] For $\gamma'$ stably conjugate to $\gamma$, $\Delta(\gamma_H, \gamma') = \Delta(\gamma_H, \gamma) \langle \inv(\gamma, \gamma'), \kappa \rangle$, where $\kappa$ is the image of $s$ under $Z(\Hat{H}) \hookrightarrow Z(\Hat{I_H}) \xrightarrow[]{} Z(\Hat{I})$ and $I_H = H_{\gamma_H}^0$.
        \item[(D2)] $\Delta(\cdot,\cdot)$ is not canonically attached to $(H,s,\eta')$: its construction in \cite{MR909227} depends on further auxiliary data, and altering them multiplies $\Delta$ by a scalar $c \in \mathbb{C}^{\times}$. To \textit{normalize} the transfer factors means to fix such data. Replacing $\Delta$ by $c\Delta$ replaces $f^H$ by $cf^H$, so the statement above is insensitive to the choice.
    \end{itemize}
    Property (D2) is harmless locally, since each local statement is insensitive to the scalar, but not globally: an adelic product $\prod_v \Delta_v$ is altered by rescaling a single factor, and nothing so far forces almost all factors to equal $1$.

    Using (D1), the transfer identity may be rewritten in terms of a single $\gamma$ related to $\gamma_H$:
    $$\SO_{\gamma_H}(f^H) = \Delta(\gamma_H, \gamma) \, \Or^{\kappa}_{\gamma}(f), \qquad \Or^{\kappa}_{\gamma}(f) = \sum_{\gamma'} \langle \inv(\gamma, \gamma'), \kappa \rangle \, e(G_{\gamma'}) \, \Or_{\gamma'}(f),$$
    the right-hand side being the \textit{$\kappa$-orbital integral} referred to in Section \ref{prestab}. Both the correspondence and the transfer factors depend on the choice of $\eta'$, and not only on $(H,s,\eta)$.

    \textbf{The fundamental lemma for the unit element.} Let $F$ be $p$-adic, let $G$ and $H$ be unramified, and assume that $\eta'$ is unramified, that is, trivial on the inertia subgroup; such a datum is called an unramified endoscopic datum (Complement \ref{satakeFL}). Let $K \subset G(F)$ and $K_H \subset H(F)$ be hyperspecial maximal compact subgroups. For the unramified normalization of the transfer factors, $\mathbf{1}_{K_H}$ is a transfer of $\mathbf{1}_K$:
    $$\SO_{\gamma_H}(\mathbf{1}_{K_H}) = \Delta(\gamma_H, \gamma) \, \Or^{\kappa}_{\gamma}(\mathbf{1}_K).$$
    Here the \textit{unramified normalization} means that the auxiliary data of (D2) are chosen compatibly with the integral structures on $G$ and $H$ underlying $K$ and $K_H$. Note that both functions are now prescribed, so that unlike the statement above there is no existential quantifier left: this is a bare identity between two explicitly given functions of $\gamma_H$. This is the fundamental lemma; see \cite{MR3051198} for its statement and history and \cite{MR2653248} for the proof.
\end{Rc}

\begin{Rc}[The Satake isomorphism and the unramified fundamental lemma]\label{satakeFL}
    At the prime $p$ the transfer of Complement \ref{transferFL} takes an explicit form: the transferred function is not merely asserted to exist, but is written down by means of the Satake isomorphism. We use the Hecke algebra conventions of Section \ref{detailssetup}.

    \textbf{The algebra $\mathcal{O}(\Hat{G}\sigma)^{\Hat{G}}$.} Let $\mathcal{G}$ be the reductive model over $\mathbb{Z}_p$ fixed in Section \ref{newtonstratdetails} and put $K_p=\mathcal{G}(\mathbb{Z}_p)$, a hyperspecial maximal compact subgroup of $G(\mathbb{Q}_p)$; normalize the Haar measure by $\vol(K_p)=1$. Since $p$ is good for $G$, the group $G_{\mathbb{Q}_p}$ is unramified, so the $L$-action of $\Gamma_{\mathbb{Q}_p}$ on $\Hat{G}$ factors through $\Gal(E/\mathbb{Q}_p)$ for some finite unramified extension $E/\mathbb{Q}_p$. Because $E/\mathbb{Q}_p$ is unramified, reduction identifies $\Gal(E/\mathbb{Q}_p)$ with the Galois group of the residue field extension, so it is cyclic with a canonical generator $\sigma$, the arithmetic Frobenius, characterized by inducing $x\mapsto x^p$ on the residue field of $E$; this is the same convention as for the Frobenius of $\mathbb{Q}_{p^j}$ over $\mathbb{Q}_p$ in Section \ref{mainproof}. The $L$-group ${}^L G = \Hat{G}\rtimes\Gamma_{\mathbb{Q}_p}$ of Definition \ref{Lgroupdef} therefore surjects onto the complex linear algebraic group
    $${}^L G_E = \Hat{G}\rtimes\Gal(E/\mathbb{Q}_p),$$
    whose identity component is $\Hat{G}$ and whose group of connected components is $\Gal(E/\mathbb{Q}_p)$. Let
    $$\Hat{G}\sigma \subset {}^L G_E$$
    denote the connected component lying over $\sigma$; passing to the finite form is what makes this well defined, a Frobenius element of $\Gamma_{\mathbb{Q}_p}$ or of $W_{\mathbb{Q}_p}$ being determined only modulo inertia. It is not a subgroup. The map $g\mapsto g\sigma$ is an isomorphism of varieties $\Hat{G}\xrightarrow{\sim}\Hat{G}\sigma$, so $\Hat{G}\sigma$ is a smooth affine variety and has a coordinate ring $\mathcal{O}(\Hat{G}\sigma)$. The group $\Hat{G}$ acts on $\Hat{G}\sigma$ by conjugation,
    $$x\cdot(g\sigma)=x(g\sigma)x^{-1}=\bigl(x\,g\,\sigma(x)^{-1}\bigr)\sigma, \qquad x,g\in\Hat{G},$$
    which under the above isomorphism is $\sigma$-twisted conjugation on $\Hat{G}$. We write $\mathcal{O}(\Hat{G}\sigma)^{\Hat{G}}$ for the algebra of regular functions on $\Hat{G}\sigma$ invariant under this action. Only the automorphism of $\Hat{G}$ induced by $\sigma$ enters, so the construction is independent of the choice of $E$.

    \textbf{The Satake isomorphism.} There is a canonical isomorphism of $\mathbb{C}$-algebras
    $$\mathcal{H}(G(\mathbb{Q}_p)//K_p)\xrightarrow{\sim}\mathcal{O}(\Hat{G}\sigma)^{\Hat{G}},$$
    for which we refer to $\S\S$ 6.6--7.1 of \cite{MR546608}. Under it the characters of the Hecke algebra correspond to the semisimple $\Hat{G}$-conjugacy classes in $\Hat{G}\sigma$.

    \textbf{Unramified endoscopic data.} An endoscopic datum $(H,s,\eta)$ for $G$ (\ref{endodef}, \ref{endotriples}) is \textit{unramified} if $H$ is unramified over $\mathbb{Q}_p$ and $\eta:\Hat{H}\xrightarrow[]{}\Hat{G}$ admits an $L$-extension $\eta'$ inflated from a homomorphism ${}^L H_E\xrightarrow[]{}{}^L G_E$ of the finite forms above, after enlarging $E$ to a common unramified splitting extension if necessary.

    Being inflated is a genuine condition on $\eta'$, and not a consequence of the $L$-actions being unramified. Let $I_p$ denote the inertia subgroup. An $L$-homomorphism lies over the group used to form the $L$-groups, so for $\tau\in I_p$ one has $\eta'(\tau)=z_{\tau}\tau$ for some $z_{\tau}\in\Hat{G}$; that $\eta'$ is a homomorphism compatible with the $L$-actions forces $z_{\tau}$ to centralize $\eta(\Hat{H})$, but does not force $z_{\tau}=1$. Thus the triviality of the action of $I_p$ on $\Hat{G}$ and on $\Hat{H}$ does not imply inflation. Inflation is the additional requirement that $z_{\tau}=1$ for every $\tau\in I_p$, that is, $\eta'(\tau)=\tau$; equivalently, that $\eta'$ factors through ${}^L H_E\xrightarrow[]{}{}^L G_E$. Such an extension exists locally at $p$ \cite[p.~180]{MR1044820}; fix one. Restricting to unramified data costs nothing: by Proposition 7.5 of \cite{MR858284}, if $H$ is not unramified then $\Or^{\kappa}_{\gamma}(f)=0$ for every $f\in\mathcal{H}(G(\mathbb{Q}_p)//K_p)$.

    \textbf{Transfer between unramified Hecke algebras.} Fix a hyperspecial maximal compact subgroup $K_H\subset H(\mathbb{Q}_p)$ and normalize $\vol(K_H)=1$. Being inflated from ${}^L H_E\xrightarrow[]{}{}^L G_E$, the map $\eta'$ carries $\Hat{H}\sigma$ into $\Hat{G}\sigma$, and $\eta'(hxh^{-1})=\eta(h)\eta'(x)\eta(h)^{-1}$ for $h\in\Hat{H}$. Pullback of functions along $\eta'$ therefore gives a homomorphism $\mathcal{O}(\Hat{G}\sigma)^{\Hat{G}}\xrightarrow[]{}\mathcal{O}(\Hat{H}\sigma)^{\Hat{H}}$, and through the two Satake isomorphisms a homomorphism
    $$b_{\eta'}:\mathcal{H}(G(\mathbb{Q}_p)//K_p)\xrightarrow[]{}\mathcal{H}(H(\mathbb{Q}_p)//K_H).$$
    Being a pullback of regular functions, $b_{\eta'}$ is a ring homomorphism and carries $\mathbf{1}_{K_p}$ to $\mathbf{1}_{K_H}$. It depends on the chosen $L$-extension $\eta'$, in keeping with Complement \ref{transferFL}.

    \textbf{The fundamental lemma in the unramified case.} Use the unramified normalization of the transfer factors from Complement \ref{transferFL} and compatible measures on matching centralizers, as in Section \ref{lvformula}. The assertion is that $b_{\eta'}(f)$ is a transfer of $f$ for \emph{every} element of the Hecke algebra, and not only for the unit element: for every $(G,H)$-regular semisimple $\gamma_H\in H(\mathbb{Q}_p)$ (\ref{GH-regular}),
    $$\SO_{\gamma_H}(b_{\eta'}(f))=\Delta(\gamma_H,\gamma)\,\Or^{\kappa}_{\gamma}(f), \qquad f\in\mathcal{H}(G(\mathbb{Q}_p)//K_p),$$
    where $\gamma\in G(\mathbb{Q}_p)$ is related to $\gamma_H$ (\ref{relatedelementsdef}) and $\kappa$ is attached to $s$ as in Complement \ref{transferFL}. Taking $f=\mathbf{1}_{K_p}$ recovers the fundamental lemma for the unit element of Complement \ref{transferFL}, since $b_{\eta'}(\mathbf{1}_{K_p})=\mathbf{1}_{K_H}$; conversely, the statement for the whole Hecke algebra follows from that special case, by \cite{MR1350645}.
\end{Rc}

\begin{Rc}[Global transfer factors]\label{globaltransfer}
    The transfer statements of Complement \ref{transferFL} are local. Section \ref{stab} requires a global input, and we record it here. Let $F$ be a number field and assume that the local transfer statement of Complement \ref{transferFL} holds at every place $v$ of $F$, with local transfer factors $\Delta_v$. Let $\gamma_H \in H(F)$ be $(G,H)$-regular semisimple, and say that $\gamma \in G(\mathbb{A})$ \textit{comes from} $\gamma_H$ if for every place $v$ the component $\gamma_v \in G(F_v)$ is related to $\gamma_H$ over $F_v$ in the sense of \ref{relatedelementsdef}. Fix an inner twisting $\psi : G_0 \xrightarrow[]{} G$ with $G_0$ quasi-split. Since $G_0$ is quasi-split and $G^{der}$ is simply connected, the $\Gamma_F$-stable semisimple conjugacy class in $G_0(\overline{F})$ determined by $\gamma_H$ contains an element $\gamma_0 \in G_0(F)$ (\ref{conjugacyclassdefoverFlemma}), so that $\gamma_H$ and $\gamma_0$ are related in the sense of \ref{relatedelementsdef}. Let $I_0$ be the centralizer of $\gamma_0$ in $G_0$ and let $\kappa \in \mathfrak{K}(I_0/F)$ be the image of $s$ under $Z(\Hat{H}) \hookrightarrow Z(\Hat{I_H}) \xrightarrow[]{} Z(\Hat{I_0})$. A $\gamma$ coming from $\gamma_H$ is then conjugate to $\psi(\gamma_0)$ under $G(\overline{\mathbb{A}})$, and there is a class
    $$\obs(\gamma) \in \mathfrak{K}(I_0/F)^D,$$
    constructed in 6.5 of \cite{MR858284} out of the local relative positions of $\gamma_v$ and $\psi(\gamma_0)$. It satisfies $\obs(\gamma') = \obs(\gamma) \cdot \inv(\gamma,\gamma')$, and it is trivial precisely when the $G(\mathbb{A})$-conjugacy class of $\gamma$ contains an element of $G(F)$.

    With this notation, the local transfer factors admit a normalization for which
    \begin{itemize}
        \item[(a)] all but finitely many of the $\Delta_v(\gamma_H, \gamma)$ equal $1$, so that $\Delta(\gamma_H,\gamma) := \prod_v \Delta_v(\gamma_H,\gamma)$ is defined, and
        \item[(b)] $\Delta(\gamma_H, \gamma) = \langle \obs(\gamma), \kappa \rangle$.
    \end{itemize}
    The scalar ambiguity of (D2) is thereby fixed globally, once and for all, so that the product in (a) is defined and the identity in (b) holds. The statement is used in the form (b) in $\S$ 7 of \cite{MR1044820}.

    The two statements above are 5.5 and 6.10 of \cite{MR858284}, where they are conjectural; both are now known. Transfer factors were constructed by Langlands and Shelstad \cite{MR909227}, and the fundamental lemma is a theorem of Ng\^o \cite{MR2653248}, from which the transfer statement follows; see \cite{MR3051198}.
\end{Rc}

\begin{Rc}[The three local components of $h$]\label{hconstruction}
    We record the identities satisfied by the three factors of $h = h^p h_p h_{\infty}$, following $\S$ 7 of \cite{MR1044820}. Throughout we use the adelic normalization $\Delta=\prod_v\Delta_v$ of Complement \ref{globaltransfer}, choosing its local representatives so that the usual untwisted factor $\Delta_p$ has the unramified normalization of Complement \ref{satakeFL}. Changing this local normalization by a scalar changes $h_p$ by the same scalar. Only the stable orbital integrals of the three factors are used and are canonical; for the constructions we refer to \textit{loc.\ cit.} Throughout, $\mathcal{H} = (H,s,\eta)$ is an endoscopic triple and $\eta':{}^L H \xrightarrow[]{} {}^L G$ is a fixed $L$-extension of $\eta$. Recall that $h$ is defined to be $0$ unless $H$ is unramified at $p$ and the elliptic maximal tori of $G_{\mathbb{R}}$ come from $H_{\mathbb{R}}$.

    \textbf{The factor $h^p$.} Write $\Delta^p = \prod_{v \neq p, \infty} \Delta_v$ for the product of the local transfer factors away from $p$ and $\infty$, in the normalization of Complement \ref{globaltransfer}. There is a function $h^p \in \mathcal{C}_c^{\infty}(H(\mathbb{A}_f^p))$ such that
    $$\SO_{\gamma_H}(h^p) = \sum_{\gamma} \Delta^p(\gamma_H, \gamma) \cdot e^p(\gamma) \cdot \Or_{\gamma}(f^p) \eqno{(7.1)}$$
    for every $(G,H)$-regular semisimple $\gamma_H \in H(\mathbb{A}_f^p)$, where this means componentwise regularity and regular reduction at almost all places. The sum is over the $G(\mathbb{A}_f^p)$-conjugacy classes of semisimple $\gamma \in G(\mathbb{A}_f^p)$ coming from $\gamma_H$, and $e^p(\gamma) = \prod_{v \neq p, \infty} e(G_{\gamma_v})$ is the product of the signs of Remark \ref{sign} attached to the centralizers of the components of $\gamma$. The global normalization of Complement \ref{globaltransfer} makes $\Delta^p$ well defined, while local transfer and the fundamental lemma for the unit element at almost all places give $h^p$; the sum has finite support by Proposition 7.1 of \cite{MR858284}.

    \textbf{The factor $h_p$.} Put $F=\mathbb{Q}_{p^j}$ and $K_F=\mathcal{G}(\mathbb{Z}_{p^j})$, and let $R = \Res_{F/\mathbb{Q}_p}(G)$ and $\theta = \sigma$ be as in \ref{twistedconjsection}. For a $(G,H)$-regular semisimple $\gamma_H \in H(\mathbb{Q}_p)$, choose a semisimple $\gamma_0 \in G(\mathbb{Q}_p)$ coming from $\gamma_H$, which is possible because $G_{\mathbb{Q}_p}$ is quasi-split, and put $I_0=G_{\gamma_0}$. The sum below is over representatives for the $\sigma$-conjugacy classes of $\delta \in G(F)$ for which the norm $N\delta$ (\ref{normmap}) is conjugate to $\gamma_0$ under $G(\overline{\mathbb{Q}}_p)$. For such a $\delta$, write $I_{\delta}^{\sigma}=\{x\in R \mid x\delta\theta(x^{-1})=\delta\}$ for its $\sigma$-centralizer (the group denoted $I_{s\delta}$ in \ref{sigmacentralizerdef}); this is the group $I(p)$ of Section \ref{lvformula}.
    The twisted orbital integral $\TO_{\delta}(\phi_j)$ vanishes unless the image of the $\sigma$-conjugacy class of $\delta$ under $B(G_{\mathbb{Q}_p}) \xrightarrow[]{} X^*(Z(\Hat{G})^{\Gamma_{\mathbb{Q}_p}})$ equals the restriction of $-\mu_h^{\natural}$, in the notation of Remark \ref{BGmudef}. For such $\delta$, the character $\alpha(\gamma_0;\delta)$ extends to a character $\beta(\gamma_0;\delta)$ of $Z(\Hat{I_0})^{\Gamma_{\mathbb{Q}_p}}Z(\Hat{G})$ by declaring it to be $-\mu_h^{\natural}$ on $Z(\Hat{G})$, as in Section \ref{lvformula}. The transfer theorem used in $\S$ 7 of \cite{MR1044820}, which we take as a black box, gives a function $h_p \in \mathcal{C}_c^{\infty}(H(\mathbb{Q}_p))$ such that
    $$\SO_{\gamma_H}(h_p) = \sum_{\delta} \langle \beta(\gamma_0;\delta), s \rangle \cdot \Delta_p(\gamma_H, \gamma_0) \cdot e(I_{\delta}^{\sigma}) \cdot \TO_{\delta}(\phi_j). \eqno{(7.3)}$$
    Passing from $\alpha$ to $\beta$ makes the pairing meaningful for a general $s \in Z(\Hat{H})^{\Gamma_{\mathbb{Q}_p}}Z(\Hat{G})$.

    We now describe $h_p$. Write $s=s_0z$ with $s_0\in Z(\Hat{H})^{\Gamma_{\mathbb{Q}_p}}$ and $z\in Z(\Hat{G})$. From $\eta':{}^L H\xrightarrow[]{}{}^L G$ and the invariant representative $s_0$, Kottwitz constructs an \textit{allowed} $L$-embedding
    $$\widetilde{\eta}:{}^L H\xrightarrow[]{}{}^L R.$$
    Roughly, if $\Hat{R}\simeq\Hat{G}^j$ and $\Hat{\theta}$ cyclically permutes the factors, then $\widetilde{\eta}(\Hat{H})$ is the identity component of the $\Hat{\theta}$-twisted centralizer of a suitable semisimple $t\in\Hat{R}$. We refer to pages 179--180 of \cite{MR1044820} for the construction. Equation (7.3) shows that the passage from $s_0$ to $s=s_0z$ multiplies the required stable orbital integrals by $\mu_h^{\natural}(z)^{-1}$, so we scale $h_p$ by this factor \cite[p.~181]{MR1044820}.

    Assume next that $\eta'$ is unramified in the sense of Complement \ref{satakeFL}; then $\widetilde{\eta}$ is unramified and induces a map of spherical Hecke algebras. For an arbitrary $\eta'$, choose an unramified reference extension $\eta'_{\mathrm{ur}}$. Their difference in $H^1(W_{\mathbb{Q}_p},Z(\Hat{H}))$ determines a quasi-character $\chi$ of $H(\mathbb{Q}_p)$, and $h_p=\chi h_{p,\mathrm{ur}}$ \cite[p.~181]{MR1044820}. It therefore remains to describe $h_{p,\mathrm{ur}}$; below $\widetilde{\eta}$ denotes the allowed embedding constructed from $s_0$ and $\eta'_{\mathrm{ur}}$.

    Under $R(\mathbb{Q}_p)=G(F)$ put $K_R=K_F$. The Satake isomorphism of Complement \ref{satakeFL} gives
    $$\mathcal{H}(R(\mathbb{Q}_p)//K_R)=\mathcal{H}(G(F)//K_F)\xrightarrow{\sim}\mathcal{O}(\Hat{R}\sigma)^{\Hat{R}}\simeq\mathcal{O}(\Hat{G}\sigma^j)^{\Hat{G}}.$$
    Restriction of unramified parameters is the norm map
    $$N_j^{\vee}:(\Hat{G}\sigma)//\Hat{G}\xrightarrow[]{}(\Hat{G}\sigma^j)//\Hat{G}, \qquad [g\sigma]\longmapsto[(g\sigma)^j],$$
    whose pullback gives the ordinary base-change homomorphism
    $$b_{F/\mathbb{Q}_p}:\mathcal{H}(R(\mathbb{Q}_p)//K_R)\xrightarrow[]{}\mathcal{H}(G(\mathbb{Q}_p)//K_p).$$
    This is dual to the norm $N\delta=\delta\sigma(\delta)\cdots\sigma^{j-1}(\delta)$ of \ref{normmap}; equivalently, it comes from the diagonal $L$-homomorphism ${}^L G\xrightarrow[]{}{}^L R$ associated with the restriction of scalars in \ref{restrictionofscalars}.

    The allowed embedding induces
    $$\widetilde{\eta}_{\sigma}:(\Hat{H}\sigma)//\Hat{H}\xrightarrow[]{}(\Hat{R}\sigma)//\Hat{R},$$
    and pullback along this map, through the Satake isomorphisms, defines
    $$b_{\widetilde{\eta}}:\mathcal{H}(R(\mathbb{Q}_p)//K_R)\xrightarrow[]{}\mathcal{H}(H(\mathbb{Q}_p)//K_H).$$
    For the trivial endoscopic datum $H=G$, $s_0=1$, the allowed embedding $\widetilde{\eta}:{}^L G\xrightarrow[]{}{}^L R$ may be taken to be the diagonal one, and then $b_{\widetilde{\eta}}=b_{F/\mathbb{Q}_p}$. Thus $b_{\widetilde{\eta}}$ for general $H$ is a generalization of ordinary unramified base change, not a composite through $\mathcal{H}(G(\mathbb{Q}_p)//K_p)$. The spherical function attached to $s_0$ and $\eta'_{\mathrm{ur}}$ is
    $$h_{p,\mathrm{ur}}=b_{\widetilde{\eta}}(\phi_j).$$
    Thus $h_{p,\mathrm{ur}}$ is the unramified transfer of $\phi_j\in\mathcal{H}(R(\mathbb{Q}_p)//K_R)$ for the datum defined by $\widetilde{\eta}$. The fundamental lemma for this transfer is the black box giving (7.3); see \cite[\S\S 7.4.12--7.4.18]{kisin2021stabletraceformulashimura}. For the construction of $\phi_j$ and its twisted orbital integrals, see \cite{MR761308}. For the original $s=s_0z$ and $L$-extension $\eta'$ this gives
    $$h_p=\chi\,\mu_h^{\natural}(z)^{-1}b_{\widetilde{\eta}}(\phi_j).$$

    \textbf{The factor $h_{\infty}$.} The function $h_{\infty}$ is compactly supported modulo $A_G(\mathbb{R})^0$ and is built as a linear combination of pseudo-coefficients for the discrete series representations of $H(\mathbb{R})$ having the same infinitesimal and central characters as the contragredient of $\xi$; see p.~185 of \cite{MR1044820}. This is the only place in the argument where the local system is used, and for $\xi$ trivial it is the packet attached to the trivial representation. Its stable orbital integrals are prescribed as follows: $\SO_{\gamma_H}(h_{\infty}) = 0$ unless $\gamma_H$ is elliptic in $H(\mathbb{R})$, in which case
    $$\SO_{\gamma_H}(h_{\infty}) = \langle \beta(\gamma_0), s \rangle \cdot \Delta_{\infty}(\gamma_H, \gamma_0) \cdot e(I(\infty)) \cdot \vol^{-1} \cdot \tr\,\xi(\gamma_0). \eqno{(7.4)}$$
    Here $\gamma_0 \in T(\mathbb{R})$ comes from $\gamma_H$ for an elliptic maximal torus $T$ of $G_{\mathbb{R}}$, $I_0$ is its centralizer, and $\beta(\gamma_0) \in X^*(Z(\Hat{I_0})^{\Gamma_{\mathbb{R}}})$ is formed as $\beta_{\infty}$ was in Section \ref{lvformula}, using $\mu_h$ for an $h \in X_{\infty}$ factoring through $T$; the pairing with $s$ uses the canonical embedding $Z(\Hat{H}) \xrightarrow[]{} Z(\Hat{I_0})$. The product $\langle \beta(\gamma_0), s\rangle \Delta_{\infty}(\gamma_H,\gamma_0)$ does not depend on the choice of $\gamma_0$. The group $I(\infty)$ is the inner form of $I_0$ with $I(\infty)/Z(G)$ anisotropic over $\mathbb{R}$, as in Section \ref{lvformula}; by (\HypTwo{}) this is equivalent to $I(\infty)/A_G$ being anisotropic over $\mathbb{R}$ (Appendix \ref{hyp2appendix}), and $\vol$ abbreviates $\vol(A_G(\mathbb{R})^0 \backslash I(\infty)(\mathbb{R}))$, so that $\vol^{-1}=c_{\infty}$ of Section \ref{prestab}. Finally, the central characters of the representations occurring in $h_{\infty}$ all restrict to the same quasi-character of $A_G(\mathbb{R})^0$, determined by $\xi$; accordingly the stable trace formula for $H$ is taken with the corresponding central character datum, which is trivial exactly when $\xi$ is. See Appendix \ref{hyp2appendix}.
\end{Rc}

\begin{Rc}[From the local identities to the stabilized formula]\label{route72}
    We indicate how the three identities of Complement \ref{hconstruction} combine, following pages 188--189 of \cite{MR1044820}. The main point is the termwise identity \eqref{endo}: for a fixed $(\mathcal H,\gamma_H)\to(\gamma_0,\kappa)$, the stable orbital integral of $h$ is the corresponding $\kappa$-orbital sum over Kottwitz pairs $(\gamma,\delta)$. Once this is established, the passage to \eqref{final} is a matter of reindexing the sum and accounting for the fibres of $(\mathcal H,\gamma_H)\mapsto(\gamma_0,\kappa)$.

    \textbf{The function $h$ and its stable orbital integrals.} Set $h = h^p h_p h_{\infty}$. With compatible choices, replacing $s$ by $sz$ for $z \in Z(\Hat{G})$ multiplies the stable orbital integrals of $h_p$ and $h_{\infty}$ by $\mu_h^{\natural}(z)^{-1}$ and $\mu_h^{\natural}(z)$ respectively, so their product is unchanged. We use throughout the adelic transfer factor attached to the fixed $L$-extension $\eta'$ of Complement \ref{hconstruction}; only the resulting stable orbital integrals of $h$ will be used.

    \textbf{Which $\gamma_H$ contribute.} Let $\gamma_H \in H(\mathbb{Q})$ be $(G,H)$-regular semisimple. Then $\SO_{\gamma_H}(h) = 0$ unless
    \begin{enumerate}
        \item[(i)] $\gamma_H$ is elliptic in $H(\mathbb{R})$, and
        \item[(ii)] $\gamma_H$ appears in $G(\mathbb{Q}_v)$ for every place $v$ of $\mathbb{Q}$.
    \end{enumerate}
    At $p$ this local matching is automatic because $G$ is quasi-split over $\mathbb{Q}_p$ and $G^{der}$ is simply connected \cite[Theorem 4.1]{MR683003}; at $\infty$ it follows from (i) and the condition on elliptic maximal tori imposed in the definition of $h$. At the remaining places it is enforced by the local transfer in (7.1).

    \textbf{Descent to $\mathbb{Q}$.} Conditions (i) and (ii) imply that $\gamma_H$ has a matching element $\gamma_0\in G(\mathbb{Q})$. This is proved in \cite[p.~188]{MR1044820}: Theorem 4.1 of \cite{MR683003} supplies a rational representative in the quasi-split inner form, the obstruction of \cite[\S6.5]{MR858284} can be killed by changing one finite local component within its stable class, and \cite[Theorem 6.6]{MR858284} then gives the required rational element of $G$. We use only this conclusion.

    \textbf{Assembling the three identities.} Fix such a $\gamma_0$; it is elliptic in $G(\mathbb{R})$. Since $\obs(\gamma_0)=1$, Complement \ref{globaltransfer} gives $\Delta(\gamma_H,\gamma_0)=1$. On comparing the local representatives in (7.1), (7.3) and (7.4) with $\gamma_0$, the variation of the local transfer factors, together with the characters $\beta(\gamma_0;\delta)$ and $\beta(\gamma_0)$, gives $\langle \alpha(\gamma_0,\gamma,\delta),s\rangle$. The image of $s$ in $\mathfrak{K}(I_0)$ is the element $\kappa$ attached to $(\mathcal H,\gamma_H)$, so this pairing equals $\langle \alpha(\gamma_0,\gamma,\delta),\kappa\rangle$. Moreover,
    $$e^p(\gamma)e(I_\delta^{\sigma})e(I(\infty))=e(\gamma,\delta), \qquad \vol^{-1}=c_\infty.$$
    Multiplying the three identities therefore gives
    $$\SO_{\gamma_H}(h) = \tr\,\xi(\gamma_0) \sum_{(\gamma, \delta)} \langle \alpha(\gamma_0, \gamma, \delta), \kappa \rangle \, e(\gamma,\delta) \, \Or_{\gamma}(f^p) \, \TO_{\delta}(\phi_j) \, c_{\infty},$$
    which is \eqref{endo}.

    \textbf{Comparison with the prestabilization.} Remark \ref{kappaisreached} shows that every pair $(\gamma_0,\kappa)$ in the prestabilized expression arises from some $(\mathcal H,\gamma_H)$. By Lemma 9.7 of \cite{MR858284}, the fibres have cardinality
    $$\mid\Aut(\mathcal H)/H_{ad}(\mathbb{Q})\mid,$$
    and
    $$\iota(G,H)\tau(H)=\tau(G)\mid\Aut(\mathcal H)/H_{ad}(\mathbb{Q})\mid^{-1}.$$
    Thus the automorphism factor corrects the fiber multiplicity and the remaining factor is the $\tau(G)$ of the prestabilization. The required twisted sum at $p$ vanishes when $H$ is ramified, by \cite[Proposition 7.4.16]{kisin2021stabletraceformulashimura}. Finally, Remark \ref{SOvanishing} allows the sum to be written over all $\mathbb{Q}$-elliptic stable classes in $H(\mathbb{Q})$: only the $(G,H)$-regular classes survive, for which $H_{\gamma_H}$ is connected. Retaining the component-group factor in the stable trace formula notation gives exactly \eqref{final}, hence Theorem 7.2 of \cite{MR1044820}.

    This separation is what will be used in Section \ref{completionmainproof}. There the Newton condition is imposed on the $\delta$-terms in the termwise identity \eqref{endo} by modifying the factor $h_p$. The descent and fibre-counting arguments used to pass from \eqref{endo} to \eqref{final} are unaffected, so the same comparison gives the stabilized formula for a fixed Newton stratum.
\end{Rc}

\subsection{Completion of the proof}\label{completionmainproof}

\textbf{The Newton restriction on the endoscopic side.}
We now combine the prestabilization for the Newton stratum in \eqref{prestabnewton} with the stabilization of Section \ref{stab}. Fix an elliptic endoscopic triple $\mathcal H=(H,s,\eta)$ and a $(G,H)$-regular semisimple $\gamma_H\in H(\mathbb{Q})$ with $(\mathcal H,\gamma_H)\to(\gamma_0,\kappa)$. For a pair $(\gamma,\delta)$ occurring on the right-hand side of \eqref{endo}, the stable norm $\mathfrak N\delta$ is the stable conjugacy class of $\gamma_0$. Here $N\delta$ is conjugate to $\gamma_0$ under $G(\overline{\mathbb{Q}}_p)$ by condition (2) of the Kottwitz triple, hence is semisimple, so Lemma \ref{normisolem2} applies. That lemma, with $[\mathbb{Q}_{p^j}:\mathbb{Q}_p]=j$, gives the equality $\nu_G(\mathfrak N \delta) = j\,\nu_G(\delta)$ of slope morphisms, and passing to Newton points
$$
    \overline\nu_G(\gamma_0)=j\,\overline\nu_G(\delta).
$$
The Kottwitz point of every $\delta$ for which $\TO_\delta(\phi_j)$ is non-zero is the restriction of $-\mu_h^{\natural}$, by Complement \ref{hconstruction}; and $\kappa_G$ is constant on $B(G_{\mathbb{Q}_p},\mu_h^{-1})$ with that same value, by Complement \ref{BGmudef}. So the Kottwitz points of $\delta$ and of $b$ agree. That the two constants coincide, rather than differing by a sign, is exactly the covariant normalisation fixed in Convention \ref{newtonnormalisation}. Hence Lemma \ref{Newonkottwitzcharisoc} shows that
\begin{equation}\label{newtonrestrictiondelta}
    \delta\mapsto b
    \quad\Longleftrightarrow\quad
    \overline\nu_G(\gamma_0)=j\,\overline\nu_G(b)
\end{equation}
for the terms occurring on the right-hand side of \eqref{endo}. Indeed, the forward implication follows from the displayed norm formula. Conversely, the equality on the right of \eqref{newtonrestrictiondelta} gives $\overline\nu_G(\delta)=\overline\nu_G(b)$, and the Kottwitz points are already equal, so the two classes in $B(G_{\mathbb{Q}_p})$ coincide.

Since $\gamma_H$ and $\gamma_0$ are related by an admissible embedding, the construction of the endoscopic Newton map in Lemma \ref{endoscopicnewtonmap}, together with Lemma \ref{isocrystalstableconj}, gives
$$
    f^{\mathcal H}\bigl(\overline\nu_H(\gamma_H)\bigr)
    =\overline\nu_G(\gamma_0).
$$
It follows from \eqref{newtonrestrictiondelta} that the restriction $\delta\mapsto b$ in \eqref{prestabnewton} is equivalent, after passage to the endoscopic side, to
$$
    \chi_{b,n}^{\mathcal H}(\gamma_H)=1.
$$
By Lemma B, the function $\chi_{b,n}^{\mathcal H}$ is smooth and constant on stable conjugacy classes. Consequently
$$
    \SO_{\gamma_H}(h_p\chi_{b,n}^{\mathcal H})
    =\chi_{b,n}^{\mathcal H}(\gamma_H)\SO_{\gamma_H}(h_p),
$$
so replacing $h_p$ by $h_p\chi_{b,n}^{\mathcal H}$ in \eqref{endo} has exactly the effect of retaining the terms with $\delta\mapsto b$. The product $h_p\chi_{b,n}^{\mathcal H}$ is compactly supported because $h_p$ is compactly supported.

\textbf{Which endoscopic triples contribute.}
It remains to determine which endoscopic groups can occur. Suppose that the term attached to $\mathcal H$ is non-zero. By Remark \ref{SOvanishing} and the preceding discussion, there are a $(G,H)$-regular $\gamma_H$, a corresponding $\gamma_0$, and a pair $(\gamma,\delta)$ occurring on the right-hand side of \eqref{endo} with $\delta\mapsto b$. Put
$$
    I_H=H_{\gamma_H}^0,
    \qquad
    I_0=G_{\gamma_0}^0,
$$
over $\mathbb{Q}_p$. Construction \ref{GHregular} identifies $I_H$ and $I_0$ as inner forms.

\begin{const}[The endoscopic isocrystal]\label{endoscopicisocrystal}
Let $(\mathcal H,\gamma_H)\to(\gamma_0,\kappa)$ and let $(\gamma,\delta)$ be a pair occurring in \eqref{endo} with $\delta\mapsto b$. We construct from these a class $b_H\in B(H_{\mathbb{Q}_p})$.

\textit{A class $b_{I_0}\in B(I_0)$ with image $b$ in $B(G)$.} By condition (2) of the Kottwitz triple, $N\delta$ is conjugate to $\gamma_0$ under $G(\overline{\mathbb{Q}}_p)$. Since $H^1(L, I_0)$ is trivial by a theorem of Steinberg, this upgrades to conjugacy under $G(L)$, so we may choose $c \in G(L)$ with
$$
    c\gamma_0c^{-1}=N\delta,
$$
as was done in Complement \ref{kottwitztripledetails}. Put
$$
    \delta_0=c^{-1}\delta\sigma(c)\in I_0(L).
$$
Let $b_{I_0}\in B(I_0)$ be its class. Its image in $B(G)$ is $[\delta]=b$.

\textit{The class $b_{I_0}$ is basic.} We must show that the centralizer of its slope morphism is all of $(I_0)_L$, so that the slope morphism is central. To see this, let $J_{\delta_0}$ be the group attached to the $I_0$-isocrystal $\delta_0$ in Construction \ref{J}. Conjugation by $c$ identifies $J_{\delta_0}$ with $I(p)$. Indeed, the relation $c\delta_0=\delta\sigma(c)$ gives
$$
    c\bigl(\delta_0\sigma(x)\delta_0^{-1}\bigr)c^{-1}
    =\delta\,\sigma(cxc^{-1})\,\delta^{-1},
    \qquad x\in I_0(L),
$$
so $\Int(c)$ intertwines the twisted descent data defining $J_{\delta_0}$ and the $\sigma$-centralizer $I(p)=I_{s\delta}$ of Definition \ref{sigmacentralizerdef}. By Lemma \ref{twistedlemma} and the equality $c\gamma_0c^{-1}=N\delta$, the latter is an inner form of $I_0$. Hence $J_{\delta_0}$ is an inner form of $I_0$.

Applying part (1) of Lemma \ref{leviofcontracted} to $\delta_0\in I_0(L)$ gives
$$
    (J_{\delta_0})_L\simeq Z_{(I_0)_L}\bigl(\nu_{I_0}(\delta_0)\bigr).
$$
Since $J_{\delta_0}$ is an inner form of $I_0$, it has the same dimension; the centralizer on the right is therefore a closed subgroup of the connected group $(I_0)_L$ of dimension $\dim I_0$, and so equals $(I_0)_L$. Thus $\nu_{I_0}(\delta_0)$ commutes with all of $(I_0)_L$, that is, it factors through $Z(I_0)_L$. By \ref{basicisocrystaldef} this says precisely that $b_{I_0}$ is basic.

\textit{The corresponding basic class $b_{I_H}\in B(I_H)$, and its image $b_H\in B(H)$.} The inner forms $I_H$ and $I_0$ have canonically identified centres of their dual groups, $Z(\Hat{I_H})\simeq Z(\Hat{I_0})$, and hence canonically identified targets $X^*(Z(\Hat{I_H})^{\Gamma_{\mathbb{Q}_p}})$ and $X^*(Z(\Hat{I_0})^{\Gamma_{\mathbb{Q}_p}})$ for the Kottwitz maps. By Corollary \ref{classificationofbasic}, there is a basic class $b_{I_H}\in B(I_H)$ whose Kottwitz point corresponds to that of $b_{I_0}$. Since the two groups have the same quasi-split inner form, their Newton maps have the same target, as in the discussion of Lemma \ref{isocrystalsandinnerforms}; the compatibility of the Newton and Kottwitz maps in Lemma \ref{Newonkottwitzcharisoc} shows that their Newton points correspond as well. Let $b_H\in B(H)$ be the image of $b_{I_H}$ under $B(I_H)\to B(H)$.

This produces the class $b_H\in B(H_{\mathbb{Q}_p})$ attached to $(\mathcal H,\gamma_H)$ and $b$.
\end{const}

We claim that the squares
\begin{center}
    \begin{tikzcd}
        \mathcal{N}(I_H) \arrow{r} \arrow{d}{\wr} & \mathcal{N}(H) \arrow{d}{f^{\mathcal H}} & & \pi_1(I_H) \arrow{r} \arrow{d}{\wr} & \pi_1(H) \arrow{d}{\psi^{\mathcal H}} \\
        \mathcal{N}(I_0) \arrow{r} & \mathcal{N}(G) & & \pi_1(I_0) \arrow{r} & \pi_1(G)
    \end{tikzcd}
\end{center}
commute, where the horizontal maps are induced by the inclusions $I_H \subset H$ and $I_0 \subset G$, and the left vertical maps are the identifications attached to the inner twisting of Construction \ref{GHregular}. Indeed, let $T_H$ be a maximal torus of $I_H$ over $F$ and let $T = \iota(T_H) \subset I_0$ be its image under the admissible embedding $\iota$ of \ref{endoscopicnewtonmap}, so that $T_H$ is also a maximal torus of $H$ and $T$ one of $G$. By \ref{GHregular} the inner twisting $I_{H,\overline{F}} \xrightarrow[]{\sim} I_{0,\overline{F}}$ extends $\iota$, so all four maps in each square are induced by the single homomorphism $\iota_*: X_*(T_H) \xrightarrow[]{} X_*(T)$: the horizontal ones by enlarging the Weyl group from $\Omega_{I_H}$ to $\Omega_H$, respectively from $\Omega_{I_0}$ to $\Omega_{G_0}$, and the vertical ones by \ref{endoscopicnewtonmap} and \ref{endoscopicpimap}. Commutativity is therefore the compatibility of $\iota_*$ with these inclusions of Weyl groups, which is the content of the proof of \ref{endoscopicnewtonmap}.

Applying the two squares to $b_{I_H}$, whose image in $\mathcal{N}(I_0)$ and $\pi_1(I_0)$ is that of $b_{I_0}$ by Construction \ref{endoscopicisocrystal}, now gives
$$
    f^{\mathcal H}\bigl(\overline\nu_H(b_H)\bigr)=\overline\nu_G(b),
    \qquad
    \psi^{\mathcal H}\bigl(\kappa_H(b_H)\bigr)=\kappa_G(b).
$$
Thus $\mathcal H\in\Eell(G)_b$ by the definition in Section \ref{details}. Note that $\gamma_H$ itself has transferred Newton point $j\,\overline\nu_G(b)$, whereas the auxiliary isocrystal $b_H$ witnessing membership in $\Eell(G)_b$ has transferred Newton point $\overline\nu_G(b)$.

We have proved that endoscopic triples outside $\Eell(G)_b$ make no contribution, and that for the remaining triples the Newton restriction is imposed by replacing $h_p$ with $h_p\chi_{b,n}^{\mathcal H}$. Applying the stabilization of Section \ref{stab} to \eqref{prestabnewton} therefore gives
$$
    c_b(n,f^p,\xi)
    =\sum_{\mathcal H\in\Eell(G)_b}\iota(G,H)\,
    \ST_e^H\bigl(h^p\cdot h_p\chi_{b,n}^{\mathcal H}\cdot h_\infty\bigr),
$$
which proves Theorem \ref{maintheorem}.

\appendix
\section{On Hypothesis 2}\label{hyp2appendix}

Hypothesis (\HypTwo{}) asserts that the maximal $\mathbb{Q}$-split and the maximal $\mathbb{R}$-split tori in the center of $G$ coincide. Its role in the proof is confined to the volume normalizations in Sections \ref{lvformula}--\ref{stab}. It makes the two quotients defining
$$
    c_1(\gamma_0,\gamma,\delta)
    =\vol\bigl(I(\mathbb{Q})\backslash I(\mathbb{A}_f)\bigr),
    \qquad
    c_{\infty}
    =\vol\bigl(A_G(\mathbb{R})^0\backslash I(\infty)(\mathbb{R})\bigr)^{-1}
$$
well defined and of finite non-zero volume. Complement \ref{prestabkey} uses the resulting identity $c_1=\tau(I)c_{\infty}$, and Complement \ref{hconstruction} uses the same $c_{\infty}$ in the archimedean factor $h_{\infty}$. Behind these volume terms is a general criterion for compactness of adelic quotients, which we now recall; it is also the criterion used in Section \ref{intro2} to discuss the trace formula when $G/A_G$ is anisotropic.

Let $\mathbb{A}_F^1\subset\mathbb{A}_F^{\times}$ be the norm-one idele group
$$
    \mathbb{A}_F^1
    =\left\{(x_v)_v\in\mathbb{A}_F^{\times}\ \middle|\
    \prod_v\lvert x_v\rvert_v=1\right\},
$$
where $v$ runs over all places of $F$. Using the convention $\lvert z\rvert=z\overline z$ at a complex place, embed $\mathbb{R}_{>0}$ diagonally in $\prod_{v\mid\infty}F_v^{\times}$ by $t\mapsto t$ at a real place and $t\mapsto\sqrt t$ at a complex place. Multiplication then gives a decomposition $\mathbb{A}_F^{\times}=\mathbb{A}_F^1\times\mathbb{R}_{>0}$. The quotient $F^{\times}\backslash\mathbb{A}_F^1$ is compact, as follows from the next two facts:
\begin{enumerate}
\item There is a map $F^{\times}\backslash\mathbb{A}_F^1 \xrightarrow[]{} \Cl_F$ sending $\prod_{v_{\mathfrak{p}} \nmid \infty} (x_{v_{\mathfrak{p}}})_{v_{\mathfrak{p}}} \times \prod_{v \mid \infty}(x_v)_v \mapsto \prod_{\mathfrak{p}} \mathfrak{p}^{v_{\mathfrak{p}}(x_{v_{\mathfrak{p}}})}$, where $\Cl_F$ denotes the ideal class group of $F$. The group $\Cl_F$ is finite.
\item The fibres of the above map are compact. Indeed, fix an ideal class $[\mathfrak a]\in \Cl_F$ and choose an idele $y=(y_v)_v\in \mathbb A_F^1$ whose associated fractional ideal is in the class $[\mathfrak a]$. Then the fibre over $[\mathfrak a]$ is represented by
$$
(\prod_{\mathfrak p} \mathcal O_{\mathfrak p}^{\times} \times \{(x_v)_{v\mid\infty}\in \prod_{v\mid\infty}F_v^\times: \prod_{v\mid\infty}\lvert x_v\rvert_v = c \}) / \mathcal O_F^\times $$
for some positive constant $c$ depending on $y$. The finite-adelic factor $\prod_{\mathfrak p}\mathcal O_{\mathfrak p}^{\times}$ is compact. On the archimedean factor the map $(x_v)\mapsto(\log\lvert x_v\rvert_v)$ has compact kernel and image an affine hyperplane in $\mathbb R^{r_1+r_2}$. The group $\mathcal O_F^\times$ acts through the logarithmic embedding $u\mapsto(\log\lvert u\rvert_v)_{v\mid\infty}$, whose image is a lattice in the hyperplane
$$
    \left\{(t_v)_{v\mid\infty}:\sum_{v\mid\infty}t_v=0\right\}
$$
by Dirichlet's unit theorem. Hence the quotient of the archimedean factor by $\mathcal O_F^\times$ is compact, and therefore every fibre is compact.
\end{enumerate}

For a connected reductive group $G$ over $F$, write $X_F^*(G)$ for the group of characters of $G$ defined over $F$ and define
$$
    H_G:G(\mathbb{A}_F)\longrightarrow\Hom(X_F^*(G),\mathbb{R}),
    \qquad
    H_G(g)(\chi)=\log\prod_v\lvert\chi(g_v)\rvert_v.
$$
The product formula gives $G(F)\subset\ker(H_G)$. The norm-one adelic subgroup is
$$
    G(\mathbb{A}_F)^1:=\ker(H_G)
    =\left\{g\in G(\mathbb{A}_F)\ \middle|\
    \prod_v\lvert\chi(g_v)\rvert_v=1
    \text{ for every }\chi\in X_F^*(G)\right\}.
$$

Let $A_G$ denote the maximal $F$-split torus in the center of $G$. The diagonal embedding of $\mathbb R_{>0}$ used above gives a subgroup
$$
    A_G^+\simeq X_*(A_G)\otimes_{\mathbb Z}\mathbb R_{>0}
    \subset A_G(F\otimes_{\mathbb Q}\mathbb R).
$$
The restriction of $H_G$ to $A_G^+$ is an isomorphism after taking logarithms, and multiplication gives the decomposition
$$
    G(\mathbb{A}_F)=G(\mathbb{A}_F)^1\times A_G^+.
$$

The decomposition extends a character of $A_G^+$ to $G(\mathbb{A}_F)$ by making it trivial on $G(\mathbb{A}_F)^1$. For a unitary character $\chi$ of $A_G^+$, the notation $L^2_{\chi}(G(F)\backslash G(\mathbb{A}_F))$ means the space of measurable functions satisfying
$$
    \varphi(ag)=\chi(a)\varphi(g),\qquad a\in A_G^+,
$$
which are square-integrable on $G(F)A_G^+\backslash G(\mathbb{A}_F)$. Restriction to $G(\mathbb{A}_F)^1$ identifies this space, as a $G(\mathbb{A}_F)^1$-representation, with $L^2(G(F)\backslash G(\mathbb{A}_F)^1)$. By \cite[\S 5.3, Theorem 5.6]{MR4615820}, the latter quotient is compact if and only if $G/A_G$ is anisotropic. This is the compactness criterion used in Section \ref{intro2}; in that case the fixed-central-character representation has discrete spectrum and the compact-quotient trace formula of \cite[\S\S 1--3]{MR2192011} applies.

We now apply the criterion to the group $I$ attached in Section \ref{lvformula} to a Kottwitz triple with trivial invariant. Recall that $I_{\mathbb{R}}\simeq I(\infty)$ and $I(\infty)/Z(G)$ is anisotropic over $\mathbb{R}$. By (\HypTwo{}), $A_G$ is the maximal $\mathbb{R}$-split torus in $Z(G)$. Every $\mathbb{R}$-split torus in $I(\infty)$ has trivial image in $I(\infty)/Z(G)$ and therefore lies in $Z(G)$. Hence
$$
    A_G(\mathbb{R})^0\backslash I(\infty)(\mathbb{R})
$$
is compact, which proves that $c_{\infty}$ is finite and non-zero.

Similarly, every $\mathbb{Q}$-split torus in $I$ becomes $\mathbb{R}$-split and hence lies in $Z(G)$, so it is contained in $A_G$. Since $A_G\subset Z(I)$, it follows that $A_G$ is the maximal $\mathbb{Q}$-split torus in the center of $I$ and that $I/A_G$ is anisotropic over $\mathbb{Q}$. Thus $I(\mathbb{Q})\backslash I(\mathbb{A})^1$ is compact. Moreover, under the same hypotheses $I(\mathbb{Q})$ is discrete in $I(\mathbb{A}_f)$, as observed on p.~172 of \cite{MR1044820}. With the compatible measures fixed in Sections \ref{lvformula} and \ref{prestab}, the Tamagawa volume decomposes as
$$
    \tau(I)
    =\vol\bigl(I(\mathbb{Q})\backslash I(\mathbb{A}_f)\bigr)
     \cdot\vol\bigl(A_G(\mathbb{R})^0\backslash I(\mathbb{R})\bigr).
$$
Consequently
$$
    c_1(\gamma_0,\gamma,\delta)=\tau(I)c_{\infty},
$$
so both volume terms used in Sections \ref{lvformula}--\ref{stab} are finite. Without (\HypTwo{}) these quotients need not be compact, and the volume terms are not available in this form. Kisin--Shin--Zhu \cite{kisin2021stabletraceformulashimura} remove (\HypOne{}) and (\HypTwo{}) by working with a more general stabilization.

\phantomsection
\addcontentsline{toc}{section}{References}
\bibliography{links}
\end{document}